\documentclass[reqno]{amsart}
\usepackage{amsfonts,amsmath,latexsym,verbatim,amscd,color,array}

\usepackage{amsmath,amssymb,amsthm,amsfonts,graphicx,color}
\usepackage{amssymb}
\usepackage{epstopdf}

\usepackage{color}

\usepackage[a4paper]{geometry}

\usepackage[colorlinks=true]{hyperref}
\hypersetup{urlcolor=blue, linkcolor=blue, citecolor=red, anchorcolor=blue]}

\usepackage[colorinlistoftodos]{todonotes}

\newtheorem{lemma}{Lemma}[section]

\newtheorem{theorem}{Theorem}[section]
\newtheorem{prop}{Proposition}[section]
\newtheorem{corollary}{Corollary}[section]
\newtheorem{remark}{Remark}[section]
\newcommand{\bremark}{\begin{remark} \em}
\newcommand{\eremark}{\end{remark} }

\newcommand{\inn}{{\quad\hbox{in } }}

\newcommand{\ass}{\quad \hbox{as } }

\newcommand{\nn}{ {\nabla} }

\newcommand{\pp}{ {\partial} }

\newcommand{\vp}{\varphi}

\newcommand{\R} {\mathbb R}

\newcommand{\cuad}{{\sqcap\kern-.68em\sqcup}}

\newcommand{\foral}{\quad\mbox{for all}\quad}

\newcommand{\be}{\begin{equation}}
\newcommand{\ee}{\end{equation}}

\newcommand{\la}{\lambda}

\newcommand{\equ}[1]{(\ref{#1})}

\newcommand{\norm}{\|}

\numberwithin{equation}{section}
\definecolor{orange}{rgb}{0.9,0.5,0}

\newcommand{\ch}[1]{#1}

\long\def\elim#1{}
\long\def\hide#1{{\color{blue}#1}}
\long\def\hide#1{}

\newcommand\blfootnote[1]{%
  \begingroup
  \renewcommand\thefootnote{}\footnote{#1}%
  \addtocounter{footnote}{-1}%
  \endgroup
}

\begin{document}

\title[Infinite time blow-up in the Keller-Segel system]{Existence and stability of infinite time blow-up in the Keller-Segel system}

\author[J.~D\'avila]{Juan D\'avila}
\address{\noindent J.~D\'avila: Department of Mathematical Sciences University of Bath, Bath BA2 7AY, United Kingdom}
\email{jddb22@bath.ac.uk}

\author[M.~del Pino]{Manuel del Pino}
\address{\noindent M.~del Pino: Department of Mathematical Sciences University of Bath, Bath BA2 7AY, United Kingdom}
\email{m.delpino@bath.ac.uk}

\author[J.~Dolbeault]{Jean Dolbeault}
\address{\noindent J.~Dolbeault: Ceremade, UMR CNRS nº7534, PSL university, Universit\'e Paris-Dauphine, Place de Lattre de Tassigny, 75775 Paris 16, France}
\email{dolbeault@ceremade.dauphine.fr}

\author[M.~Musso]{Monica Musso}
\address{\noindent M.~Musso: Department of Mathematical Sciences University of Bath, Bath BA2 7AY, United Kingdom.}
\email{m.musso@bath.ac.uk}

\author[J.~Wei]{Juncheng Wei}
\address{\noindent J.~Wei: Department of Mathematics University of British Columbia, Vancouver, BC V6T 1Z2, Canada}
\email{jcwei@math.ubc.ca}

\begin{abstract}
Perhaps the most classical diffusion model for chemotaxis is the Keller-Segel system
\begin{equation}\tag{$\ast$}
\label{ks0}
\left\{ \begin{aligned}
u_t =&\; \Delta u - \nabla \cdot(u \nabla v) \quad \inn \R^2\times(0,\infty),\\
v =&\; (-\Delta_{\R^2})^{-1} u := \frac 1{2\pi} \int_{\R^2} \, \log \frac 1{|x-z|}\,u(z,t)\, dz,
\\ & \qquad\ u(\cdot ,0) = u_0 \ge 0\quad\hbox{in } \R^2.
\end{aligned}
\right.
\end{equation}
We consider the {\em critical mass case} $\int_{\R^2} u_0(x)\, dx = 8\pi$ which corresponds to the exact threshold between finite-time blow-up and self-similar diffusion towards zero. We find a radial function $u_0^*$ with mass $8\pi$ such that for any initial condition $u_0$ sufficiently close to $u_0^*$ the solution $u(x,t)$ of \equ{ks0} is globally defined and blows-up in infinite time. As $t\to+\infty $ it has the approximate profile
$$
u(x,t) \approx \frac 1{\la^2} \ch{U}\left (\frac {x-\xi(t)}{\la(t)} \right ), \quad \ch{U}(y)= \frac{8}{(1+|y|^2)^2},
$$
where $\lambda(t) \approx \frac c{\sqrt{\log t}}, \ \xi(t)\to q $ for some $c>0$ and $q\in \R^2$.
This result answers affirmatively the nonradial stability conjecture raised in \cite{g}.

\end{abstract}

\keywords{Patlak-Keller-Segel system; chemotaxis; critical mass; blow-up; infinite time blow-up; inner-outer gluing scheme; rate; blow-up profile}
\subjclass[2020]{35K15; 35B40; 35B44}
\maketitle

\thispagestyle{empty}

\section{Introduction}

This paper deals with the classical Keller-Segel problem in $\R^2$,
\begin{equation}
\label{ks}
\left\{ \begin{aligned}
u_t =& \Delta u - \nabla \cdot(u \nabla v) \quad \inn \R^2\times(0,\infty),\\
v =& (-\Delta_{\R^2})^{-1} u := \frac 1{2\pi} \int_{\R^2} \, \log \frac 1{|x-z|}\,u(z,t)\, dz,
\\ & \qquad\ u(\cdot ,0) = u_0 \quad\hbox{in } \R^2,
\end{aligned}
\right.
\end{equation}
which is a well-known model for the dynamics of a population density $u(x,t)$ evolving by diffusion with a chemotactic drift. We consider positive solutions which are well defined, unique and smooth up to a maximal time $0< T \le +\infty$.
This problem formally preserves mass, in the sense that
$$
\int_{\R^2} u(x,t)dx = \int_{\R^2} u_0(x)\,dx =: M \foral t\in (0,T).
$$
An interesting feature of \equ{ks} is the connection between the second moment of the solution and its mass which is precisely given by
$$
\frac d{dt} \int_{\R^2}|x|^2\,u(x,t)\,dx = 4M - \frac {M^2} {2\pi},
$$
provided that the second moments are finite. If $M>8\pi$, the negative rate of production of the second moment and the positivity of the solution implies finite blow-up time. If $M<8\pi$ the solution lives at all times and diffuses to zero with a self similar profile according to~\cite{bdp}.
When $M = 8\pi$ the solution is globally defined in time. If the initial second moment is finite, it is preserved in time, and there is {\em infinite time blow-up} for the solution, as was shown in~\cite{bcm}.

Globally defined in time solutions of \equ{ks} are of course its positive finite mass steady states, which consist of the family
\be\label{st}
U_{\la,\xi}(x) = \frac 1{\la^2} \ch{U}\left ( \frac{x-\xi}{\la} \right ), \quad \ch{U}(y) = \frac 8{(1+|y|^2)^2 }, \quad \la>0, \ \xi\in \R^2.
\ee
We observe that all these steady states have the exact mass $8\pi$ and infinite second moment
$$
\int_{\R^2} U_{\la,\xi}(x)\, dx\, = \, 8\pi, \quad \int_{\R^2}|x|^2\,U_{\la,\xi}(x)\, dx\, = \, +\infty.
$$
As a consequence, if a solution of~\equ{ks} is attracted by the family $(U_{\la,\xi})$, its mass must be larger than $8\pi$ and if the initial second moment is finite, then blow-up occurs in a singular limit corresponding to $\lambda\to0_+$.

In the critical mass $M=8\pi$ case, the infinite-time blow-up in \equ{ks} when the second moment is finite, takes place in the form of a bubble in the form \equ{st} with $\la=\la(t)\to 0$ according to~\cite{biler2,bcm}. Formal rates and precise profiles were derived in~\cite{sire,campos} to be
\[
\lambda(t) \sim \frac c{\sqrt{\log t}}\ass t\to+\infty.
\]
A radial solution with this rate was built by Ghoul and Masmoudi in~\cite{g} and its stability within the radial class was established.  The framework of the construction in \cite{g} was actually fully nonradial, but for stability a spectral gap inequality only known in the radial case was used. Numerical
evidence for this inequality was obtained in~\cite{MR3196188}, and stability for general nonradial perturbation was conjectured in \cite{g}.
A related spectral estimate, useful in the analysis of finite time blow-up was found in \cite{collot2019spectral}.

In this paper we construct an infinite-time blow-up solution with a different method to that in~\cite{g}, which in particular leads to a proof of the stability assertion among non-radial functions. The following is our main result.



\begin{theorem}\label{teo1} There exists a nonnegative, radially symmetric function $u_0^*(x)$ with critical mass $\int_{\R^2} u_0^*(x)\,dx =8\pi$ and finite second moment $\int_{\R^2}|x|^2\, u_0^*(x)\,dx <+\infty$ such that for every $u_1(x)$ sufficiently close (in suitable sense) to $u_0^*$ with $\int_{\R^2} u_1\,dx =8\pi$, we have that the solution $u(x,t)$ of system $\equ{ks}$ with initial condition
$u(x,0)= u_1(x) $
has the form
\begin{align}
\label{exp1}
u(x,t)\, =\, \frac 1{\la(t)^2} \ch{U\Bigl(\frac {x-\xi(t)}{\la(t)} \Bigr) (1+ o(1))} , \quad \ch{U}(y)= \frac{8}{(1+|y|^2)^2}
\end{align}
uniformly on bounded sets of \ch{$\R^2$}, and
$$
\la(t)\ =\ \frac c{\sqrt{\log t}}\,(1+o(1)), \quad \xi(t)\to q \ass t\to+\infty,
$$
for some number $c>0$ and some $q\in \R^2$.
\end{theorem}

{\em Sufficiently close} for the perturbation $u_1(x) := u_0^*(x) + \vp(x)$ in this result is measured in the $C^1$-weighted norm for some $\sigma>1$
\[
\| \vp\|_{*\sigma} := \| (1+|\cdot| ^{4+\sigma}) \vp \|_{L^\infty(\R^2)} + \| (1+|\cdot|^{5+\sigma}) \nn \vp (x) \|_{L^\infty(\R^2)}<+\infty .
\]
\ch{The perturbation $\varphi$ must have zero mass too.}

\ch{``Uniformly on bounded sets'' of $\R^2$ in \eqref{exp1} means that for any bounded $K\subset \R^2$
\[
\lim_{t\to \infty}
\sup_{x\in K}
\la(t)^2 U\Bigl(\frac {x-\xi(t)}{\la(t)} \Bigr)^{-1}
\left| u(x,t)- \frac 1{\la(t)^2} U\Bigl(\frac {x-\xi(t)}{\la(t)} \Bigr)
\right| =0.
\]
The expansion of $u(x,t)$ can be made more precise though, and this is explained along the proof of theorem.}

The scaling parameter is rather simple to find at main order from the approximate conservation of second moment, \ch{see Section~\ref{sect-formal}.} The center $\xi(t)$ actually obeys a relatively simple system of nonlocal ODEs.

We devote the rest of this paper to the proof of Theorem \ref{teo1}. Our approach borrows elements of constructions in the works~\cite{1,2,3,4} based on the so-called {\em inner-outer gluing scheme}, where a system is derived for an inner equation defined near the blow-up point and expressed in the variable of the blowing-up bubble, and an outer problem that sees the whole picture in the original scale. The result of Theorem \ref{teo1} has already been announced in~\cite{d} in connection with~\cite{1,2,3}.

\medskip
There is a huge literature on chemotaxis in biology and in mathematics. The Patlak-Keller-Segel model~\cite{Patlak_1953,ks} is used in mathematical biology to describe the motion of mono-cellular organisms, like Dictyostelium Discoideum, which move randomly but experience a drift in presence of a chemo-attractant. Under certain circumstances, these cells are able to emit the chemo-attractant themselves. Through the chemical signal, they coordinate their motion and eventually aggregate. Such a self-organization scenario is at the basis of many models of chemotaxis and is considered as a fundamental mechanism in biology. Of course, the aggregation induced by the drift competes with the noise associated with the random motion so that aggregation occurs only if the chemical signal is strong enough. A classical survey of the mathematical problems in chemotaxis models can be found in~\cite{MR2013508,MR2073515}. After a proper adimensionalization, it turns out that all coefficients in the Patlak-Keller-Segel model studied in this paper can be taken equal to $1$ and that the only free parameter left is the total mass. For further considerations on chemotaxis, we shall refer to~\cite{Hillen_2008} for biological models and to~\cite{Chavanis2008} for physics backgrounds.

In many situations of interest, cells are moving on a substrate. The two-dimen\-sional case is therefore of special interest in biology, but also turns out to be particularly interesting from the mathematical point of view as well, because of scaling properties, at least in the simplest versions of the Keller-Segel model. Boundary conditions induce various additional difficulties. In the idealized situation of the Euclidean plane $\R^2$, it is known since the early work of W.~J\"{a}ger and S.~Luckhaus in~\cite{j} that solutions globally exist if the mass $M$ is small and blow-up in finite time if $M$ is large. The blow-up in a bounded domain is studied in~\cite{j,biler1,MR1361006,nagai2001blowup,s}. The precise threshold for blow-up, $M=8\pi$, has been determined in~\cite{MR2103197,bdp}, with sufficient conditions for global existence if $M\le8\pi$ in~\cite{bdp} (also see \cite{MR1620286} in the radial case). The key estimate is the boundedness of the free energy, which relies on the logarithmic Hardy-Littlewood-Sobolev inequality established in optimal form in~\cite{MR1143664}. We refer to~\cite{MR3379847} for a review of related results. If $M<8\pi$, diffusion dominates: intermediate asymptotic profiles and exact rates of convergence have been determined in~\cite{MR3196188}. Also see~\cite{MR2850755,MR3466844}. In the supercritical case $M>8\pi$, various formal expansions are known for many years, starting with~\cite{MR1478048,MR1415081,MR1918569} which were later justified in~\cite{raphael,mizoguchi}, in the radial case, and in~\cite{collot2019refined}, in the non-radially symmetric regime. This latter result is based on the analysis of the spectrum of a linearized operator done in~\cite{collot2019spectral}, based on the earlier work~\cite{Dejak_2012}, and relies on a scalar product already considered in~\cite{raphael} and similar to the one used in~\cite{MR2996772,MR3196188} in the subcritical mass regime. An interesting subproduct of the blow-up mechanism in~\cite{raphael,MR1627338} is that the blow-up takes the form of a concentration in the form of a Dirac distribution with mass exactly $8\pi$ at blow-up time, as was expected from~\cite{MR1627338,0728}, but it is still an open question to decide whether this is, locally in space, the only mechanism of blow-up.

The critical mass case $M=8\pi$ is more delicate. If the second moment is infinite, there is a variety of behaviors as observed for instance in~\cite{MR3004770,MR3165232,MR3781311}. For solutions with finite second moment, blow-up is expected to occur as $t\to+\infty$: see~\cite{kavallaris} for grow-up rates in $\R^2$, and~\cite{sw} for the higher-dimensional radial case. The existence in $\R^2$ of a global radial solution and first results of large time asymptotics were established in~\cite{biler2} using cumulated mass functions. In~\cite{bcm}, the infinite time blow-up was proved without symmetry assumptions using the free energy and an assumption of boundedness of the second moment. Also see~\cite{MR3524608,MR3781311} for an existence result under weaker assumptions, and further estimates on the solutions. Asymptotic stability of the family of steady states determined by \equ{st} under the mass constraint $M=8\pi$ has been determined in~\cite{figalli}. The blow-up rate $\lambda(t)$ and the shape of the limiting profile $U$ were identified in formal asymptotic expansions in~\cite{v1,v2,PhysRevE.66.046133,sire,MR2276287} and also in~\cite[Chapter~8]{campos}. 
As already mentioned, a radial solution with rate $\lambda(t)\sim(\log t)^{-1/2}$ was built and its stability within the radial class was established in~\cite{g}.

\section{\texorpdfstring{\ch{Formal derivation of the behavior of the parameters}}{}}
\label{sect-formal}

\ch{
We consider here a first approximation to a solution $u(x,t)$ of \eqref{ks},  globally defined in time, such that on bounded sets in $x$,}
\be
u(x,t) = \frac {1}{\la(t)^2}U\left ( \frac{x- \xi(t)}{\la(t)} \right)(1+ o(1)) \quad \text{as } t\to+\infty
\label{beh}\ee
for certain functions $0<\la(t) \to 0$ and $\xi(t) \to q\in \R^2$, \ch{where we recall that}
$$
U(y) = \frac 8{(1+|y|^2)^2}.
$$


We know that \equ{beh} can only happen in the critical mass, finite second moment case:
$$ \int_{\R^2} u(x,t)dx = 8\pi , \quad
\int_{\R^2} |x|^2 u(x,t)dx <+\infty,
$$
which according to the results in \cite{bcm,g,sire} is consistent with a behavior of the form \equ{beh}. Since the second moment of $U$ is infinite, we do not expect the approximation \equ{beh} be uniform in $\R^2$ but sufficiently far, a faster decay in $x$ should take place as we shall see next. We will find an approximate asymptotic expression for the scaling parameter $\la(t)$ that matches with this behavior.

\medskip
Let us introduce the function
$
\Gamma_0 : = (-\Delta )^{-1} U .
$
We directly compute
$$
\Gamma_0 (y) = \log \frac 8 {(1+ |y|^2)^2}
$$
and hence $\Gamma_0$ solves the Liouville equation
$$
-\Delta \Gamma_0\, = \, e^{\Gamma_0} = U \inn \R^2.
$$
Then $\nabla \Gamma_0(y) \approx -\frac {4y}{|y|^2} $ for all large $y$, and hence we get, away from $x=\xi$,
$$
-\nabla \cdot ( u\nabla (-\Delta )^{-1} u) \approx 4\nabla u \cdot \frac {x-\xi }{|x-\xi|^2}.
$$
Therefore, defining
\be
\mathcal E(u) : = \Delta u - \nabla\cdot ( u\nabla (-\Delta )^{-1} u)
\label{E(u)}
\ee
and writing in polar coordinates
$$
u(r,\theta, t) = u(x,t) , \quad x= \xi(t) + re^{i\theta},
$$
we find
$
\mathcal E(u) \approx \pp_r^2 u + \frac 5 r \pp_r u $.
Hence, assuming \ch{that} $\dot \xi (t)\to 0 $ sufficiently fast, equation \ch{\equ{ks}} approximately reads
\[
\pp_t u = \pp_r^2 u + \frac 5 r \pp_r u ,
\]
which can be idealized as a homogeneous heat equation in $\R^6$ for radially symmetric functions. It is therefore reasonable to believe
that beyond the self-similar region $r\gg \sqrt{t}$ the behavior changes into a function of $r/\sqrt{t}$ with fast decay at $+\infty$ that yields finiteness of the second moment. To obtain a first global approximation, we simply cut-off the bubble \equ{beh} beyond the self-similar zone.
We introduce a further parameter $\alpha(t) $
and set
\begin{align}
\label{u1}
\ch{\bar u}(x,t) = \frac {\alpha(t) }{\la^2}U\Bigl ( \frac{\ch{x-\xi}}{\la} \Bigr) \chi(x,t) ,
\end{align}
where
\begin{align}
\label{defchi-1}
\index{$\chi$}
\chi(x,t)= \chi_0 \Bigl( \frac {\ch{x-\xi}}{\sqrt{t}} \Bigr)
\end{align}
with $\chi_0$
a smooth radial cut-off function such that
\begin{align}
\label{chi0}
\index{$\chi_0$}
\ch{
\chi_0(z) = \begin{cases} 1 \ & \hbox{ if } |z|\le 1, \\ 0\ & \hbox{ if } |z|\ge 2 .
\end{cases}}
\end{align}
\ch{We} introduce the parameter $\alpha(t)$ because the total mass of the actual solution
should equal $8\pi$ for all $t$.
But
\begin{align}
\label{expansion-mass}
\frac{1}{\lambda^2}\int_{\R^2} U\Bigl(\frac{x-\xi}{\lambda}\Bigr) \chi(x,t)\, dx =   8 \pi + 16\pi \Upsilon \frac{\lambda^2}{t} + O\Bigl( \frac{\lambda^4}{t^2}\Bigr),
\end{align}
as $t\to\infty$,
where
\begin{align}
\label{defUpsilon}
\Upsilon = \int_0^\infty (\tilde \chi_0(s)-1) s^{-3}ds <0,
\end{align}
\ch{and $\chi_0(x) = \tilde \chi_0(|x|)$.}
To achieve
$  \int_{\R^2} \ch{\bar u}(x,t)\,dx = 8\pi $
we set $\ch{\alpha=\bar\alpha}$ where
\begin{align}
\nonumber
\ch{\bar\alpha}(t) = 1 - 2 \Upsilon \frac{\la^2} {t} + O\Bigl( \frac{\lambda^4}{t^2}\Bigr) .
\end{align}
Next we will obtain an approximate value of the scaling parameter $\la(t)$ that is consistent with the existence of a solution $u(x,t)\approx \ch{\bar u}(x,t)$ where $\ch{\bar u}$ is the function in \equ{u1} with $\alpha =\ch{\bar\alpha}$.
Let us consider the ``error operator''
\be
\label{Su}
S(u) = - u_t + \mathcal E(u),
\ee
where $\mathcal E(u) $ is defined in \eqref{E(u)}.
We have the following well-known identities,
valid for an arbitrary function $\omega(x)$ of class $C^2(\R^2)$ with finite mass
and $D^2 \omega(x) = O(|x|^{-4-\sigma})$ for large $|x|$.
We have
\be \label{moment}
\int_{\R^2} \mathcal |x|^2\mathcal E(\omega )\, dx = 4M - \frac {M^2}{2\pi}, \quad M= \int_{\R^2} \omega(x) dx
\ee
and
\be\label{moment2}
\int_{\R^2} x\mathcal E(\omega )\, dx = 0, \quad \int_{\R^2} \mathcal E(\omega )\, dx = 0.
\ee
Let us recall the simple proof of \equ{moment}. Integrating by parts on finite balls with large radii and using the behavior of the boundary terms we get
the identities
\begin{align}
\nonumber
\int_{\R^2} \mathcal |x|^2\Delta \omega\, dx & = 4M , \\
\nonumber
\int_{\R^2} |x|^2 \nabla\cdot (\ch{\omega} \nabla (-\Delta)^{-1})\omega )  \, dx &= - 2 \int_{\R^2} x\cdot \ch{\omega} \nabla (-\Delta)^{-1}\omega\,dx\\
\nonumber
&= \frac 1\pi \int_{\R^2} \int_{\R^2} \omega (x)\omega (y) \frac {x\cdot (x-y)}{|x-y|^2} dx\,dy \\
\nonumber
&= \frac 1{2\pi} \int_{\R^2} \int_{\R^2}\omega (x)\omega (y) \frac {(x-y)\cdot (x-y)}{|x-y|^2} dx\,dy \\
\label{momento}
& = \frac{M^2}{2\pi}
\end{align}
and then \equ{moment} follows. The proof of \equ{moment2} is even simpler. For a solution $u(x,t)$ of \ch{\equ{ks}} we then get
\be
\nonumber
\frac d{dt} \int_{\R^2} u(x,t)|x|^2 dx = 4M - \frac{M^2} {2\pi} , \quad M= \int_{\R^2} u(x,t) dx .
\ee

In particular, if $u(x,t)$ is sufficiently close to $\ch{\bar u}(x,t)$ and since $\int_{\R^2} \ch{\bar u}(x,t) dx =8\pi$, we get the approximate validity of the identity
$$
\frac d{dt} \int_{\R^2} \ch{\bar u}(x,t)|x|^2 dx = 0 .
$$
This means
$$
a I(t):= \int_{\R^2} \frac {\ch{\bar\alpha}}{\la^2 }U\left (\frac {x-\xi}\la \right ) \chi_0 \left (\frac {\ch{x-\xi}}{\sqrt{t}} \right ) |x|^2 dx \, =\, constant .
$$
We readily check that for some constant $\kappa$
$$
I(t) = 16\pi \la^2 \int_0^{\frac{\sqrt{t}}\la } \frac { \rho^3 d\rho } { (1+\rho^2)^2 } + \kappa + o(1) = 16\pi \la^2\log \frac{\sqrt{t}}\la + \kappa + o(1)
\quad \text{as } \lambda\to 0.
$$
Then we conclude that $\la(t)$ approximately satisfies
$$
\la^2 \log t \, = \, c^2 \,=\, constant
$$
and hence we get at main order
$$
\la(t) = \frac c{\sqrt{\log t}} .
$$
We also notice that the center of mass is preserved for a true solution, thanks to \equ{moment2}:
$$
\frac d{dt} \int_{\R^2} x u(x,t) dx = 0 .
$$
\ch{Since} the center of mass of $\ch{\bar u}(x,t)$ is exactly $\xi(t)$ we then get that approximately $$\xi(t) =  constant  = q. $$

\section{The approximations \texorpdfstring{$u_0$}{u0} and \texorpdfstring{$u_1$}{u1}}
\label{sect-approx2}

\ch{
From now on we to consider the Keller-Segel system starting at a large $t_0$:
\begin{equation}
\label{ks1}
\left\{ \begin{aligned}
u_t =& \Delta u - \nabla \cdot(u \nabla v) \quad \inn \R^2\times(t_0,\infty),\\
v =& (-\Delta_{\R^2})^{-1} u := \frac 1{2\pi} \int_{\R^2} \, \log \frac 1{|x-z|}\,u(z,t)\, dz,
\\ & \qquad\ u(\cdot ,t_0) = u_0 \quad\hbox{in } \R^2,
\end{aligned}
\right.
\end{equation}
which is equivalent to \eqref{ks}. We do this so that some expansions for $t$ large take a simpler form.}

In this section we will define a basic approximation to a solution of the Keller-Segel system \ch{\equ{ks1}}. Let us consider parameter functions
$$
0<\la(t)\to 0 ,\quad  \xi(t)\to q, \quad \alpha (t)\to 1  \ass  t\to +\infty
$$
that we will later specify.
Let us consider the functions
\[
U (y) = \frac{8}{(1+|y|^2)^2}, \quad
\Gamma_0(y) = \log U(y)
\]
and define the approximate solution $u_0(x,t)$ as
\begin{align}
\label{defu0}
\index{$u_0$}
u_0(x,t) &= \frac{\alpha}{\lambda^2} U \Bigl( \frac{x-\xi}{\lambda} \Bigr) \chi(x,t),
\\
\nonumber
v_0(x,t) &= (-\Delta_x)^{-1} u_0 = \frac 1{2\pi}\int_{\R^2 } \log \frac 1{|x-\bar x|} \, u_0(\bar x,t)\, d\bar x,
\end{align}
\ch{where $\chi$ is the cut-off function \eqref{defchi}.}
\ch{
We consider the error operator
\begin{align}
\nonumber
S(u)
&= -\pp_t u + \mathcal E(u ),
\end{align}
where
\begin{align*}
\mathcal E(u ) = \Delta_x u - \nabla_x \cdot ( u \nabla_x v),
\quad v = (-\Delta_{\ch{x}})^{-1} u .
\end{align*}
and next measure the error of approximation $S(u_0)$.}

We have
\begin{align}
\nonumber
 -\pp_t u_0 (x,t)
&=
- \frac{\dot{\alpha}}{\lambda^2} U(y) \chi_0(z)
+ \alpha \frac{\dot\lambda}{\lambda^3} Z_0 \chi_0(z)
+\frac{\alpha}{\lambda^3} \dot\xi \cdot \nn_y U (y) \, \chi_0(z)
\\
\label{ptu0}
& \quad
+ \frac{\alpha}{\ch{\lambda^2} \sqrt t}
U(y)
\dot\xi \cdot  \nabla_z \chi_0(z)
+  \frac{\alpha }{2\lambda^2 t } U(y)
\nabla_z\chi_0(z) \cdot z,
\end{align}
\[
z = \frac{x-\xi}{\sqrt t}
\]
where
\begin{align}
\label{Z0}
Z_0(y) = 2 U(y) + y\cdot \nabla_y U(y)  , \quad y = \frac{x-\xi}{\lambda}.
\end{align}
We also have
\begin{align*}
\mathcal E(u_0) & =
\Delta_x u_0 - \nabla_x \cdot ( u_0 \nabla_x v_0 )
\\
& = \frac{2\alpha}{\lambda^3 t^{1/2} }
\nabla_z \chi_0(z) \cdot \nabla_y U (y)
+ \frac{\alpha }{t} \frac{1}{\lambda^2}
\Delta_z \chi_0(z)  U(y)
- \frac{\alpha}{\lambda^2 \sqrt t} U(y) \nabla_z \chi_0(z)\cdot \nabla_x v_0
\\
&
+ \frac{\alpha \chi_0(z)}{\lambda^4}
\Bigl[ (\chi_0(z) \alpha-1) U^2(y)
- \nabla_y U(y) \cdot ( \nabla_y v_0- \nabla_y  \Gamma_0 )
\Bigr]\ .
\end{align*}

Let us decompose
\be \label{v0}
v_0(y) = \alpha \Gamma_0(y) + \mathcal R(y).
\ee
For the term $\mathcal R$ in \equ{v0} we directly estimate
\begin{align}
\label{est-mathcalR}
|\nabla_y \mathcal R(y)| \leq
\begin{cases}
\frac{\lambda^2}{t} \frac{1}{|y|}  & |y| \geq \frac{\sqrt t}{\lambda},
\\
0  & |y| \leq \frac{\sqrt t}{\lambda}.
\end{cases}
\end{align}
Then
\begin{align*}
\mathcal E(u_0)
& = \frac{2\alpha}{\lambda^3 t^{1/2} }
\nabla_z \chi_0(z) \cdot \nabla_y U (y)
+ \frac{\alpha }{t} \frac{1}{\lambda^2}
\Delta_z \chi_0(z)  U(y)
- \frac{\alpha}{\lambda^2 \sqrt t} U(y) \nabla_z \chi_0(z) \nabla_x v_0
\\
& \quad
+ \frac{\alpha \chi_0(z)}{\lambda^4}
\Bigl[ (\alpha-1) U^2(y) - (\alpha-1) \nabla_y U(y) \cdot\nabla_y \Gamma_0(y)
+ \alpha (\chi_0(z)-1)U^2(y)
\\
& \quad
- \nabla_y U(y) \cdot \nabla_y \mathcal R(y)
\Bigr].
\end{align*}
and thus
\begin{align}
\index{$S(u_0)$}
\nonumber
S(u_0)
&=
- \frac{\dot{\alpha}}{\lambda^2} U(y) \chi_0(z)
+ \alpha \frac{\dot\lambda}{\lambda^3} Z_0 \chi_0(z)
+\frac{\alpha}{\lambda^3} \dot\xi \cdot \nn_y U (y) \, \chi_0(z)
\\
\nonumber
& \quad
+ \frac{\alpha}{\ch{\lambda^2 \sqrt t}} U(y)
\dot\xi \cdot  \nabla_z \chi_0(z)
+  \frac{\alpha }{2\lambda^2 t } U(y)
\nabla_z\chi_0(z) \cdot z
\\
\nonumber
& \quad
+ \frac{2\alpha}{\lambda^3 t^{1/2} }
\nabla_z \chi_0(z) \cdot \nabla_y U (y)
+ \frac{\alpha }{t} \frac{1}{\lambda^2}
\Delta_z \chi_0(z) U(y)
- \frac{\alpha}{\lambda^2 \sqrt t} U(y) \nabla_z \chi_0(z) \cdot \nabla_x v_0
\\
\nonumber
& \quad
- \frac{\alpha(\alpha-1) \chi_0(z)}{\lambda^4}
\nabla_y \cdot ( U(y) \nabla_y \Gamma_0(y))
\\
\nonumber
& \quad
+ \frac{\alpha \chi_0(z)}{\lambda^4}
\Bigl[
 \alpha (\chi-1)U^2(y)
- \nabla_y U(y) \cdot \nabla_y \mathcal R(y)
\Bigr].
\end{align}

For a function $v(\zeta)$ defined for $\zeta \in \R^2$ consider the operator
\begin{align}
\label{def-laplacian6}
\Delta_6 \ch{v}(\zeta)  = \Delta v(\zeta) +  4\frac{\zeta}{|\zeta|^2} \cdot \nn_\zeta v (\zeta)  . \end{align}
The reason for the notation is that for radial functions $\ch{ v= v(r)}$, $r=|\zeta|$,  we have
$$
\Delta_6 \ch{ v} = \pp^2_r \ch{ v} + \frac 5r  \pp_r \ch{ v},
$$
which corresponds to Laplace's operator in $\R^6$ on radial functions.

Let  $\tilde \varphi_\lambda(\zeta,t)$ be the (radial) solution to
\begin{align}
\label{deftildephilambda}
\left\{
\begin{aligned}
\partial_t \tilde \varphi_\lambda
&= \Delta_6 \tilde \varphi_\lambda
+  E(\zeta,t)
\quad \text{in }\R^2 \times ( \frac{t_0}{2},\infty),
\\
\tilde \varphi_\lambda(\cdot,\frac{t_0}{2})&= 0
\quad \text{in }\R^2,
\end{aligned}
\right.
\end{align}
given by Duhamel's formula,
where $E(\zeta,t)$ is the radial function
\begin{align}
\label{defE}
E( \zeta ,t;\lambda) =
\frac{\dot \lambda}{\lambda^3} Z_0\Bigl(\frac{\zeta}{\lambda}\Bigr)
\chi_0\Bigl( \frac{\zeta}{\sqrt t}\Bigr)
\ch{+} \frac{1}{2\lambda^2 t} U \Bigl(\frac{\zeta}{\lambda}\Bigr) \nabla_z \chi_0(z) \cdot z
+ \tilde E (x,t)  ,
\end{align}
and
\begin{align}
\nonumber
\tilde E(\zeta,t;\lambda)
&=
\frac{2}{\lambda^3 t^{1/2}}  \nabla_z \chi_0(z) \cdot \nabla_y U (y)
+ \frac{1}{\lambda^2 t}
\Delta_z \chi_0(z) U (y)
\\
\label{defTildeE}
& \quad
- \frac{1}{\lambda^3 t^{1/2}} U(y) \nabla_z \chi_0(z) \cdot \nabla_y \Gamma_0(y) ,
\end{align}
with $z = \frac{\zeta}{\sqrt t}$, $ y = \frac{\zeta}{\lambda} $.

We then define
\begin{align}
\label{defphilambda}
\index{$\varphi_\lambda$}
\varphi_\lambda(x,t) = \tilde \varphi_\lambda(x-\xi(t),t) .
\end{align}

\ch{The reason to define $\varphi_\lambda$ for $t> \frac{t_0}{2}$ is that it gives better properties for the first approximation of $\lambda$ constructed in Section~\ref{sect-mass-varphilambda}.}
Since $\lambda(t)$ is defined naturally for $t>t_0$, we will need to define $\lambda(t)$ for $\frac{t_0}{2}<t<t_0$ in an appropriate way \ch{(see Proposition~\ref{prop-lambda0} and Section~\ref{sect-mass-varphilambda}).}
We will write $\lambda = \lambda_0 + \lambda_1$ where both of these functions are constructed so that they are  defined for $t>\frac{t_0}{2}$. The construction of $\lambda_0$ is given in Proposition~\ref{prop-lambda0}.
In particular $\lambda_0(t) = \frac{c_0}{\sqrt{\log t}}(1+o(1))$ as $t\to \infty$.
\ch{Note that  $\varphi_\lambda(\cdot,t_0)$ is not zero.}

We define the approximate solution
\begin{align}
\label{defu1}
\index{$u_1$}
u_1 := u_0 + \varphi_\lambda
\end{align}
which depends on the parameter functions $\alpha(t)$, $\xi(t)$, $\la(t)$. Correspondingly, we write
$$
v_1 \ :=\  (-\Delta_{\ch{x}} )^{-1} (u_1)\, .
$$

We will establish in the next sections that a suitable choice of these functions makes it possible to find an actual solution  \ch{of \equ{ks1}} as a lower order perturbation of $u_1$.

\section{The first error of approximation}

We will assume the following conditions on $\lambda$, $\alpha$, $\xi$
\begin{align}
\label{conditions}
\left\{
\begin{aligned}
&
|\lambda(t)| + t \log (t) |\dot \lambda(t)|  \leq \frac{C}{\sqrt{\log (t)}}
\\
&
|\dot \xi(t) | \leq \frac{C}{t^{\gamma}}
\\
&
|\alpha(t)-1|\leq \frac{C}{t \log t} ,
\quad
|\dot \alpha(t)|\leq \frac{C}{t^2 \ch{\, \log t}} ,
\end{aligned}
\right.
\end{align}
where $\frac{3}{2}<\gamma<2$.

We compute
\begin{align*}
S(u_1) & = S(u_0+\varphi_\lambda)
= S(u_0) - \partial_t \varphi_\lambda
+ \mathcal L_{u_0}[\varphi_\lambda]
- \nabla\cdot( \varphi_\lambda \nabla \psi_\lambda).
\end{align*}
where
\begin{align*}
\mathcal L_{u_0}[\varphi]
&= \Delta \varphi - \nabla\cdot ( \varphi \nabla v_0) - \nabla \cdot ( u_0 \nabla \psi ),
\\
\psi_\lambda &= (-\Delta)^{-1} \varphi_\lambda,
\quad
v_0 = ( -\Delta)^{-1} u_0 .
\end{align*}
Then
\begin{align}
\nonumber
S(u_1)
&=
- \frac{\dot{\alpha}}{\lambda^2} U(y) \chi
+ (\alpha-1) \frac{\dot\lambda}{\lambda^3} Z_0 \chi
+\frac{\alpha}{\lambda^3} \dot\xi \cdot \nn_y U (y) \, \chi
+ \frac{\alpha}{\ch{\lambda^2} \sqrt t} U(y)
\dot\xi \cdot  \nabla \chi_0
\\
\nonumber
& \quad
+  \frac{(\alpha -1)}{2 t }  \frac{1}{\lambda^2} U
\nabla_z\chi_0 \cdot \frac{x-\xi}{\sqrt t}
+ \frac{2(\alpha-1)}{\lambda^3 t^{1/2} }
\nabla_z \chi_0 \cdot \nabla_y U
\\
\nonumber
& \quad
+ \frac{(\alpha-1) }{t} \Delta \chi_0 \frac{1}{\lambda^2} U
\ch{
-\frac{\alpha^2-1}{\lambda^3 \sqrt t} U \nabla_z \chi_0 \cdot \nabla_y \Gamma_0
- \frac{\alpha}{\lambda^3 \sqrt t} U \nabla_z \chi_0 \cdot \nabla_y \mathcal R
}
\\
\nonumber
& \quad
- \frac{\alpha (\alpha-1) \chi}{\lambda^4}
\nabla_y \cdot ( U \nabla_y \Gamma_0)
+ \frac{\alpha^2 \chi ( 1-\chi)}{\lambda^4} U^2
- \frac{\alpha \chi}{\lambda^4} \nabla _y U \cdot \nabla_y \mathcal R
\\
\label{Su1}
\index{$S(u_1)$}
&\quad
+ \nabla \varphi_\lambda \cdot \dot \xi
- \frac{4}{r} \partial_r \varphi_\lambda
- \nabla\cdot ( \varphi_\lambda \nabla v_0) - \nabla \cdot ( u_0 \nabla \psi_\lambda )
-  \nabla\cdot ( \varphi_\lambda \nabla \psi_\lambda),
\end{align}
where $\mathcal R$ is defined in the decomposition \eqref{v0}.

\begin{lemma}
\label{lemma-est-varphilambda}
Let $\varphi_\lambda$ be defined by \eqref{defphilambda}\ch{-\eqref{deftildephilambda}} with $\lambda$ satisfying \eqref{conditions}.
Then
\begin{align}
\label{est-varphi-lambda1}
|\varphi_\lambda(x,t)|
+ ( |x-\xi|+\lambda) |\nabla \varphi_\lambda(x,t)|
\leq C \frac{1}{t  \ch{ \log t}  } \begin{cases}
\frac{1}{\lambda^2 + |x-\xi|^2} & |x-\xi|\leq \sqrt t
\\
\frac{1}{t}
e^{-\frac{|x-\xi|^2}{4t}} & |x-\xi|\geq \sqrt t.
\end{cases}
\end{align}
We also have
\begin{align}
\label{est-varphi-lambda2}
|\nabla \varphi_\lambda(x,t)|
\leq  \frac{\ch{C}}{t \ch{ \log t} } \frac{|x-\xi|}{(\lambda + |x-\xi|)^4} , \quad |x-\xi| \leq \sqrt t.
\end{align}
\end{lemma}
\begin{proof}
In terms of the function $\tilde \varphi_\lambda$ defined in \eqref{deftildephilambda}, with $r = |x-\xi|$ we claim that
\begin{align*}
|\tilde \varphi_\lambda(r,t)|
\leq C \frac{1}{t \ch{ \log t} } \begin{cases}
\frac{1}{\lambda^2+r^2} & r\leq \sqrt t,
\\
\frac{1}{t}
e^{-\frac{r^2}{4t}} & r\geq \sqrt t.
\end{cases}
\end{align*}

For the proof of this we use barriers.
Consider
\[
\psi_1(r,t) = \frac{1}{t \ch{ \log t}} \frac{1}{\lambda^2+r^2}
\]
and note that
\begin{align*}
\partial_t \psi_1
- \Bigl(\partial_{rr} +\frac{5}{r}\partial_r\Bigr)
\psi_1 \geq c \frac{\lambda^{-4} }{t \ch{ \log t} (1+r/\lambda)^4} ,\quad
r \leq 2 \delta \sqrt t
\end{align*}
for some $c>0$, $\delta>0$.

Let $\chi_{\delta\sqrt t}(r,t ) = \ch{\tilde \chi_0}(\frac{r}{\delta \sqrt t}) $ where \ch{$\tilde\chi_0\in C^\infty(\R)$ is such that $\tilde\chi_0(s) = 1$ for $s \leq 1$ and $\tilde\chi_0(s) = 0$ for $s \geq 2$.}
Consider
\begin{align*}
\psi (r,t) =
\psi_1 (r,t) \chi_{\delta \sqrt t}(r,t) + \frac{C_1}{t^2 \ch{ \log t}} e^{-\frac{r^2}{4t}} .
\end{align*}
The function $\tilde E$ \eqref{defTildeE} can be estimated by
\begin{align*}
|\tilde E(\zeta,t) |\leq \frac{1}{\lambda^2 t^3} h_1\Bigl(\frac{\zeta}{\sqrt t}\Bigr)
\end{align*}
where $h_1(z)$ is a smooth function with compact support.
Then  $E$  \eqref{defE} has the estimate
\begin{align*}
|\tilde E(\zeta,t) | \leq C \frac{|\lambda \dot \lambda|}{(r^2+\lambda^2)^2} +  \frac{1}{\lambda^2 t^3} h_2\Bigl(\frac{\zeta}{\sqrt t}\Bigr)
\end{align*}
where $h_2(z)$ is a smooth function with compact support.

Then for $C_1$ sufficiently large
\begin{align*}
\partial_t \psi
- \Bigl(\partial_{rr} +\frac{5}{r}\partial_r\Bigr)
\psi \geq
c |E(r,t)|,
\end{align*}
where $c>0$.

By the comparison principle,
\[
|\tilde \varphi_\lambda(r,t) | \leq C \ch{\psi(r,t)} ,
\]
for some uniform constant $C$.
After a suitable scaling,
from standard parabolic estimates we also get
\begin{align*}
(\lambda + r)
|\nabla_x \tilde \varphi_\lambda(r,t)|
\leq C \ch{\psi(r,t)}.
\end{align*}
With these two inequalities we obtain \eqref{est-varphi-lambda1}.

To prove \eqref{est-varphi-lambda2} we change variables $y = \frac{x-\xi}{\lambda}$ in the equation \eqref{deftildephilambda} and define
\[
\tilde \varphi_\lambda(r, t)
= \frac{1}{\lambda^2}
\hat \varphi_\lambda \Bigl( \frac{r}{\lambda},t \Bigr).
\]
We get the equation, after interpreting $\rho = |y|$, $y\in \R^6$
\[
\lambda^2 \partial_t \hat \varphi
= \Delta_{\R^6}  \hat \varphi
+ \lambda \dot \lambda ( 2 \hat \varphi_\lambda
+ y \cdot \nabla_y \hat \varphi_\lambda)
+ \lambda^4 E (\lambda y,t),
\]
where $E$ is defined in \eqref{defE}. Differentiating with respect to $y$ and using the bound we already have for $\nabla_y \hat \varphi_\lambda $ from \eqref{est-varphi-lambda2}, and using standard parabolic estimates, we get
\begin{align*}
|D^2_y \hat \varphi_\lambda(y,t)|
\leq \frac{C}{t \ch{\log t} } \frac{1}{(1+|y|)^4} ,
\quad |y|\leq \sqrt{t \log t}.
\end{align*}
Using that $\nabla \hat \varphi_\lambda(0,t)=0$ we deduce
that
\begin{align*}
|\nabla_y \hat \varphi_\lambda(y,t)|
\leq \frac{C}{t \ch{\log t} } \frac{|y|}{(1+|y|)^4}
,
\quad |y|\leq \sqrt{t \log t},
\end{align*}
\ch{which readily gives \eqref{est-varphi-lambda2}.}

\end{proof}

\begin{lemma}
Assuming \eqref{conditions} we have
\begin{align}
\label{est-Su1-inner}
\lambda^4 |S(u_1)| \chi(x,t)  \leq C \frac{1}{t \ch{\,\log t} } \frac{\ch{\log(2+|y|)}}{1+|y|^6} , \quad y = \frac{x-\xi}{\lambda},
\end{align}
and
\begin{align}
\label{est-Su1-outer}
|S(u_1)| (1-\chi) \leq C \frac{1}{t^4 \ch{ \log t}}
e^{-c\frac{|x|^2}{t}},
\end{align}
for some $c\in (0,\frac{1}{4})$.
\end{lemma}
\begin{proof}
Let us analyze the terms involving $\varphi_\lambda$.
We estimate, using \ch{Lemma~\ref{lemma-est-varphilambda}},
\begin{align*}
\left| \lambda^2 U(y) \varphi_\lambda (\xi+ \lambda y) \right|
\leq
C
\frac{1}{t  \ch{\log t} }\frac{1}{(1+|y|)^6} ,
\quad |y|\leq \sqrt {t \log t}.
\end{align*}

Similarly, by \eqref{v0}
\begin{align}
\nonumber
-\frac{4}{r} \partial_r \tilde \varphi_\lambda
- \nabla \tilde \varphi_\lambda \cdot \nabla v_0
&=
-\frac{4}{r} \partial_r \tilde \varphi_\lambda
- \nabla \tilde \varphi_\lambda \cdot \nabla \Gamma_0
- (\alpha-1) \nabla \tilde \varphi_\lambda \cdot \nabla \Gamma_0
- \nabla \tilde \varphi_\lambda \cdot \nabla \mathcal R
\\
\label{S1a}
&=
4\Bigr( \frac{r}{r^2+\lambda^2} - \frac{1}{r}\Bigl) \partial_r \tilde  \varphi_\lambda
- (\alpha-1) \nabla \tilde \varphi_\lambda \cdot \nabla \Gamma_0
- \nabla \tilde \varphi_\lambda \cdot \nabla \mathcal R.
\end{align}
By \eqref{est-varphi-lambda2}
\begin{align*}
\Bigl|
\lambda^4
4\Bigr( \frac{r}{r^2+\lambda^2} - \frac{1}{r}\Bigl) \partial_r \varphi_\lambda
\Bigr|
&\leq
\frac{C}{t  \ch{\log t}} \frac{1}{(1+|y|)^6} , \quad |y|\leq \sqrt{t \log t}.
\end{align*}
The other terms in \eqref{S1a} are estimated similarly, using the hypotheses on $\alpha$  and the estimate on $\mathcal R$ \eqref{est-mathcalR}, and we get
\begin{align*}
\left|
-\frac{4}{r} \partial_r \tilde \varphi_\lambda
- \nabla \tilde \varphi_\lambda \cdot \nabla v_0
\right| \leq \frac{C}{t  \ch{\log t}} \frac{1}{(1+|y|)^6} , \quad |y|\leq \sqrt{t \log t}.
\end{align*}

The terms involving $\psi_\lambda = (-\Delta)^{-1}\varphi_\lambda$ are estimated using the formula
\begin{align*}
\partial_r \psi_\lambda(r,t) = \frac{1}{r} \int_0^r \varphi_\lambda (s,t)sds .
\end{align*}

\ch{
In $ \lambda^4 S(u_1)$ we have also the term $- \dot{\alpha}\lambda^2  U(y) \chi  $, which thanks to \eqref{conditions} can be estimated as
\begin{align*}
\Bigl| \lambda^2 \dot{\alpha} U(y) \chi \Bigr|
& \leq \frac{C \lambda^2}{t^2 \log t} \frac{1}{(1+|y|)^4}
\chi(y,t)
\leq \frac{C }{t \log t} \frac{1}{(1+|y|)^6}
\chi(y,t).
\end{align*}
}

The remaining terms are estimated similarly, and we obtain \eqref{est-Su1-inner}.

The stated inequality \eqref{est-Su1-outer}  follows from the Gaussian decay of $\varphi_\lambda$ in Lemma~\ref{lemma-est-varphilambda}.
\end{proof}

\section{The inner-outer gluing system}

Let us consider the initial approximation
\[
u_1(x,t) = u_0(x,t) + \varphi_\lambda(x,t)
\]
built in \ch{Section~\ref{sect-approx2}}
for a given choice of the parameter functions $\la(t)$, $\alpha(t)$, $\xi(t) $ \ch{satisfying \eqref{conditions}}. Here
$u_0$ is the function defined in \eqref{defu0} and $\varphi_\lambda$ that in \eqref{defphilambda}.
We look for a solution of the Keller-Segel equation \ch{\equ{ks1}}
in the form of a small perturbation of $u_1$,
namely
\begin{align}
\label{def-u}
u (x,t)= u_1 (x,t)+ \Phi(x,t).
\end{align}
We write the perturbation $\Phi$ \ch{as a sum of an}  ``inner" contribution, better expressed in the scale of $u_0$, and
a remote effect that takes into consideration the ``outer" regime.
Precisely, we write
\begin{align}
\label{def-varphi}
\Phi(x,t) = \frac{1}{\lambda^2} \phi^i(y,t)\chi(x,t)  + \varphi^o(x,t), \quad y= \frac{x-\xi}{\lambda} ,
\end{align}
where $\chi$ is the smooth cut-off
\begin{align}
\label{defchi}
\chi(x,t)= \chi_0 \Bigl( \frac {\ch{x-\xi}}{\sqrt{t}} \Bigr)
\end{align}
with $\chi_0$ a smooth radial cut-off function such that $ \chi_0(z) = 1 $ if $|z|\le 1$, $ \chi_0(z) = 1 $ if $ |z|\ge 2 $.
(The same as defined in \eqref{defchi-1}.)


\medskip
\ch{Recall} $S(u)$ given by
\[
S(u) = - \partial_t u + \Delta u -
\nabla \cdot (u \nabla v), \quad v = ( -\Delta)^{-1} u ,
\]
where the operators act on the original variable $x$ unless otherwise indicated.
\ch{In the computations that follow we will express the equation}
$$
S(u_1+\Phi) = 0
$$
for $\Phi$ given by \equ{def-varphi}, as a parabolic system in its inner and outer contributions  $\phi^i$ and
$\vp^o$. The coupling in that system will be small if $\phi^i(y,t)$ decays sufficiently fast in space and time. That can only be achieved for suitable choices of the parameters $\alpha, \la, \xi$ that yield certain solvability conditions
satisfied.  The set of all these relations is what we call the inner-outer gluing system.
Next we formulate this system. It will be necessary to successively refine  its original expression by further decomposing $\phi^i$ into two contributions with separate space decay, finally arriving at the equations
\eqref{inner4a}, \eqref{inner4b}, \eqref{outer4} and
\eqref{eq-param} which are the ones we will actually solve.

\medskip
Let us observe that
\begin{align*}
S(u_1+\Phi)
&= S(u_1) - \partial_t \Bigl( \frac{1}{\lambda^2} \phi^i \chi \Bigr) - \partial_t \varphi^o
+ \mathcal L_{u_1}\Bigr[\frac{1}{\lambda^2} \phi^i \chi \Bigl]
+ \mathcal L_{u_1}[\varphi^o]
\\
& \quad
- \nabla\cdot( \Phi \nabla (-\Delta)^{-1} \Phi ) ,
\end{align*}
where
\begin{align*}
\mathcal L_{u_1}[\varphi]
&= \Delta \varphi - \nabla\cdot ( \varphi \nabla v_1) - \nabla \cdot ( u_1 \nabla (-\Delta)^{-1} \varphi )
, \qquad
v_1 = ( -\Delta)^{-1} u_1 .
\end{align*}
We use the notation
\[
\psi = \frac{1}{\lambda^2}(-\Delta)^{-1} \phi^i , \quad \hat \psi = \ch{\frac{1}{\lambda^2} (-\Delta)^{-1} (  \phi^i \chi )} ,
\]
in the expressions that follow.
We expand
\begin{align*}
\mathcal L_{u_1}[\frac{1}{\lambda^2} \phi^i \chi ]
&= \chi \frac{1}{\lambda^2} \Delta \phi^i
+  \frac{2}{\lambda^2}\nabla \chi \cdot \nabla \phi^i + \frac{1}{\lambda^2} \phi^i \Delta \chi
- \nabla\cdot ( \frac{1}{\lambda^2} \phi^i \chi \nabla v_1)
- \nabla \cdot ( u_1 \nabla \hat \psi ) .
\end{align*}
We have
\begin{align*}
\nabla \cdot ( u_1 \nabla \hat \psi )
&=
\nabla \cdot ( \frac{\alpha }{\lambda^2} U \nabla \psi ) \chi
+
\nabla \cdot ( \frac{\alpha }{\lambda^2} U \nabla ( \hat \psi -\psi) ) \chi
+
\frac{\alpha }{\lambda^2} U \nabla \chi \cdot \nabla \hat \psi
\\
&
\quad
+
\nabla \cdot ( \varphi_\lambda \nabla  \psi )
+ \nabla \cdot ( \vp_\la \nn (\hat \psi -\psi) )
\end{align*}
and
\begin{align*}
\nabla\cdot ( \frac{1}{\lambda^2} \phi^i \chi \nabla v_1)
&=
\nabla\cdot ( \frac{1}{\lambda^2} \phi^i  \nabla v_1) \chi
+
\frac{1}{\lambda^2} \phi^i \nabla  \chi \cdot \nabla v_1 .
\end{align*}
Recall the notation
\begin{align*}
v_1 = v_0 + \psi_\lambda, \quad
v_0 = \ch{\frac{\alpha}{\lambda^2} (-\Delta)^{-1} ( U \chi ) } , \quad
\psi_\lambda = (-\Delta)^{-1} \varphi_\lambda ,
\end{align*}
and also \eqref{v0}
\[
v_0 = \alpha \Gamma_0 + \mathcal{R} ,
\quad
\mathcal{R} = \ch{ \frac{\alpha}{\lambda^2} (-\Delta)^{-1} \bigl( U
\ch{(\chi-1) } \bigl) }.
\]
Then
\begin{align*}
\nabla\cdot ( \frac{1}{\lambda^2} \phi^i \chi \nabla v_1)
&=
\nabla\cdot ( \frac{1}{\lambda^2} \phi^i  \nabla v_0) \chi
+
\nabla\cdot ( \frac{1}{\lambda^2} \phi^i  \nabla \psi_\lambda ) \chi
+
\frac{1}{\lambda^2} \phi^i \nabla  \chi \cdot \nabla v_0
\\
& \quad
+ \frac{1}{\lambda^2} \phi^i \nabla  \chi \cdot \nabla \psi_\lambda
\\
&=
\frac{\alpha}{\lambda^2}  \nabla\cdot ( \phi^i  \nabla  \Gamma_0) \chi
+
\nabla\cdot ( \frac{1}{\lambda^2} \phi^i  \nabla \mathcal R) \chi
+
\nabla\cdot ( \frac{1}{\lambda^2} \phi^i  \nabla \psi_\lambda ) \chi
\\
& \quad
+
\frac{\alpha}{\lambda^2}
 \phi^i \nabla  \chi \cdot \nabla  \Gamma_0
+
\frac{1}{\lambda^2} \phi^i \nabla  \chi \cdot \nabla \mathcal R
+
\frac{1}{\lambda^2} \phi^i \nabla  \chi \cdot \nabla \psi_\lambda .
\end{align*}
Therefore
\begin{align*}
\mathcal L_{u_1}[\frac{1}{\lambda^2} \phi^i \chi ]
&= \chi \frac{1}{\lambda^2} \Delta \phi^i
+  \frac{2}{\lambda^2}\nabla \chi \cdot \nabla \phi^i + \frac{1}{\lambda^2} \phi^i \Delta \chi
\\
& \quad
- \Bigl[
\nabla\cdot ( \ch{\frac{\alpha}{\lambda^2} \phi^i  \nabla  \Gamma_0} ) \chi
+
\nabla\cdot ( \frac{1}{\lambda^2} \phi^i  \nabla \mathcal R) \chi
+
\nabla\cdot ( \frac{1}{\lambda^2} \phi^i  \nabla \psi_\lambda ) \chi
\\
& \quad \quad
+
\ch{\frac{\alpha}{\lambda^2} \phi^i \nabla  \chi \cdot \nabla  \Gamma_0}
+
\frac{1}{\lambda^2} \phi^i \nabla  \chi \cdot \nabla \mathcal R
+
\frac{1}{\lambda^2} \phi^i \nabla  \chi \cdot \nabla \psi_\lambda
\Bigr]
\\
& \quad
-
\Bigl[
\nabla \cdot ( \frac{\alpha }{\lambda^2} U \nabla \psi ) \chi
+
\nabla \cdot ( \frac{\alpha }{\lambda^2} U \nabla ( \hat \psi -\psi) ) \chi
+
\frac{\alpha }{\lambda^2} U \nabla \chi \cdot \nabla \hat \psi
\\
&
\quad \quad
+
\nabla \cdot ( \varphi_\lambda \nabla  \psi )
+
\nabla \cdot ( \varphi_\lambda \nabla (\hat \psi - \psi)  )
\Bigr] .
\end{align*}

Next we expand
\begin{align*}
\mathcal L_{u_1}[\varphi^o]
&= \Delta \varphi^o - \nabla\cdot ( \varphi^o \nabla v_1) - \nabla \cdot (u_1 \nabla \psi^o) , \quad \psi^{o} = (-\Delta)^{-1} \varphi^o.
\end{align*}
We have
\begin{align*}
\nabla \cdot (u_1 \nabla \psi^o)
&=
\nabla \cdot (  \frac{\alpha }{\lambda^2} U \chi  \nabla \psi^o )
+
\nabla \cdot ( \varphi_\lambda \nabla \psi^o )
\\
&=
\nabla \cdot ( \frac{\alpha }{\lambda^2} U \nabla \psi^o ) \chi
+
\frac{\alpha }{\lambda^2} U \nabla \chi \cdot \nabla  \psi^o
+
\nabla \cdot ( \varphi_\lambda \nabla  \psi^o ) \chi
\\
& \quad
+
\nabla \cdot ( \varphi_\lambda \nabla  \psi^o ) (1-\chi ),
\end{align*}
\ch{and}
\begin{align*}
\nabla\cdot ( \varphi^o \nabla v_1)
&=
\nabla\cdot ( \varphi^o \nabla v_0)
+\nabla\cdot ( \varphi^o \nabla \psi_\lambda)
\\
&=
\alpha \nabla\cdot ( \varphi^o \nabla \Gamma_0)
+ \nabla\cdot ( \varphi^o \nabla \ch{\mathcal R})
+\nabla\cdot ( \varphi^o \nabla \psi_\lambda)
\\
&=
\nabla \varphi^o \cdot \nabla \Gamma_0
-
\frac{1}{\lambda^2} U \varphi^o
+
(\alpha-1) \nabla\cdot ( \varphi^o \nabla \Gamma_0)
\\
& \quad
+
\nabla\cdot ( \varphi^o \nabla \mathcal R)
+
\nabla\cdot ( \varphi^o \nabla \psi_\lambda).
\end{align*}

Therefore,
\begin{align*}
\mathcal L_{u_1}[\varphi^o]
&=
\Delta \varphi^o
- \Bigl[
\nabla \cdot ( \frac{\alpha }{\lambda^2} U \nabla \psi^o ) \chi
+
\frac{\alpha }{\lambda^2} U \nabla \chi \cdot \nabla  \psi^o
+
\nabla \cdot ( \varphi_\lambda \nabla  \psi^o ) \chi
\\
& \quad
+
\nabla \cdot ( \varphi_\lambda \nabla  \psi^o ) (1-\chi )
\Bigr]
\\
& \quad
-\Bigl[
\nabla \varphi^o \cdot \nabla \Gamma_0
-
\frac{1}{\lambda^2} U \varphi^o
+
(\alpha-1) \nabla\cdot ( \varphi^o \nabla \Gamma_0)
\\
& \quad
+
\nabla\cdot ( \varphi^o \nabla \mathcal R)
+
\nabla\cdot ( \varphi^o \nabla \psi_\lambda)
\Bigr].
\end{align*}
Based on the previous formulas we formulate the inner equation
\begin{align*}
\lambda^4 \partial_t( \frac{1}{\lambda^2} \phi^i )
 & = L [\phi^i]
- (\alpha-1) \nabla_y \cdot ( U \nabla_y \psi)
- (\alpha-1) \nabla_y \cdot ( \phi^i \nabla \Gamma_0)
+ \lambda^4 S(u_1)
\\
& \quad
- \lambda^2 \nabla_y \cdot( \varphi_\lambda \nabla_y \psi^o)
- \lambda^2 \nabla_y \cdot( \varphi^o \nabla_y \psi_\lambda)
+ \lambda^2 U \varphi^o
- \alpha \nabla_y \cdot ( U \nabla_y \psi^o)
\\
& \quad
- \lambda^2 \nabla_y \cdot ( \varphi_\lambda \nabla_y \psi)
- \nabla_y \cdot( \phi^i \nabla_y \psi_\lambda)
- (\alpha-1) \lambda^2 \nabla \cdot ( \varphi^o \nabla \Gamma_0 )
\\
& \quad
\ch{- \alpha
\nabla_y \cdot ( U \nabla_y ( \hat \psi - \psi ) ) }
\ch{-\lambda^2 \nabla_y \cdot (  \varphi_\lambda  \nabla_y (\hat \psi - \psi))}
- \nabla_y \cdot ( (\phi^i \ch{\chi} + \ch{\lambda^2}\varphi^o) \nabla_y ( \hat \psi + \psi^o)) ,
\end{align*}
where
\begin{align}
\label{def-L}
\index{$L$, linerized operator around $U$}
L[\phi] = \Delta_y \phi
-  \nabla_y \cdot ( U \nabla_y \psi)
-  \nabla_y \cdot ( \phi \nabla \Gamma_0) .
\end{align}

We slightly modify the inner equation into the form
\begin{align}
\label{inner}
\lambda^2 \partial_t \phi^i
& = L [\phi^i]
+ B_0 [\phi^i]
+ E_1 \tilde \chi + F (\phi^i,\varphi^o,\textbf{p})\tilde \chi
\end{align}
where
\[
\textbf{p} = (\lambda,\alpha,\xi),
\]
\begin{align}
\nonumber
E_1(y,t)  = \lambda^4 S(u_1(\textbf{p}))(x,t), \quad y = \frac{x-\xi}{\lambda},
\end{align}
\begin{align}
\nonumber
F(\phi^i,\varphi^o,\textbf{p}) &=
- \lambda^2 \nabla_y \cdot( \varphi_\lambda \nabla_y \psi^o)
- \lambda^2 \nabla_y \cdot( \varphi^o \nabla_y \psi_\lambda)
+ \lambda^2 U \varphi^o
\\
\nonumber
 & \quad
- (\alpha-1) \lambda^2 \nabla_y \cdot ( \varphi^o \nabla_y \Gamma_0 )
- \alpha \nabla_y \cdot ( U \nabla_y \psi^o)
\\
\nonumber
& \quad
+ \lambda \dot \xi \cdot \nabla_y \phi^i
- \lambda^2 \nabla_y \cdot ( \varphi_\lambda \nabla_y \psi)
- \nabla_y \cdot( \phi^i \nabla_y \psi_\lambda)
\\
\nonumber
& \quad
- (\alpha-1) \nabla_y \cdot ( U \nabla_y \psi)
- (\alpha-1) \nabla_y \cdot ( \phi^i \nabla_y \Gamma_0)
\\
\nonumber
& \quad
- \alpha
\nabla_y \cdot ( U \nabla_y ( \hat \psi - \psi ) )
- \lambda^2 \nabla_y \cdot ( \varphi_\lambda \nabla_y ( \hat \psi - \psi ) )
\\
\label{def-F}
&\quad
- \nabla_y \cdot ( (\phi^i \chi + \lambda^2 \varphi^o) \nabla_y ( \hat \psi + \psi^o)) , \qquad
\hat \psi = (-\Delta_y)^{-1} ( \phi^i \chi) ,
\end{align}
\index{$F$, RHS of the inner equation v.1}
\begin{align}
\label{defB0}
B_0[\phi^i] &= \lambda \dot \lambda ( 2\phi^i + y \cdot \nabla_{\ch{y}} \phi^i ) ,
\end{align}
and
\begin{align}
\label{tilde-chi}
\tilde \chi(y,t) = \chi_0\Bigl( \frac{\lambda y}{2\sqrt t}\Bigr) ,
\end{align}
\index{$\tilde\chi$}
\ch{with $\chi_0$ as in \eqref{chi0}.}
Similarly we formulate the outer equation as
\begin{align}
\label{outer}
\partial_t \varphi^o
&= \Delta \varphi^o - \nabla \Gamma_0 \cdot \nabla \varphi^o
+ G(\phi^i,\varphi^o,\mathbf{p})
\end{align}
where
\begin{align}
\nonumber
G(\phi^i,\varphi^o,\mathbf{p}) &=
S(u_1,\textbf{p} ) (1-\chi)
+ \frac{2}{\lambda^2} \nabla \chi \cdot \nabla  \phi^i + \frac{1}{\lambda^2} \phi^i \Delta \chi -
\frac{1}{\lambda^2} \phi^i \partial_t \chi
- \frac{\alpha}{\lambda^2} \phi^i
\nabla \chi \cdot \nabla \Gamma_0
\\
\nonumber
& \quad
+ \frac{1}{\lambda^2} U \varphi^o (1-\chi)
- \alpha \lambda^2 U \nabla \chi\cdot \nabla \psi^o
- \nabla \cdot ( \varphi_\lambda \nabla \psi^o)
(1-\chi)
\\
\nonumber
& \quad
-(\alpha-1) \nabla\cdot(\varphi^o\nabla \Gamma_0) (1-\chi)
- \nabla \cdot (\varphi^o \nabla \mathcal R)
- \nabla\cdot(\varphi^o\nabla \psi_\lambda) (1-\chi)
\\
\nonumber
& \quad
- \frac{1}{\lambda^2}
\nabla \cdot( \phi^i \nabla \mathcal R ) \chi
- \frac{1}{\lambda^2} \phi^i \nabla \chi \cdot \nabla \mathcal R
- \frac{1}{\lambda^2} \phi^i \nabla \chi \cdot \nabla \psi_\lambda
\\
\nonumber
& \quad
- \frac{\alpha}{\lambda^2} U \nabla \chi \cdot \nabla \hat \psi
- \nabla \cdot ( \varphi_\lambda \nabla ( \hat \psi - \psi ) )(1-\chi)
\\
\label{def-G}
& \quad
\ch{
-\nabla( \varphi_\lambda \nabla  \psi ) (1-\chi)}
- \nabla \cdot ( ( \frac{1}{\lambda^2}\phi^i\chi + \varphi^o) \nabla ( \hat \psi + \psi^o)) (1-\chi) .
\end{align}
\index{$G$, RHS of the inner equation v.1}

%

\ch{If $\phi^i$, $\varphi^o$ is a solution to system \eqref{inner}, \eqref{outer}, then $u$} given by \eqref{def-u}, \eqref{def-varphi} satisfies the Keller-Segel system \eqref{ks1}.
\index{gluing 1:  \eqref{inner}, \eqref{outer}}

\subsection{Choice of \texorpdfstring{$\lambda_0$}{lambda0} and \texorpdfstring{$\alpha_0$}{alpha0}}
\label{section-lambda0-alpha0}
%
%
We explain the choice of $\lambda_0$ in the context of the elliptic equation
\begin{align}
\label{eqh}
L[\phi]=h \quad \text{in }\R^2,
\end{align}
where $h$ is radial.

\begin{lemma}
\label{lemma-radial-L}
Let $h(y)$ be a radial function such that
\begin{align*}
\| (1+|y|)^{\gamma} h(y) \|_{L^\infty(\R^2)} < \infty,
\end{align*}
for some 
$\gamma >4$
and satisfying
\begin{align}
\label{masszero}
\int_{\R^2} h(y) dy = 0
\\
\label{secondMomenth}
\int_{\R^2} h(y) |y|^2 dy = 0 .
\end{align}
Then there exists a radial solution $\phi(y) $ of equation \eqref{eqh} such that
\begin{align}
\label{estPhi}
|\phi(y) | \leq C  \| (1+|y|)^{\gamma} h(y) \|_{L^\infty(\R^2)}
\frac{1}{(1+|y|)^{\gamma-2}}
, \quad \text{if } \gamma\not=6
\end{align}
\begin{align}
\label{est-gamma-6}
|\phi(y) | \leq C  \| (1+|y|)^{\gamma} h(y) \|_{L^\infty(\R^2)}
\frac{\log(1+|y|)}{(1+|y|)^{4}}
, \quad \text{if }\gamma=6 ,
\end{align}
and
\begin{align}
\label{zeroMassphi}
\int_{\R^2} \phi(y) dy = 0.
\end{align}
\end{lemma}
\begin{proof}
Defining $g = \frac{\phi}{U}- (-\Delta)^{-1} \phi$ we obtain the equation
\begin{align}
\label{eq-g}
\nabla \cdot ( U \nabla g ) = h  .
\end{align}
Assuming $\gamma>6$ we choose the radial function $g$ defined by
\[
g(\rho) = - \int_\rho^\infty \frac{1}{r U(r)} \int_0^r h(s) s ds dr , \quad \rho = |y|,
\]
and using \eqref{masszero} we get
\[
|g(\rho) | \leq C \| (1+|y|)^{\gamma} h \|_{L^\infty(\R^2)}
\frac{1}{(1+|y|)^{\gamma-6} }.
\]
\ch{Now} we solve Liouville's equation
\begin{align}
\label{Liouville-0}
-\Delta \psi - U \psi = U g \quad \text{in }\R^2, \quad\psi(\rho)\to 0 \quad \text{as }\rho\to\infty .
\end{align}
\ch{Multiplying \eqref{eq-g} by $|y|^2$ and using} \eqref{secondMomenth} we see that
\[
\int_{\R^2} \ch{g Z_0} dy \ch{= \frac{1}{2} \int_{\R^2} h(y) |y|^2 dy}= 0 ,
\]
\ch{with $Z_0$ defined in \eqref{Z0}.}
Then by the variations of parameter formula we find that \eqref{Liouville-0} has a unique solution $\psi$, which satisfies
\begin{align}
\label{decayPsi}
|\psi(y)| + (1+|y|)|\nabla \psi(y)| \leq \| (1+|y|)^{\gamma} h \|_{L^\infty(\R^2)}
\frac{1}{(1+|y|)^{\gamma-4} }.
\end{align}
Then we see that $\phi $ defined by $\phi = U g + U \psi $ satisfies \eqref{eqh}, \eqref{estPhi}
and \eqref{zeroMassphi}
because \ch{$\phi = - \Delta \psi$} and $\psi$ has the decay \eqref{decayPsi}.

If \ch{$4<\gamma \leq 6$} we do almost the same, except that we define
\[
g(\rho) = \int_0^\rho \frac{1}{r U(r)} \int_0^r h(s) s ds dr.
\]
\end{proof}

\begin{remark}
\ch{We observe that $L[Z_0]=0$.
This can also be seen in the context of the Lemma~\ref{lemma-radial-L}, where $\phi=Z_0$ which corresponds to $g$ being constant. Indeed, suppose $g \equiv 1$. Then from \eqref{Liouville-0} $\psi = -1 - \frac{1}{2}z_0$, where $z_0$ is defined in  \eqref{defZLiouville}.
This gives $\phi = U g + U \psi = -\frac{1}{U}z_0 = -\frac{1}{2}Z_0$. This shows that $L[Z_0]=0$.}

\ch{If $h$ doesn't satisfy the zero second moment condition \eqref{secondMomenth}, then a solution still exists but with worse decay and non-zero mass.
More precisely,  if $h$ is radial, $\| (1+|y|)^{\gamma} h(y) \|_{L^\infty(\R^2)} < \infty$
for some $\gamma > 6$, and satisfies only \eqref{masszero}, then one can construct a solution $\phi$ to \eqref{eqh}, but any such solution has the estimate
\[
|\phi(y)|\leq  C  \| (1+|y|)^{\gamma} h(y) \|_{L^\infty(\R^2)}
\frac{\log(1+|y|)}{(1+|y|)^{4}} ,
\]
so worse decay than the one in \eqref{estPhi}.
Moreover, the mass of $\phi$ becomes
\begin{align*}
\int_{\R^2} \phi  = -\int_{\R^2} \Delta \psi = - \int_{\R^2} g Z_0 = - \frac{1}{2} \int_{\R^2} h(y) |y|^2 dy .
\end{align*}}
\end{remark}

For the inner equation \eqref{inner} it is then natural to impose that the first error $S(u_1) \chi$ satisfies the second moment condition
\begin{align*}
\int_{\R^2} S(u_1) \chi |y|^2 dy = 0,
\quad \text{for all }t>t_0.
\end{align*}

The next lemma gives a way of expressing the second moment of $u_1$.

\begin{lemma}
\label{second-moment-u1}
Let $u_1$ be defined in \eqref{defu1}. Then
\begin{align}
\nonumber
\int_{\R^2} S(u_1) |x-\xi|^2dx
&=
4 \int_{\R^2} \varphi_\lambda dx
-  \alpha \int_{\R^2} \tilde E(x-\xi,t;\lambda) |x-\xi|^2 dx
\\
\nonumber
& \quad
+ \int_{\R^2} \nabla \varphi_\lambda \, dx
\cdot \dot \xi
- \frac{\dot \alpha}{\lambda^2} \int_{\R^2} U \chi |x-\xi|^2 dx
- (1-\alpha) \int_{\R^2} E (x-\xi,t;\lambda) |x-\xi|^2dx
\\
\label{second-moment-Su1}
& \quad
+ 4 \Bigl( \int_{\R^2} u_0 + \int_{\R^2} \varphi_\lambda \Bigr)
\Bigl( 1 - \frac{1}{8\pi} \int_{\R^2} u_0
 - \frac{1}{8\pi} \int_{\R^2} \varphi_\lambda \Bigr).
\end{align}
where $E$, $\tilde E$ are defined in \eqref{defE}, \eqref{defTildeE}.
\end{lemma}

\begin{proof}[Proof of Lemma~\ref{second-moment-u1}]
Using \eqref{momento} we see that
\begin{align*}
\int_{\R^2} S(u_1) |x-\xi|^2dx
&=
-\int_{\R^2} \partial_t u_0 |x-\xi|^2dx
-\int_{\R^2} \partial_t \varphi_\lambda|x-\xi|^2dx
\\
& \quad
+ 4 \Bigl( \int_{\R^2} u_0 + \int_{\R^2} \varphi_\lambda \Bigr)
\Bigl( 1 - \frac{1}{8\pi} \int_{\R^2} u_0
 - \frac{1}{8\pi} \int_{\R^2} \varphi_\lambda \Bigr) .
\end{align*}
But recall that $\varphi_\lambda(x,t) = \tilde \varphi_\lambda(x-\xi(t),t)$ where $\tilde \varphi_\lambda$ satisfies \eqref{deftildephilambda}. Multiplying that equation by $|\zeta|^2$ and integrating on $\R^2$ results in
\begin{align*}
\int_{\R^2} \partial_t \tilde \varphi_\lambda |\zeta|^2 \,d\zeta = - 4 \int_{\R^2} \tilde \varphi_\lambda \,d\zeta
+ \int_{\R^2} E(\zeta,t)|\zeta|^2 \,d\zeta.
\end{align*}
Therefore
\begin{align*}
\int_{\R^2} \partial_t\varphi_\lambda|x-\xi|^2\,dx
=
-4\int_{\R^2} \varphi_\lambda \,dx
- \frac{(x-\xi)\cdot\dot\xi}{|x-\xi|} \int_{\R^2} \partial_r \tilde\varphi_\lambda
+ \int_{\R^2} E(\zeta,t)|\zeta|^2 \,d\zeta
\end{align*}
and then
\begin{align}
\nonumber
\int_{\R^2} S(u_1) |x-\xi|^2dx
&=  -\int_{\R^2} \partial_t u_0 |x-\xi|^2dx
+ 4 \int_{\R^2} \varphi_\lambda dx
+  \int_{\R^2} \nabla \varphi_\lambda \, dx
\cdot \dot \xi
-  \int_{\R^2} E(x-\xi,t)|x-\xi|^2dx
\\
\label{Su1-a}
& \quad
+ 4 \Bigl( \int_{\R^2} u_0 + \int_{\R^2} \varphi_\lambda \Bigr)
\Bigl( 1 - \frac{1}{8\pi} \int_{\R^2} u_0
 - \frac{1}{8\pi} \int_{\R^2} \varphi_\lambda \Bigr) .
\end{align}
But from the formula for $ \partial_t u_0$ \eqref{ptu0} and the definitions of $E$ and $\tilde E$ \eqref{defE}, \eqref{defTildeE} we get
\begin{align*}
-\partial_t u_0(x,t) =
- \frac{\dot{\alpha}}{\lambda^2} U(y) \chi_0(z) + \alpha E (x-\xi,t) - \alpha \tilde E(x-\xi,t).
\end{align*}
Hence
\begin{align*}
& \int_{\R^2} ( \partial_t u_0 +E(x-\xi,t))|x-\xi|^2dx
\\
& \quad =
\int_{\R^2} ( \partial_t u_0 +\alpha E(x-\xi))|x-\xi|^2dx
+ (1-\alpha) \int_{\R^2} E(x-\xi,t) |x-\xi|^2dx
\\
& \quad =
\frac{\dot \alpha}{\lambda^2} \int_{\R^2} U \chi |x-\xi|^2 dx
+  \alpha \int_{\R^2} \tilde E(x-\xi,t) |x-\xi|^2 dx
+(1-\alpha) \int_{\R^2} E(x-\xi,t) |x-\xi|^2dx .
\end{align*}
Replacing this in \eqref{Su1-a} we obtain \eqref{second-moment-Su1}.
\end{proof}

In the definition \eqref{defu1} of $u_1$  we will stress the dependence on the parameters by writing $\mathbf{p} = (\lambda,\alpha,\xi)$ and $u_1 = u_1(\mathbf{p})$.
\ch{At this point we would like to construct $\lambda_0$ and $\alpha_0$} so that setting $\mathbf{p}_0 = ( \lambda_0 , \alpha_0 , 0 ) $ we have
\begin{align}
\label{massU1}
\int_{\R^2} u_1(\mathbf{p}_0) dx
& = 8 \pi,
\\
\label{secondMomentSu1}
\int_{\R^2} S(u_1(\mathbf{p}_0))|x-\xi|^2 dx & = O \Bigl( \frac{1}{t^{\frac{3}{2}+\sigma}}\Bigr) ,
\end{align}
for some $\sigma>0$. The reason for allowing in \eqref{secondMomentSu1} an error is that it is difficult to solve with right hand side equal to 0 and a remainder of size $\ch{O(t^{-\frac{3}{2}-\sigma})}$ with $\sigma>0$ is sufficiently small to proceed with the rest of the construction.

\ch{Assuming that \eqref{massU1} holds,}
we get
\begin{align*}
\int_{\R^2} S(u_1) |x-\xi|^2dx
&=
4 \int_{\R^2} \varphi_\lambda dx
-  \alpha \int_{\R^2} \tilde E(x-\xi,t;\lambda) |x-\xi|^2 dx
\\
\nonumber
& \quad
+ \int_{\R^2} \nabla \varphi_\lambda \, dx
\cdot \dot \xi
- \frac{\dot \alpha}{\lambda^2} \int_{\R^2} U \chi |x-\xi|^2 dx
- (1-\alpha) \int_{\R^2} E (x-\xi,t;\lambda) |x-\xi|^2dx .
\end{align*}
It turns out that the main terms in the expression for $\int_{\R^2} S( \ch{u_1} ) |x-\xi|^2 dx$ are the first two.
So the equation
\[
\int_{\R^2} S(u_1(\mathbf{p_0})) |x-\xi|^2 dx = 0
\]
is at main order given by
\begin{align*}
4 \int_{\R^2} \varphi_\lambda dx
- \int_{\R^2} \tilde E |x-\xi|^2dx
&= 0 .
\end{align*}
It will be shown later that
\begin{align}
\label{secondMtildeE}
\int_{\R^2}
\tilde E |x-\xi|^2 dx = - 64 \pi \Upsilon \frac{\lambda^2}{t}
+ O \Big( \frac{ \lambda^4 }{t^2} \Big) ,
\end{align}
see Lemma~\ref{lemma-second-moment-tE}, where $\Upsilon$ is given in \eqref{defUpsilon}, so that the equation we want to solve becomes at main order,
\begin{align*}
\int_{\R^2} \varphi_\lambda dx + 16 \pi \Upsilon \frac{\lambda^2}{t} =0.
\end{align*}
In
\S\ref{sect-mass-varphilambda} we will show that
\begin{align}
\label{massVP1}
\int_{\R^2} \varphi_\lambda dx
=  -4\pi \int_{t/2}^{t-\lambda^2} \frac{\lambda \dot \lambda}{t-s}ds
-2\pi \frac{\lambda^2}{t}
- 16 \pi \Upsilon \frac{\lambda^2}{t}
+ O \Big( \frac{ \lambda^4 \log \log t}{t} \Big)
\end{align}
see Corollary~\ref{coro-mass-vp-lambda}.
Using \eqref{massVP1} we see that
\ch{
\begin{align}
\label{mass-vp-plus-16pil2ot}
\int_{\R^2} \varphi_\lambda dx + 16 \pi \Upsilon \frac{\lambda^2}{t}
=
-4 \pi \Bigl[
\int_{t/2}^{t-\lambda^2} \frac{\lambda \dot \lambda}{t-s}ds
+ \frac{\lambda^2}{2t}\Bigr]  + O \Big( \frac{ \lambda^4 \log \log t}{t} \Big)
\end{align}
}
\ch{so that the equation for $\lambda$ is at main order}
\ch{
\begin{align}
\nonumber
\int_{t/2}^{t-\lambda^2} \frac{\lambda \dot \lambda}{t-s}ds
+ \frac{\lambda^2}{2t} = 0.
\end{align}
}
\ch{One can check that $\lambda^*(t) = \frac{c_0}{\sqrt{\log t}}$, where $c_0>0$ is an arbitrary constant, is an approximate solution. Indeed
\begin{align*}
\int_{t/2}^{t-(\lambda^*)^2} \frac{\lambda^*(s) \dot \lambda^*(s)}{t-s}ds
+ \frac{(\lambda^*)^2}{2t}
&\approx
\lambda^*(t) \dot \lambda^*(t)\int_{t/2}^{t-(\lambda^*)^2} \frac{ds}{t-s}
+ \frac{\lambda^*(t)^2}{2t}
\\
&\approx
\lambda^*(t) \dot \lambda^*(t) \log t
+ \frac{\lambda^*(t)^2}{2t}
\\
&= \frac{1}{2}
\frac{d}{dt} \bigl[
\lambda^*(t)^2 \log t
\Bigr]=0.
\end{align*}
}

\ch{The error left out in the  approximation \eqref{mass-vp-plus-16pil2ot} is too big. We give next a result that shows that for an appropriate modification of $\lambda^*$ we can achieve a smaller error.}
Let us write $\tilde E(\lambda)$ the expression defined in \eqref{defTildeE} with the explicit dependence on $\lambda$.
\begin{prop}
\label{prop-lambda0}
Let $c_0>0$ be fixed.
For $t_0>0$ sufficiently large there exists $\lambda_0:[\frac{t_0}{2},\infty) \to (0,\infty) $ such that
\begin{align}
\label{eq-lambda0}
\int_{\R^2} \varphi_{\lambda_0} dx
- \frac{1}{4}\int_{\R^2} \tilde E(\lambda_0) |x-\xi|^2dx
& = O\Bigl( \frac{1}{t^{\frac{3}{2}+\sigma}}\Bigr), \quad t >t_0 ,
\end{align}
for some $\sigma>0$. Moreover, for arbitrarily $\varepsilon>0$ small, $\lambda_0$ has the expansion
\begin{align*}
\lambda_0(t) &= \frac{c_0}{\sqrt{\log t}}
+ O\Bigl( \frac{1}{(\log t)^{\frac{3}{2}-\varepsilon}}\Bigr),
\\
\dot\lambda_0(t) &= -\frac{c_0}{2 t (\log t)^{3/2}}
+ O\Bigl(   \frac{1}{t (\log t)^{\frac{5}{2} -\varepsilon} }  \Bigr),
\\
|\ddot \lambda_0(t) | & \leq  \frac{C}{t^2 (\log t)^{3/2}} ,
\end{align*}
as $  t \to \infty$.
\end{prop}

We will prove this result in \S\ref{subsect-prop-lambda0}.

\ch{Once $\lambda_0$ is constructed in Proposition~\ref{prop-lambda0} we choose $\alpha_0$ so that \eqref{massU1} holds, by imposing
\begin{align}
\label{def-alpha0}
\alpha_0(t) \int_{\R^2} U(y) \chi_0\Bigl(\frac{\lambda_0(t) y}{\sqrt t}\Bigr)\,dy + \int_{\R^2} \varphi_{\lambda_0}(x,t)\,dx=8\pi, \quad t>t_0.
\end{align}
}

We note that by \eqref{expansion-mass}, \eqref{eq-lambda0} and \eqref{secondMtildeE} we get
\begin{align*}
\alpha_0(t)
&= O\Bigl( \frac{1}{t^{\frac{3}{2}+\sigma}}\Bigr)
\end{align*}
as $t\to\infty$.
A byproduct of the proof of Proposition~\ref{prop-lambda0} is that
\begin{align}
\label{bound-dt-mass-varphilambda0}
\left| \frac{d}{dt} \int_{\R^2} \varphi_{\lambda_0}dx\right|\leq \frac{C}{t^2} ,
\end{align}
and from this and \eqref{def-alpha0} we get
\begin{align}
\label{bound-dt-alpha0}
|\dot\alpha_0(t)|\leq  \frac{C}{t^2} .
\end{align}

As a corollary \ch{of Proposition~\ref{prop-lambda0}} we get:
\begin{corollary}
\label{cor-second-moment-Su10}
Let $\textbf{p}_0=(\lambda_0,\alpha_0,0)$
with
$\alpha_0$ defined by \eqref{massU1} and $\lambda_0$ be given by Proposition~\ref{prop-lambda0}. Then
\begin{align*}
\int_{\R^2} S(u_1(\textbf{p}_0)) |x-\xi|^2 dx
=  O \Bigl( \frac{1}{t^{\frac{3}{2}+\sigma}}\Bigr) ,
\end{align*}
for some $\sigma>0$.
\end{corollary}
\begin{proof}

Using Lemma~\ref{second-moment-u1}  we have
\begin{align*}
\int_{\R^2} S(u_1) |x-\xi|^2 dx
& =4 \int_{\R^2} \varphi_{\lambda_0} dx
-  \int_{\R^2} \tilde E(x-\xi,t;\lambda_0) |x-\xi|^2 dx
\\
& \quad
- \frac{\dot \alpha_0}{\lambda^2} \int_{\R^2} U \chi |x-\xi|^2 dx
- (1-\alpha_0) \int_{\R^2} E(x-\xi,t;\lambda_0) |x-\xi|^2dx
\\
&= O \Bigl( \frac{1}{t^{\frac{3}{2}+\sigma}}\Bigr) ,
\end{align*}
for some $\sigma>0$, since $\dot \alpha_0(t) = O( \frac{1}{t^2 \log t})$ and
\begin{align*}
\int_{\R^2} E(x-\xi,t;\lambda_0) |x-\xi|^2dx = O\Bigl( \frac{\lambda_0^2}{t}\Bigr)
\end{align*}
by \eqref{secondMtildeE} and a direct estimate for the remaining terms in $E$ (c.f. \eqref{defE}).
\end{proof}

\subsection{A further improvement of the approximation}

We introduce a correction $\phi_0^i(y) $, $y=\frac{x-\xi}{\lambda}$ in the inner approximation to eliminate the radial part of $S(u_1(\textbf{p})) $ (defined in \eqref{Su1}), which we define as
\begin{align}
\nonumber
S_0(u_1(\textbf{p}))
&=
- \frac{\dot{\alpha}}{\lambda^2} U(y) \chi
+ (\alpha-1) \frac{\dot\lambda}{\lambda^3} Z_0 \chi
+  \frac{(\alpha -1)}{2 t }  \frac{1}{\lambda^2} U
\nabla_z\chi_0 \cdot \frac{x-\xi}{\sqrt t}
\\
\nonumber
& \quad
+ \frac{2(\alpha-1)}{\lambda^3 t^{1/2} }
\nabla_z \chi_0 \cdot \nabla_y U
+ \frac{(\alpha-1) }{t} \Delta \chi_0 \frac{1}{\lambda^2} U
\ch{
-\frac{\alpha^2-1}{\lambda^3 \sqrt t} U \nabla_z \chi_0 \cdot \nabla_y \Gamma_0
- \frac{\alpha}{\lambda^3 \sqrt t} U \nabla_z \chi_0 \cdot \nabla_y \mathcal R }
\\
\nonumber
& \quad
- \frac{\alpha (\alpha-1) \chi}{\lambda^4}
\nabla_y \cdot ( U \nabla_y \Gamma_0)
+ \frac{\alpha^2 \chi ( 1-\chi)}{\lambda^4} U^2
- \frac{\alpha \chi}{\lambda^4} \nabla _y U \cdot \nabla_y \mathcal R.
\\
\label{S0u1}
&\quad
- \frac{4}{r} \partial_r \varphi_\lambda
- \nabla\cdot ( \varphi_\lambda \nabla v_0) - \nabla \cdot ( u_0 \nabla \psi_\lambda )
-  \nabla\cdot ( \varphi_\lambda \nabla \psi_\lambda).
\end{align}
\index{$S_0$}
With this definition
\begin{align*}
S(u_1) = S_0(u_1) +  \frac{\alpha}{\lambda^3} \dot\xi \cdot \nn_y U (y) \, \chi
+ \frac{\alpha}{\ch{\lambda^2} \sqrt t} U(y)
\dot\xi \cdot  \nabla \chi_0 ,
\end{align*}
and the terms not in $S(u_1)$ correspond to $ \frac{\alpha}{\lambda^3} \dot\xi \cdot \nn_y U (y) \, \chi
+ \frac{\alpha}{\ch{\lambda^2} \sqrt t} U(y)$ which are in mode 1.

Then we want $\phi_0^i$ to be an appropriate solution to the equation
\begin{align}
\label{eq-phii0}
L[\phi^i_0] + \lambda^4 S_0(u_1(\textbf{p}_0))(x,t) = c_0(t) W_2 \quad \text{in }\R^2,
\quad
x = \xi + \lambda y,
\end{align}
where $L$ is the linear operator \eqref{def-L},
$t>t_0$ is regarded as a parameter,
$W_2(y)$ is a fixed smooth radial function with compact support, and
\begin{align}
\label{defW2}
\int_{\R^2} W_2(y)dy=0, \quad
\int_{\R^2} W_2(y)|y|^2dy=1.
\end{align}
\index{$W_2$}

By Lemma~\ref{second-moment-u1} and Proposition~\ref{prop-lambda0}, the choice $\textbf{p} = \textbf{p}_0$ is so that \eqref{massU1}, \eqref{secondMomentSu1} hold. Since the difference between $S(u_1)$ and $S_0(u_1)$ contains terms in mode 1 only, we get from Corollary~\ref{cor-second-moment-Su10}
\begin{align}
\label{secondMomentS0u1}
\int_{\R^2} \lambda^4 S_0(u_1(\mathbf{p}_0))|y|^2 dy & = O \Bigl( \frac{1}{t^{\frac{3}{2}+\sigma}}\Bigr) .
\end{align}
In \eqref{eq-phii0} we select $c_0(t)$ such that
\begin{align*}
\int_{\R^2}
[ \lambda^4S_0(u_1(\textbf{p}_0)) + c_0(t) W_2] |y|^2dy = 0, \quad t>t_0
\end{align*}
and thanks to \eqref{secondMomentS0u1} we have
\begin{align}
\label{estc0}
|c_0(t)|\leq \frac{C}{t^{\frac{3}{2}+\sigma}} , \quad t>t_0 .
\end{align}

Note that we have
\[
\int_{\R^2} S_0(u_1(\textbf{p}_0)) d x =0,
\]
which follows from the constant mass in time of $u_1(\mathbf{p}_0)$ in \eqref{massU1} and the form of the operator $S_0$ \eqref{S0u1}.

We let $\phi^i_0$ be the solution
to \eqref{eq-phii0} constructed in Lemma~\ref{lemma-radial-L}.
\ch{By \eqref{est-gamma-6} and \eqref{est-Su1-inner}}
\begin{align}
\label{est-Phii0}
|\phi^i_0(y,t)|
\leq  \ch{\frac{C}{t}}
\frac{\log ( 1+|y|)}{1+|y|^4},
\end{align}
and
\begin{align*}
\int_{\R^2} \phi^i_0 (y,t)dy=0, \quad t>t_0.
\end{align*}

\subsection{Reformulation of the system}

In the outer problem \eqref{outer} we would like to separate the effect of the initial condition from the coupling $G(\phi^i,\varphi^o,\mathbf{p})$.

We take the initial condition in \eqref{outer} to be
\[
\varphi^o(\cdot,t_0)= \varphi_0^* ,
\]
and let $\ch{\varphi^*(x,t)}$ denote the solution of
\begin{align}
\label{eq-cond-in}
\left\{
\begin{aligned}
\partial_t \ch{\varphi^*}
&= \Delta \ch{\varphi^*}
-  \nabla_x \Gamma_0 \Big (\frac{x-\xi}{\lambda}\Big)
\cdot \nabla \ch{\varphi^*}
\quad \text{in }\R^2 \times (t_0,\infty)
\\
\ch{\varphi^*} (\cdot,t_0)&=\ch{\varphi^*_0}
\quad \text{in }\R^2.
\end{aligned}
\right.
\end{align}
The initial condition $\varphi^*_0(x)$  will be later used to prove the stability claimed in Theorem~\ref{teo1}. The topology for $\varphi_0^*$ will be specified later on.

Note that $\ch{\nabla_x \Gamma_0 (\frac{x-\xi}{\lambda})= -4 \frac{x-\xi}{|x-\xi|^2 +\lambda^2}}$ so that $\ch{\varphi^*} $ is a function of the parameters \ch{$\lambda,\xi$}. Therefore we will write $\ch{ \varphi^*}(x,t;\textbf{p})$ when convenient.

\medskip

We decompose
\begin{align}
\label{decompPhii}
\left\{
\begin{aligned}
\phi^i &= \phi^i_0 + \phi
\\
\varphi^o &= \ch{\varphi^*} + \varphi
\\
\textbf{p} &= \textbf{p}_0 + \textbf{p}_1
\end{aligned}
\right.
\end{align}
where
\begin{align}
\nonumber
\textbf{p}_0 = (\lambda_0,\alpha_0,0) , \quad
 \textbf{p}_1 = (\lambda_1,\alpha_1,\xi_1) ,
\end{align}
with $\lambda_0$ the function constructed in Proposition~\ref{prop-lambda0} and $\alpha_0$ chosen so that \ch{\eqref{massU1} holds}.

%

\medskip

We substitute the expressions for $\phi^i$, $\varphi^o$ and $\textbf{p}$ in
\eqref{decompPhii} into the equations \eqref{inner}, \eqref{outer},
\ch{and are led} to the following problem for  $\phi$, $\varphi$
\begin{align}
\label{inner3}
\left\{
\begin{aligned}
\lambda^2 \partial_t \phi
& = L [\phi]
+ B_0 [\phi]
+ E_2 \tilde \chi_2+ F_2 (\phi,\varphi,\textbf{p}_1,\ch{\varphi^*_0})\tilde \chi
\quad \text{in }\R^2 \times (t_0,\infty)
\\
\phi(\cdot,t_0) &=  \phi_0 \quad \text{in }\R^2
\end{aligned}
\right.
\end{align}
\begin{align}
\label{outer3}
\left\{
\begin{aligned}
\partial_t \varphi
&= \Delta \varphi - \ch{\nabla_x \Gamma_0(\frac{x-\xi}{\lambda})} \cdot \nabla \varphi
+ G_2(\phi,\varphi,\mathbf{p}_1,\ch{\varphi^*_0})
\quad \text{in }\R^2 \times (t_0,\infty)
\\
\varphi (\cdot,t_0) &=  0
\quad \text{in }\R^2 ,
\end{aligned}
\right.
\end{align}
\index{gluing 2:  \eqref{inner3}, \eqref{outer3}}
where \ch{$\tilde\chi$ is defined in \eqref{tilde-chi}},
\begin{align}
\nonumber
E_2 &= - \partial_t \phi^i_0 + B_0[\phi^i_0]
+ c_0(t) W_2
\end{align}
\index{$E_2$}
\begin{align}
\nonumber
F_2 (\phi,\varphi,\textbf{p}_1,\ch{\varphi^*_0})
&= F(\phi^i_0 + \phi ,\ch{ \varphi^*} + \varphi ,\textbf{p}_0+\textbf{p}_1)
+ \lambda^4 [ S_0( u_1( \textbf{p}_0 + \textbf{p}_1 ) ) - S_0( u_1( \textbf{p}_0) ) ]
\\
\label{def-F2}
& \quad
+\lambda \alpha  \dot\xi_1 \cdot \nn_y U (y) \, \chi
+ \frac{\alpha \ch{\lambda^2}}{ \sqrt t} U(y)
\dot\xi_1 \cdot  \nabla \chi_0
\end{align}
\index{$F_2$, RHS of the inner equation v.2}
\begin{align}
\label{def-G2}
G_2(\phi,\varphi,\mathbf{p}_1,\ch{\varphi^*_0})
&=  G(\phi^i_0 + \phi , \ch{\varphi^*} + \varphi ,\textbf{p}_0+\textbf{p}_1)
+ \lambda^{-4} E_2 (1-\tilde \chi_2)\chi
\end{align}
\index{$G_2$, RHS of the outer equation v.2}
\begin{align}
\nonumber
\tilde \chi_2(x,t) = \chi_0\Bigl(\frac{x-\xi}{t^{\frac{1}{2}-\delta}}\Bigr)  ,
\end{align}
\index{$\tilde\chi_2$}
$\delta>0$ is a small constant to be fixed later on, and $\chi_0$ is as in \eqref{chi0}.
We recall that $F$ and $G$ are defined in \eqref{def-F} and \eqref{def-G}.
\ch{The expressions for $F_2$ and $G_2$ depend on the initial condition $\varphi^*_0$ through $\varphi^*$ \eqref{eq-cond-in} and $\phi_0$. The role of $\phi_0$ will be clarified later on.}

By the estimate for $\ddot\lambda_0$ in Proposition~\ref{prop-lambda0} and \eqref{estc0} we get
\begin{align}
\label{est-E2}
|E_2(y,t)| \leq
\frac{C}{t^2 \ch{(\log t)^2}} \frac{\log(1+|y|)}{1+|y|^4}
+ \frac{C}{t^{\frac{3}{2}+\sigma}} |W_2(y)|,
\quad |y| \leq C \sqrt {t \log t}.
\end{align}
The reason that we introduce the cut-off $\tilde \chi_2$ is to achieve
\begin{align}
\nonumber
|E_2 \tilde \chi_2(y,t)|\leq \frac{C}{t^\nu (1+|y|)^{6+\sigma}} ,
\end{align}
if $\nu < 1 +2\delta - \frac{\sigma}{2}$. We will choose $\delta$ and $\sigma$ positive small numbers such that $2\delta-\frac{\sigma}{2}>0$ so that we can find $1<\nu< 1 +2\delta - \frac{\sigma}{2}$.

\subsection{Splitting the inner solution \texorpdfstring{$\phi$}{phi}}

We perform one more change in the formulation \eqref{inner3}, \eqref{outer3}, which consists in decomposing
\[
\phi = \phi_1 + \phi_2 .
\]
The function $\phi_1$ will solve an equation with part of the right hand side of \eqref{inner3}, which will be projected  so that it satisfies the zero second moment condition.

For any $h(y,t)$ with sufficient spatial decay we define
\begin{align}
\label{def-m0-m2}
m_0[h](t) =  \int_{\R^2} h(y,t)  dy ,\quad
m_2[h](t) = \int_{\R^2} h(y,t) |y|^2 dy ,
\end{align}
and
\begin{align*}
m_{1,j}[h](t) =
\int_{\R^2} h(y,t)y_j dy, \quad j=1,2,
\end{align*}
which denote the mass, second moment and center of mass of $h$.

Let $ W_0 \in C^\infty(\R^2)$ be radial with compact support such that
\index{$W_0$}
\[
\int_{\R^2} W_0 dy =1, \quad
\int_{\R^2}  W_0|y|^2 dy =0 .
\]
Let $W_{1,j}$, $j=1,2$ be a smooth functions with compact support and with the form $W_{1,j} (y) = \tilde W(|y|) y_j$ so that
\[
\int_{\R^2} W_{1,j}(y) y_j = 1.
\]
We recall that $W_2$ defined in \eqref{defW2}.

Then, $h - m_0[h]W_0$ has zero mass, $h-m_2[h]W_2$ has zero second moment, and $h - \ch{m_{1,1}}[h] W_{1,1} - \ch{m_{1,2}}[h]W_{1,2}$ has zero center of mass.


\medskip

We modify of the operator $B_0$ appearing in \eqref{inner3}, and defined in \eqref{defB0}.
The idea is to work with a variant of it, which coincides with it for radial functions, but for functions without radial part it is cutoff outside the region $|y|\lesssim \frac{\sqrt t}{\lambda}$.
More precisely, we decompose $\phi $ in a  radial part $[\phi]_{rad}$ defined
by
\begin{align}
\label{def-radial-part}
[\phi]_{rad}(\rho,t) = \frac 1{2\pi} \int_0^{2\pi}  \phi(\rho e^{i\theta}, t) d\theta
\end{align}
 and a term with no radial mode $\phi_1 = \phi - [\phi]_{rad}$.
We note that the other linear terms in the equation behave well with this decomposition.
Then we define
\begin{align}
\label{defB}
B[\phi] =  \lambda \dot \lambda ( 2 [\phi]_{rad} +  y \cdot \nabla [\phi]_{rad})
+ \lambda \dot \lambda ( 2\phi_1 +  y \cdot \nabla \phi_1)  \chi_0\Bigl( \frac{\lambda y}{5 \sqrt t} \Bigr)
\end{align}
where $\chi_0$ is a smooth cut-off in $\R$ with $\chi_0(s) = 1$ for $s\leq 1$ and $ \chi_0(s)=1$ for $s \geq 2$.

\medskip

With these definitions we introduce the following system for $\phi_1$, $\phi_2$,  $\varphi $, $\textbf{p}_1$,
\begin{align}
\label{inner4a}
\left\{
\begin{aligned}
\lambda^2 \partial_t \phi_1
& = L [\phi_1]
+ B [\phi_1]
+ F_3(\phi_1+\phi_2,\varphi,\textbf{p}_1,\ch{\varphi_0^*})
\\
& \quad
- m_0[F_3(\phi_1+\phi_2,\varphi,\textbf{p}_1,\ch{\varphi_0^*})] W_0
- m_2[F_3(\phi_1+\phi_2,\varphi,\textbf{p}_1,\ch{\varphi_0^*})] W_2
\\
& \quad
- m_{1,1}[F_3(\phi_1+\phi_2,\varphi,\textbf{p}_1,\ch{\varphi_0^*})] W_{1,1}
- m_{1,2}[F_3(\phi_1+\phi_2,\varphi,\textbf{p}_1,\ch{\varphi_0^*})] W_{1,2}
\\
& \qquad \qquad
\quad \text{in }\R^2 \times (t_0,\infty)
\\
\phi_1(\cdot,t_0) &= 0\quad \text{in }\R^2,
\end{aligned}
\right.
\end{align}
\begin{align}
\label{inner4b}
\left\{
\begin{aligned}
\lambda^2 \partial_t \phi_2
& = L [\phi_2]
+ B [\phi_2]
+ m_2[F_3(\phi_1+\phi_2,\varphi,\textbf{p}_1,\ch{\varphi_0^*})] W_2
\quad \text{in }\R^2 \times (t_0,\infty)
\\
\phi_2(\cdot,t_0) &=  \phi_0 \quad \text{in }\R^2,
\end{aligned}
\right.
\end{align}
\begin{align}
\label{outer4}
\left\{
\begin{aligned}
\partial_t \varphi
&= \Delta \varphi - \nabla \Gamma_0 \cdot \nabla \varphi
+ \ch{G_2(\phi_1+\phi_2,\varphi,\mathbf{p}_1,\varphi_0^*)}
\quad \text{in }\R^2 \times (t_0,\infty)
\\
\varphi (\cdot,t_0) &= 0
\quad \text{in }\R^2,
\end{aligned}
\right.
\end{align}
\index{gluing 3: \eqref{inner4a}, \eqref{inner4b}, \eqref{outer4}}
where
\begin{align}
\label{def-F3}
F_3(\phi,\varphi,\textbf{p}_1 ,\ch{\varphi_0^*})
= E_2 \tilde \chi_2 + F_2(\phi,\varphi,\textbf{p}_1 ,\ch{\varphi_0^*}) \tilde \chi,
\end{align}
\index{$F_3$, RHS of the inner equation v.3}

A solution $\phi_1$, $ \phi_2$, $\varphi$ to \eqref{inner4a}, \eqref{inner4b} and \eqref{outer4} gives a solution to the system \eqref{inner3}, \eqref{outer3} provided $\textbf{p}_1$ is such that the following equations are satisfied
\begin{align}
\label{eq-param}
\left\{
\begin{aligned}
0&=m_0[ F_3(\phi_1+\phi_2,\varphi,\textbf{p}_1,\ch{\varphi_0^*}) ]
\\
0&=m_{1,j}[ F_3(\phi_1+\phi_2,\varphi,\textbf{p}_1,\ch{\varphi_0^*}) ] , \quad j=1,2.
\end{aligned}
\right.
\end{align}

\hide{\color{red}
\todo{eliminate?}
In order for the first equation in  \eqref{eq-param} to be satisfied, we impose that
\begin{align*}
\lambda \dot \alpha_1(t) =
\frac{1}{\int_{\R^2} U \chi dy}
\int_{\R^2}
F_3^{(0)}(\phi_1+\phi_2,\varphi,\textbf{p}_1,\ch{\varphi_0^*}) dy ,
\end{align*}
where
\begin{align}
\nonumber
F_3^{(0)}(\phi,\varphi,\textbf{p}_1,\ch{\varphi_0^*})
&=
E_2 \tilde \chi_2
+ F(\phi^i_0 + \phi ,\tilde \varphi_0^* + \varphi ,\textbf{p}_0+\textbf{p}_1)
\\
\nonumber
& \quad
+
\lambda^4 [ S_0(u_1(\lambda_0+\lambda_1,\alpha_0,0)) - S_0(u_1(\textbf{p}_0)) ] \tilde \chi
\\
\nonumber
& \quad
- \alpha \alpha_1 \nabla_y \cdot ( U \nabla_y \Gamma_0 ) \chi
+\lambda \alpha  \dot\xi_1 \cdot
\nabla_y U (y) \, \chi
\\
\nonumber
& \quad + \tilde F_3(\phi,\varphi,\textbf{p}_1,\ch{\varphi_0^*}) ,
\end{align}
where \ch{$F$ is defined in \eqref{def-F} and $\tilde F_3$ is defined in \eqref{def-tildeF3}}.

We choose $\alpha_1$ so that $\alpha_1(t)\to 0$ as $t\to \infty$, that is,
\begin{align}
\label{eq-alpha1}
\alpha_1(t)
= - \int_t^\infty
\frac{1}{\ch{\lambda} \int_{\R^2} U \chi dy}
\int_{\R^2}
F_3^{(0)}(\phi_1+\phi_2,\varphi,\textbf{p}_1,\ch{\varphi_0^*}) dy ds.
\end{align}

The third equation in \eqref{eq-param} is satisfied if
\begin{align*}
\lambda \alpha \dot \xi_{1,j}(t)
= -\frac{1}{\int_{\R^2} \partial_{y_j} U y_j \chi dy }
\int_{\R^2} F_3^{(1,j)} (\phi_1+\phi_2,\varphi,\textbf{p}_1,\ch{\varphi_0^*}) y_j dy,
\end{align*}
where $\xi_1 = (\xi_{1,1},\xi_{1,2})$ and
\begin{align*}
F_3^{(1,j)} (\phi,\varphi,\textbf{p}_1,\ch{\varphi_0^*})
&=
E_2 \tilde \chi_2
+ F(\phi^i_0 + \phi ,\tilde \varphi_0^* + \varphi ,\textbf{p}_0+\textbf{p}_1)
\\
& \quad
+
\lambda^4 [ S_0(u_1(\lambda_0+\lambda_1,\alpha_0,0)) - S_0(u_1(\textbf{p}_0)) ] \tilde \chi
\\
& \quad
- \lambda^2 \dot  \alpha_1 U \chi
- \alpha \alpha_1 \nabla_y \cdot ( U \nabla_y \Gamma_0 ) \chi
\\
& \quad + \tilde F_3(\phi,\varphi,\textbf{p}_1,\ch{\varphi_0^*}) .
\end{align*}
We select $\xi_1$ given by the formula
\begin{align}
\label{eq-xi1}
\xi_{1,j}(t) = q_j -
\int_{t_0}^t \frac{1}{\lambda \alpha \int_{\R^2} \partial_{y_j} U  y_j \chi dy }
\int_{\R^2} F_3^{(1,j)} (\phi_1+\phi_2,\varphi,\textbf{p}_1,\ch{\varphi_0^*}) y_j dy ds ,
\end{align}
where $q = (q_1,q_2) \in \R^2$ is a fixed point (as in the statement of Theorem~\ref{teo1}).
}

\subsection{Mass and second moment}
In this section we derive some formulas for the mass and second moment appearing in the right hand side of \eqref{inner4a}.

In the computation of  $m_0[ F_3(\phi,\varphi,\textbf{p}_1 ,\varphi_0^*) ]$ and  $m_2[ F_3(\phi,\varphi,\textbf{p}_1 ,\varphi_0^*) ]$, the following formulas will be useful.


\begin{lemma}
\label{lemma-mass-Su1}
We have
\begin{align}
\nonumber
\int_{\R^2} S(u_1(\mathbf{p})) dx
&=  - \partial_t \int_{\R^2} u_0 d x - \partial_t \int_{\R^2} \varphi_\lambda dx
\\
\nonumber
& = -  \partial_t \Bigl\{ 8 \pi \alpha \Bigl[ 1+2\Upsilon \frac{\lambda^2}{t} \Bigr]+ \alpha e_1 \Bigl( \frac{\lambda^2}{t}\Bigr)  +  \int_{\R^2} \varphi_\lambda dx \Bigr\}
\end{align}
and
\begin{align}
\nonumber
\int_{\R^2}
( S(u_1(\mathbf{p}))-S(u_1(\mathbf{p}_0))) dx
& = -  \partial_t \Bigl\{  \alpha_1 \Bigl[ 8 \pi \Bigl( 1+2\Upsilon \frac{\lambda^2}{t}\Bigr)+e_1 \Bigl( \frac{\lambda^2}{t}\Bigr)   \Bigr]
+ 16 \pi \alpha_0 \Upsilon \frac{\lambda^2-\lambda_0^2}{t}
\\
\nonumber
& \quad
+ \alpha_0 \Bigl( e_1\Bigl(\frac{\lambda^2}{t}\Bigr)
-e_1\Bigl(\frac{\lambda_0^2}{t}\Bigr) \Bigr)
+ \int_{\R^2} (\varphi_\lambda-\varphi_{\lambda_0}) dx \Bigr\} ,
\end{align}
where $e_1(s)$ is defined by
\begin{align}
\label{massu0a}
\int_{\R^2} u_0 dx =
8 \pi \alpha \Bigl[ 1+2\Upsilon \frac{\lambda^2}{t} \Bigr] + \alpha e_1\Bigl( \frac{\lambda^2}{t}\Bigr) .
\end{align}
\end{lemma}
Recall that  $\Upsilon$ is given in \eqref{defUpsilon} and note that
\begin{align*}
e_1(s) = O(s^2) , \quad \text{as }s\to 0.
\end{align*}
\begin{proof}
For this we recall that (c.f. \eqref{Su})
\begin{align*}
S(u_1(\mathbf{p})) = - \partial_t u_0 - \partial_t \varphi_\lambda + \mathcal{E}(u_0+\varphi_\lambda) ,
\end{align*}
so
\begin{align}
\nonumber
\int_{\R^2} S(u_1(\mathbf{p})) dx
&=  - \partial_t \int_{\R^2} u_0 d x - \partial_t \int_{\R^2} \varphi_\lambda dx
\\
\nonumber
& = -  \partial_t \Bigl\{ 8 \pi \alpha \Bigl[ 1+2\Upsilon \frac{\lambda^2}{t} \Bigr]+ \alpha e_1 \Bigl( \frac{\lambda^2}{t}\Bigr)  +  \int_{\R^2} \varphi_\lambda dx \Bigr\} .
\end{align}
Therefore
\begin{align}
\nonumber
\int_{\R^2}
( S(u_1(\mathbf{p}))-S(u_1(\mathbf{p}_0))) dx
& = -  \partial_t \Bigl\{  \alpha_1 \Bigl[ 8 \pi \Bigl( 1+2\Upsilon \frac{\lambda^2}{t}\Bigr)+e_1 \Bigl( \frac{\lambda^2}{t}\Bigr)   \Bigr]
+ 16 \pi \alpha_0 \Upsilon \frac{\lambda^2-\lambda_0^2}{t}
\\
\nonumber
& \quad
+ \alpha_0 \Bigl( e_1\Bigl(\frac{\lambda^2}{t}\Bigr)
-e_1\Bigl(\frac{\lambda_0^2}{t}\Bigr) \Bigr)
+ \int_{\R^2} (\varphi_\lambda-\varphi_{\lambda_0}) dx \Bigr\}.
\end{align}
\end{proof}

\begin{lemma}
\label{lemma-second-m}
We have
\begin{align}
\nonumber
\lambda^4 m_2 & [S_0( u_1( \textbf{p}_0 + \textbf{p}_1 ) ) - S_0( u_1( \textbf{p}_0) )  ]
\\
\nonumber
& = - 32\pi \alpha_1
- \frac{\dot\alpha}{\lambda^2}
\int_{\R^2} U(\frac{x-\xi}{\lambda}) \chi_0(\frac{x-\xi}{\lambda}) |x-\xi|^2\,dx
+ \frac{\dot\alpha_0}{\lambda_0^2}
\int_{\R^2} U(\frac{x}{\lambda_0}) \chi(\frac{x}{\lambda_0}) |x|^2\,dx
\\
\nonumber
& \quad- 4 \Bigl[ \alpha e_1\Bigl(\frac{\lambda^2}{t}\Bigr) -  \alpha_0 e_1\Bigl(\frac{\lambda_0^2}{t} \Bigr)\Bigr]
- \Bigl[ \alpha e_2\Bigl(\frac{\lambda^2}{t}\Bigr) -  \alpha_0 e_2\Bigl(\frac{\lambda_0^2}{t} \Bigr)\Bigr]
\\
\nonumber
& \quad - \Bigl( \int_{\R^2} (\varphi_{\lambda}-\varphi_{\lambda_0})\,dx\Bigr)^2
\\
\nonumber
& \quad
-(1-\alpha) \int_{\R^2} E(x-\xi,t,\lambda)|x-\xi|^2\, dx
+(1-\alpha_0) \int_{\R^2} E(x,t,\lambda_0)|x|^2\, dx
\\
\nonumber
& \quad
- |\xi|^2  \int_{\R^2} S(u_1(\mathbf{p}_0))dx .
\end{align}
\end{lemma}
\begin{proof}
We have defined the second moment $m_2$ \eqref{def-m0-m2} integrating with respect to $y$. Note that
\[
\lambda^4
\int_{\R^2} f(y) |y|^2dy = \int_{\R^2} f\Bigl( \frac{x-\xi}{\lambda}\Bigr)|x-\xi|^2 dx ,
\]
and therefore
\begin{align*}
\lambda^4 m_2 [S_0( u_1( \textbf{p}_0 + \textbf{p}_1 ) ) - S_0( u_1( \textbf{p}_0) )  ]
&=
\lambda^4 \int_{\R^2} S_0( u_1( \textbf{p}_0 + \textbf{p}_1 ) ) (\xi + \lambda y) |y|^2 dy
\\
& \quad
- \lambda^4 \int_{\R^2} S_0( u_1( \textbf{p}_0 ) ) (\xi + \lambda y) |y|^2 dy
\\
&= \int_{\R^2} S_0( u_1( \textbf{p}_0 + \textbf{p}_1 ) ) (x) |x-\xi|^2 dy
\\
& \quad
-  \int_{\R^2} S_0( u_1( \textbf{p}_0 ) ) (x) |x-\xi|^2 dy.
\end{align*}

We have by Lemma~\ref{second-moment-u1},
\begin{align}
\nonumber
\int_{\R^2} S(u_1) |x-\xi|^2dx
&=
4 \int_{\R^2} \varphi_\lambda dx
-  \alpha \int_{\R^2} \tilde E(x-\xi,t;\lambda) |x-\xi|^2 dx
\\
\nonumber
& \quad
+ \int_{\R^2} \nabla \varphi_\lambda \, dx
\cdot \dot \xi
- \frac{\dot \alpha}{\lambda^2} \int_{\R^2} U \chi |x-\xi|^2 dx
\\
\nonumber
& \quad
- (1-\alpha) \int_{\R^2} E (x-\xi,t;\lambda) |x-\xi|^2dx
\\
\label{Su1-2a}
& \quad
+ 4 \Bigl( \int_{\R^2} u_0 + \int_{\R^2} \varphi_\lambda \Bigr)
\Bigl( 1 - \frac{1}{8\pi} \int_{\R^2} u_0
 - \frac{1}{8\pi} \int_{\R^2} \varphi_\lambda \Bigr).
\end{align}
where $E$, $\tilde E$ are defined in \eqref{defE}, \eqref{defTildeE}.
Let
\begin{align*}
m = \int_{\R^2} (u_0 + \varphi_\lambda) dx,
\quad \delta m = m - 8 \pi.
\end{align*}
Since
\begin{align*}
\int_{\R^2} (u_0 + \varphi_{\lambda_0}) dx = 8 \pi,
\end{align*}
by \eqref{massU1}, we have
\begin{align*}
\delta m= \int_{\R^2} (\varphi_\lambda - \varphi_{\lambda_0})\,dx.
\end{align*}
Replacing $m$ in \eqref{Su1-2a} we get
\begin{align}
\nonumber
\int_{\R^2} S(u_1(\mathbf{p})) |x-\xi|^2dx
&=
32 \pi
- 4 \int_{\R^2} u_0 dx
- \frac{1}{2\pi}(\delta m)^2
-  \alpha \int_{\R^2} \tilde E(x-\xi,t;\lambda) |x-\xi|^2 dx
\\
\nonumber
& \quad
+ \int_{\R^2} \nabla \varphi_\lambda \, dx \cdot \dot \xi
- \frac{\dot \alpha}{\lambda^2} \int_{\R^2} U \chi |x-\xi|^2 dx
\\
\label{Su1-3a}
& \quad
- (1-\alpha) \int_{\R^2} E (x-\xi,t;\lambda) |x-\xi|^2dx .
\end{align}

Also under \eqref{conditions} we have by
\eqref{secondMtildeE}:
\begin{align}
\label{secondMtildeE-2a}
\int_{\R^2}
\tilde E |x-\xi|^2 dx = - 64 \pi \Upsilon \frac{\lambda^2}{t} +e_2\Bigl( \frac{\lambda^2}{t}\Bigr),
\end{align}
where
\begin{align*}
e_2(s) = O(s^2) , \quad \text{as }s\to 0.
\end{align*}

Combining \eqref{Su1-3a}, \eqref{massu0a} and \eqref{secondMtildeE-2a} we get
\begin{align*}
\int_{\R^2} S(u_1(\mathbf{p})) |x-\xi|^2dx
&=
32 \pi (1-\alpha)
- \frac{1}{2\pi}(\delta m)^2
- \frac{\dot \alpha}{\lambda^2} \int_{\R^2} U \chi |x-\xi|^2 dx
\\
& \quad
+ \int_{\R^2} \nabla \varphi_\lambda \, dx \cdot \dot \xi
- (1-\alpha) \int_{\R^2} E (x-\xi,t;\lambda) |x-\xi|^2dx
\\
& \quad
-4 \alpha e_1 \Big( \frac{ \lambda^2 }{t} \Big)
- \alpha  e_2  \Big( \frac{ \lambda^2 }{t} \Big) .
\end{align*}
We can apply this formula to $\mathbf{p} = \mathbf{p}_0$ and get
\begin{align*}
\int_{\R^2} S(u_1(\mathbf{p}_0)) |x|^2dx
&=
32 \pi (1-\alpha_0)
- \frac{\dot \alpha_0}{\lambda_0^2} \int_{\R^2} U(\frac{x}{\lambda_0}) \chi |x|^2 dx
\\
& \quad
- (1-\alpha_0) \int_{\R^2} E (x,t;\lambda_0) |x|^2dx
\\
& \quad
-4 \alpha_0 e_1 \Big( \frac{ \lambda_0^2 }{t} \Big)
- \alpha_0  e_2  \Big( \frac{ \lambda_0^2 }{t} \Big) .
\end{align*}
Note that
\begin{align*}
\int_{\R^2} S_0(u_1(\mathbf{p}_0))|x-\xi|^2dx
&= \int_{\R^2} S_0(u_1(\mathbf{p}_0))|x|^2dx
+ |\xi|^2  \int_{\R^2} S_0(u_1(\mathbf{p}_0))dx
\\
&= \int_{\R^2} S_0(u_1(\mathbf{p}_0)|x|^2dx
+ |\xi|^2  \int_{\R^2} S(u_1(\mathbf{p}_0)dx
\end{align*}
because
\begin{align*}
 \int_{\R^2} S_0(u_1(\mathbf{p}_0) x_j dx =0.
\end{align*}
Therefore,
\begin{align}
\nonumber
\int_{\R^2} & [ S(u_1(\mathbf{p}))-S(u_1(\mathbf{p}_0))] |x-\xi|^2dx
\\
\nonumber
& = - 32\pi \alpha_1
- \frac{\dot\alpha}{\lambda^2}
\int_{\R^2} U(\frac{x-\xi}{\lambda}) \chi_0(\frac{x-\xi}{\lambda}) |x-\xi|^2\,dx
+ \frac{\dot\alpha_0}{\lambda_0^2}
\int_{\R^2} U(\frac{x}{\lambda_0}) \chi(\frac{x}{\lambda_0}) |x|^2\,dx
\\
\nonumber
& \quad- 4 \Bigl[ \alpha e_1\Bigl(\frac{\lambda^2}{t}\Bigr) -  \alpha_0 e_1\Bigl(\frac{\lambda_0^2}{t} \Bigr)\Bigr]
- \Bigl[ \alpha e_2\Bigl(\frac{\lambda^2}{t}\Bigr) -  \alpha_0 e_2\Bigl(\frac{\lambda_0^2}{t} \Bigr)\Bigr]
\\
\nonumber
& \quad - \Bigl( \int_{\R^2} (\varphi_{\lambda}-\varphi_{\lambda_0})\,dx\Bigr)^2
\\
\nonumber
& \quad
-(1-\alpha) \int_{\R^2} E(x-\xi,t,\lambda)|x-\xi|^2\, dx
+(1-\alpha_0) \int_{\R^2} E(x,t,\lambda_0)|x|^2\, dx
\\
\nonumber
& \quad
- |\xi|^2  \int_{\R^2} S(u_1(\mathbf{p}_0)dx .
\end{align}

\end{proof}

\medskip

\section{Proof of Theorem~\ref{teo1}}
\label{sect-proof-existence}

Next we define norms, which are suitably adapted to the  terms  in  the inner linear problems \eqref{inner4a}, \eqref{inner4b}.
Let us write the linearized versions of these problems as
\begin{align}
\label{linear-000}
\left\{
\begin{aligned}
\lambda^2 \partial_t \phi &= L[\phi] + B[\phi]+  h(y,t)  \quad \text{in }\R^2 \times (t_0,\infty),
\\
\phi(\cdot,t_0) &= 0 \quad \text{in }\R^2 .
\end{aligned}
\right.
\end{align}
Given positive numbers $\nu$, \ch{$p$}, $\epsilon$ and $m\in \R$, we let
\begin{align}
\label{normh0}
\|h\|_{0,\nu,m,p,\epsilon} = & \inf K \quad\text{such that }
\\
\nonumber
& |h(y,t)| \leq
\frac{K}{t^{\nu} (\log t)^{m}}
\ch{\frac{  1  }{(1+|y|)^p} }
\begin{cases}
\displaystyle
1
& |y|\leq \sqrt { t  \log t },
\vspace{1mm}
\\
\displaystyle
\frac{\ch{(t \log t)^{\epsilon/2}}}{|y|^\epsilon}
& |y|\geq  \sqrt { t  \log t } .
\end{cases}
\end{align}
\index{norm for the inner problem $\norm\ \norm_{0,\nu,m,p,\epsilon}$}
We also defie
\begin{align}
\nonumber
\| \phi \|_{1,\nu,m,p,\epsilon} = & \inf K \quad\text{such that }
\\
\nonumber
& | \phi(y,t)| + (1+|y|) |\nabla_y \phi(y,t)| \leq
\frac{K}{t^{\nu} (\log t)^{m}}
\ch{\frac{  1  }{(1+|y|)^p} }
\begin{cases}
\displaystyle
1
& |y|\leq \sqrt { t  \log t },
\vspace{1mm}
\\
\displaystyle
\frac{\ch{(t \log t)^{\epsilon/2}}}{|y|^\epsilon}
& |y|\geq  \sqrt { t  \log t } .
\end{cases}
\end{align}
\index{norm for the inner problem $\norm\ \norm_{1,\nu,m,p,\epsilon}$}

We develop a solvability theory of problem \eqref{linear-000} that involves uniform space-time bounds in terms of the above norms.
We will establish two results: one in which
the solution ``loses" one power of $t$ on bounded sets with respect to the time-decay of $h$, under radial symmetry and the condition of spatial average 0 at all times. Our second result states that for a general $h$
this loss is only $t^{\frac 12} $ if in addition the center of  mass and second-moment of $h$ are zero at all times.



For the first result we introduce a parameter in the problem in order to get a fast decay of the solution:
\begin{align}
\label{linear-000a}
\left\{
\begin{aligned}
\lambda^2 \partial_t \phi &= L[\phi] + B[\phi]+  h(y,t)  \quad \text{in }\R^2 \times (t_0,\infty),
\\
\phi(\cdot,t_0) &= c_1 \tilde Z_0 \quad \text{in }\R^2 ,
\end{aligned}
\right.
\end{align}
where $\tilde Z_0$ is defined as
\begin{align}
\nonumber
\tilde Z_0(\rho ) &= ( Z_0(\rho) - m_{Z_0} U ) \chi_0\Bigl( \frac{\rho}{3\lambda(t_0)\sqrt{t_0} }\Bigr),
\end{align}
where $m_{Z_0}$ is such that
\begin{align*}
\int_{\R^2} \tilde Z_0 = 0.
\end{align*}

\begin{prop}
\label{prop-linear-without-second-moment-1}
Assume \eqref{conditions}.
Let $\sigma>0$, $\epsilon>0$ with $\sigma+\epsilon<2$ and  $1<\nu<\frac{7}{4}$. Let $0<q<1$.
Then there exists a number $C>0$ such that for $t_0$ sufficiently large and all radially symmetric $h=h(|y|,t)$ with $\|h\|_{0,\nu,m,6+\sigma,\epsilon}<\infty$ and
\begin{align*}
\int_{\R^2} h(y,t)dy &= 0
,\quad \text{for all } t>t_0 ,
\end{align*}
there exists  $c_1 \in \R$ and  solution $\phi(y,t) = \mathcal T^{i,2}_{\textbf{p}} [h]$ of problem
\ch{\eqref{linear-000a}}
that defines a linear operator of $h$ and
satisfies the estimate
\begin{align*}
\|  \phi \|_{1,\nu-1,m+q-1,4,2+\sigma+\epsilon}  \leq \frac{C}{(\log t_0)^{1-q}}  \|h\|_{0,\nu,m,6+\sigma,\epsilon}.
\end{align*}
Moreover $c_1$ is a linear operator of $h$ and
\begin{align*}
|c_1| & \leq  C \frac{1 }{ t_0^{\nu-1} (\log t_0)^{m}}  \|h\|_{0,\nu,m,6+\sigma,\epsilon}.
\end{align*}
\end{prop}

\begin{prop}
\label{prop-linear-with-second-moment-1}
Assume \eqref{conditions}.
Let $0<\sigma<1$, $\epsilon>0$ with $\sigma+\epsilon<\frac{3}{2}$
and  $1<\nu< \min( 1+\frac{\epsilon}{2},3-\frac{\sigma}{2}, \frac{5}{4})$.
Let $0<q<1$.
Then there is $C$ such that for \ch{$t_0$} large the following holds.
Suppose that  $h$ satisfies  $\|h\|_{0,\nu,m,6+\sigma,\epsilon}<\infty$ and
\begin{align}
\nonumber
\int_{\R^2} h(y,t)dy&=0 , \quad
\int_{\R^2} h(y,t)|y|^2dy=0,
\\
\nonumber
\int_{\R^2} h(y,t) y_j dy&=0 , \quad j=1,2 ,\quad \text{for all } t>t_0 .
\end{align}
Then there exists a solution $\phi(y,t)= \mathcal T^{i,1}_{\textbf{p}} [h]$ of problem
\ch{\eqref{linear-000}}
that defines a linear operator of $h$ and
satisfies
\begin{align}
\nonumber
\|  \phi \|_{1,\nu-\frac 12 ,m+\frac {q-1}{2},4,2+\sigma+\epsilon} \leq C \|h\|_{0,\nu,m,6+\sigma,\epsilon}.
\end{align}
\end{prop}

The proof of the Propositions~\ref{prop-linear-without-second-moment-1} and \ref{prop-linear-with-second-moment-1} is divided into different steps and presented in sections~\ref{sectLT1}--\ref{sect-linear-nonradial}.

Next we consider the linear outer problem:
\begin{align}
\label{outer1b}
\left\{
\begin{aligned}
\partial_t \phi^o
&= L^o  [\phi^o] +  g(x,t), \quad \text{in }\R^2 \times(t_0,\infty)
\\
\phi^o(\cdot,t_0)&=\phi^o_0, \quad \text{in }\R^2 .
\end{aligned}
\right.
\end{align}
where
\begin{align}
\nonumber
L^o [\varphi ] :=\Delta_x \varphi- \nabla_x\Big[ \Gamma_0 \Big (\frac{x-\xi(t)}{\lambda(t)}\Big) \Big]\cdot \nabla_x \vp .
\end{align}

For a given function $g(x,t) $ we consider the norm $\| g \|_{**,o}$ defined as the least $K\ge 0$ such that for all
$(x,t)\in \R^2\times (t_0,\infty)$
\begin{align}
\label{norm-outer}
|g(x,t)|  \le   K\frac 1{t^a (\log t)^\beta}  \frac 1 { 1+ |\zeta |^b} , \quad \zeta = \frac{x-\xi(t)}{\sqrt{t}}.
\end{align}
\index{norm of the outer RHS, $\norm \ \norm_{**,o}$}
Accordingly, we consider for a function $\phi^o(x,t)$ the norm  $\| \phi \|_{*,o}$ defined as the least $K\ge 0$ such that
\begin{align}
\label{norm-phio}
|\phi^o(x,t)| +  (\la +|x-\xi|) |\nn_x \phi^o (x,t)| \le    K\frac 1{t^{a-1} (\log t)^\beta}  \frac 1 { 1+ |\zeta |^{b}} , \quad \zeta = \frac{x-\xi}{\sqrt{t}}
\end{align}
\index{norm of the outer solution, $\norm \ \norm_{*,o}$}
for all $(x,t)\in \R^2\times (t_0,\infty)$.


We assume that  the parameters $a,b,\beta$  satisfy the constraints
\begin{align}
\label{cond2b}
1<a < 4,\quad  2<b< 6, \quad  a< 1+ \frac b2, \quad
\quad
\beta \in \R.
\end{align}

\begin{prop}
\label{thmOuter}
Assume that the parameter functions $\textbf{p} =(\lambda, \alpha,\xi)$ satisfy conditions \eqref{conditions} and the numbers  $a,b,\beta$ satisfy \eqref{cond2b}.
Then there is a constant $C$ so that for $t_0$ sufficiently large and
for $\|g\|_{**,o}<\infty$,
there exists a solution  $\phi^o= \mathcal T ^o_{\textbf{p}}[g ]$  of \eqref{outer1b} with $\phi^o_0=0$,
which defines a linear operator of  $g$ and satisfies
\begin{align*}
\|\phi^o\|_{*,o}
\leq  C \|g\|_{ **,o} .
\end{align*}
\end{prop}

For the initial condition $\phi_0^o$ in \eqref{outer1b} we consider the norm $\| \varphi_0^o \|_{*,b}$ defined as
\begin{align}
\nonumber
\| \phi^o_0 \|_{*,b} =& \inf K \quad \text{such that }
\\
\nonumber
& |\phi^o_0(x)| + (\lambda(t_0) +|x| )  |\nabla_x \phi^o(x)| \leq \frac{K}{( 1+ \frac{|x|}{\sqrt{t}} )^b} .
\end{align}
\index{norm initial condition, v1, $\norm \ \norm_{*,b}$}
We have an estimate for the solution of \eqref{outer1b} with $g=0$ and $\| \phi^o_0 \|_{*,b}<\infty$.
\begin{prop}
\label{thmOuter-ci}
Assume that the parameter functions $\textbf{p} =(\lambda, \alpha,\xi)$ satisfy conditions \eqref{conditions} and the numbers  $a,b,\beta$ satisfy \eqref{cond2b}.
Then there is a constant $C$ so that for $t_0$ sufficiently large and for $\| \phi^o_0 \|_{*,b}<\infty$ there exists a solution  $\phi^o$  of \eqref{outer1b}, which defines a linear operator of $ \phi^o_0$ and satisfies
\begin{align*}
\|\phi^o\|_{*,o}
\leq  C t_0^{a-1} (\log t_0)^\beta
\| \phi^o_0 \|_{*,b} .
\end{align*}
\end{prop}

The proofs of Propositions~\ref{thmOuter} and \ref{thmOuter-ci} are contained in Section~\ref{sect-thmOuter}.


%

\medskip

In what follows we work with $\textbf{p}_1 $ of the form
\[
\textbf{p}_1 = (0, \alpha_1,\xi_1)  ,
\]
that is, we take $\lambda = \lambda_0$, $\alpha = \alpha_0+\alpha_1$, $\xi = \xi_1$, where $\lambda_0$ and $\alpha_0$ have been fixed in Section~\ref{section-lambda0-alpha0},
and we write
\[
\textbf{p} = \textbf{p}_0+\textbf{p}_1.
\]

Next we define suitable operators that allow us to formulate
the system of equations \eqref{inner4a}, \eqref{inner4b}, \eqref{outer4}, and \eqref{eq-param} as a fixed point problem.
We let
\begin{align*}
\mathcal A_{i1}[\phi_1,\phi_2,\varphi, \textbf{p}_1]
&= \mathcal T^{i,1}_{\textbf{p}} \Bigl[   F_3(\phi_1+\phi_2,\varphi,\textbf{p}_1,\ch{\varphi_0^*})
\\
& \quad
- m_0[F_3(\phi_1+\phi_2,\varphi,\textbf{p}_1,\ch{\varphi_0^*})] W_0
- m_2[F_3(\phi_1+\phi_2,\varphi,\textbf{p}_1,\ch{\varphi_0^*})] W_2
\\
& \quad
- m_{1,1}[F_3(\phi_1+\phi_2,\varphi,\textbf{p}_1,\ch{\varphi_0^*})] W_{1,1}
- m_{1,2}[F_3(\phi_1+\phi_2,\varphi,\textbf{p}_1,\ch{\varphi_0^*})] W_{1,2} \Bigr]
\end{align*}
\begin{align*}
\mathcal A_{i2}[\phi_1,\phi_2,\varphi,\textbf{p}_1,\ch{\varphi_0^*}]
&=
\mathcal T^{i,2}_{\textbf{p}} \left[
 m_2[F_3(\phi_1+\phi_2,\varphi,\textbf{p}_1),\ch{\varphi_0^*}] W_2
\right]
\end{align*}
\begin{align*}
A_o[\phi_1,\phi_2,\varphi,\textbf{p}_1, \ch{\varphi_0^*} ]
&=
\ch{
\mathcal T^o_{\textbf{p}}
[ G_2(\phi_1+\phi_2,\varphi,\textbf{p}_1, \varphi_0^*)] } .
\end{align*}
Then the equations \eqref{inner4a}, \eqref{inner4b},\eqref{outer4} can be written as
\begin{align*}
\phi_1 &= \mathcal A_{i1}[\phi_1,\phi_2,\varphi,\textbf{p}_1,
\ch{\varphi_0^*}]
\\
\phi_2 &= \mathcal A_{i2}[\phi_1,\phi_2,\varphi,\textbf{p}_1,
\ch{\varphi_0^*}]
\\
\varphi &=  \mathcal A_o[\phi_1,\phi_2,\varphi,\textbf{p}_1,
\ch{\varphi_0^*}.
]
\end{align*}

Next we consider the equations \eqref{eq-param}, that is, $m_0[F_3(\phi_1+\phi_2,\varphi,\textbf{p}_1,\ch{\varphi_0^*})] (t)\equiv 0$ and $m_{1,j}[F_3(\phi_1+\phi_2,\varphi,\textbf{p}_1,\ch{\varphi_0^*})] (t) \equiv 0$.
By \eqref{def-F3} and  \eqref{def-F2}
\begin{align*}
m_0[F_2 (\phi,\varphi,\textbf{p}_1, \varphi^*_0 ) \tilde \chi ]
&= \lambda_0^4 m_0[S_0( u_1( \textbf{p}_0 + \textbf{p}_1 ) ) - S_0( u_1( \textbf{p}_0) )  ]
+m_0[ E_2 \tilde \chi_2]
\\
& \quad
+ m_0[ F(\phi^i_0 + \phi ,  \varphi^* + \varphi ,\textbf{p}_0+\textbf{p}_1) \tilde \chi]
\\
& \quad
+ \lambda_0^4 m_0[ ( S_0( u_1( \textbf{p}_0 + \textbf{p}_1 ) ) - S_0( u_1( \textbf{p}_0) )   ) (\tilde \chi-1) ] ,
\end{align*}
and using Lemma~\ref{lemma-mass-Su1},
\begin{align*}
m_0[F_2 (\phi,\varphi,\textbf{p}_1, \varphi^*_0 ) \tilde \chi ]
&=  -  \lambda_0^2 \partial_t \Bigl\{  \alpha_1 \Bigl[ 8 \pi \Bigl( 1+2\Upsilon \frac{\lambda_0^2}{t}\Bigr)+e_1 \Bigl( \frac{\lambda_0^2}{t}\Bigr)   \Bigr]
+m_0[ E_2 \tilde \chi_2]
\\
& \quad
+ m_0[ F(\phi^i_0 + \phi ,  \varphi^* + \varphi ,\textbf{p}_0+\textbf{p}_1) \tilde \chi]
\\
& \quad
+ \lambda_0^4 m_0[ ( S_0( u_1( \textbf{p}_0 + \textbf{p}_1 ) ) - S_0( u_1( \textbf{p}_0) )   ) (\tilde \chi-1) ] .
\end{align*}
This motivates the definition
\begin{align}
\nonumber
& \mathcal A_{p,\alpha_1} [\phi_1,\phi_2,\varphi,\textbf{p}_1,\ch{\varphi_0^*}]
\\
\nonumber
& \quad
 = -
\frac{1}{8\pi( 1+2\Upsilon\frac{\lambda_0^2}{t}) + e_1(\frac{\lambda_0^2}{t})}
\int_t^\infty
\frac{1}{\lambda_0^2}
\Bigl\{
m_0[ E_2 \tilde \chi_2] (s)
+ m_0[ F(\phi^i_0 + \phi ,  \varphi^* + \varphi ,\textbf{p}_0+\textbf{p}_1) \tilde \chi](s)
\\
\label{def-Aalpha}
& \qquad\qquad\qquad\qquad\qquad\qquad\qquad\qquad
+ \lambda_0^4 m_0[ ( S_0( u_1( \textbf{p}_0 + \textbf{p}_1 ) ) - S_0( u_1( \textbf{p}_0) )   ) (\tilde \chi-1) ] (s)
\Bigr\}ds
\end{align}

Similarly, by \eqref{def-F3} and \eqref{def-F2}
\begin{align*}
m_{1,j}[F_2 (\phi,\varphi,\textbf{p}_1, \varphi^*_0 ) \tilde \chi ]
&=
\lambda_0 \alpha\dot \xi_{1,j} \int_{\R^2} \partial_{y_j}U(y) y_j \tilde \chi dy
+  \frac{\alpha \lambda_0^2}{\sqrt t} \dot \xi_{1,j}
\int_{\R^2} U(y) \partial_{z_j} \chi_0(\frac{\lambda y}{\sqrt t}) y_j dy
\\
& \quad
+m_{1,j}[ E_2 \tilde \chi_2] + m_{1,j}[ F(\phi^i_0 + \phi ,  \varphi^* + \varphi ,\textbf{p}_0+\textbf{p}_1) \tilde \chi] .
\end{align*}
This motivates the definition
\begin{align}
\nonumber
& \mathcal A_{p,\xi_1} [\phi_1,\phi_2,\varphi,\textbf{p}_1,\ch{\varphi_0^*}]
\\
\nonumber
& \quad =
\int_t^\infty
\frac{1}{ \lambda_0 \alpha \int_{\R^2} \partial_{y_j}U(y) y_j \tilde \chi dy}
\Bigl\{
\frac{\alpha \lambda_0^2}{\sqrt t} \dot \xi_{1,j}
\int_{\R^2} U(y) \partial_{z_j} \chi_0(\frac{\lambda y}{\sqrt t}) y_j dy
\\
\label{Apxi1}
& \qquad  \qquad  \qquad
+m_{1,j}[ E_2 \tilde \chi_2](s) + m_{1,j}[ F(\phi^i_0 + \phi ,  \varphi^* + \varphi ,\textbf{p}_0+\textbf{p}_1) \tilde \chi](s) \Big\}ds
\end{align}

Then we define $\mathcal A_p$ by
\begin{align}
\label{def-Ap}
\mathcal A_p [\phi_1,\phi_2,\varphi,\textbf{p}_1,\ch{\varphi_0^*}]
= ( 0,  \mathcal A_{p,\alpha_1} [\phi_1,\phi_2,\varphi,\textbf{p}_1,\ch{\varphi_0^*}],
\mathcal A_{p,\xi_1} [\phi_1,\phi_2,\varphi,\textbf{p}_1,\ch{\varphi_0^*}]).
\end{align}
Then
\[
\mathbf{p}_1 = \mathcal A_p [\phi_1,\phi_2,\varphi,\textbf{p}_1,\ch{\varphi_0^*}]
\]
is equivalent to the equations \eqref{eq-param}.

We write
\[
\vec \phi = (\phi_1,\phi_2,\varphi,\textbf{p}_1) ,
\]
and
\[
\mathcal A[\vec \phi] =
(\mathcal A_{i1}[\vec \phi,\ch{\varphi_0^*}] ,
\mathcal A_{i2}[\vec \phi,\ch{\varphi_0^*}] ,
\mathcal A_{o}[\vec \phi,\ch{\varphi_0^*}] ,
\mathcal A_{p}[\vec \phi,\ch{\varphi_0^*}] ),
\]
and the objective is to find $\vec\phi$ such that
\[
\vec\phi = \mathcal A[\vec \phi].
\]
The operator $\mathcal A$ depends on the initial condition $\varphi_0^*$ appearing in the parabolic problem \eqref{eq-cond-in}, and we will stress its dependence later on when proving the stability assertion in Theorem~\ref{teo1}.

We define the spaces on which we will consider the operator $\mathcal A$ to set up the fixed point problem.
For certain choices of constants $\nu$,  $q$, $\sigma$, $\epsilon$, $a$, $b$, $\beta$, $\gamma$, $\Upsilon$
that we will make precise later,
we let
\begin{align*}
X_{i} = \Bigl\{ \, \phi \in L^\infty( \R^2 \times (t_0,\infty)  ) \ &| \ \nabla_y \phi \in L^\infty( \R^2 \times (t_0,\infty)  ) , \, \|\phi\|_{1,\nu-\frac{1}{2},\frac{q-1}{2},4,2+\sigma+\epsilon}<\infty,
\\
& \quad
\int_{\R^2} \phi(y,t)dy=0, \
\int_{\R^2} \phi(y,t) y dy = 0,  \
\ t>t_0 \, \Bigr\} ,
\end{align*}
\begin{align*}
X_o = \{ \varphi \in  L^\infty( \R^2 \times (t_0,\infty)  ) \  | \ \nabla_y \phi \in L^\infty( \R^2 \times (t_0,\infty)  ) , \,
\|\varphi\|_{*,o}<\infty
\,
\} ,
\end{align*}
\begin{align*}
X_p = \{ \,  ( 0,\alpha_1,\xi_1) \in C^1([t_0,\infty))\   |\
\|\alpha_1\|_{C^1,\nu+\frac{1}{2},\Upsilon}<\infty, \
\|\xi_1\|_{C^1,\gamma,0}<\infty\,
\}
\end{align*}
where the norms $\|\phi\|_{1,\nu-\frac{1}{2},\frac{q-1}{2},4,2+\sigma+\epsilon}$ and   $\|\varphi\|_{*,o}$ are defined in \eqref{normh0}, \eqref{norm-outer} and  $\|\xi_1\|_{C^1,\mu,m}$ is defined by
\begin{align}
\nonumber
\|g\|_{C^0,\mu,m} = \sup_{t\geq t_0} \,  t^\mu (\log t)^m |g(t)|.
\end{align}
\begin{align}
\nonumber
\| g \|_{C^1,\mu,m}
= \| g \|_{C^0,\mu,m} +  \|\dot g\|_{C^0,\mu+1,m}.
\end{align}
\index{norm for $\alpha_1$, $\xi_1$, $\norm \ \norm_{C^1,\mu,m} $}
for a function $g \in C^1([t_0 , \infty))$.

We choose in the definition of the outer norm \eqref{norm-phio}
\begin{align}
\label{cond-nu-a1}
a =  \nu +\frac{5}{2} , \quad 2\nu+3<b<6, \quad \beta < \frac{1+q}{2}.
\end{align}
With these choices we see that \eqref{cond2b} are satisfied.
Also $\nu$ will be in the range $1<\nu<\frac{3}{2}$ so the interval for $b$ is not empty in \eqref{cond-nu-a1}.

We use the following notation: for
$\textbf{p}_1= (0,\alpha_1,\xi_1)$
\begin{align*}
\|\textbf{p}_1 \|_{X_p}
=
\|\alpha_1\|_{C^1,\nu+\frac{1}{2},\Upsilon}
+ \|\xi_1\|_{C^1,1+\gamma,0},
\end{align*}
\index{norm $\norm \textbf{p}_1 \norm_{X_p}$}
and for
$ \vec \phi = ( \phi_1,\phi_2, \varphi, \textbf{p}_1) $
\begin{align}
\label{norma}
\| \vec \phi \|_X =
\|\phi_1\|_{1,\nu-\frac{1}{2},\frac{q-1}{2},4,2+\sigma+\epsilon}
+
\|\phi_2\|_{1,\nu-\frac{1}{2},\frac{q-1}{2},4,2+\sigma+\epsilon}
+ \|\varphi\|_{*,o}
+ \| \textbf{p}_1 \|_{X_p} .
\end{align}
\index{norm $\norm \vec \phi \norm_X$}

\ch{With the above notation,
given $\varphi_0^*$ with $\|\varphi_0^*\|_{*,b} $ sufficiently small,
we consider the fixed point problem
\begin{align}
\label{fixed1}
\vec \phi = \mathcal A[\vec \phi]  ,
\end{align}
with $\vec \phi $ in a suitable close ball of $X$.
A solution of this fixed point problem yields a solution of the system of equations \eqref{inner4a}, \eqref{inner4b}, \eqref{outer4}, \eqref{eq-param}, which in turn gives a solution to \eqref{ks1}.
}

We claim that for some constant $C$ independent of $t_0\gg 1$,  if
$ t_0^{a-1}  (\log t_0)^\beta  \| \varphi_0^* \|_{*,b}\leq 1 $,
and  $\|\vec\phi\|_X\leq 1$, then
\begin{align}
\label{est-Ai1}
\| \mathcal A_{i1}[\phi_1,\phi_2,\varphi,\textbf{p}_1,\varphi_0^*] \|_{1,\nu-\frac{1}{2},\frac{q-1}{2},4,2+\sigma+\epsilon}
\leq \frac{C}{t_0^\vartheta}
+  C (\log t_0)^{\frac{\sigma}{2}} t_0^{\nu+1+\frac{\sigma}{2}}
\| \varphi_0^* \|_{*,b},
\end{align}
for some $\vartheta>0$ small, a constant $C$ independent of $t_0$, and $t_0$ sufficiently large.

Indeed,  by Proposition~\ref{prop-linear-with-second-moment-1} we have
\begin{align*}
\| \mathcal A_{i1}[\phi_1,\phi_2,\varphi,\textbf{p}_1,\ch{\varphi_0^*}] \|_{1,\nu-\frac{1}{2},\frac{q-1}{2},4,2+\sigma+\epsilon}
\leq
C
\| F_3(\phi_1+\phi_2,\varphi,\textbf{p}_1,\ch{\varphi_0^*})  \|_{0,\nu,6+\sigma,\epsilon} .
\end{align*}
We recall the expansion of $F_3$ in \eqref{def-F3}. To estimate $ E_2 \tilde \chi_2$ we use \eqref{est-E2} to get
\begin{align}
\label{estimate-E2-tildechi}
\| E_2 \tilde \chi \|_{0,\nu,0,6+\sigma,\epsilon}
\leq \frac{C }{t_0^{1+2\delta-\frac{\sigma}{2}-\nu} (\log t_0)^{2}}
\end{align}
where $\delta$, $\sigma$ are positive small constants and are assumed to satisfy $2\delta-\frac{\sigma}{2}>0$. Then we take $\nu$ in the range
\begin{align}
\label{firstdelta-sigma}
1<\nu<1+2\delta-\frac{\sigma}{2},
\end{align}
with $\nu$ close to 1.

Let us consider the term $ \lambda^4 [ S_0(\textbf{p}_0 + \textbf{p}_1) - S_0(\textbf{p}_0)]$ in  $F_3(\phi_1+\phi_2,\varphi,\textbf{p}_1,\ch{\varphi_0^*}) $ (c.f. \eqref{def-F3}).
The formula $ \lambda^4 [ S_0(\textbf{p}_0 + \textbf{p}_1) - S_0(\textbf{p}_0)]$ (c.f. \eqref{S0u1}) contains for example the term, evaluated at $y = \frac{x-\xi_1}{\lambda_0}$,
\begin{align}
\nonumber
& -\lambda_0^2 \dot \alpha U(y) \chi_0\Bigl( \frac{\lambda_0 y }{\sqrt t}\Bigr)
+\lambda_0^2 \dot \alpha_0 U\Bigl(\frac{\xi_1+\lambda_0y}{\lambda_0}\Bigr) \chi_0\Bigl( \frac{\xi_1+\lambda_0 y }{\sqrt t}\Bigr)
\\
\nonumber
& \quad = -\lambda_0^2 \dot \alpha_1 U(y) \chi_0\Bigl( \frac{\lambda_0 y }{\sqrt t}\Bigr)
-\lambda_0^2 \dot \alpha_0 \Big[ U(y)-U\Bigl(\frac{\xi_1+\lambda_0y}{\lambda_0}\Bigr)    \Bigr] \chi_0\Bigl( \frac{\lambda_0 y }{\sqrt t}\Bigr)
\\
\label{dotalphaU}
& \qquad
-\lambda_0^2 \dot \alpha_0 U\Bigl(\frac{\xi_1+\lambda_0y}{\lambda_0}\Bigr)
\Bigl[ \chi_0\Bigl( \frac{\xi_1+\lambda_0 y }{\sqrt t}\Bigr)-\chi_0\Bigl( \frac{\lambda_0 y }{\sqrt t}\Bigr) \Bigr]
\end{align}
But
\begin{align*}
\left|  -\lambda_0^2 \dot \alpha_1 U(y) \chi_0\Bigl( \frac{\lambda_0 y }{\sqrt t}\Bigr) \right|
& \leq C \frac{1}{t^{\nu+\frac{1-\sigma}{2}} (\log t)^{\Upsilon-\frac{\sigma}{2}}}
\frac{1}{(1+|y|)^{6+\sigma }}
\chi_0\Bigl( \frac{\lambda_0 y }{\sqrt t}\Bigr)\|\alpha_1\|_{C^1,\nu+\frac{1}{2},\Upsilon} ,
\end{align*}
so
\begin{align*}
\Bigl\|  -\lambda_0^2 \dot \alpha_1 U(y) \chi_0\Bigl( \frac{\lambda_0 y }{\sqrt t}\Bigr) \Bigr\|_{0,\nu,6+\sigma,\epsilon}
\leq \frac{C}{t_0^\vartheta} \|\alpha_1\|_{C^1,\nu+\frac{1}{2},\Upsilon} ,
\end{align*}
for some $\vartheta>0$.

Similarly,
\begin{align*}
\left| -\lambda_0^2 \dot \alpha_0 \Big[ U(y)-U\Bigl(\frac{\xi_1+\lambda_0y}{\lambda_0}\Bigr)    \Bigr] \chi_0\Bigl( \frac{\lambda_0 y }{\sqrt t}\Bigr)  \right|
& \leq C \frac{1}{\log t} \frac{1}{t^2 \log t}
\frac{1}{(1+|y|)^5} \frac{|\xi_1|}{\lambda_0}\chi_0\Bigl( \frac{\lambda_0 y }{\sqrt t}\Bigr) \\
\\
& \leq C \frac{1}{t^{2+\gamma} (\log t)^{\frac{3}{2}}}
\frac{(t\log t)^{\frac{1+\sigma}{2}}}{(1+|y|)^{6+\sigma}}
\chi_0\Bigl( \frac{\lambda_0 y }{\sqrt t}\Bigr) \|\xi_1\|_{C^1,\gamma,0}
\\
&  \leq C \frac{1}{t^{\frac{3-\sigma}{2}+\gamma} (\log t)^{1-\frac{\sigma}{2}}}
\frac{1}{(1+|y|)^{6+\sigma}}
\chi_0\Bigl( \frac{\lambda_0 y }{\sqrt t}\Bigr) \|\xi_1\|_{C^1,\gamma,0}
\end{align*}
so
\begin{align*}
\Bigl\|   -\lambda_0^2 \dot \alpha_0 \Big[ U(y)-U\Bigl(\frac{\xi_1+\lambda_0y}{\lambda_0}\Bigr)    \Bigr] \chi_0\Bigl( \frac{\lambda_0 y }{\sqrt t}\Bigr) \Bigr\|_{0,\nu,6+\sigma,\epsilon}
\leq \frac{C}{t_0^\vartheta} \|\alpha_1\|_{C^1,\nu+\frac{1}{2},\Upsilon} ,
\end{align*}
for some $\vartheta>0$.
The last term in the expression \eqref{dotalphaU} is similar.

The terms in  $ \lambda^4 [ S_0(\textbf{p}_0 + \textbf{p}_1) - S_0(\textbf{p}_0)]$ that contain the function $\varphi_{\lambda_0}$ are
\begin{align*}
&
{\lambda_0}^4
\Bigl[ - \frac{4}{r} \partial_r \varphi_{\lambda_0}
- \nabla\cdot ( \varphi_{\lambda_0} \nabla v_0) - \nabla \cdot ( u_0 \nabla \psi_{\lambda_0} )
-  \nabla\cdot ( \varphi_{\lambda_0} \nabla \psi_{\lambda_0})
\Bigr]
\\
&= 4 \frac{{\lambda_0}^2}{\rho (\rho^2+1)} \partial_\rho \varphi_{\lambda_0}
-(\alpha-1) {\lambda_0}^2 \nabla_y \varphi_{\lambda_0} \cdot \nabla_y \Gamma_0
+ {\lambda_0}^2 \nabla_y \varphi_{\lambda_0} \cdot \nabla_y \mathcal R
+ 2 {\lambda_0}^2 U \chi \varphi_{\lambda_0}
\\
& \quad
-  \alpha \nabla_y (U \chi)  \cdot \nabla_y \psi_{\lambda_0}
- {\lambda_0}^2 \nabla_y ( \varphi_{\lambda_0} \nabla_y \psi_{\lambda_0}) .
\end{align*}
In $ {\lambda_0}^4 [ S_0(\textbf{p}_0 + \textbf{p}_1) - S_0(\textbf{p}_0)]$
these terms appear evaluated at $y$ and then at $\frac{\xi_1}{\lambda_0}+y$.
Using estimates for the the second derivative of $\varphi_{\lambda_0}$ similar to Lemma~\ref{lemma-est-varphilambda} and assuming
\begin{align}
\label{nu2}
\sigma<1, \quad \nu<1+\gamma,
\end{align}
we get
\begin{align*}
\| \lambda^4 [ S_0(\textbf{p}_0 + \textbf{p}_1) - S_0(\textbf{p}_0)] \|_{0,\nu,6+\sigma,\epsilon}  \leq C \frac{1}{t_0^\vartheta} \| \vec\phi\|_{X}.
\end{align*}


\medskip

The main term in $F_3(\phi_1+\phi_2,\varphi,\textbf{p}_1,\ch{\varphi_0^*}) $ that depends on the outer solution is $
\lambda^2 U \varphi^o $
with $\varphi^o = \varphi^* + \varphi $ defined in \eqref{decompPhii}.
Then we have
\begin{align*}
| \lambda^2 U \varphi (y,t) \tilde \chi|
& \leq \frac{\lambda^2}{t^{a-1} (\log t)^\beta} \frac{1}{(1+|y|)^4} \tilde \chi \| \varphi \|_{*,o}
\\
&
\leq C \frac{1}{t^{\nu +\frac{3}{2}}  (\log t)^{\beta+1} }  \frac{1}{(1+|y|)^{4}} \tilde \chi \| \varphi \|_{*,o}
\\
&
\leq C \frac{(t \log t)^{1+\frac{\sigma}{2}}}{t^{\nu + \frac{3}{2}}  (\log t)^{\beta+1} }  \frac{1}{(1+|y|)^{6+\sigma}} \tilde \chi \| \varphi \|_{*,o}
\\
&
\leq C \frac{1}{t^{\nu + \frac{1-\sigma}{2}}  (\log t)^{\beta - \frac{\sigma}{2}} }  \frac{1}{(1+|y|)^{6+\sigma}} \tilde \chi \| \varphi \|_{*,o}.
\end{align*}
Therefore
\begin{align*}
\| \lambda^2 U \varphi \ch{\tilde \chi} \|_{0,\nu,6+\sigma,\epsilon}
\leq
C \frac{1}{t_0^{\frac{1-\sigma}{2}}(\log t_0)^{\beta-\frac{\sigma}{2}}}
\ch{ \| \varphi \|_{*,o}  } .
\end{align*}

Regarding the function $\varphi^*$ (c.f. \eqref{eq-cond-in}) we note that it has the estimate
\begin{align}
\label{est-tilde-phi-o0}
|\varphi^*(x,t)|  \leq
\ch{t_0^{a-1}  (\log t_0)^\beta
\| \varphi_0^* \|_{*,b}
\frac 1{t^{a-1} (\log t)^\beta }  \frac 1 { 1+ |\zeta |^{b}} , \quad \zeta = \frac{x-\xi}{\sqrt{t}} }
\end{align}
by Proposition~\ref{thmOuter-ci}, provided \eqref{cond2b} \ch{holds}, and therefore
\begin{align*}
\| \lambda^2 U \varphi^* \tilde \chi
\|_{0,\nu,6+\sigma,\epsilon}
\leq
\ch{
C t_0^{\nu+1+\frac\sigma2}(\log t_0)^{\frac{\sigma}{2}}
\| \varphi_0^* \|_{*,b} .}
\end{align*}

Let us analyze some of the terms in  $F_3(\phi_1+\phi_2,\varphi,\textbf{p}_1,\ch{\varphi_0^*}) $ that depend on the inner solutions $\phi_1$ and $\phi_2$.
For instance
\begin{align*}
(\alpha-1) \nabla_y \cdot ( \phi_j \nabla_y \Gamma_0)
=
(\alpha-1) \nabla_y  \phi_j \cdot \nabla_y \Gamma_0
-  (\alpha-1) \phi_j U .
\end{align*}
We have the estimate
\begin{align*}
 |(\alpha-1) \nabla_y  \phi_j \cdot \nabla_y \Gamma_0 \tilde \chi|
&\leq \frac{C}{t \log t} \frac{1}{t^{\nu-\frac{1}{2}} (\log t)^{\frac{q-1}{2}}} \frac{1}{(1+|y|)^6}
\| \phi_j \|_{1,\nu-\frac{1}{2},\frac{q-1}{2},4,2+\sigma+\epsilon}
\\
&
\leq C \frac{1}{t^{\nu+\frac{1}{2}-\frac{\sigma}{2}} (\log t)^{1+\frac{q-1}{2}-\frac{\sigma}{2}}} \frac{1}{(1+|y|)^{6+\sigma}}
\| \phi_j \|_{1,\nu-\frac{1}{2},\frac{q-1}{2},4,2+\sigma+\epsilon} ,
\end{align*}
and we get
\begin{align*}
\|(\alpha-1) \nabla_y  \phi_j \cdot \nabla_y \Gamma_0 \tilde \chi\|_{0,\nu,6+\sigma,\epsilon}
\leq \frac{C}{t_0^\vartheta}\| \phi_j \|_{1,\nu-\frac{1}{2},\frac{q-1}{2},4,2+\sigma+\epsilon} ,
\end{align*}
for some $\vartheta>0$.

We also have, writing $\phi = \phi_1+\phi_2$,
\[
\| \lambda \dot \xi_1 \nabla \phi \tilde \chi\|_{0,\nu,6+\sigma,\epsilon}
\leq \frac{C}{t_0^\vartheta} \| \phi\|_{1,\nu-\frac{1}{2},\frac{q-1}{2},4,2+\sigma+\epsilon}
\]
for some $\vartheta>0$, if
\begin{align}
\label{gamma0}
\gamma > \frac{\sigma}{2}.
\end{align}

\ch{Let us estimate the term
$ \nabla_y \cdot ( U \nabla_y( \hat \psi - \psi)) \tilde \chi $ appearing in \eqref{def-F}, where $ \hat\psi = (-\Delta)^{-1} (\lambda^{-2} \phi^i \chi)$,  $ \psi = (-\Delta)^{-1} (\lambda^{-2} \phi^i )$. We recall that $\phi^i = \phi^i_0 + \phi $ , c.f. \eqref{decompPhii}, and therefore we can decompose $\hat \psi = \hat\psi^i_0 + \hat \psi_1$ where $\hat \psi^i_0 = ( -\Delta)^{-1}( \lambda^{-2} \phi^i_0 \chi)$ and $\hat \psi_1 = (-\Delta)^{-1} ( \lambda^{-2} \phi \chi)$. Similarly, we can decompose $\psi = \psi^i_0 + \psi_1$ where $\psi^i_0 = ( -\Delta)^{-1}( \lambda^{-2} \phi^i_0 )$ and $\psi_1 = (-\Delta)^{-2} ( \lambda^{-1} \phi)$.
By linearity we need to estimate separately
$
\nabla_y \cdot ( U \nabla_y( \hat \psi^i_0 - \psi^i_0)) $ and
$\nabla_y \cdot ( U \nabla_y( \hat \psi_1 - \psi_1)) $.
Let us consider the latter one.
Note that
\[
\hat \psi_1 - \psi_1 = (-\Delta)^{-1}[\lambda^{-2} \phi (1-\chi)].
\]
From the definition of the norm $\|\phi\|_{1,\nu-\frac{1}{2},\frac{q-1}{2},4,2+\sigma+\epsilon}$
\begin{align}
\label{decay-phi}
|\phi(y,t)|\leq
\|\phi\|_{1,\nu-\frac{1}{2},\frac{q-1}{2},4,2+\sigma+\epsilon}
\frac{1}{t^{\nu-\frac{1}{2}} (\log t)^{\frac{q-1}{2}}}
\frac{  1  }{(1+|y|)^{4}}
\end{align}
and so
\[
|\nabla_y( \hat\psi_1-\psi_1)(y,t)|
\leq
C \|\phi\|_{1,\nu-\frac{1}{2},\frac{q-1}{2},4,2+\sigma+\epsilon}
\frac{1}{t^{\nu-\frac{1}{2}} (\log t)^{\frac{q-1}{2}}}
\frac{1}{(t\log t)^{\frac{3}{2}}}, \quad
\text{for } |y|\leq 2 \frac{\sqrt t}{\lambda}.
\]
Then
\begin{align*}
| \nabla_y  U \cdot \nabla_y( \hat \psi_1 - \psi_1)) (y,t)|
&\leq
C \|\phi\|_{1,\nu-\frac{1}{2},\frac{q-1}{2},4,2+\sigma+\epsilon}
\frac{1}{t^{\nu + \frac{1-\sigma}{2}} (\log t)^{\frac{q+1-\sigma}{2}}}
\frac{1}{(1+|y|)^{6+\sigma}}
, \quad
\text{for } |y|\leq 2 \frac{\sqrt t}{\lambda}.
\end{align*}
This and a similar estimate for $U \phi( 1-\chi) $ give
\begin{align*}
\| \nabla_y \cdot ( U \nabla_y(\hat\psi_1 - \psi_1)) \tilde \chi \|_{0,\nu,6+\sigma,\epsilon}
\leq \frac{C}{t_0^{\vartheta}} \|\phi\|_{1,\nu-\frac{1}{2},\frac{q-1}{2},4,2+\sigma+\epsilon}
\end{align*}
for some $\vartheta>0$.
A similar estimate is obtained for $\| \nabla_y \cdot ( U \nabla_y(\hat\psi^i_0 - \psi^i_0)) \tilde \chi \|_{0,\nu,6+\sigma,\epsilon}$ using \eqref{est-Phii0}.
}

\ch{Let us estimate next the term $\lambda^2 \nabla_y \cdot (\varphi_\lambda \nabla_y \psi) \tilde \chi$, where we recall, $\psi = (-\Delta)^{-1} (\lambda^{-2} \phi)$. To do this we use that $\phi = \phi_1+\phi_2$ has zero mass and center of mass, that is,
\[
\int_{\R^2} \phi(y,t)\,dy = \int_{\R^2} \phi(y,t)y_j\,dy= 0 , \quad t>t_0.
\]
This and the estimate \eqref{decay-phi} imply
\begin{align*}
|\psi(y,t)| + (1+|y|) |\nabla_y\psi(y,t)| \leq
C \|\phi\|_{1,\nu-\frac{1}{2},\frac{q-1}{2},4,2+\sigma+\epsilon}
\frac{1}{t^{\nu-\frac{1}{2}} (\log t)^{\frac{q-1}{2}}}
\frac{  \log(2+|y|) }{(1+|y|)^2} ,
\end{align*}
by an argument similar to Remark~\ref{rem-newtonian}.
On the other hand, from \eqref{est-varphi-lambda1}
\begin{align*}
|\nabla_y \varphi_\lambda(y,t)|\leq
\frac{C }{t \log t} \frac{1}{(1+|y|)^3}, \quad
|y|\leq 2 \frac{\sqrt t}{\lambda}.
\end{align*}
Therefore
\begin{align*}
|\lambda^2 (\nabla_y \varphi_\lambda \cdot \nabla_y \psi)(y,t)|
& \leq C\|\phi\|_{1,\nu-\frac{1}{2},\frac{q-1}{2},4,2+\sigma+\epsilon}
\frac{\lambda^2}{t^{\nu+\frac{1}{2}} (\log t)^{\frac{q+1}{2}}}
\frac{  \log(2+|y|) }{(1+|y|)^6}
\\
& \leq C \|\phi\|_{1,\nu-\frac{1}{2},\frac{q-1}{2},4,2+\sigma+\epsilon}
\frac{1}{t^{\nu+\frac{1-\sigma}{2}} (\log t)^{\frac{q+1-\sigma}{2}}}
\frac{  1 }{(1+|y|)^{6+\sigma}} , \quad
|y|\leq 2 \frac{\sqrt t}{\lambda}.
\end{align*}
From this coupled with a similar estimate for $\lambda^2 \varphi_\lambda \phi$ we get
\begin{align*}
\| \lambda^2 \nabla_y \cdot (\varphi_\lambda \nabla_y \psi) \tilde \chi \|_{0,\nu,6+\sigma,\epsilon}
\leq \frac{C}{t_0^{\vartheta}} \|\phi\|_{1,\nu-\frac{1}{2},\frac{q-1}{2},4,2+\sigma+\epsilon}
\end{align*}
for some $\vartheta>0$.}

The remaining terms in $F_3(\phi_1+\phi_2,\varphi,\textbf{p}_1,\ch{\varphi_0^*}) $ are estimated in a similar way and we get the validity \eqref{est-Ai1}.

\medskip

Proceeding in the same way we get a Lipschitz bound.
\ch{
Assuming
$ t_0^{a-1}  (\log t_0)^\beta  \| \varphi_0^* \|_{*,b}\leq 1 $,}
for $\|\vec \phi_1\|_X \leq 1$ and $\|\vec \phi_2\|_X\leq 1$ we have
\begin{align}
\nonumber
\| \mathcal A_{i1}[\vec\phi_1,\ch{\varphi_0^*}]-\mathcal A_{i1}[\vec\phi_2,\ch{\varphi_0^*}] \|_{1,\nu-\frac{1}{2},\frac{q-1}{2},4,2+\sigma+\epsilon}
\leq \frac{C}{t_0^\vartheta}
\|\vec \phi_1 - \vec\phi_2\|_{X},
\end{align}
for some $\vartheta>0$ small, a constant $C$ independent of $t_0$, and $t_0$ sufficiently large.
Indeed, the Lipschitz estimate with respect to $\phi_1$, $\phi_2$, and $\varphi$ is direct from the explicit dependence of $F_3(\phi_1+\phi_2,\varphi,\textbf{p}_1,\ch{\varphi_0^*}) $ on these variables, which is either linear or quadratic. The Lipschitz dependence on $\xi_1$ (where $\textbf{p}_1 = ( \alpha_1,\xi_1))$ is also direct from the explicit form of  $F_3(\phi_1+\phi_2,\varphi,\textbf{p}_1,\ch{\varphi_0^*}) $. The Lipschitz condition with respect to $\alpha_1$ appears as an explicit dependence on this variable in $F_3(\phi_1+\phi_2,\varphi,\textbf{p}_1,\ch{\varphi_0^*}) $.

\medskip

Let us estimate the operator $\mathcal A_{i2}$.
We claim that \ch{ if $ t_0^{a-1}  (\log t_0)^\beta  \| \varphi_0^* \|_{*,b}\leq 1 $  and  $\|\vec\phi\|_X\leq 1$, then}
\begin{align}
\label{est-Ai2}
\| \mathcal A_{i2}[\phi_1,\phi_2,\varphi,\textbf{p}_1,\ch{\varphi_0^*}]
\|_{1,\nu-\frac{1}{2},\frac{q-1}{2},4,2+\sigma+\epsilon}
\leq
C ( \log t_0)^{-\frac{1-q}{2}-\Upsilon}
+C t_0^{a-1} (\log t_0)^{\frac{1-q}{2}}
\| \varphi_0^*\|_{*,b}.
\end{align}

Indeed, we apply Proposition~\ref{prop-linear-without-second-moment-1} to get
\begin{align*}
\| \mathcal A_{i2}[\phi_1,\phi_2,\varphi,\textbf{p}_1,\ch{\varphi_0^*}]
\|_{1,\nu-\frac{1}{2},\frac{q-1}{2},4,2+\sigma+\epsilon}
\leq
\frac{C}{(\log t_0)^{1-q}}
 \left\|
 m_2[F_3(\phi_1+\phi_2,\varphi,\textbf{p}_1,\ch{\varphi_0^*})] W_2
\right\|_{0,\nu+\frac{1}{2},\frac{1-q}{2},6+\sigma,\epsilon}
\end{align*}
and since $W_2$ has compact support,
\begin{align*}
\| & \mathcal A_{i2}[\phi_1,\phi_2,\varphi,\textbf{p}_1,\ch{\varphi_0^*}]
\|_{1,\nu-\frac{1}{2},\frac{q-1}{2},4,2+\sigma+\epsilon}
\\
&
\leq
\frac{C}{(\log t_0)^{1-q}}
\sup_{t>t_0} t^{\nu+\frac{1}{2}} (\log t)^{\frac{1-q}{2}}
| m_2[F_3(\phi_1+\phi_2,\varphi,\textbf{p}_1,\ch{\varphi_0^*})] (t)|.
\end{align*}
Using the definition of $F_3$ \eqref{def-F3}
\begin{align*}
m_2[ F_3(\phi,\varphi,\textbf{p}_1 ,\ch{\varphi_0^*}) ]
= m_2[ E_2 \tilde \chi_2] + m_2[F_2(\phi,\varphi,\textbf{p}_1 ,\ch{\varphi_0^*}) \tilde \chi ]
\end{align*}
We have by \eqref{est-E2} (assuming $\sigma<\frac{1}{2}$),
\begin{align*}
|m_2[ E_2 \tilde \chi_2] (t) | \leq \frac{C}{t^{\frac{3+\sigma}{2}}}.
\end{align*}
Therefore, asking that
\begin{align}
\label{nu3}
\nu + \frac{1}{2}<\frac{3+\sigma}{2} \Leftrightarrow \nu < 1+ \frac{\sigma}{2}
\end{align}
we get
\begin{align*}
\sup_{t>t_0} t^{\nu+\frac{1}{2}} (\log t)^{\frac{1-q}{2}}
|m_2[ E_2 \tilde \chi_2] (t) | \leq \frac{C}{t_0^\vartheta},
\end{align*}
for some $\vartheta>0$.

By \eqref{def-F2}
\begin{align*}
m_2[F_2 (\phi,\varphi,\textbf{p}_1,\ch{\varphi^*_0}) \tilde \chi ]
&= \lambda^4 m_2[S_0( u_1( \textbf{p}_0 + \textbf{p}_1 ) ) - S_0( u_1( \textbf{p}_0) )  ]
+ m_2[ F(\phi^i_0 + \phi ,\ch{ \varphi^*} + \varphi ,\textbf{p}_0+\textbf{p}_1) \tilde \chi]
\\
& \quad
+ \lambda^4 m_2[ ( S_0( u_1( \textbf{p}_0 + \textbf{p}_1 ) ) - S_0( u_1( \textbf{p}_0) )   ) (\tilde \chi-1) ].
\end{align*}
Of these terms, the largest is the first one.
By Lemma~\ref{lemma-second-m}, and since $\lambda = \lambda_0$, we get
\begin{align}
\nonumber
\lambda^4 m_2 & [S_0( u_1( \textbf{p}_0 + \textbf{p}_1 ) ) - S_0( u_1( \textbf{p}_0) )  ]
\\
\nonumber
& = - 32\pi \alpha_1
- \frac{\dot\alpha_1}{\lambda_0^2}
\int_{\R^2} U(\frac{x-\xi}{\lambda_0}) \chi_0(\frac{x-\xi}{\lambda_0}) |x-\xi|^2\,dx
\\
\label{second-m2}
& \quad
+ \alpha_1 \int_{\R^2} E(x-\xi,t,\lambda_0)|x-\xi|^2\, dx
- |\xi|^2  \int_{\R^2} S(u_1(\mathbf{p}_0))dx .
\end{align}
But
\begin{align}
\label{ax1}
\sup_{t>t_0} t^{\nu+\frac{1}{2}} (\log t)^{\frac{1-q}{2}}
|\alpha_1(t)|\leq C (\log t_0)^{\frac{1-q}{2}-\Upsilon} \|\alpha_1\|_{C^1,\nu+\frac{1}{2},\Upsilon} ,
\end{align}
under the assumption
\begin{align}
\label{Upsilon1}
\Upsilon> \frac{1-q}{2}.
\end{align}
The second  term in \eqref{second-m2} ia much smaller.
For the last term in \eqref{second-m2} we use Lemma~\ref{lemma-mass-Su1}
and \eqref{bound-dt-mass-varphilambda0}, \eqref{bound-dt-alpha0} to get
\begin{align}
\label{ax-xi1}
 \left| \int_{\R^2} S(u_1(\mathbf{p}_0))dx
\right| \leq \frac{C}{t^2}
\end{align}
and therefore
\begin{align*}
|\xi(t)|^2 \left| \int_{\R^2} S(u_1(\mathbf{p}_0))dx
\right|
\leq \frac{C}{t^{2+2\gamma}} \|\xi_1\|_{C^1,\gamma,0}^2 .
\end{align*}
Combining \eqref{second-m2}, \eqref{ax1} and \eqref{ax-xi1} we get
\begin{align*}
\frac{C}{(\log t_0)^{1-q}}
\sup_{t>t_0} t^{\nu+\frac{1}{2}} (\log t)^{\frac{1-q}{2}}
\lambda^4 |m_2[S_0( u_1( \textbf{p}_0 + \textbf{p}_1 ) ) - S_0( u_1( \textbf{p}_0) )  ] (t)|
\leq C ( \log t_0)^{-\frac{1-q}{2}-\Upsilon} \|\mathbf{p}_1\|_{X_p}.
\end{align*}

Let's estimate the remaining terms in $m_2[ F_3(\phi,\varphi,\textbf{p}_1 ,\ch{\varphi_0^*}) ]$. Consider
\begin{align*}
A(t):=
\int_{\R^2 }
\nabla_y \cdot ( \lambda^2 \varphi_\lambda \nabla_y (-\Delta_y)^{-1} \phi ) \tilde \chi |y|^2 dy
+
\int_{\R^2 }
\nabla_y \cdot (\phi \nabla_y \psi_\lambda) \tilde \chi |y|^2 dy
\end{align*}
which appears in the definition of $F$, where $\phi = \phi_1 + \phi_2$.
It is convenient to write
\begin{align*}
A(t) =
A_1(t) - A_2(t)- A_3(t)
\end{align*}
where
\begin{align*}
A_1(t)
&= \int_{\R^2 }
\nabla_y \cdot [ ( \lambda^2 \varphi_\lambda +\phi) \nabla_y (\psi_\lambda + (-\Delta_y)^{-1} \phi) ] \tilde \chi |y|^2 dy
\\
A_2(t) &=
\int_{\R^2 }
\nabla_y \cdot [ \lambda^2 \varphi_\lambda \nabla_y \psi_\lambda] \tilde \chi |y|^2 dy
\\
A_3(t) &=
\int_{\R^2 }
\nabla_y \cdot [ \phi \nabla_y  \left( (-\Delta_y)^{-1} \phi\right) ] \tilde \chi |y|^2 dy .
\end{align*}
We have, writing $\psi = (-\Delta_y)^{-1} \phi $,
\begin{align*}
A_2(t)
&= \int_{\R^2} \Delta_y \psi \nabla_y \psi \cdot ( |y|^2 \nabla \tilde \chi + 2\tilde \chi y)dy
\\
&= \int_{\R^2} \Delta_y \psi \nabla_y \psi \cdot y
\Bigl[  \frac{\lambda |y|}{2 \sqrt t} \chi_0'\Bigl(\frac{\lambda y}{2 \sqrt t}\Bigr) + 2\tilde \chi_0 \Bigl(\frac{\lambda y}{2 \sqrt t}\Bigr) \Bigr] dy .
\end{align*}
Using Pohozaev's identity
\begin{align*}
|A_2(t) |
&=
\left|
\int_{\R^2}
\nabla_y
\Bigl[
\nabla_y \psi ( y\cdot \nabla \psi) -y \frac{|\nabla_y \psi|^2}{2}
\Bigr]
\Bigl[  \frac{\lambda |y|}{2 \sqrt t} \chi_0'\Bigl(\frac{\lambda y}{2 \sqrt t}\Bigr) + 2\tilde \chi_0 \Bigl(\frac{\lambda y}{2 \sqrt t}\Bigr) \Bigr] dy
\right|
\\
&=
\left|
\int_{\R^2}
\Bigl[
\nabla_y \psi ( y\cdot \nabla \psi) -y \frac{|\nabla_y \psi|^2}{2}
\Bigr]
\cdot
\nabla_y
\Bigl[  \frac{\lambda |y|}{2 \sqrt t} \chi_0'\Bigl(\frac{\lambda y}{2 \sqrt t}\Bigr) + 2\tilde \chi_0 \Bigl(\frac{\lambda y}{2 \sqrt t}\Bigr) \Bigr] dy
\right|
\\
& \leq
C
\int_{2 \sqrt{t}/\lambda \leq |y| \leq 4 \sqrt{t}/\lambda}
|\nabla_y \psi|^2 dy .
\end{align*}
Using that $\psi = (-\Delta)^{-1} \phi$, and
\[
\int_{\R^2} \phi(y,t)dy=0,\quad
\int_{\R^2} \phi(y,t)ydy=0,
\]
we have (see Remark~\ref{rem-newtonian}) for any $\varrho>0$ small,
\begin{align*}
|\nabla \psi(y,t)| \leq \frac{C}{1+|y|^{2-\varrho}}
\frac{1}{t^{\nu-\frac{1}{2}}(\log t)^{\frac{q-1}{2}}}
\|\phi\|_{1,\nu-\frac{1}{2},\frac{q-1}{2},4,2+\sigma+\epsilon}.
\end{align*}
Then
\begin{align*}
\int_{2 \sqrt{t}/\lambda \leq |y| \leq 4 \sqrt{t}/\lambda}
|\nabla_y \psi|^2 dy
& \leq
\frac{C}{t^{2\nu-\varrho} (\log t)^{q-\varrho}}
\|\phi\|_{1,\nu-\frac{1}{2},\frac{q-1}{2},4,2+\sigma+\epsilon}^2
\end{align*}
and so
\begin{align*}
|A_2(t) |
& \leq  \frac{C}{t^{\nu+\frac{1}{2}+\vartheta}} .
\end{align*}
The estimates of the other terms, $A_1$ and $A_2$ is similar, except that $\varphi_\lambda$ doesn't have mass equal to zero. Instead, we use that $\varphi_\lambda$, and $\psi_\lambda$ are radial and
\[
\left|\int_{\R^2} \varphi_\lambda dx \right| \leq
\frac{C}{t \log t},
\]
by Lemmas~\ref{mass-varphi1-lambda1} and \ref{mass-varphi1-lambda2}, to get
\[
|\psi_\lambda(y,t)|\,\le\,  \frac{C}{t \sqrt{ \log t}} \frac{1}{1+|y|}
\]


Let us consider the contribution of the term $\lambda^2 U \varphi^*$. Thanks to \eqref{est-tilde-phi-o0}
\begin{align*}
\| m_2[\lambda^2 U  \varphi^* \tilde \chi] W_2 \|_{0,\nu+\frac{1}{2},\frac{1-q}{2},6+\sigma,\epsilon}
\leq C t_0^{a-1} (\log t_0)^{\frac{1-q}{2}}
\| \varphi_0^*\|_{*,b},
\end{align*}
under the condition
\begin{align}
\label{beta3}
\beta>\frac{1-q}{2}.
\end{align}

The other terms in $m_2$ are estimated in a similar way and we get \eqref{est-Ai2}.

\medskip

Similarly we get that
\ch{ if  $ t_0^{a-1}  (\log t_0)^\beta  \| \varphi_0^* \|_{*,b}\leq 1 $, then}
for $\|\vec \phi_1\|_X \leq 1$ and $\|\vec \phi_2\|_X\leq 1$ we have
\begin{align}
\nonumber
\| \mathcal A_{i2}[\vec\phi_1]- \mathcal A_{i2}[\vec\phi_2] \|_{0,\nu-\frac{1}{2},\frac{q-1}{2},4,2+\sigma+\epsilon}
\leq C ( \log t_0)^{-\frac{1-q}{2}-\Upsilon}
\|\vec \phi_1 - \vec\phi_2\|_{X},
\end{align}
for a constant $C$ independent of $t_0$, where $t_0$ sufficiently large.

\medskip

Let us estimate the operator $
\ch{\mathcal A_o[\phi_1,\phi_2,\varphi,\textbf{p}_1,\varphi_0^*]}$.
We claim that \ch{ if $ t_0^{a-1}  (\log t_0)^\beta  \| \varphi_0^* \|_{*,b}\leq 1 $, then} for $\|\vec \phi\|_X\leq 1$,
\begin{align}
\label{est-Ao}
\| \ch{\mathcal A_o[\vec \phi,\varphi_0^*]} \|_{*,o}
\leq  \frac{C}{(\log t_0)^{\frac{q+1}{2}-\beta}} + C t_0^{a-2}  (\log t_0)^{\beta-1} \|\varphi_0^*\|_{*,b},
\end{align}
and for $\|\vec \phi_1\|_X\leq 1$, $\|\vec \phi_2\|_X\leq 1$
and \ch{$ t_0^{a-1}  (\log t_0)^\beta  \| \varphi_0^* \|_{*,b}\leq 1 $},
\begin{align*}
\| \ch{ \mathcal A_o[\vec \phi_1,\varphi_0^*]
- \mathcal A_o[\vec \phi_2,\varphi_0^*] } \|_{*,o}
\leq   \frac{C}{(\log t_0)^{\frac{q+1}{2}-\beta}}\|\vec \phi_1 - \vec \phi_2\|_X.
\end{align*}
Note that $\frac{q+1}{2}-\beta>0$ by \eqref{cond-nu-a1}.

Indeed, by Proposition~\ref{thmOuter}
\begin{align*}
\| \mathcal A_o[\phi_1,\phi_2,\varphi,\textbf{p}_1,
\ch{\varphi_0^*} ] \|_{*,o}
\leq  C
\|  G_2(\phi_1+\phi_2,\varphi,\textbf{p}_1, \ch{\varphi_0^*} ) \|_{ **,o} ,
\end{align*}
where we recall $G_2$ defined in \eqref{def-G2}.

We start with the term $ \lambda^{-4} E_2 (1-\tilde \chi_2)\chi$. Using the estimate \eqref{est-E2} we get
\begin{align*}
\| \lambda^{-4} E_2 (1-\tilde \chi_2)\chi \|_{**,o}
\leq \frac{C}{t_0^\vartheta}
\end{align*}
for some $\vartheta>0$ provided
\[
a< 4 (1-\delta).
\]
We also directly get from \eqref{est-Su1-outer}
\[
\| S(u_1) (1-\chi) \|_{**,o} \leq \frac{C}{t_0^\vartheta}
\]
for some $\vartheta>0$ if $a<4$.

Regarding the \ch{terms} in $G$ (c.f. \eqref{def-G}) that the depend linearly on $\phi^i  = \phi^i_0 + \phi$, we have for $\|\phi\|_{1,\nu-\frac{1}{2},\frac{q-1}{2},4,2+\sigma+\epsilon} < \infty$
\begin{align}
\nonumber
\Bigl| \frac{1}{\lambda^2} \phi  \Delta \chi \Bigr|(x,t)
& \leq \frac{C}{\lambda^2} \frac{1}{t^{\nu-\frac{1}{2}} (\log t)^{\frac{q-1}{2}}}
\frac{1}{( |x-\xi|/\lambda|) ^4} \frac{1}{t} \Bigl|\Delta_z \chi_0( \frac{x-\xi}{\sqrt t}) \Bigr|
\|\phi\|_{1,\nu-\frac{1}{2} ,\frac{q-1}{2},4,2+\sigma+\epsilon}
\\
\label{computation-o1}
& \leq C \frac{1}{t^{\nu+\frac{5}{2}} (\log t)^{\frac{q+1}{2}}}
\frac{1}{(1+|x-\xi|/\sqrt{t})^b}
\|\phi\|_{1,\nu-\frac{1}{2} ,\frac{q-1}{2},4,2+\sigma+\epsilon}
\end{align}
which implies
\begin{align*}
\Bigl\|
\frac{1}{\lambda^2} \phi  \Delta \chi
\Bigr\|_{**,o} \leq \frac{C}{(\log t_0)^{\frac{q+1}{2}-\beta}} \|\phi\|_{1,\nu-\frac{1}{2} ,\frac{q-1}{2},4,2+\sigma+\epsilon},
\end{align*}
since $\beta<\frac{q+1}{2}$, which is one of the conditions in \eqref{cond-nu-a1}.

We also have, using \eqref{est-Phii0},
\begin{align*}
\Bigl\|  \frac{1}{\lambda^2} \phi_0^i  \Delta \chi \Bigr\|_{**,o} \leq
\frac{C}{t_0^\vartheta}
\end{align*}
for some $\vartheta>0$ if
\begin{align*}
a<4.
\end{align*}
A similar estimate holds for the other terms depending on $\phi^i$.

Some of the terms in $G$ that depend on $\varphi^o = \varphi^* + \varphi $ are
\begin{align*}
\Bigl|
\frac{1}{\lambda^2} U \varphi^o (1-\chi)
\Bigr|
& \leq C \frac{\lambda^2}{|x-\xi|^4} \frac{1}{t^{a-1} (\log t)^\beta}
\frac{1}{(1+|x-\xi|/\sqrt t)^b}
\ch{ (1-\chi)}
\|\varphi^o\|_{*,o}
\\
& \leq \frac{C}{t_0 \log t_0}
\frac{1}{t^a (\log t)^\beta }
\frac{1}{(1+|x-\xi|/\sqrt t)^b}
( \| \varphi^*\|_{*,o} + \| \varphi\|_{*,o} )
\end{align*}
which implies that
\begin{align*}
\Bigl\| \frac{1}{\lambda^2} U \varphi^o (1-\chi)
\ch{\Bigr\|_{**,o}}
\leq \frac{C}{t_0 \log t_0} \|\varphi\|_{*,o}
+ C t_0^{a-2} (\log t_0)^{\beta-1} \|\varphi_0^*\|_{*,b},
\end{align*}
by Proposition~\ref{thmOuter-ci}.
Other terms are estimated in a similar way.

\medskip

Let us estimate the operator $\mathcal A_p$, which is defined by the equations \eqref{def-Ap}.
We claim that if
\[
(0,\tilde \alpha_1 , \tilde \xi_1)  = \mathcal A_p [\phi_1,\phi_2,\varphi,\textbf{p}_1]
\]
and \ch{$ t_0^{a-1}  (\log t_0)^\beta \| \varphi_0^* \|_{*,b}\leq 1$}, $\|\vec \phi\|_X \leq 1$, $\vec \phi = (\phi_1,\phi_2,\varphi,\mathbf{p}_1)$, then
\begin{align}
\nonumber
\| \tilde \alpha_1 \|_{C^1,\nu+\frac{1}{2},\Upsilon}
&\leq C ( \log t_0)^{\Upsilon-\beta}
+ C t_0^{a-1} (\log t_0)^{ \Upsilon} \| \varphi_0^*\|_{*,b}
\\
\label{est-tildexi1}
\| \tilde \xi_1 \|_{*,\gamma,0}&\leq \frac{C}{t_0^\vartheta} +  C t_0^{1+\gamma} (\log t_0)^{\frac{1}{2}}\| \varphi_0^*\|_{*,b},
\end{align}
for some $\vartheta>0$.
\ch{Similarly, we have the following Lipschitz estimate.  If $ t_0^{a-1}  (\log t_0)^\beta  \| \varphi_0^* \|_{*,b}\leq 1$, then for some $\vartheta>0$, and for $\|\vec \phi_1\|_X\leq 1$, $\|\vec \phi_2\|_X\leq 1$,
\begin{align}
\label{est-Ap}
\|  \mathcal A_p[\vec \phi_1,\varphi_0^*]
- \mathcal A_p[\vec \phi_2,\varphi_0^*]  \|_{X_p}
\leq   C ( \log t_0)^{\Upsilon-\beta} \|\vec \phi_1 - \vec \phi_2\|_X ,
\end{align}
for some $\vartheta>0$.
}

Indeed, by \eqref{def-Aalpha}
\begin{align}
\nonumber
|\mathcal A_{p,\alpha_1} [\phi_1,\phi_2,\varphi,\textbf{p}_1,\ch{\varphi_0^*}](t)|
\leq |I_1(t)| +   |I_2(t)| +   |I_3(t)|
\end{align}
where
\begin{align*}
I_1(t) &= \int_t^\infty
\frac{1}{\lambda_0^2}
m_0[ E_2 \tilde \chi_2] (s)ds
\\
I_2(t) &= \int_t^\infty
\frac{1}{\lambda_0^2}
 m_0[ F(\phi^i_0 + \phi ,  \varphi^* + \varphi ,\textbf{p}_0+\textbf{p}_1) \tilde \chi](s)ds
\\
I_3(t) &= \int_t^\infty
\lambda_0^2
m_0[ ( S_0( u_1( \textbf{p}_0 + \textbf{p}_1 ) ) - S_0( u_1( \textbf{p}_0) )   ) (\tilde \chi-1) ] (s) ds.
\end{align*}
Using \eqref{est-E2} and $\int_{\R^2} E_2 dy=0$ we get
\begin{align*}
\left| \frac{1}{\lambda_0^2 } m_0[E_2\tilde \chi_2](t) \right|\leq C\frac{1}{t^{3-2\delta}}.
\end{align*}
This gives
\begin{align}
\label{Aa1}
\| I_1 \|_{C^1,\nu+\frac{1}{2},\Upsilon} \leq C t_0^{\nu-\frac{3}{2}+2\delta},
\end{align}
under the assumption
\begin{align*}
\nu<\frac{3}{2}-2\delta.
\end{align*}
The largest contribution in $I_2$ comes from the term $\lambda^2 U \varphi^o$ in $F(\phi^i_0 + \phi ,  \varphi^* + \varphi ,\textbf{p}_0+\textbf{p}_1) $ (c.f. \eqref{def-F}).
The estimate of this term is
\begin{align*}
\left|
\frac{1}{\lambda_0^2(t)}
\int_{\R^2} \lambda_0(t)^2 U(y) \varphi^o (y,t) dy
\right|
& \leq
C \frac{1}{t^{\nu+\frac{3}{2}} (\log t)^\beta} \|\varphi^o\|_{*,o}
\end{align*}
and so
\begin{align*}
\Bigl\| \int_{\R^2}  U(y) \varphi^o (y,t) dy\Bigr\|_{C^1,\nu+\frac{1}{2},\Upsilon} \leq C ( \log t_0)^{\Upsilon-\beta} \|\varphi^o\|_{*,o}  ,
\end{align*}
under the assumption
\begin{align}
\label{Upsilon2}
\Upsilon<\beta.
\end{align}
Similar estimates for the remaining terms give
\begin{align}
\label{Aa2}
\| I_2 \|_{C^1,\nu+\frac{1}{2},\Upsilon} \leq C ( \log t_0)^{\Upsilon-\beta} \|\vec \phi\|_{X}
+ C t_0^{a-1} (\log t_0)^{ \Upsilon} \| \varphi_0^*\|_{*,b}.
\end{align}
Regarding $I_3$, using \eqref{est-Su1-inner} we have
\begin{align}
\label{Aa3}
\lambda_0^2 m_0[S_0(u_1(\mathbf{p}) ) (\tilde \chi-1)]
&\leq \frac{C}{t^3 \ch{ \log t }}.
\end{align}
Putting together \eqref{Aa1}, \eqref{Aa2}, and \eqref{Aa3} we get
\begin{align*}
\| \mathcal A_{p,\alpha_1} [\phi_1,\phi_2,\varphi,\textbf{p}_1,\ch{\varphi_0^*} ] \|_{C^1,\nu+\frac{1}{2},\Upsilon} \leq  C ( \log t_0)^{\Upsilon-\beta} \|\vec \phi\|_{X}
+ C t_0^{a-1} (\log t_0)^{ \Upsilon} \| \varphi_0^*\|_{*,b}
\end{align*}
assuming also that
\begin{align*}
\nu<\frac{3}{2}.
\end{align*}

The computations leading to \eqref{est-tildexi1} are very similar, under the assumption
\begin{align}
\label{gamma1}
\gamma < \nu-\frac{1}{2}.
\end{align}
This restriction arises when considering the largest term in the expression \eqref{Apxi1}, namely comes from estimating the term $\lambda_0^2 m_{1,j}[\varphi_{\lambda_0} \phi \tilde \chi] $ ($\lambda_0^2 \varphi_{\lambda_0} \phi $ is one of the terms in \eqref{def-F})
\begin{align*}
\frac{1}{\lambda_0} \lambda_0^2 |m_{1,j}[\varphi_{\lambda_0} \phi \tilde \chi](t)|
& \leq C \lambda_0 \int_{\R^2} |\varphi_{\lambda_0} \phi y_j| dy
\\
& \leq C \lambda_0 \frac{1}{t (\log t)^2}  \frac{1}{t^{\nu-\frac{1}{2}} (\log t)^{\frac{q-1}{2}}} \|\phi\|_{1,\nu-\frac{1}{2},\frac{q-1}{2},4,2+\sigma+\epsilon}
\end{align*}

\medskip

Let us summarize the restrictions on the parameters.
We let $0<q<1$ be fixed. We take
\begin{align*}
0<\delta<\sigma<\min(1,4\delta),
\end{align*}
and
\begin{align*}
1<\nu< \min\Bigl( 1+2\delta - \frac{\sigma}{2}  , \frac{3}{2} , 1+\gamma , 1+\frac{\sigma}{2}\Bigr).
\end{align*}
because of  \eqref{firstdelta-sigma}, \eqref{nu2}, \eqref{nu3}.
We also need
\begin{align}
\nonumber
\frac{1-q}{2} < \Upsilon < \beta < \frac{1+q}{2}
\end{align}
by \eqref{Upsilon1}, \eqref{Upsilon2} and  by \eqref{beta3} and \eqref{cond-nu-a1}.
We take
\begin{align*}
\frac{\sigma}{2}<\gamma < \nu-\frac{1}{2}
\end{align*}
by \eqref{gamma0} and \eqref{gamma1}.

Together with the above inequalities we want also the relations
$ \sigma+\epsilon<2$, $\nu+\frac{1}{2}<\frac{7}{4}$ for Proposition~\ref{prop-linear-without-second-moment-1} and $\sigma+\epsilon<\frac{3}{2}$, $\nu<\min( 1  + \frac{\epsilon}{2}, 3-\frac{\sigma}{2},\frac{5}{4}) $ for Proposition~\ref{prop-linear-with-second-moment-1}.  The condition  \eqref{cond2b} for Propositions~\ref{thmOuter} and \ref{thmOuter-ci} hold  by \eqref{cond-nu-a1}. We see that all these restrictions are satisfied by choosing first $\delta$, $\sigma>0$ small so that $2\delta - \frac{\sigma}{2}>0$. Then we take $\nu>1$ close to 1, then let $a=\nu+\frac{5}{2}$ and $b$ satisfying \eqref{cond-nu-a1}.
Then $\Upsilon$, $\beta$ and $\gamma$ can be selected.
Note that with the above procedure we are getting the restriction $b>5$.

\medskip
We already have all elements to solve the fixed point \ch{problem \eqref{fixed1}}, which we recall
\begin{align*}
\vec \phi  = \mathcal A [ \vec \phi ], \quad \vec\phi\in \mathcal B ,
\end{align*}
where  $\mathcal B$ is the closed
unit ball in the Banach space of
functions $\vec \phi $ with
$\|\vec \phi\|_X < +\infty $ and the norm defined in \equ{norma}. Thus
$$
\mathcal B = \{ \, \vec \phi \in X \, | \,  \|\vec \phi\|_X \le 1  \, \}.
$$
%
\ch{Let $\varphi_0^*$ be such that  $t_0^{a-1} (\log t_0)^{\beta}  \|\varphi_0^*\|_{*,b}\leq 1$.
Estimates \eqref{est-Ai1}, \eqref{est-Ai2}, \eqref{est-Ao} and \eqref{est-Ap}, imply that, enlarging the parameter $t_0$ if necessary,
$\mathcal A$ maps  $\mathcal B$ into itself.
We also get that $\mathcal A$ is a contraction mapping on $\mathcal B$.}
The contraction mapping principle yields
the existence of a unique fixed point in $\mathcal B$, which then yields the required existence result.

\begin{remark}
\em
The computations above provide the estimates necessary to show that $\mathcal{A}$ has a fixed point on $\mathcal{B}$. We show here why the scheme works restricting our attention to the most delicate terms only.
We adopt a strategy slightly different to that of finding a fixed point of $\mathcal{A}$:
we consider the inner equations \eqref{inner4a} and \eqref{inner4b} as a system where $\varphi$ is an operator of $\phi = \phi_1+\phi_2$, which is found by solving the outer problem \eqref{outer4}.
In the system  \eqref{inner4a} and \eqref{inner4b} we regard $\alpha_1$ and $\xi_1$ as operators of $\phi$, $\varphi[\phi]$ given by the formulas $\alpha_1 =  \mathcal A_{p,\alpha_1} $ defined in \eqref{def-Aalpha} and $\xi_1 =  \mathcal A_{p,\xi_1} $ defined in \eqref{Apxi1}.

It is natural so use the same norm for $\phi_1$ and $\phi_2$, since it is $\phi = \phi_1 +\phi_2$ that appears in the right hand side of the outer problem \eqref{outer4}. Then estimate \eqref{estimate-E2-tildechi} and Proposition~\ref{prop-linear-with-second-moment-1} suggest that
\begin{align*}
\| \phi_j\|_{1,\nu-\frac{1}{2},\frac{q-1}{2},4,2+\sigma+\varepsilon}<\infty ,
\end{align*}
where we have ignored all other terms in \eqref{inner4a}.
The term $\frac{1}{\lambda^2} \phi \Delta \chi$ in \eqref{outer4}, the computation in \eqref{computation-o1} and Proposition~\ref{thmOuter} give that
\begin{align}
\label{bd-varphi-s}
| \varphi(x,t)|  \leq C \frac{1}{t^{\nu+\frac{3}{2}} (\log t)^{\frac{q+1}{2}}}
\frac{1}{(1+|x-\xi|/\sqrt t)^b} [ \| \phi \|_{1,\nu-\frac{1}{2},\frac{q-1}{2},4,2+\sigma+\varepsilon} + C(\varphi_0^*) ],
\end{align}
where  $C(\varphi_0^*)$ denotes a constant that depends on the initial condition $\varphi_0^*$, and whose exact form we don't need now.

Considering $\varphi $ as an operator of $\phi$ we examine the effect of the therm $\lambda^2 U \varphi$. This term appears in the right hand side of \eqref{inner4a}, where the effect is less important, and in the computation of $\alpha_1$. Ignoring all but this term, we find from \eqref{bd-varphi-s} the estimate
\begin{align*}
|\alpha_1[\phi](t)|
& \leq C \int_t^\infty \int_{\R^2} U(y) |\varphi( \xi + \lambda y) | dy
\\
& \leq C \frac{1}{t^{\nu+\frac{1}{2}} (\log t)^{\frac{q+1}{2}}}
[ \| \phi \|_{1,\nu-\frac{1}{2},\frac{q-1}{2},4,2+\sigma+\varepsilon} + C(\varphi_0^*) ].
\end{align*}
We consider now the effect of $|\alpha_1[\phi](t)|$ in the right hand side of \eqref{inner4b}, where thanks to Lemma~\ref{lemma-second-m} appears mainly as $ \alpha_1(t) W_2(y)$, where $W_2$ is radial with compact support. Then Proposition~\ref{prop-linear-without-second-moment-1} gives
\begin{align*}
| \phi_2(y,t) |
& \leq C \frac{1}{(\log t_0)^{1-q}}\frac{1}{t^{\nu-\frac{1}{2}} (\log t)^{\frac{q+1}{2}+q-1} }
\frac{1}{(1+|y|)^4} \min\Bigl( 1 , \frac{(t \log t)^{1/2}}{|y|} \Bigr)^{2+\sigma+\epsilon}
\\
& \quad \quad \cdot [ \| \phi \|_{1,\nu-\frac{1}{2},\frac{q-1}{2},4,2+\sigma+\varepsilon} + C(\varphi_0^*) ].
\end{align*}
Because we want to use the norm $\| \ \|_{1,\nu-\frac{1}{2},\frac{q-1}{2},4,2+\sigma+\varepsilon}$ we get
\begin{align*}
\| \phi_2 \|_{1,\nu-\frac{1}{2},\frac{q-1}{2},4,2+\sigma+\varepsilon}
\leq C \frac{1}{\log t_0}  [ \epsilon_e + \| \phi \|_{1,\nu-\frac{1}{2},\frac{q-1}{2},4,2+\sigma+\varepsilon} + C(\varphi_0^*) ].
\end{align*}
The $\epsilon_e>0$ on the right hand is there because of the term $E_2 \tilde \chi_2$ in equation \eqref{inner3}. It can be chosen to be a small constant taking $t_0$ large.
The estimate for $\phi_1$ is actually better, and therefore
\begin{align*}
&
\| \phi_1\|_{1,\nu-\frac{1}{2},\frac{q-1}{2},4,2+\sigma+\varepsilon}
+\| \phi_2\|_{1,\nu-\frac{1}{2},\frac{q-1}{2},4,2+\sigma+\varepsilon}
\\
& \qquad \qquad
\leq C \frac{1}{\log t_0}  [ \| \phi_1 \|_{1,\nu-\frac{1}{2},\frac{q-1}{2},4,2+\sigma+\varepsilon} +  \| \phi_2 \|_{1,\nu-\frac{1}{2},\frac{q-1}{2},4,2+\sigma+\varepsilon} + C(\varphi_0^*) ].
\end{align*}
This indicates that the problem for $\phi_1$, $\phi_2$ may be solved by the contraction mapping principle.

\end{remark}

\subsection*{Stability}


If we solve system \eqref{inner4a}, \eqref{inner4b}, \eqref{outer4}, \eqref{eq-param} for a given $\varphi_0^*$ and write $\mathbf{p}=\mathbf{p}(\varphi_0^*)$, we have a solution of problem  \equ{ks1}, which blows up in infinite time as described in Theorem~\ref{teo1}, with initial condition
\begin{align*}
u^*(x,;\varphi_0^*)
&= \frac{\alpha(t_0;\varphi_0^*)}{\lambda_0(t_0)^2}
\Big[ U\Bigl( \frac{x-\xi(t_0;\varphi_0^*)}{\lambda_0(t_0)} \Bigr)
+ \phi_0^i \Bigl( \frac{x-\xi(t_0;\varphi_0^*)}{\lambda_0(t_0)}  \Bigr)
+ c_1(\varphi_0^*) \tilde Z_0\Bigl( \frac{x-\xi(t_0)}{\lambda_0(t_0)}  \Bigr)
\Big]
\\
& \quad \cdot \chi_0 \Bigl( \frac{x-\xi(t_0;\varphi_0^*)}{\sqrt{t_0}}\Bigr)
+ \tilde \varphi_{\lambda_0}(x-\xi(t_0;\varphi_0^*),t_0)  + \varphi_0^*(x).
\end{align*}
where we recall that $\tilde \varphi_\lambda$ was defined in \eqref{deftildephilambda} and $\varphi_\lambda(x,t) = \tilde \varphi_\lambda(x-\xi(t),t) $.
The function $\tilde \varphi$ doesn't depend on $\xi$ and is radial about the origin.

We let $u_0^*(x) = u^*(x;0)$ and
\begin{align*}
I_0 = \int_{\R^2} u_0^*(x)|x|^2 dx.
\end{align*}
Note that $u_0^*$ is radial and so it has center of mass at the origin.

To prove stability we first show the following claim:
if $v:\R^2 \to \R$ satisfies $\|v\|_{*,b}<t_0^{1-a} (\log t_0)^{-\beta}$, has mass zero, and
\begin{align*}
\int_{\R^2} v(x)x_j dx = 0 , \quad \int_{\R^2} v(x)|x|^2 dx = 0 , \quad
\end{align*}
then $u_0^*+v = u^*(\varphi_0^*) $ for some $\varphi_0^*$ with mass zero and $\|\varphi_0^*\|_{*,b}<t_0^{1-a} (\log t_0)^{-\beta}$. Indeed, the equation for $\varphi_0^*$ has the form
\begin{align*}
& \frac{\alpha(t_0;\varphi_0^*)}{\lambda_0(t_0)^2}
\Big[ U\Bigl( \frac{x-\xi(t_0;\varphi_0^*)}{\lambda_0(t_0)} \Bigr)
+ \phi_0^i \Bigl( \frac{x-\xi(t_0;\varphi_0^*)}{\lambda_0(t_0)}  \Bigr)
+ c_1(\varphi_0^*) \tilde Z_0\Bigl( \frac{x-\xi(t_0)}{\lambda_0(t_0)}  \Bigr)
\Big]
\\
& \quad \cdot \chi_0 \Bigl( \frac{x-\xi(t_0;\varphi_0^*)}{\sqrt{t_0}}\Bigr)
+  \tilde \varphi_{\lambda_0}(x-\xi(t_0;\varphi_0^*),t_0)   + \varphi_0^*(x)
\\
&=
 \frac{\alpha(t_0;0)}{\lambda_0(t_0)^2}
\Big[ U\Bigl( \frac{x}{\lambda_0(t_0)} \Bigr)
+ \phi_0^i \Bigl( \frac{x}{\lambda_0(t_0)}  \Bigr)
+ c_1(0) \tilde Z_0\Bigl( \frac{x}{\lambda_0(t_0)}  \Bigr)
\Big]
\\
& \quad \cdot \chi_0 \Bigl( \frac{x}{\sqrt{t_0}}\Bigr)
+ \varphi_{\lambda_0}(x,t_0)  + v .
\end{align*}
A solution to this equation is $\varphi_0^*=v$. Indeed, with this choice,  computing the center of mass we find that $\xi(t_0;\varphi_0^*)=0$, computing the mass we find that $\alpha(t_0;\varphi_0^*)=\alpha(t_0;0)$ and computing the second moment we obtain $ c_1(\varphi_0^*) =c_1(0)$.

Now consider a general $v$ with  $\|v\|_{*,b}< \frac{1}{M}	t_0^{1-a} (\log t_0)^{-\beta}$ and mass zero (where $M>0$ is a constant to be chosen).
We want to show that the initial condition $u_0^*+v$ produces a solution to \eqref{ks1} with infinite time blow as described in Theorem~\ref{teo1}.
To prove this, consider
\begin{align*}
u_{\Lambda,p}(x) = \frac{1}{\Lambda^2} \Bigl[ u_0^*\Bigl(\frac{x-p}{\Lambda}\Bigr) + v\Bigl(\frac{x-p}{\Lambda}\Bigr) \Bigr] ,
\end{align*}
where $p\in \R^2$ and $\Lambda>0$.
Note that $u_{\Lambda,p}$ has mass $8\pi$.
Then we select $\Lambda$ and $p$ such that
\begin{align*}
\int_{\R^2} u_{\Lambda,p}(x) x_j dx =  0 , \quad \int_{\R^2} u_{\Lambda,p}(x) |x|^2 d x = I_0.
\end{align*}
Note that $|\Lambda-1| \leq \frac{C}{M} t_0^{1-a} (\log t_0)^{-\beta}$ and $|p|\leq \frac{C}{M} t_0^{1-a} (\log t_0)^{-\beta}$.
Then we expand
\begin{align*}
u_{\Lambda,p} (x) = u_0^* + w
\end{align*}
and $w$ satisfies $\|w\|_{*,b}\leq \frac{C}{M} t_0^{1-a} (\log t_0)^{-\beta} $, has mass zero, center of mass zero and second moment equal to $I_0$.
Choosing $M= C$, by the previous claim, there is $\varphi_0^*$ with $\|\varphi_0^*\|_{*,b}<t_0^{1-a} (\log t_0)^{-\beta}$ such that $u_{\Lambda_p} = u^*(\varphi_0^*)$. The initial condition $u^*(\varphi_0^*)$ is such that the solution to \eqref{ks1} blows up  as in Theorem~\ref{teo1}. Then the same is true for the initial condition $u_0^*+v$ after a scaling and translation in space.

\section{The mass of \texorpdfstring{$\varphi_\lambda$}{varphi lambda}}
\label{sect-mass-varphilambda}

We devote this section to prove  Proposition~\ref{prop-lambda0}. To that purpose,
a basic step is to derive a formula for the mass of $\varphi_\lambda$ defined in \eqref{defphilambda}.

\medskip
Let us write
\begin{align}
\label{decomp-varphilambda}
\varphi_\lambda =
\varphi_\lambda^{(1)}
+\varphi_\lambda^{(2)}
\end{align}
where $\varphi_\lambda^{(1)}$ and $\varphi_\lambda^{(2)}$ are the solutions, given by Duhamel's formula, of the following problems
\begin{align}
\left\{
\begin{aligned}
\label{eq-phi1}
\partial_t \varphi_\lambda^{(1)}
&= \Delta_6 \varphi_\lambda^{(1)}
+
\frac{\dot \lambda}{\lambda^3} Z_0(\frac{x}{\lambda}) \chi_0(z)
\quad
\text{in }\R^2 \times ({\textstyle\frac{t_0}{2}},\infty)
\\
\varphi_\lambda^{(1)}(\cdot,{\textstyle \frac{t_0}{2}}) &= 0
\end{aligned}
\right.
\end{align}
\begin{align}
\label{eq-phi2}
\left\{
\begin{aligned}
\partial_t \varphi_\lambda^{(2)}
&= \Delta_6 \varphi_\lambda^{(2)}
\ch{+}\frac{1}{2 \lambda^2 t} U \nabla_z \chi_0(z)\cdot z
+ \tilde E ,
\quad
\text{in }\R^2 \times ({\textstyle\frac{t_0}{2}},\infty) ,
\quad z = \frac{x}{\sqrt t} ,
\\
\varphi_\lambda^{(2)}(\cdot,{\textstyle \frac{t_0}{2}}) &= 0
\end{aligned}
\right.
\end{align}
where the operator  $\Delta_6$ is defined in \eqref{def-laplacian6} and $\tilde E$ in \eqref{defTildeE}.
We let $\varphi[p,\lambda](r,t)$ be the solution of the problem
\begin{align}
\left\{
\label{def-varphip}
\begin{aligned}
\partial_t \varphi[p,\lambda]  &= \Delta_6 \ch{ \varphi}[p,\lambda] + \frac{p}{\lambda^4} Z_0\Bigl(\frac{r}{\lambda}\Bigr) \chi\Bigr( \frac{r}{\sqrt t}\Bigl)
\quad \text{in } \R^2 \times ({\textstyle\frac{t_0}{2}},\infty),
\\
\varphi[p,\lambda] (\cdot,{\textstyle\frac{t_0}{2}}) &= 0 \quad \text{in }\R^2 ,
\end{aligned}
\right.
\end{align}
%
%
given by Duhamel's formula.
By definition, we have
\begin{align}
\nonumber
\varphi_{\lambda}^{(1)} = \varphi[\lambda\dot\lambda,\lambda] .
\end{align}
In definitions \eqref{eq-phi1}, \eqref{eq-phi2}, \eqref{def-varphip}, the parameter function $\lambda(t)$ is assumed to be defined for $t>\frac{t_0}{2}$.
In the rest of this section we also assume the validity of the condition stated for $\lambda$ in \eqref{conditions}, namely
\begin{align}
\label{cond-lambda}
|\lambda(t)| + t \log (t) |\dot \lambda(t)|  \leq \frac{C}{\sqrt{\log (t)}}, \quad
t >\frac{t_0}{2},
\end{align}
for some fixed constant $C$.
Let us define
\begin{align}
\label{norm1}
\|p\|_{\gamma,m} = \sup_{t\geq t_0/2} \,  t^\gamma (\log t)^m |p(t)|.
\end{align}
In what follows we shall only deal with radial functions on $\R^2$ and sometimes we will consider them as radial functions on $\R^6$.
For a fixed constant $c_0>0$
we let
\begin{align}
\label{def-lambda-star}
\lambda^*(t) = \frac{c_0}{\sqrt{\log t} } .
\end{align}
The following expansion holds.

\begin{lemma}
\label{mass-varphi1-lambda1}
Assume that $\lambda $ satisfies \eqref{cond-lambda}.
Let $0<\gamma<2$, $m\in \R$ and suppose that  $\|p\|_{\gamma,m} <\infty$.
Then
\begin{align*}
\int_{\R^2} \varphi[p,\lambda](x)dx
= - 4 \pi \int_{t/2}^{t-\ch{\lambda(t)}^2} \frac{p(s)}{t-s}ds + R[p,\lambda]
\end{align*}
where $R[p,\lambda]$ satisfies
\[
\| R[p,\lambda] \|_{\gamma,m}
\leq C \|p\|_{\gamma,m} .
\]
If $\lambda_1,\lambda_2$ satisfy
\[
\Bigl\|
\frac{\lambda_j}{\lambda^*}
\Bigl\|_{L^\infty(t_0/2,\infty) } <\frac{1}{2}, \quad j=1,2,
\]
then we also have
\begin{align}
\label{R}
\| R[p,\lambda^*+\lambda_1] -
R[p,\lambda^*+\lambda_2] \|_{\gamma,m}
\leq C \|p\|_{\gamma,m} \Bigl\|
\frac{\lambda_1-\lambda_2}{\lambda^*}
\Bigl\|_{L^\infty(t_0/2,\infty) }.
\end{align}
\end{lemma}
For the proof of the above result we will need the following calculation.
\begin{lemma}
Let
\begin{align}
\nonumber
f(w)  = \frac{1}{(4\pi)^3}
\int_{\R^6}
e^{-\frac{|z|^2}{4}} \frac{1}{|w-z|^4}dz , \quad
w \in \R^6.
\end{align}
Then
\begin{align}
\label{formula-b}
f(w) = \frac{1}{|w|^4}\Bigl[
1- e^{ \frac{-|w|^2}{4} }\Bigl(
1+\frac{|w|^{\ch{2}}}{4}\Bigr) \Bigr] .
\end{align}
\end{lemma}
\begin{proof}
Let $\varphi_0$ be given by
\[
\varphi_0(x,t) = \frac{1}{(4\pi)^3} \frac{1}{t^3}
\int_{\R^6} e^{-\frac{|y|^2}{4t}} \frac{1}{|x-y|^4}dy ,
\quad x \in \R^6, \ t>0,
\]
which solves
\begin{align*}
\partial_t \varphi_0 & = \Delta_{\R^6} \varphi_0 \quad \text{in }\R^6\times (0,\infty)
\\
\varphi_0(x,0) &= \frac{1}{|x|^4}.
\end{align*}
Then
\[
f(w) = \varphi_0(w,1).
\]
Write
\begin{align*}
\varphi_0(x,t) = \frac{1}{t^2}q\Bigl(\frac{|x|}{\sqrt t}\Bigr) .
\end{align*}
Then
\[
q''(s) + \frac{5}{s}q'(s)+ \frac{s}{2} q'(s) +2 q(s) = 0
\]
and we want $q(s) $ bounded for $s\to 0$, $q(s) = s^{-4}(1+o(1))$ as $s\to \infty$.
A calculation using the explicit element in the kernel of the linear operator, $s^{-4}$, gives
\[
q(s) = \frac{1}{s^4}\Bigl[ 1-e^{-\frac{s^2}{4}}\Bigl( 1+\frac{s^2}{4}\Bigr)\Bigr] , \quad s>0,
\]
and then \eqref{formula-b} follows.

\end{proof}

\begin{proof}[Proof of Lemma~\ref{mass-varphi1-lambda1}]

The solution $\varphi[p,\lambda]$ of \eqref{def-varphip} has the formula
\begin{align*}
\varphi[p,\lambda](x,t)
=
\frac{1}{(4\pi)^3}
\int_{t_0/2}^t
\frac{p(s)}{\lambda^4(s)} \frac{1}{(t-s)^3}
\int_{\R^6}
e^{-\frac{|x-y|^2}{4(t-s)}}
Z_0\Bigl(\frac{y}{\ch{\lambda(s)}}\Bigr) \chi\Bigr( \frac{y}{\sqrt s}\Bigl) dy ds , \quad x \in \R^6.
\end{align*}

Writing
\begin{align*}
\varphi = \varphi[p,\lambda]
\end{align*}
we have
\begin{align*}
\int_{\R^2} \ch{\varphi (x,t)}\,dx
&= \frac{2}{\pi^2}
\int_{\R^6} \varphi(x,t)|x|^{-4}dx
\\
&=
\frac{2}{\pi^2}
\frac{1}{(4\pi)^3}
\int_{t_0/2}^t
\frac{p(s)}{\lambda^4(s)} \frac{1}{(t-s)^3}
\int_{\R^6}
\int_{\R^6}
e^{-\frac{|x-y|^2}{4(t-s)}}
|x|^{-4}dx
Z_0\Bigl(\frac{y}{ \ch{\lambda(s)} }\Bigr) \chi\Bigr( \frac{y}{\sqrt s}\Bigl) dy ds
\\
&=\frac{2}{\pi^2}
\frac{1}{(4\pi)^3}
\int_{t_0/2}^t
\frac{p(s)}{\lambda(s)^4}
\int_{\R^6}
\int_{\R^6}
e^{-\frac{|z|^2}{4}}
\frac{1}{|y-\sqrt{t-s}z|^4}
dz
Z_0\Bigl(\frac{y}{ \ch{\lambda(s)} }\Bigr) \chi\Bigr( \frac{y}{\sqrt s}\Bigl) dy ds
\end{align*}
Using \ch{\eqref{formula-b}} we have
\begin{align*}
\int_{\R^2} \varphi(x,t)dx
&= \frac{2}{\pi^2 }
\int_{t_0/2}^t \frac{p(s)}{\lambda(s)^4}
\int_{\R^6} \ch{ \frac{1}{(t-s)^2}} f( (t-s)^{-1/2}|y|)
Z_0\Bigl(\frac{y}{ \ch{\lambda(s)} }\Bigr) \chi\Bigr( \frac{y}{\sqrt s}\Bigl) dy ds
\\
&= 2\pi
\int_{t_0/2}^t
\frac{p(s)}{\lambda(s)^4}
\int_0^\infty \Bigl[ 1-e^{-\frac{r^2}{4(t-s)}}
\Bigl( 1+\frac{r^2}{4(t-s)}\Bigr) \Bigr]
Z_0\Bigl(\frac{r}{ \ch{\lambda(s)} }\Bigr) \chi\Bigr( \frac{r}{\sqrt s}\Bigl) r dr ds  .
\end{align*}

Let us notice  that
\begin{align*}
&
\frac{1}{2\pi}
\int_{\R^2} \varphi(x,t)dx
\\
& \quad =
\int_{t_0/2}^t \frac{p(s)(t-s)}{\lambda(s)^4}
\int_0^\infty
 \Bigl[ 1-e^{-\frac{z^2}{4}}
\Bigl( 1+\frac{z^2}{4}\Bigr) \Bigr]
Z_0\Bigl(\frac{z\sqrt{t-s}}{ \ch{\lambda(s)} }\Bigr) \chi\Bigr( \frac{z \sqrt{t-s}}{\sqrt s}\Bigl) z dz ds .
\end{align*}
We decompose
\begin{align*}
\frac{1}{2\pi}
\int_{\R^2} \varphi(x,t)dx
&= I_1 + I_2 + I_3
\end{align*}
where
\begin{align*}
I_1 &=
\int_{t_0/2}^{t/2} ...
\\
I_2 &=
\int_{t/2}^{t-\ch{\lambda(t)^2}} ...
\\
I_3 &=
\int_{t-\ch{\lambda(t)^2}}^t ...
\end{align*}
and separately estimate each term.
To estimate  $I_1$ we note that for $s \leq t/2$ we have $\frac{s}{t-s}\leq 1$.
Assuming that $\chi(x) = 0 $ for $x\geq 2$ we obtain
\begin{align*}
\int_0^\infty
\Bigl[ 1-e^{-\frac{z^2}{4}}
\Bigl( 1+\frac{z^2}{4}\Bigr) \Bigr]
Z_0\Bigl(\frac{z\sqrt{t-s}}{ \ch{\lambda(s)} }\Bigr) \chi\Bigr( \frac{z \sqrt{t-s}}{\sqrt s}\Bigl) z dz
&=
\int_{ 0 } ^{2 \frac{\sqrt s}{\sqrt {t-s}}	}  ...
\end{align*}
We estimate for $s \leq t/2$,
\begin{align*}
& \Bigl|
\int_{0} ^{2 \frac{\sqrt s}{\sqrt {t-s}}	}
\Bigl[ 1-e^{-\frac{z^2}{4}}
\Bigl( 1+\frac{z^2}{4}\Bigr) \Bigr]
Z_0\Bigl(\frac{z\sqrt{t-s}}{ \ch{\lambda(s)} }\Bigr) \chi\Bigr( \frac{z \sqrt{t-s}}{\sqrt s}\Bigl) z dz
\Bigr|
\\
& \quad
\leq C
\int_{0} ^{2 \frac{\sqrt s}{\sqrt {t-s}}	}
z^4
\frac{\ch{\lambda(s)^4}}{(t-s)^2 z^4} z dz
\\
& \quad
\leq C
\frac{\lambda(s)^4 s}{(t-s)^3 } ,
\end{align*}
where we have used that
$Z_0(\rho) \leq C / \rho^4$ and
$1-e^{-\frac{z^2}{4}}
( 1+\frac{z^2}{4}) \leq C z^4$.
Therefore
\begin{align*}
|I_1 |
& \leq
\int_{t_0/2}^{t/2}
\frac{|p(s)| s }{(t-s)^2}ds \leq
\|p\|_{\gamma,m}
\int_0^{t/2}
\frac{ s^{1-\gamma} }{(t-s)^2 (\log s)^m}ds
\leq \frac{C}{t^\gamma (\log t)^m } \|p\|_{\gamma,m}.
\end{align*}

Let us analyze $I_2$.
We write
\begin{align*}
I_2
= I_{2,*}+ I_{2,a} + I_{2,b} + I_{2,c} + I_{2,d}
\end{align*}
where
\begin{align*}
I_{2,*}
&=
-16
\int_{t/2}^{t-\ch{\lambda(t)^2}} \frac{p(s)(t-s)}{\lambda(s)^4}
\int_0^\infty
\Bigl[ 1-e^{-\frac{z^2}{4}}
\Bigl( 1+\frac{z^2}{4}\Bigr) \Bigr]
\frac{\lambda(s)^4}{(t-s)^2 z^4}
z  dz ds
\end{align*}
and
\begin{align*}
I_{2,a}
&=
\int_{t/2}^{t-\ch{\lambda(t)^2}}
\frac{p(s)(t-s)}{\lambda(s)^4}
\int_0^{\frac{\ch{\lambda(s)}}{\sqrt{t-s}}}
\Bigl[ 1-e^{-\frac{z^2}{4}}
\Bigl( 1+\frac{z^2}{4}\Bigr) \Bigr]
Z_0\Bigl(\frac{z\sqrt{t-s}}{\ch{\lambda(s)}}\Bigr) \chi\Bigr( \frac{z \sqrt{t-s}}{\sqrt s}\Bigl) z dz ds
\\
I_{2,b}
&=
\ch{16}
\int_{t/2}^{t-\ch{\lambda(t)^2}}
\frac{p(s)(t-s)}{\lambda(s)^4}
\int_0^{\frac{\ch{\lambda(s)}}{\sqrt{t-s}}}
\Bigl[ 1-e^{-\frac{z^2}{4}}
\Bigl( 1+\frac{z^2}{4}\Bigr) \Bigr]
\frac{\lambda(s)^4}{(t-s)^2z^4}
\chi\Bigr( \frac{z \sqrt{t-s}}{\sqrt s}\Bigl) z dz ds
\\
I_{2,c}
&=
\int_{t/2}^{t-\ch{\lambda(t)^2}}
\frac{p(s)(t-s)}{\lambda(s)^4}
\int_{\frac{\ch{\lambda(s)}}{\sqrt{t-s}}} ^\infty
\Bigl[ 1-e^{-\frac{z^2}{4}}
\Bigl( 1+\frac{z^2}{4}\Bigr) \Bigr]
\Bigl[
Z_0\Bigl(\frac{z\sqrt{t-s}}{ \ch{\lambda(s)} }\Bigr)
+ 16\frac{\lambda(s)^4}{(t-s)z^4}
\Bigr]
z dz ds
\\
I_{2,d}
&=
\int_{t/2}^{t-\ch{\lambda(t)^2}}
\frac{p(s)(t-s)}{\lambda(s)^4}
\int_0^\infty
\Bigl[ 1-e^{-\frac{z^2}{4}}
\Bigl( 1+\frac{z^2}{4}\Bigr) \Bigr]
Z_0\Bigl(\frac{z\sqrt{t-s}}{ \ch{\lambda(s)} }\Bigr) \Bigl[
\chi\Bigr( \frac{z \sqrt{t-s}}{\sqrt s}\Bigl) -1
\Bigr] z dz ds
\end{align*}

\ch{A} calculation gives that
\begin{align}
I_{2,*}
\label{I2star}
&=
-2
\int_{t/2}^{t-\ch{\lambda(t)^2}}
\frac{p(s)}{t-s}
ds .
\end{align}

Next we find a bound for $I_{2,a}$.
Using that $Z_0$ is a bounded function and $|  1-e^{-\frac{z^2}{4}}
( 1+\frac{z^2}{4}) |\leq C z^4$, we get
\begin{align*}
& \Bigl|
\int_0^{\frac{\ch{\lambda(s)}}{\sqrt{t-s}}}
\Bigl[ 1-e^{-\frac{z^2}{4}}
\Bigl( 1+\frac{z^2}{4}\Bigr) \Bigr]
Z_0\Bigl(\frac{z\sqrt{t-s}}{\ch{\lambda(s)}}\Bigr) \chi\Bigr( \frac{z \sqrt{t-s}}{\sqrt s}\Bigl) z dz
\Bigr|
\\
& \quad
\leq C \int_0^{\frac{\ch{\lambda(s)}}{\sqrt{t-s}}} z^5 dz \leq C \frac{\ch{\lambda(s)^6}}{(t-s)^3} .
\end{align*}

It follows that
\begin{align*}
|I_{2,a}| &\leq
C \int_{t/2}^{t-\ch{\lambda(t)}^2}
\frac{|p(s)| \lambda(s)^2}{(t-s)^2}ds
\leq \frac{C}{t^\gamma (\log t)^m} \|p\|_{\gamma,m}
\int_{t/2}^{t-\ch{\lambda(t)}^2}
\frac{\lambda(s)^2}{(t-s)^2}ds
\\
& \leq
\frac{C}{ t^\gamma (\log t)^m}
\|p\|_{\gamma,m} .
\end{align*}

Using that  $|  1-e^{-\frac{z^2}{4}}
( 1+\frac{z^2}{4}) |\leq C z^4$, we get
\begin{align*}
\Bigl|
\int_0^{\frac{\ch{\lambda(s)}}{\sqrt{t-s}}}
\Bigl[ 1-e^{-\frac{z^2}{4}}
\Bigl( 1+\frac{z^2}{4}\Bigr) \Bigr]
\frac{\lambda(s)^4}{(t-s)^2z^4}
\chi\Bigr( \frac{z \sqrt{t-s}}{\sqrt s}\Bigl) z dz
\Bigr|
& C \leq \frac{\lambda(s)^4}{(t-s)^2}
\int_0^{\frac{\ch{\lambda(s)}}{\sqrt{t-s}}}  z dz
\\
&\leq C \frac{\ch{\lambda(s)^6}}{(t-s)^3} ,
\end{align*}
and similarly as before,
\begin{align*}
|I_{2,b}|
& \leq \frac{C}{t^\gamma (\log t)^m} \|p\|_{\gamma,m} .
\end{align*}

Using that
\[
Z_0\Bigl(\frac{z\sqrt{t-s}}{\ch{\lambda(s)}}\Bigr)
= - 16 \frac{\ch{\lambda(s)^4}}{(t-s)^2 z^4}
+ O\Bigl(  \frac{\ch{\lambda(s)^6}}{(t-s)^3 \ch{z^6}}\Bigr) , \quad \frac{z\sqrt{t-s}}{\ch{\lambda(s)}}\geq 1,
\]
we get
\begin{align*}
|I_{2,c}|
&
\leq C
\int_{t/2}^{t-\ch{\lambda(t)}^2}
\frac{p(s)(t-s)}{\lambda(s)^4}
\int_{\frac{\ch{\lambda(s)}}{\sqrt{t-s}}}^\infty
\Bigl[ 1-e^{-\frac{z^2}{4}}
\Bigl( 1+\frac{z^2}{4}\Bigr) \Bigr]
\frac{\lambda(s)^6}{(t-s)^3 z^6}
\chi\Bigr( \frac{z \sqrt{t-s}}{\sqrt s}\Bigl) z dz
\\
&
\leq C
\int_{t/2}^{t-\ch{\lambda(t)}^2}
\frac{p(s) \lambda(s)^2}{(t-s)^2}
\int_{\frac{ \ch{\lambda(s)}}{\sqrt{t-s}}}^\infty
\Bigl[ 1-e^{-\frac{z^2}{4}}
\Bigl( 1+\frac{z^2}{4}\Bigr) \Bigr]
\frac{1}{ z^5}
 dz .
\end{align*}
But $\frac{\ch{\lambda(s)}}{\sqrt{t-s}}\leq 2$ in the considered range of $s$, and then
\begin{align*}
|I_{2,c}| & \leq C
\int_{t/2}^{t-\ch{\lambda(t)}^2}
\frac{|p(s)| \lambda(s)^2 }{(t-s)^2}
\log\Bigl( \frac{ \ch{\lambda(s)^2}}{t-s} \Bigr)
ds
\\
& \leq
\frac{C}{t^\gamma (\log t)^m}
\|p\|_{\gamma,m}
\int_{t/2}^{t-\ch{\lambda(t)}^2}
\frac{\lambda(s)^2 }{(t-s)^2}
\log\Bigl( \frac{ \ch{\lambda(s)^2}}{t-s} \Bigr)
ds
\\
& \leq \frac{C}{t^\gamma (\log t)^m}\|p\|_{\gamma,m}.
\end{align*}

Finally, for $I_{2,d}$,
\begin{align*}
& \Bigl|
\int_0^\infty
\Bigl[ 1-e^{-\frac{z^2}{4}}
\Bigl( 1+\frac{z^2}{4}\Bigr) \Bigr]
Z_0\Bigl(\frac{z\sqrt{t-s}}{\ch{\lambda(s)}}\Bigr) \Bigl[
\chi\Bigr( \frac{z \sqrt{t-s}}{\sqrt s}\Bigl) -1
\Bigr] z dz
\Bigr|
\\
& \quad
\leq
\int_{2\sqrt s/\sqrt{t-s}}^\infty
\Bigl[ 1-e^{-\frac{z^2}{4}}
\Bigl( 1+\frac{z^2}{4}\Bigr) \Bigr]
\frac{\lambda(s)^4}{(t-s)^2 z^4}
z dz
\\
& \quad
\leq
\frac{\lambda(s)^4}{(t-s)^2}
\int_{2\sqrt s/\sqrt{t-s}}^\infty
z^{-3} dz
\\
& \quad
\leq
C \frac{\lambda(s)^4}{(t-s) s} .
\end{align*}
Then
\begin{align*}
|I_{2,d}|
&\leq
C \int_{t/2}^{t-\ch{\lambda(t)}^2}
\frac{|p(s)|(t-s)}{\lambda(s)^4}
\frac{\lambda(s)^4}{(t-s) s} ds
\leq
\frac{C}{t^\gamma (\log t)^m}\|p\|_{\gamma,m}.
\end{align*}

Finally we estimate
\begin{align*}
|I_3| &=
\Bigl|
\int_{t-\ch{\lambda(t)}^2} ^t
\frac{p(s)}{\lambda(s)^2}
\int_0^\infty
\Bigl[ 1-e^{-\frac{\rho^2 \lambda^2 }{4(t-s)}}
\Bigl( 1+\frac{\rho^2 \lambda^2 }{4(t-s)}\Bigr) \Bigr]
Z_0(\rho) \chi\Bigr( \frac{\lambda \rho}{\sqrt s}\Bigl) \rho d \rho ds
\Bigr|
\\
& \leq
C \int_{t-\ch{\lambda(t)}^2} ^t
\frac{|p(s)|}{\lambda(s)^2} ds
\\
&\leq \frac{C}{ t^\gamma (\log t)^m} \|p\|_{\gamma,m}.
\end{align*}

In summary, by \eqref{I2star} we have written
\begin{align*}
\frac{1}{2\pi}
\int_{\R^2} \varphi(x,t)dx
&= -2
\int_{t/2}^{t- \ch{\lambda(t)}^2}
\frac{p(s)}{t-s}
ds  + I_1 + I_{2,a} + I_{2,b} + I_{2,c} + I_{2,d} + I_3 ,
\end{align*}
and each of the expressions $I_1$, $ I_{2,a}$, $ I_{2,b} $, $ I_{2,c} $, $ I_{2,d} $, $ I_3$ are linear operators of $p$ with the estimate
\[
\| I_j[p] \|_{\gamma,m}
\leq C \|p\|_{\gamma,m}.
\]

The proof of \eqref{R} follows from the explicit expressions for the  terms $I_j$ in $R$, and similar estimates as before.

\end{proof}

\begin{lemma}
Suppose that  $\lambda$ satisfies \eqref{cond-lambda} and $\varphi_\lambda^{(2)}$ be given by
\eqref{eq-phi2}.
Then
\begin{align}
\label{vp1b}
\varphi_\lambda^{(2)}(0,t;\lambda)
&= -\frac{\lambda(t)^2}{4 t^2}
+ O\Bigl( \frac{1}{t^2 (\log t)^2} \Bigr) ,
\end{align}
as $t\to \infty$, where $ O( \frac{1}{t^2 (\log t)^2} ) $ is uniform in $t_0$.
With $\lambda^*$ given by \eqref{def-lambda-star},
if  $\lambda_1,\lambda_2$ satisfy
\[
\Bigl\|
\frac{\lambda_j}{\lambda^*}
\Bigl\|_{L^\infty(t_0/2,\infty) } <\frac{1}{2}, \quad j=1,2,
\]
then we also have
\begin{align}
\label{R2}
| \varphi_{\lambda^*+\lambda_1}^{(2)}(0,t)
-
\varphi_{\lambda^*+\lambda_2}^{(2)}(0,t)
| \leq \frac{C}{t^2 \log t }
\Bigl\| \frac{\lambda_1-\lambda_2}{\lambda^*} \Bigr\|_{L^\infty(t_0/2,\infty)}.
\end{align}
\end{lemma}

\begin{proof}
For simplicity of notation let us write
$\varphi(x,t;\lambda) = \varphi_\lambda^{(2)}(x,t)$.
Let us write the right hand side of equation \eqref{eq-phi2} in the following form
\begin{align*}
E_2(x,t;\lambda) &= - \frac{1}{2\lambda^2 t} U (y) \nabla_z z_0(z) \cdot z
+
\frac{2}{\lambda^3 t^{1/2}}  \nabla_z \chi_0(z) \cdot \nabla_y U (y)
+ \frac{1}{\lambda^2 t}
\Delta_z \chi_0(z) U (y)
\\
& \quad
- \frac{1}{\lambda^3 t^{1/2}} U(y) \nabla_z \chi_0(z) \cdot \nabla_y \Gamma_0(y) , \quad
y = \frac{x}{\lambda}, \ z = \frac{x}{\sqrt t}.
\end{align*}
To compute $\varphi(0,t;\lambda)$
let us define the following approximation of it
\begin{align*}
\hat \varphi(r,t) = \lambda^2 \tilde \varphi(r,t) ,
\end{align*}
where $\tilde \varphi(r,t)$ solves the radial heat equation in dimension 6:
\begin{align}
\label{eq-phi-o}
\left\{\begin{aligned}
\pp_t \tilde \varphi &=      \pp_r^2  \tilde \varphi+ \frac 5r \pp_r \tilde \varphi
+
\frac{1}{t^3} h\Bigl(\frac{r}{\sqrt t}\Bigr),
\\
\tilde \varphi(r,0)&= 0,
\end{aligned}
\right.
\end{align}
and
\begin{align}
\nonumber
h(\zeta) =
\frac{8}{\zeta^4}
\left[
\chi_0''
- \frac{3}{\zeta} \chi_0'(\zeta)+
\frac{\zeta}{2}\chi_0'(\zeta)
\right] .
\end{align}
The solution $\tilde \varphi(r,t)$ to problem  \eqref{eq-phi-o} can be expressed in self-similar form as
\[
\tilde \varphi(r,t) = \frac{1}{t^2} g(\zeta), \quad \zeta = \frac{r}{\sqrt t}.
\]
We find for $g$ the equation
\be\label{gg}
g'' + \frac{5}{\zeta}g' + \frac{\zeta}{2} g ' + 2 g + h(\zeta) = 0, \quad \zeta\in (0,\infty).
\ee
Using that the function  $\frac{1}{\zeta^4}$ is in the kernel of the homogeneous equation, we find the explicit solution of \equ{gg},
\[
g_0(\zeta) =
- \frac{1}{\zeta^4} \int_0^\zeta x^3 e^{-\frac{1}{4}x^2}
\int_0^x h(y) e^{\frac{1}{4}y^2}y\,dy dx.
\]
To find the solution $\tilde \varphi$ with suitable decay at infinity we let
\be\label{gpx}
g(\zeta) = g_0(\zeta) + \frac{1}{8} \bar z(\zeta) I ,
\ee
where
\[
\bar z(\zeta)  = \frac{1}{\zeta^4} \int_0^\zeta x^3 e^{-\frac{1}{4}x^2}\,dx
\]
is a second solution of the homogeneous equation, linearly independent of $\frac{1}{\zeta^4}$ and
\begin{align*}
I =  \int_0^\infty x^3 e^{-\frac{1}{4}x^2}
\int_0^x h(y) e^{\frac{1}{4}y^2}y\,dy dx.
\end{align*}
We observe that
\[
g(\zeta) =  O( e^{- \frac 14 \zeta^2} )\ass \zeta \to +\infty,
\]
which makes the solution \equ{gpx} the only one with decay faster than $O(\zeta^{-4})$ as $\zeta\to+\infty$.
An explicit calculation gives that
$
I = -8 ,
$
and therefore
\begin{align}
\label{vp1}
\hat \varphi (0,t) = -\frac{\lambda(t)^2}{4 t^2}.
\end{align}

Then, using a barrier for the equation satisfied by $\varphi(x,t;\lambda) - \hat \varphi(x,t)$ we get
\begin{align}
\label{cota-dvp1}
|\varphi(x,t;\lambda) - \hat \varphi(x,t)| \leq C \frac{1}{t^2 (\log t)^2} e^{-c\frac{|x|^2}{t}},
\end{align}
for $t\geq 2$, where $0<c<\frac{1}{4}$.
From \eqref{vp1} and \eqref{cota-dvp1} we obtain \eqref{vp1b}.

The proof of \eqref{R2} is similar.

\end{proof}

\begin{lemma}
\label{mass-varphi1-lambda2}
Suppose that  $\lambda$ satisfies \eqref{cond-lambda} and $\varphi_\lambda^{(2)}$ be given by  \eqref{eq-phi2}.
Then
\begin{align}
\label{int-phi2}
\int_{\R^2}
\varphi_{\lambda}^{(2)}
&= -2\pi \frac{\lambda^2}{t}
- 16 \pi \Upsilon \frac{\lambda^2}{t}
+ O \Bigl( \frac{1}{t^2 ( \log t)^2} \Bigr).
\end{align}
where $\Upsilon$ is defined in \eqref{defUpsilon}, that, is,
$\Upsilon = \int_0^\infty (\chi_0(s)-1) s^{-3}ds$.
\end{lemma}

\begin{proof}
Integrating  \eqref{eq-phi2}
\begin{align*}
\frac{d}{dt} \int_{\R^2}
\varphi_{\lambda}^{(2)}
&=
- 4 \varphi_{\lambda}^{(2)}(0,t)
-  \frac{1}{2 \lambda^2 t} \int_{\R^2}
U(y) \nabla_z \chi_0(z)\cdot z dx
+  \int_{\R^2} \tilde E dx .
\end{align*}

From \eqref{vp1b}
\begin{align*}
\varphi_{\lambda}^{(2)}(0,t)
&= -\frac{\lambda(t)^2}{4 t^2}
+ O\Big( \frac{1}{t^2 (\log t)^2}\Bigr)
\end{align*}
and we compute
\begin{align}
\nonumber
& -  \frac{1}{2 \lambda^2 t}
 U(y) \nabla_z \chi_0(z)\cdot z
+ \tilde E
\\
\nonumber
& \quad =
- \frac{1}{\lambda^2} U(y) \nabla_z \chi_0(z)\cdot z
+\frac{2}{\lambda^2}  \nabla_x \chi \cdot \nabla_x U
+ \frac{1}{\lambda^2}
\Delta_x \chi U
- \frac{1}{\lambda^2} U \nabla \chi \cdot \nabla \Gamma_0
\\
\nonumber
& \quad =
\Bigl[
4 \frac{\lambda^2}{t^3} \chi_0'(s)  \frac{1}{s^3}
- 64 \frac{\lambda^2}{t^3}\chi_0'(s) \frac{1}{s^5}
+ 8 \frac{\lambda^2}{t^3} (\chi_0''(s)+ \frac{1}{s}\chi_0'(s))  \frac{1}{s^4}
\\
\nonumber
& \quad  \qquad
+32 \frac{\lambda^2}{t^3}\chi_0'(s) \frac{1}{s^5}
\Bigr] + O\Bigl( \frac{\lambda^4}{t^4}\Bigr) \chi_{\{ 1\leq s \leq 2 \} }
\\
\nonumber
& \quad =
8 \frac{ \lambda^2}{t^3}
\frac{1}{s^4}
\Big[
\frac{s}{2} \chi_0'(s)
-  \frac{3}{s} \chi_0'(s)
+  \chi_0''(s)
\Bigr]  + O\Bigl( \frac{\lambda^4}{t^4}\Bigr) \chi_{\{ 1\leq s \leq 2 \} }
\end{align}
where $s = \frac{r}{\sqrt t}$.
Then
\begin{align*}
& -  \frac{1}{2 \lambda^2 t} \int_{\R^2}
U(y) \nabla_z \chi_0(z)\cdot z dx
+  \int_{\R^2} \tilde E dx
\\
& \quad =
2\pi\frac{8\lambda^2}{t^2}
\int_0^\infty
\frac{1}{s^4}
\Big[
\frac{s}{2} \chi_0'(s)
-  \frac{3}{s} \chi_0'(s)
+  \chi_0''(s)
\Bigr]sds  + O\Bigl( \frac{\lambda^4}{t^3}\Bigr)
\\
& \quad =
16 \pi \frac{\lambda^2}{t^2}
\Bigl[
\int_0^\infty (\chi_0(s)-1)s^{-3}ds
+ \int_0^\infty (s^{-3}\chi_0')' ds
\Bigl] + O\Bigl( \frac{\lambda^4}{t^3}\Bigr)
\\
& \quad =
16 \pi \frac{\lambda^2}{t^2} \Upsilon
+ O\Bigl( \frac{\lambda^4}{t^3}\Bigr).
\end{align*}

Therefore
\begin{align*}
\frac{d}{dt} \int_{\R^2}
\varphi_{\lambda}^{(2)}
&= 2\pi \frac{\lambda(t)^2}{t^2}
+ 16 \pi \Upsilon \frac{\lambda^2}{t^2}
+ O \Bigl( \frac{1}{t^3 ( \log t)^2} \Bigr)
\end{align*}
and integrating we get
\begin{align*}
\int_{\R^2}
\varphi_{\lambda}^{(2)}
&= -2\pi \frac{\lambda^2}{t}
- 16 \pi \Upsilon \frac{\lambda^2}{t}
+ O \Bigl( \frac{1}{t^2 ( \log t)^2} \Bigr).
\end{align*}
This is the desired expansion \eqref{int-phi2}.
\end{proof}

As a corollary from Lemma~\ref{mass-varphi1-lambda1} and Lemma~\ref{mass-varphi1-lambda2} we get:
\begin{corollary}
\label{coro-mass-vp-lambda}
Assume $\lambda$ satisfies \eqref{cond-lambda}.
Then
\begin{align*}
\int_{\R^2} \varphi_\lambda dx
= - 4 \pi \int_{t/2}^{t-\ch{\lambda(t)}^2} \frac{\lambda \dot \lambda (s)}{t-s}ds
-2\pi \frac{\lambda^2(t)}{t}
- 16 \pi \Upsilon \frac{\lambda^2(t)}{t}
+ O \Bigl( \frac{1}{t^2 ( \log t)^2} \Bigr)
+ R[\lambda\dot\lambda,\lambda],
\end{align*}
where $R$ is as in Lemma~\ref{mass-varphi1-lambda1}.
\end{corollary}

\begin{lemma}
\label{lemma-second-moment-tE}
Let $\tilde E$ be defined by \eqref{defTildeE}.
Assume that $\lambda$ satisfies \eqref{cond-lambda}.
Then
\begin{align}
\label{second-moment-tE}
\int_{\R^2}
\tilde E |x|^2 dx = - 64 \pi \Upsilon \frac{\lambda^2}{t} + O\Bigl( \frac{1}{t^2 (\log t)^2} \Bigr).
\end{align}
\end{lemma}

\begin{proof}
Similarly to the proof of Lemma~\ref{mass-varphi1-lambda2} we have
\begin{align}
\nonumber
\tilde E
&=
8 \frac{ \lambda^2}{t^3}
\frac{1}{s^4}
\Big[
-  \frac{3}{s} \chi_0'(s)
+  \chi_0''(s)
\Bigr] + O\Bigr( \frac{\lambda^4}{t^4} \Bigr) \chi_{ \{ 1 \leq s \leq 2 \} } ,
\end{align}
where $ r = |x|$, $s = \frac{r}{\sqrt t}$,
and so
\begin{align*}
\int_{\R^2}
\tilde E |x|^2 dx
&=
16 \pi  \frac{ \lambda^2}{t}
\int_0^\infty
\frac{1}{s^4}
\Big[
-  \frac{3}{s} \chi_0'(s)
+  \chi_0''(s)
\Bigr]s^3ds  + O\Bigr( \frac{\lambda^4}{t^2} \Bigr)
\\
&= -64 \pi   \frac{ \lambda^2}{t} \Upsilon  + O\Bigr( \frac{\lambda^4}{t^2} \Bigr).
\end{align*}
This is \eqref{second-moment-tE}.
\end{proof}

\begin{lemma}
Let $E$ be defined by \eqref{defE}.
Assume that $\lambda$ satisfies \eqref{cond-lambda}.
Then
\begin{align}
\nonumber
\left| \int_{\R^2}  E |x|^2 dx \right|\leq  \frac{C}{t\log(t)}.
\end{align}
\end{lemma}
\begin{proof}
We have from \eqref{defE}
\begin{align}
\nonumber
E( \zeta ,t;\lambda) =
\frac{\dot \lambda}{\lambda^3} Z_0\Bigl(\frac{\zeta}{\lambda}\Bigr)
\chi_0\Bigl( \frac{\zeta}{\sqrt t}\Bigr)
\ch{+} \frac{1}{2\lambda^2 t} U \Bigl(\frac{\zeta}{\lambda}\Bigr) \nabla_z \chi_0(z) \cdot z
+ \tilde E (x,t)  ,
\end{align}
and we have already computed $\int_{\R^2} \tilde E |x|^2 dx$.
We have
\begin{align*}
\int_{\R^2} Z_0\Bigl(\frac{\zeta}{\lambda}\Bigr)
\chi_0\Bigl( \frac{\zeta}{\sqrt t}\Bigr) |\zeta|^2\,d\zeta
&= 2\pi \lambda^4 \int_0^\infty Z_0(\rho)
\chi_0\Bigl( \frac{\lambda\rho}{\sqrt t}\Bigr)\rho^3 \, d\rho
\\
&= O ( \lambda^4 \log(t) ),
\end{align*}
and so
\begin{align*}
\left| \frac{\dot \lambda}{\lambda^3} \int_{\R^2} Z_0\Bigl(\frac{\zeta}{\lambda}\Bigr)  \chi_0\Bigl( \frac{\zeta}{\sqrt t}\Bigr)|\zeta|^2\,d\zeta\right| \leq  \frac{C}{t \log t}.
\end{align*}

\end{proof}

\subsection{Proof of Proposition~\ref{prop-lambda0}}
\label{subsect-prop-lambda0}

Let
\begin{align}
\nonumber
I[ \lambda ] = 4 \int_{\R^2} \varphi_{\lambda} dx
- \int_{\R^2} \tilde E(\lambda) |x|^2dx.
\end{align}

For the proof we proceed by linearization, that is we look for a function $\lambda_0$
satisfying
\[
|I[\lambda_0]\ch{(t)}|\leq C \frac{1}{t^{\frac{3}{2}+\sigma}}, \quad t>t_0
\]
with the expansion
\[
\lambda_0(t) = \lambda^*(t) + \tilde \lambda_0\ch{(t)}
\]
where $\lambda^*$ was defined in \eqref{def-lambda-star}, that is,
$\lambda^*(t) = \frac{c_0 }{\sqrt{\log t}}$
and $\tilde \lambda_0(t)$, $t>\frac{t_0}{2}$, is a correction. Here $c_0>0$ is a fixed constant.

We claim that
\begin{align}
\label{error-lambda0}
|I[\lambda^*](t)| \leq C \frac{\log (\log t)}{t (\log t)^2}, \quad t>\frac{t_0}{2} ,
\end{align}
with $C$ independent of $t_0$. In the rest of the proof $C$ will be a constant independent of $t_0$ (for $t_0$ large).

Indeed, using the decomposition \eqref{decomp-varphilambda} and the notation \eqref{def-varphip} we have
\begin{align*}
\int_{\R^2} \varphi_{\lambda^*} dx =
\int_{\R^2} \varphi_{\lambda^*}^{(1)} dx
+\int_{\R^2} \varphi_{\lambda^*}^{(2)} dx
\end{align*}
and
\[
\int_{\R^2} \varphi_{\lambda^*}^{(1)} dx
= \int_{\R^2} \varphi[p^*,\lambda^*] dx , \quad p^* = \lambda^*\dot\lambda^*.
\]
By  Lemma~\ref{mass-varphi1-lambda1} we have
\begin{align*}
\left|
\int_{\R^2} \varphi[p^*,\lambda^*] dx
+ 4\pi \int_{t/2}^{t-\lambda^*(t)^2}
\frac{p^*(s)}{t-s}ds
\right|
\leq C \frac{1}{t (\log t)^2} , \quad t>\frac{t_0}{2}.
\end{align*}
Therefore
\begin{align*}
\left|
\int_{\R^2} \varphi[p^*,\lambda^*] dx
+4\pi \log(t) p^*(t) \right|
\leq C \frac{\log(\log t)}
{t (\log t)^2}, \quad t>\frac{t_0}{2}.
\end{align*}
On the other hand, by Lemma~\ref{mass-varphi1-lambda2} we have
\begin{align}
\nonumber
\int_{\R^2} \varphi_{\lambda^*}^{(2)}dx
&=
-2\pi \frac{\lambda^*(t)^2}{t}
- 16 \pi \Upsilon \frac{\lambda^*(t)^2}{t}
+ O\Bigl( \frac{1}{t(\log t)^2} \Bigr),
\end{align}
and by Lemma~\ref{lemma-second-moment-tE}
\begin{align}
\nonumber
\int_{\R^2}
\tilde E(\lambda^*) |x|^2 dx = - 64 \pi \Upsilon \frac{\lambda^*(t)^2}{t} + O\Bigl( \frac{1}{t^2 (\log t)^2} \Bigr).
\end{align}
Using the explicit form of $\lambda^* $ and the previous formulas we deduce \eqref{error-lambda0}.

Next let us rewrite slightly the operator
$I [\lambda]$ as follows. We have
\begin{align}
\nonumber
I[\lambda] = 4
\int_{\R^2}
\varphi[\lambda \dot\lambda,\lambda]
dx
+ 4 \int_{\R^2} \varphi_{\lambda}^{(2)} dx
- \int_{\R^2} \tilde E(\lambda) |x|^2dx.
\end{align}
Let us define
\begin{align*}
R[p,\lambda]
= \int_{\R^2}
\varphi[p,\lambda]
dx
+  4 \pi \int_{t/2}^{t-\lambda^*(t)^2} \frac{p(s)}{t-s}ds  .
\end{align*}
This is similar to the decomposition given in Lemma~\ref{mass-varphi1-lambda1}, but we have changed the interval of integration to $[\frac{t}{2},t-\lambda^*(t)^2]$.
We decompose the integral
\begin{align*}
\int_{t/2}^{t-\lambda^*(t)^2}
\frac{p(s)}{t-s}ds
& = \int_{t/2}^{t-t^{1-\vartheta}} \frac{p(s)}{t-s}ds
+
\int_{t-t^{1-\vartheta}}^{t-\lambda^*(t)^2}  \frac{p(s)}{t-s}ds
\\
& = \int_{t/2}^{t-t^{1-\vartheta}} \frac{p(s)}{t-s}ds
+
p(t)
\int_{t-t^{1-\vartheta}}^{t-\lambda^*(t)^2}  \frac{1}{t-s}ds
\\
&\quad
- \int_{t-t^{1-\vartheta}}^{t-\lambda^*(t)^2}  \frac{p(t)-p(s)}{t-s}ds
\end{align*}
where $0<\vartheta<\frac{1}{2}$ is a fixed constant.

We change variables $\mu = \lambda^2$, so that
\begin{align*}
I[\lambda]
& =
 - 8 \pi \dot \mu (t)
 ( (1-\vartheta) \log(t)- 2\log(\lambda^*(t)) \ch{)}
-
8 \pi  \int_{t/2}^{t-t^{1-\vartheta}} \frac{\dot \mu(s)}{t-s}ds
\\
& \quad
+ 4 \int_{\R^2} \varphi_{\sqrt \mu }^{(2)}dx
+ 2 R[ \dot \mu ,\sqrt \mu]
- \int_{\R^2} \tilde E(\sqrt \mu) |x|^2 dx
\\
& \quad
+ 8 \pi \int_{t-t^{1-\vartheta}}^{t-\lambda^*(t)^2}
\frac{\dot\mu(t)-\dot\mu(s)}{t-s}ds.
\end{align*}
Let $\eta$ be a smooth cut-off such that $\eta(t) = 0 $ for $t<\frac{3}{4}t_0$, $\eta(t) =1 $ for $t>t_0$.
We define
\begin{align*}
\tilde I [\mu]
& =
 - 8 \pi \dot \mu (t)
 ( (1-\vartheta) \log(t)- 2\log(\lambda^*(t)) \ch{)}
-
8 \pi \eta(t) \int_{t/2}^{t-t^{1-\vartheta}} \frac{\dot \mu(s)}{t-s}ds
\\
& \quad
+ 4 \eta(t) \int_{\R^2} \varphi_{\sqrt \mu }^{(2)}dx
+ 2 \eta (t) R[ \dot \mu ,\sqrt \mu]
- \eta(t) \int_{\R^2} \tilde E(\sqrt \mu) |x|^2 dx
\\
& \quad
+ 8 \pi \eta(t) \int_{t-t^{1-\vartheta}}^{t-\lambda^*(t)^2}
\frac{\dot\mu(t)-\dot\mu(s)}{t-s}ds
\end{align*}
which we write
\begin{align*}
\tilde I[\mu] = \ell[\mu] + N[\mu]+ R[\mu],
\end{align*}
where
\begin{align*}
\ell[\mu](t) &=
 - 8 \pi \dot \mu (t)
 ( (1-\vartheta) \log(t)- 2\log(\lambda^*(t)) )
-
8 \pi \eta(t) \int_{t/2}^{t-t^{1-\vartheta}} \frac{\dot \mu(s)}{t-s}ds
\\
N[\mu](t) &=
 4 \eta(t) \int_{\R^2} \varphi_{\sqrt \mu }^{(2)}dx
+ 2 \eta (t) R[ \dot \mu ,\sqrt \mu]
- \eta(t) \int_{\R^2} \tilde E(\sqrt \mu) |x|^2 dx
\\
R[\mu](t) &=
8 \pi \eta(t) \int_{t-t^{1-\vartheta}}^{t-\lambda^*(t)^2}
\frac{\dot\mu(t)-\dot\mu(s)}{t-s} ds.
\end{align*}

Note that $I[\lambda](t) = \tilde I[\lambda^2](t)$ for $t\geq t_0$.

Instead of finding $\lambda$ such that  $I[\lambda]=0$ for $t>t_0$ we are going to construct $\mu$ such that
\[
|\tilde I[\mu](t) |\leq \frac{C}{t^{\frac{3}{2}+\sigma}}, \quad t> \frac{t_0}{2} ,
\]
for some $\sigma>0$.

Let $\mu^* = (\lambda^*)^2$ where $\lambda^*$ is defined in \eqref{def-lambda-star}.
In a first step we will find $\mu_1$ so that
\begin{align}
\label{eq-mu1-a}
\ell[\ch{\mu^*}+    \mu_1] + N[\ch{\mu^*}+\mu_1] + R[\ch{\mu^*}]= 0 , \quad  t> \frac{t_0}{2} .
\end{align}
We will look for $\mu_1$
with $\|\mu_1\|_{*,\gamma,m}<\infty$ where, for a function  $\mu_1 \in C^1([\frac {t_0}2 , \infty))$  with  $\lim_{t\to \infty} \mu_1(t)=0$ we define
\begin{align}
\nonumber
\| \mu_1 \|_{*,\gamma,m}
= \sup_{t\geq t_0/2} t^\gamma ( \log t)^m   |\dot \mu_1(t)| = \|\dot \mu_1\|_{\gamma,m}.
\end{align}

Equation \eqref{eq-mu1-a} takes the form
\begin{align}
\nonumber
0&=- 8 \pi \dot \mu_1
 ( (1-\vartheta) \log(t)- 2\log(\lambda^*(t)) \ch{)}
-
8 \pi \eta(t) \int_{t/2}^{t-t^{1-\vartheta}} \frac{\dot \mu_1(s)}{t-s}ds
\\
\label{f2-b}
& \quad
+ \eta(t)  e_1(t) + \eta(t) F_1[\mu_1](t)  ,
\quad t>\frac{t_0}{2},
\end{align}
where
\[
e_1(t) = \tilde I[\mu^*]
\]
and $F_1$ is an operator with the following properties:
\begin{align}
\label{F1a}
\| F_1[\tilde \mu_1]  \|_{\gamma,m} &\leq C \| \tilde \mu_1 \|_{*,\gamma,m} ,
\\
\label{F1b}
\| F_1[\tilde \mu_1] -F_1[\tilde \mu_2] \|_{\gamma,m} &\leq C \| \tilde \mu_1 -\tilde \mu_2\|_{*,\gamma,m} ,
\end{align}
for $\tilde \mu_j$ satisfying $\|\tilde \mu_j\|_{*,\gamma,m} \leq 1$, with $0<\gamma<2$, $m\in \R$, where $\| \ \|_{\gamma,m}$ is defined in \eqref{norm1}.
From \eqref{error-lambda0} we find
\[
|e_1(t) |\leq C \frac{\log( \log t)}{t (\log t)^2} , \quad t > \frac{t_0}{2} .
\]

Now we apply the contraction mapping principle to the equation \eqref{f2-b} written in the form
\begin{align}
\label{eq-mu1}
\dot \mu_1 &=
-
\eta(t) I_r[\dot\mu_1]
+ \frac{1}{8\pi( (1-\vartheta) \log(t)- 2\log(\lambda^*(t)))}
\eta(t) \bigl[  e_1(t) +  F[\mu_1](t) \bigr] , \  t>\frac{t_0}{2},
\end{align}
where
\begin{align*}
I_r[\dot\mu_1] =
\frac{1}{( (1-\vartheta) \log(t)- 2\log(\lambda^*(t))}  \int_{t/2}^{t-t^{1-\vartheta}} \frac{\dot \mu_1(s)}{t-s}ds .
\end{align*}
We directly check that
\begin{align*}
\|I_r[\dot\mu] \|_{\gamma,m}
\leq \frac{\vartheta}{1-\vartheta}
\| \dot \mu_1\|_{\gamma,m} .
\end{align*}
Let $X$ be the space  $X = \{ \mu_1 \in  C^1([ \frac{t_0}{2},\infty) ) \, | \, \lim_{t\to \infty} \mu_1(t) = 0 \}$ with the norm $\| \mu_1 \|_X = \| \mu_1 \|_{*,1,3-\varepsilon}$, where $0<\varepsilon<1$.
It follows that if $\vartheta<\frac{1}{2}$ the equation \eqref{eq-mu1} has a unique solution $\mu_1 $ in the ball $\overline B_1(0)$ of $X$.

Therefore we have found $\mu_1$ with
$\| \mu_1 \|_{*,1,3-\varepsilon} \leq 1$
so that $\mu = \mu^* + \mu_1$ satisfies
\begin{align}
\label{remainder}
\tilde I[\mu] = - 8 \pi \eta(t) \int_{t-t^{1-\vartheta}}^{t-\lambda^*(t)^2}
\frac{\dot\mu_1(t)-\dot\mu_1(s)}{t-s}ds.
\end{align}
To estimate this remainder we then need  a bound for $\ddot \mu$.
Differentiating with respect to $t$ in the decompositions used in Lemmas~\ref{mass-varphi1-lambda1}, \ref{mass-varphi1-lambda2}, \ref{lemma-second-moment-tE}
we obtain
\begin{align*}
|\dot e_1(t)|\leq C \frac{\log(\log t)}{t^2 (\log t)^2} , \quad t > \frac{t_0}{2}.
\end{align*}
Differentiating in $t$ equation \eqref{eq-mu1} and using the contraction mapping principle we get that for any $\varepsilon>0$ small
\begin{align*}
|\ddot \mu_1(t)| \leq \frac{C}{t^{2-\varepsilon}}.
\end{align*}
Using this  we find that the remainder \eqref{remainder} has the estimate
\begin{align*}
\left| \int_{t-t^{1-\vartheta}}^{t-\lambda^*(t)^2}
\frac{\dot \mu(t)-\dot \mu(s)}{t-s}ds
\right|
& \leq \frac{C}{t^{1+\vartheta-\varepsilon} }, \quad t>\frac{t_0}{2},
\end{align*}
where $\mu = \mu^* + \mu_1$.

Next we introduce another correction $\mu_2$ to improve the decay of the remainder.
We consider $\mu = \mu^* + \mu_1 + \mu_2$ and we consider the following equation for $\mu_2$:
\begin{align*}
\ell[\mu^*+\mu_1+\mu_2] + N[\mu^*+\mu_1+\mu_2] + R[\mu^*+\mu_1] = 0 , \quad t > \frac{t_0}{2}.
\end{align*}
Similarly as before, this equation can be written as
\begin{align}
\nonumber
0&=- 8 \pi \dot \mu_2
 ( (1-\vartheta) \log(t)- 2\log(\lambda^*(t)) )
-
8 \pi \eta(t) \int_{t/2}^{t-t^{1-\vartheta}} \frac{\dot \mu_2(s)}{t-s}ds
\\
\label{f2c}
& \quad
+ \eta(t)  e_2(t) + \eta(t) F_2[\mu_2](t)  ,
\quad t>\frac{t_0}{2},
\end{align}
where $F_2$ satisfies the same estimate \eqref{F1a}  \eqref{F1b}, and $e_2$ has the estimate
\begin{align*}
|e_2(t) |\leq \frac{C}{t^{1+\vartheta-\varepsilon}} , \quad t > \frac{t_0}{2}.
\end{align*}
Using again the contraction mapping principle we find a solution $\mu_2$ of \eqref{f2c} with $ \|\mu_2\|_{*,1+\vartheta-\varepsilon,1} \leq 1$.
Then for $\mu= \mu^* + \mu_1 + \mu_2$
\begin{align*}
\tilde I [\mu](t) = - 8 \pi \eta(t) \int_{t-t^{1-\vartheta}}^{t-\lambda^*(t)^2}
\frac{\dot\mu_2(t)-\dot\mu_2(s)}{t-s}ds.
\end{align*}
To estimate this remainder we need the following bound for $\ddot \mu_2$
\begin{align}
\label{ddotmu2}
|\ddot \mu_2(t) |\leq \frac{C}{t^{2+\vartheta-\varepsilon}}
\end{align}
which is obtained from an estimate for $\dot e_2$, differentiating with respect to $t$ equation \eqref{f2c}. The estimate for $\dot e_2$ is obtained from an analogous estimate for $\frac{d^3 \mu_1}{dt^3}$.

From \eqref{ddotmu2} we find
\begin{align*}
|\tilde I [\mu](t)|
\leq \frac{C}{t^{1+2\vartheta-\varepsilon}} \quad t>\frac{t_0}{2},
\end{align*}
where we recall that $0<\vartheta<\frac{1}{2}$ is arbitrary.

Thus letting $\lambda_0 = \sqrt{\mu}$, $\mu = \mu^* + \mu_1 + \mu_2$ we obtain
\begin{align*}
|I[\lambda_0] |\leq  \frac{C}{t^{1+2\vartheta-\varepsilon}} \quad t>t_0.
\end{align*}
Choosing $\vartheta>\frac{1}{4}$ and $\varepsilon>0$ small, we obtain the properties stated in Proposition~\ref{prop-lambda0}.

\qed

\section{Inner linear theory}
\label{sectLT1}

In this section we consider the problem
\begin{align}
\label{linear1-0}
\left\{
\begin{aligned}
\lambda^2 \partial_t \phi &= L[\phi] + B[\phi]+  h(y,t)  \quad \text{in }\R^2 \times (t_0,\infty)
\\
\phi(\cdot,t_0) &= 0 \quad \text{in }\R^2 .
\end{aligned}
\right.
\end{align}
that appears in the inner equations \eqref{inner4a} and  \eqref{inner4b},
where, we recall
\begin{align}
\nonumber
L[\phi] &= \nabla \cdot \Bigl[ U \nabla \Bigl( \frac{\phi}{U}- (-\Delta)^{-1}\phi \Bigr) \Bigr] ,
\\
\label{newtonian}
(-\Delta)^{-1}\phi(y, t)
&= \frac{1}{2\pi} \int_{\R^2} \log\Bigl(\frac{1}{|y-z|}\Bigr) \phi(z,t )dz.
\end{align}

Slightly more general than the operator $B$ defined in \eqref{defB} we will consider
\begin{align*}
B[\phi] =  \zeta_1(t) [\phi]_{rad} + \zeta_2(t) y \cdot \nabla [\phi]_{rad}
+ ( \zeta_1(t) \phi_1 + \zeta_2(t) y \cdot \nabla \phi_1)  \chi_0\Bigl( \frac{\lambda y}{5 \sqrt t} \Bigr)
\end{align*}
where $[\phi]_{rad}$ is the radial part of $\phi$ (defined in \eqref{def-radial-part}) and $\phi_1 = \phi - [\phi]_{rad}$, and
where $\chi_0$ is the smooth cut-off function \ch{defined in \eqref{chi0}.}
In the sequel we will keep the same notation for $B$.

In what follows we will analyze the linear initial value problem   \eqref{linear1-0}
where we assume that the functions $\lambda(t)$, $\zeta_i(t)$  are continuous, $t_0>1$ and  that for some positive numbers $c$, $C$  we have
\[
\frac c {\sqrt{\log t}}\, \leq \,   \lambda(t) \, \leq\,    \frac C {\sqrt{\log t}}\quad\text{for all } t>t_0,
\]
\begin{align}
\nonumber
|\zeta_i(t)| \le   \frac C {t\log^2 t  }\quad\text{for all } t>t_0.
\end{align}

Next we change the time variable into
\begin{align}
\nonumber
\tau = \tau_0 + \int_{t_0}^t \frac{1}{\lambda(s)^2}ds ,
\end{align}
where $\tau_0 = t_0 \log t_0$.
Then
\begin{align}
\nonumber
\tilde c_1 t \log t \, \leq\,
\tau \, \leq\,  \tilde c_2 t \log t
\end{align}
for some $\tilde c_1,\tilde c_2>0$.
Identifying $\phi(y,t)$ and $h(y,t)$ with  $\phi(y,\tau)$ and $h(y,\tau)$ we rewrite \eqref{linear1-0} as
\begin{align}
\label{linear1}
\left\{
\begin{aligned}
\partial_\tau \phi &= L[\phi] + B[\phi] + h  \quad \text{in }\R^2 \times \ch{(}\tau_0,\infty)
\\
\phi(\cdot,\tau_0) &= 0 \quad \text{in }\R^2 ,
\end{aligned}
\right.
\end{align}

We consider problem  \eqref{linear1} for functions  $h(y,\tau)$ that have fast decay in space. More precisely, we assume that for all $T>0$ there is $C_T$ such that
\begin{align*}
|h(y,\tau)| \leq  \frac{C_T}{1+|y|^6}
\quad
\text{for all }
(y,\tau) \in \R^2 \times \ch{(}\tau_0,T\ch{)}.
\end{align*}
In this case, by a solution $\phi(y,\tau)$ of \eqref{linear1}
we understand a continuous function $\phi(y,\tau)$, of class $C^1$ in $y$, such that for any $T>\tau_0$ there exists a $C_T>0$ with
\begin{align}
\label{gaussian}
|\phi(y,\tau)|+ (1+|y|)|\nabla_y\phi(y,\tau)|\, \leq\,
 \frac{C_T}{1+|y|^6}
\quad \text{for all }
(y,\tau) \in \R^2 \times \ch{(}\tau_0,T\ch{)},
\end{align}
and satisfies the integral equation
\begin{align}
\label{lin}
\phi(y,\tau)
=
\int_{\tau_0}^\tau \int_{\R^2} G(y-z,\tau-s)\,
[ -\nabla & \phi \nabla \Gamma_0 - \nabla U \nabla (-\Delta)^{-1} \phi
\\
\nonumber
& + 2 U \phi
+ B[\phi] + h ](z,s)\, dz ds   ,
\end{align}
where $(-\Delta)^{-1}\phi$ is defined in \eqref{newtonian} and  $G(y,\tau)$ designates the two-dimensional heat kernel,
\begin{align*}
G(y,\tau) = \frac{1}{4 \pi \tau} e^{-\frac{|y|^2}{4 \tau}}.
\end{align*}
From the formula
\begin{align*}
\nabla (-\Delta)^{-1} h (y)
= - \frac{1}{2\pi} \int_{\R^2} \frac{y-z}{|y-z|^2}h(z)dz
\end{align*}
we see that if $|\phi(y)|\leq \frac{C}{1+|y|^6}$ then
\begin{align*}
|\nabla (-\Delta)^{-1} \phi (y)|\leq \frac{C}{1+|y|}
\| (1+|y|^6) \phi\|_{L^\infty(\R^2)}.
\end{align*}
Using this estimate,
existence and uniqueness of a  solution of \eqref{lin} satisfying \eqref{gaussian} are standard. For a short time $T>\tau_0$ this is established by a contraction mapping argument in an appropriate
$L^\infty$-weighted space.  Then a direct linear continuation procedure applies.

\medskip
A first natural condition to impose on $h$ in \eqref{linear1} is that
\begin{align}
\nonumber
\int_{\R^2} h(y,\tau)dy=0\quad \text{for all }\tau>\tau_0,
\end{align}
in order to achieve that the solution has also zero mass at all times.

\medskip

We want to find solutions to \eqref{linear1} that have fast decay in space and time.  For this we need to assume fast space-time decay of the right hand side, which we do by working with the following class of norms.

Given positive numbers $\nu$, \ch{$p$}, $\epsilon$ and $m\in \R$, we let
$\|h\|_{\nu,m,p,\epsilon}$ denote the least $K\ge 0$ such that  for all $\tau>\tau_0$  and for all $y\in \R^2$
\begin{align}
\label{normh0b}
|h(y,\tau)| \ \leq
\ \,
\frac{K}{\tau^{\nu} (\log \tau)^{m}}
\ch{\frac{  1  }{(1+|y|)^p} }
\begin{cases}
\displaystyle
1
& |y|\leq \sqrt { \tau },
\vspace{1mm}
\\
\displaystyle
\frac{\ch{\tau ^{\epsilon/2}}}{|y|^\epsilon}
& |y|\geq  \sqrt {\tau } .
\end{cases}
\end{align}
This is similar to the norm introduced in \eqref{normh0} but defined using $\tau$ instead of $t$.
We will give the results in Sections~\ref{sect-prelim-linear}--\ref{sect-linear-nonradial} using the norm \eqref{normh0b}.

Still, fast decay of the right hand side doesn't imply fast decay of the solution.
For example, consider equation \eqref{linear1-0} without the operator $B$, that is,
\begin{align}
\label{linear1-1}
\left\{
\begin{aligned}
\partial_\tau \phi &= L[\phi] +  h(y,t)  \quad \text{in }\R^2 \times (\tau_0,\infty)
\\
\phi(\cdot,\tau_0) &= 0 \quad \text{in }\R^2 ,
\end{aligned}
\right.
\end{align}
and suppose that $h$ has compact support in space and time, and that $\phi$ has sufficient space-time decay. Then, multiplying \eqref{linear1-1} by $|y|^2$ and integrating in $\R^2 \times (\tau_0,\infty) $ gives
\begin{align*}
\int_{\tau_0}^\infty \int_{\R^2} h (y,\tau) |y|^2 dy d\tau = 0,
\end{align*}
because if $\phi$ is a regular function with fast decay, then
\begin{align*}
\int_{\R^2} L[\phi]|y|^2 dy = 0,
\end{align*}
see Remark~\ref{remark-secondm-Lphi} below.
It is then necessary to impose a condition on $h$, or to adjust a parameter in the problem in order to get a fast decay of the solution.
We develop here the theory by adjusting the parameter $c_1$ in the equation below
\begin{align}
\label{linear-000b}
\left\{
\begin{aligned}
\partial_\tau \phi &= L[\phi] + B[\phi]+  h(y,t)  \quad \text{in }\R^2 \times (\tau_0,\infty),
\\
\phi(\cdot,t_0) &= c_1 \tilde Z_0 \quad \text{in }\R^2 ,
\end{aligned}
\right.
\end{align}
where $\tilde Z_0$ is defined as
\begin{align}
\nonumber
\tilde Z_0(\rho ) &= ( Z_0(\rho) - m_{Z_0} U ) \chi_0\Bigl( \frac{\rho}{\sqrt{\tau_0}}\Bigr),
\end{align}
where $m_{Z_0}$ is such that
\begin{align*}
\int_{\R^2} \tilde Z_0 = 0.
\end{align*}

\begin{prop}
\label{prop-linear-without-second-moment}
Let $\sigma>0$, $\epsilon>0$ with $\sigma+\epsilon<2$ and  $1<\nu<\frac{7}{4}$. Let $0<q<1$.  Then there exists a number $C>0$ such that for $t_0$ sufficiently large and all radially symmetric $h=h(|y|,\tau)$ with $\|h\|_{\nu,m,6+\sigma,\epsilon}<\infty$ and
\begin{align*}
\int_{\R^2} h(y,\tau)dy &= 0
,\quad \text{for all } \tau>\tau_0 ,
\end{align*}
there exists  $c_1 \in \R$ and  solution $\phi(y,\tau) = \mathcal T^{i,2}_{\textbf{p}} [h]$ of problem
\ch{\eqref{linear-000b}}
that defines a linear operator of $h$ and
satisfies the estimate
\begin{align*}
\|  \phi \|_{\nu-1,m+q,4,2+\sigma+\epsilon}  \leq \frac{C}{(\log\tau_0)^{1-q}}  \|h\|_{\nu,m,6+\sigma,\epsilon}.
\end{align*}
Moreover $c_1$ is a linear operator of $h$ and
\begin{align*}
|c_1| & \leq  C \frac{1 }{ \tau_0^{\nu-1} (\log\tau_0)^{m+1}}  \|h\|_{\nu,m,6+\sigma,\epsilon}.
\end{align*}
\end{prop}

We have stated this result only in the radial setting, because this is what is needed, but there is a version of it in the non-radial case.

The next result is for the problem
\begin{align}
\label{linear-002}
\left\{
\begin{aligned}
\partial_\tau \phi &= L[\phi] + B[\phi]+  h(y,\tau)  \quad \text{in }\R^2 \times (\tau_0,\infty),
\\
\phi(\cdot,\tau_0) &= 0 \quad \text{in }\R^2 ,
\end{aligned}
\right.
\end{align}
and holds without the radial symmetry assumption but assuming orthogonality of the right hand side with respect to all elements in the kernel and for all times.

\begin{prop}
\label{prop-linear-with-second-moment}
Let $0<\sigma<1$, $\epsilon>0$ with $\sigma+\epsilon<\frac{3}{2}$
and  $1<\nu< \min( 1+\frac{\epsilon}{2},3-\frac{\sigma}{2}, \frac{5}{4})$.
Let $0<q<1$.
Then there is $C$ such that for \ch{$\tau_0$} large the following holds.
Suppose that  $h$ satisfies  $\|h\|_{\nu,m,6+\sigma,\epsilon}<\infty$ and
\begin{align}
\nonumber
\int_{\R^2} h(y,\tau)dy&=0 , \quad
\int_{\R^2} h(y,\tau)|y|^2dy=0,
\\
\nonumber
\int_{\R^2} h(y,\tau) y_j dy&=0 , \quad j=1,2 ,\quad \text{for all } \tau>\tau_0 .
\end{align}
Then there exists a solution $\phi(y,\tau)= \mathcal T^{i,1}_{\textbf{p}} [h]$ of problem
\ch{\eqref{linear-002}}
that defines a linear operator of $h$ and
satisfies
\begin{align}
\nonumber
\|  \phi \|_{\nu-\frac 12 ,m+\frac {q}{2},4,2+\sigma+\epsilon} \leq C \|h\|_{\nu,m,6+\sigma,\epsilon}.
\end{align}
\end{prop}

Propositions~\ref{prop-linear-without-second-moment-1} and \ref{prop-linear-with-second-moment-1} given in Section~\ref{sect-proof-existence} are direct corollaries of Propositions~\ref{prop-linear-without-second-moment} and  \ref{prop-linear-with-second-moment}. The only changes are due to the change in the time variable, because $\tau \sim t \log t$, and the fact that the norms for the solutions in Propositions~\ref{prop-linear-without-second-moment-1} and \ref{prop-linear-with-second-moment-1} include a gradient term. The estimate for the gradient follows from  the weighted $L^\infty$ estimate, scaling and standard parabolic estimates.

The proofs of Propositions~\ref{prop-linear-without-second-moment} and  \ref{prop-linear-with-second-moment} are contained in Sections~\ref{sect-prelim-linear}--\ref{sect-linear-nonradial}. They are based on an energy inequality obtained by multiplying the equation by a suitable test function, and using an inequality for a quadratic form.
Section~\ref{sect-prelim-linear} contains some preliminaries on this quadratic form.

In Proposition~\ref{prop-linear-with-energy}, we obtain an additive decomposition of the solution $\phi(y,\tau)$ of \eqref{linear-000b} into a part with a relatively slow space decay that loses $\tau^{1/2}$ with respect to the time decay of the right hand side, and a term along $Z_0(y)$ that loses an entire power of $\tau$.
This is the key element for the proof of Proposition~\ref{prop-linear-without-second-moment} in Section~\ref{sect-linear-decomp} (p.\pageref{proof-linear-without-second-moment}).

Then the proof of Proposition~\ref{prop-linear-with-second-moment}
in the radial case uses Proposition~\ref{prop-linear-with-energy} after formally applying the operator $L^{-1}$ to the original equation and performing a {\em concentration procedure} that improves the space decay of the resulting error. This is done on Section~\ref{sect-theorem-linear-with-second-moment}, and we give there a proof of  Proposition~\ref{prop-linear-with-second-moment} in the case of radial functions.

The proof of Proposition~\ref{prop-linear-with-second-moment} in the general case is in Section~\ref{sect-linear-nonradial} (p.\pageref{proof-linear-with-second-moment}).. The idea is that the decomposition obtained in  Proposition~\ref{prop-linear-with-energy} for solutions with no radial mode does not contain the term along $Z_0$, which allows us to obtain a much better estimate.

\section{Preliminaries for the linear theory}
\label{sect-prelim-linear}

A central ingredient in obtaining good estimates for the linearized parabolic operator associated to the inner problem is the analysis
of the quadratic form
\begin{align}
\label{q-form}
\phi\mapsto  \int_{\R^2} g\phi ,\quad  g =  \frac {\phi} U  - (-\Delta)^{-1} \phi.
\end{align}
\ch{This quadratic form arises when considering the linearized Keller-Segel problem \eqref{linear1-0}.
Indeed, $L[\phi] = \nabla \cdot ( U \nabla g) $ and it is natural to test the equation \eqref{linear1-0} with $g$, since
\[
\int_{\R^2} L[\phi] g = \int_{\R^2}  \nabla \cdot ( U \nabla g) = - \int_{\R^2} U |\nabla g|^2.
\]
But from the time derivative we get $\lambda^2 \int_{\R^2} \partial_t \phi g $, which leads to \eqref{q-form}.}

We observe that $g$ has degeneracy directions.
Indeed,  if $\psi = (-\Delta)^{-1} \phi$ then
$$
\Delta \psi + U(y)\psi = -Ug \inn \R^2 .
$$
The operator $\Delta \psi + U(y) \psi $ is classical. It   corresponds to linearizing the Liouville equation
\begin{align*}
\Delta v + e^{v} =0\quad \text{in }\R^2,
\end{align*}
around the solution $\Gamma_0 = \log U$. It is well known that the bounded kernel of this linearization is spanned by the generators of rigid motions, namely dilation and translations of the equation, which are precisely the functions $z_0,z_1,z_2$ defined by
\begin{align}
\label{defZLiouville}
\left\{
\begin{aligned}
z_0(y) &= \nabla \Gamma_0(y) \cdot y + 2 \\
z_j(y) & = \partial_{y_j} \Gamma_0(y), \quad j=1,2 .
\end{aligned}
\right.
\end{align}
\ch{Note that} $g$ is precisely annihilated at the linear combinations of these functions. In the rest of this section we will state and prove several estimates that take into account this issue, which will be crucial later on.

%
\ch{The quadratic form \eqref{q-form} can be naturally transformed into a similar one in $S^2$ by stereographic projection $\Pi:S^2\setminus\{(0,0,1)\}\to \R^2$}
\begin{align}
\nonumber
\Pi(y_1,y_2,y_3) = \Bigl( \frac{y_1}{1-y_3},\frac{y_2}{1-y_3}\Bigr).
\end{align}
For $\varphi:\R^2 \to \R$ we write
\[
\tilde \varphi = \varphi \circ \Pi , \quad \tilde \varphi : S^2 \setminus \{(0,0,1)\}\to \R.
\]
Then we have the following formulas
\begin{align*}
\int_{S^2} \tilde \varphi &= \frac{1}{2}\int_{\R^2} \varphi	U
\\
\int_{S^2} \tilde U |\nabla_{S^2} \tilde \varphi|^2 &=
\int_{\R^2} U |\nabla_{\R^2} \varphi|^2
\\
\frac{1}{2} \tilde U \Delta_{S^2} \tilde \varphi
&= ( \Delta_{\R^2} \varphi ) \circ \Pi.
\end{align*}

\subsection{The Liouville equation}

Here we consider the linearized Liouville equation
\begin{align}
\label{Liouville}
\Delta \psi + U \psi + h=0 \quad \text{in }\R^2 .
\end{align}
The stereographic projection transforms
the linearized Liouville equation \eqref{Liouville}
into
\begin{align}
\label{liouvilleSphere}
\Delta_{S^2} \tilde \psi +2 \tilde \psi + 2 \tilde h = 0
\end{align}
in $S^2\setminus\{P\}$, $P = (0,0,1)$,
where $\tilde \psi = \psi \circ \Pi$, $\tilde h = (U^{-1} h) \circ \Pi$.

The functions in \eqref{defZLiouville} are transformed through the stereographic projection into constant multiples of the coordinate functions
\[
\tilde z_j(\omega) = c_j \omega_j ,\quad j=1,2, \quad \tilde z_0(\omega) = c_0 \omega_3, \quad \omega = (\omega_1,\omega_2,\omega_3) \in S^2.
\]
By standard elliptic theory, if $\tilde h \in L^p(S^2)$, $p>2$, then exists a solution $\tilde \psi_0 \in W^{2,p}(S^2)$ to \eqref{liouvilleSphere} in $S^2$ if and only if $\tilde h$ satisfies
\begin{align}
\nonumber
\int_{S^2} \tilde h \tilde z_j = 0 , \quad j=1,2,3.
\end{align}
This solution is unique if we normalize it such that
\[
\int_{S^2} \tilde \psi_0 \tilde z_j = 0 , \quad j=1,2,3,
\]
and then satisfies the estimate
\[
\| \tilde \psi_0 \|_{C^{1,\alpha}(S^2)} \leq C \|\tilde h\|_{L^p(S^2)}
\]
where $\alpha =  1-  \frac{2}{p}$.
By subtracting off a suitable linear combination of the functions $\tilde z_j$, $j=0,1,2$ we obtain the unique solution $\tilde \psi_1$ to \eqref{liouvilleSphere} in $S^2$ satisfying
\begin{align}
\label{tildePsi1}
\tilde \psi_1(P) = 0, \quad \nabla_{S^2} \tilde \psi_1(P) = 0 .
\end{align}
For this solution we also have the estimate
\begin{align}
\label{LpH}
\| \tilde \psi_1 \|_{C^{1,\alpha}(S^2)} \leq C \|\tilde h\|_{L^p(S^2)}.
\end{align}

\begin{lemma}
\label{lemma-liouville}
Let  $0<\sigma<1$. Then  there is $C$ such that if
$\psi $ satisfies \eqref{Liouville} and $\psi(y)\to 0$ as $|y|\to \infty$ with $h$ satisfying
$\|(1+|y|)^{3+\sigma}h\|_{L^\infty(\R^2)}<+\infty$ and
\begin{align}
\label{conditions-h}
\int_{\R^2} (U \psi + h(y)) dy= 0, \quad
\int_{\R^2} (U\psi + h(y)) y_j \, dy = 0, \quad \ch{j=1,2 },
\end{align}
then
\begin{align*}
\|(1+|y|)^{1+\sigma}\psi\|_{L^\infty(\R^2)}\, \leq\,
C\|(1+|y|)^{3+\sigma}h\|_{L^\infty(\R^2)}.
\end{align*}
\end{lemma}

\begin{remark}
\label{rem-newtonian}
Let $h:\R^2\to \R$ satisfy $\|(1+|y|)^{2+\sigma}h\|_{L^\infty(\R^2)}<+\infty$ where $0<\sigma<1$.
If
\[
\int_{\R^2} h(y) dy = 0
\]
then
\begin{align*}
|(-\Delta)^{-1} h(y)| \leq
\frac{C}{(1+|y|)^\sigma} \|(1+|y|)^{2+\sigma}h\|_{L^\infty(\R^2)}.
\end{align*}

If $h:\R^2\to \R$ satisfy $\|(1+|y|)^{3+\sigma}h\|_{L^\infty(\R^2)}<+\infty$ where $0<\sigma<1$ and in addition to mass zero we have
\begin{align*}
\int_{\R^2} h(y)y_jdy=0, \quad j=1,2,
\end{align*}
then
\begin{align*}
|(-\Delta)^{-1} h(y)| \leq
\frac{C}{(1+|y|)^{1+\sigma}} \|(1+|y|)^{3+\sigma}h\|_{L^\infty(\R^2)}.
\end{align*}
The first claim is standard.
For the second, write
$$
(-\Delta)^{-1} h(x) = \frac 1{2\pi }  \int_{\R^2}
(\log |x| - \log|x-y|  + \frac {y\cdot x}{|x|^2} ) h(y)dy
$$
and estimate the integral after splitting it into the regions $|y|<\frac{|x|}{2}$ and its complement.

\end{remark}

\begin{proof}[Proof of Lemma~\ref{lemma-liouville}]
We claim that $\psi = (-\Delta)^{-1} (U\psi +h)$. Indeed the function $\psi - (-\Delta)^{-1} (U\psi +h)$ is harmonic in $\R^2$ and decays to 0 at infinity, and therefore it is equal to 0.
The assumptions \eqref{conditions-h} and Remark~\ref{rem-newtonian} imply that
\begin{align}
\label{estPsi2}
\|(1+|y|)^{1+\sigma}\psi\|_{L^\infty(\R^2)}\, \leq\,
C\|(1+|y|)^{3+\sigma}h\|_{L^\infty(\R^2)}
+
C\| \psi \|_{L^\infty(\R^2)}.
\end{align}

Let $\tilde \psi = \psi \circ \Pi$, so that it satisfies \eqref{liouvilleSphere} in $S^2\setminus\{P\}$ with $\tilde h = (U^{-1} h)\circ \Pi$. Note that $\tilde h \in L^p(S^2)$ for some $p>2$. More precisely
\begin{align}
\label{lp2}
\|\tilde h \|_{L^p(S^2)} \leq C \|(1+|y|)^{3+\sigma} h\|_{L^\infty( \R^2)} ,
\end{align}
with $p<\frac{2}{1-\sigma}$.
The singularity at $P$ is removable and thus $\tilde \psi$ satisfies \eqref{liouvilleSphere} in $S^2$.
By elliptic regularity $\tilde \psi \in C^{1,\alpha}(S^2)$ for some $\alpha>0$.
Since $\psi$ decays at infinity, $\tilde \psi(P)=0$.
By \eqref{estPsi2} we have also $\nabla_{S^2} \tilde \psi(P)= 0$.

We let $\tilde \psi_1$ denote the solution to \eqref{liouvilleSphere} satisfying \eqref{tildePsi1}. The solution to \eqref{liouvilleSphere} in $S^2$ satisfying \eqref{tildePsi1} is unique, so that we have $\tilde \psi = \tilde \psi_1$ and by estimate \eqref{LpH}, \eqref{lp2} and \eqref{estPsi2} we obtain
\begin{align*}
\|(1+|y|)^{1+\sigma}\psi\|_{L^\infty(\R^2)}\, \leq\,
C\|(1+|y|)^{3+\sigma}h\|_{L^\infty(\R^2)}.
\end{align*}

\end{proof}

\subsection{\ch{A quadratic form}}
\label{subsect-q-form}

Here we discuss properties of the quadratic form \eqref{q-form}.
For this we consider a function $\phi:\R^2 \to \R$ with sufficient decay, in the form,
\begin{align}
\label{decayPhi}
|\phi(y)|\leq \frac{1}{(1+|y|)^{2+\sigma}} ,
\end{align}
with $0<\sigma<1$, and zero mass:
\begin{align}
\label{intPhi0}
\int_{\R^2} \phi\,dy &=0 .
\end{align}
\ch{We recall $g$ defined in \eqref{q-form} $g = \frac{\phi}{U} - (-\Delta)^{-1}\phi $, and
use the notation
\begin{align*}
\psi = (-\Delta)^{-1}\phi
\end{align*}
so that
\begin{align}
\nonumber
-\Delta \psi - U \psi = U g \quad\text{in }\R^2.
\end{align}}
\ch{We next introduce a normalized version of $g$, namely $g^\perp$ defined by}
\begin{align}
\nonumber
g^\perp = g + a,
\end{align}
where $a \in \R$ is chosen so that
\begin{align}
\nonumber
\int_{\R^2} g^\perp U d y = 0 .
\end{align}
\ch{As shown in Lemma~\ref{lemma-q} below, the quadratic form $\int_{\R^2} \phi g$ is equivalent to $\int_{\R^2} U(g^\perp)^2  $.}

\ch{It will be convenient to work with functions $\phi^\perp$, $\psi^\perp $, which are  analogues of $\phi$, $\psi$ but associated to $g^\perp$.
In particular, we want a choice of $\psi^\perp$ such that
\begin{align}
\label{eqPsiPerp}
-\Delta \psi^\perp - U \psi^\perp = U g^\perp, \quad \psi^\perp(y) \to 0\quad \text{as }|y|\to \infty.
\end{align}
Let  $\psi_0=1+\frac{1}{2}z_0$, where $z_0$ is defined in \eqref{defZLiouville}, and observe that
\[
-\Delta \psi_0 - U \psi_0 = -U, \quad \psi_0(y) \to 0\quad \text{as }|y|\to \infty.
\]
Then $\psi^\perp$ defined by
\[
\psi^\perp = \psi -a \Bigl(1+\frac{1}{2}z_0\Bigl) = \psi - a \psi_0,
\]
indeed satisfies \eqref{eqPsiPerp}.}

Define
\[
\phi^\perp = U( g^\perp + \psi^\perp ),
\]
and obtain the relations
\begin{align}
\nonumber
\phi = \phi^\perp  + \frac{a}{2} U z_0  ,
\quad
\ch{-\Delta \psi^\perp = \phi^\perp},
\quad
\ch{\int_{\R^2} \phi^\perp =0}.
\end{align}
\ch{We note that $\phi - \phi^\perp = \frac{a}{2} U z_0$ is a constant times $Z_0 = U z_0$, which is in the kernel of the operator $L$.}

\begin{lemma}
\label{lemma-orthG}
If $\phi:\R^2 \to \R$ satisfies \eqref{decayPhi} and \eqref{intPhi0}, then
\begin{align*}
\int_{\R^2} g U z_j = \int_{\R^2} g^\perp U z_j = 0, \quad j=0,1,2,
\end{align*}
where $z_j$ are the functions defined in \eqref{defZLiouville}.
\end{lemma}

\begin{proof}
By the definition of $\psi$ and from \eqref{decayPhi}, \eqref{intPhi0} we have
\begin{align}
\nonumber
|\psi(y)| + (1+|y|) |\nabla \psi(y)| \leq  \frac{C}{(1+|y|)^{\sigma}} ,
\end{align}
and hence also
\begin{align}
\label{estPsiPerp}
|\psi^\perp(y)| + (1+|y|) |\nabla \psi^\perp(y)| \leq  \frac{C}{(1+|y|)^{\sigma}}.
\end{align}

We multiply \eqref{eqPsiPerp} by $z_j$, integrate in the ball $B_R(0)$ and let $R\to \infty$.
Since $z_j$ is in the kernel of $\Delta+U$ we just have to check that
\begin{align*}
\int_{\partial B_R} \Bigl( \frac{\partial \psi^\perp}{\partial \nu} z_j
- \psi^\perp \frac{\partial z_j}{\partial \nu}\Bigr) \to 0 , \quad \text{as }R\to \infty,
\end{align*}
where $\nu$ is the exterior normal vector to $\partial B_R$.
This follows from \eqref{estPsiPerp}, and the explicit bounds
\begin{align*}
|z_0(y)| & \leq C, \quad
|z_j(y)| \leq \frac{C}{(1+|y|)}, \quad j=1,2,
\\
|\nabla z_j(y)| &\leq \frac{C}{(1+|y|)^{2}}.
\end{align*}
\end{proof}

A consequence of the previous lemma is the following.
\begin{remark}
\label{remark-secondm-Lphi}
Suppose that $\phi:\R^2 \to \R$ satisfies \eqref{decayPhi} and \eqref{intPhi0}.
Then
\begin{align*}
\int_{\R^2} L[\phi]|y|^2dy=0.
\end{align*}
Indeed, integrating on $B_R$, with the notation $g = \frac{\phi}{U}-(-\Delta)^{-1}\phi$,
\begin{align*}
\int_{B_R} L[\phi]|y|^2dy
&=
\int_{B_R} \nabla \cdot ( U \nabla g ) |y|^2dy
\\
&= -2 \int_{B_R} U \nabla g \cdot y dy + R^2 \int_{\partial B_R} U \nabla g \cdot \nu dS(y)
\\
&= 2 \int_{B_R} g Z_0 dy
- 2 \int_{\partial B_R} U g  y \cdot \nu  dy
+ R^2 \int_{\partial B_R} U \nabla g \cdot \nu dS(y) .
\end{align*}
By \eqref{decayPhi} and \eqref{intPhi0}, $g(y) = O ( |y|^{2-\sigma} )$, $\nabla g(y) = O ( |y|^{1-\sigma} )$ as $|y|\to\infty$. Therefore the boundary terms tend to 0 as $R\to \infty$, and we get
\begin{align*}
\int_{\R^2}  L[\phi]|y|^2dy = 2 \int_{\R^2} g Z_0 dy =0,
\end{align*}
by Lemma~\ref{lemma-orthG}.
\end{remark}

\begin{lemma}
\label{lemma-q}
There are constants $c_1>0$, $c_2>0$ such that if
$\phi:\R^2 \to \R$ satisfies
\begin{align}
\nonumber
|\phi(y)|\leq \frac{1}{(1+|y|)^{3+\sigma}} , \quad 0<\sigma<1
\end{align}
and
\eqref{intPhi0}, then
\begin{align}
\label{q1}
c_1 \int_{\R^2} U (g^\perp)^2
\leq
\int_{\R^2} \phi g^\perp \leq c_2 \int_{\R^2} U (g^\perp)^2 .
\end{align}
\end{lemma}

\begin{proof}


By Lemma~\ref{lemma-orthG}
\begin{align*}
\int_{\R^2} \phi g
&= \int_{\R^2} ( \phi^\perp + \frac{a}{2} U z_0) g
= \int_{\R^2} \phi^\perp (g^\perp+a) =  \int_{\R^2} \phi^\perp g^\perp
\\
&=  \int_{\R^2} U( g^\perp  + \psi^\perp) g^\perp.
\end{align*}
We transform $\tilde g^\perp = g^\perp \circ \Pi$, $\tilde \psi^\perp = \psi^\perp\circ\Pi$ and write \eqref{eqPsiPerp} as
\begin{align}
\label{eqTildePsiP}
-\Delta_{S^2} \tilde \psi^\perp - 2 \tilde \psi^\perp = 2 \tilde g^\perp , \quad
\text{in }S^2.
\end{align}
We also get
\begin{align*}
\frac{1}{2}\int_{\R^2} \phi g
&=  \int_{S^2} [ (\tilde g^\perp)^2  + \ch{ \tilde \psi^\perp \tilde g^\perp} ].
\end{align*}

Multiplying \eqref{eqTildePsiP} by $\tilde \psi^\perp$ we find that
\begin{align*}
\int_{S^2} \tilde g^\perp \tilde \psi^\perp =
\frac{1}{2} \int_{S^2}|\nabla_{S^2} \tilde \psi^\perp |^2 - \int_{S^2}(\tilde \psi^\perp)^2
\end{align*}
and hence
\begin{align*}
\frac{1}{2}\int_{\R^2} \phi g
=
\int_{S^2} (\tilde g^\perp)^2
+\frac{1}{2} \int_{S^2} |\nabla_{S^2} \tilde \psi^\perp |^2
- \int_{S^2}(\tilde \psi^\perp)^2 .
\end{align*}
We recall that the eigenvalues of $-\Delta $ on $S^2$ are given by
$
\{ k ( k+1) \ | \ k \geq 0 \}
$.
The eigenvalue 0 has a constant eigenfunction and the eigenvalue $2$ has eigenspace spanned by the coordinate functions $\pi_i(x_1,x_2,x_3) = x_i$, for $(x_1,x_2,x_3) \in S^2$ and $i=1,2,3$.
Let $(\lambda_j)_{j\geq 0}$ denote all eigenvalues, repeated according to multiplicity, with $\lambda_0 = 0$, $\lambda_1=\lambda_2=\lambda_3=2$, and let $(e_j)_{j\geq 0}$ denote the corresponding eigenfunctions so that they form an orthonormal system in $L^2(S^2)$, and $e_1,e_2,e_3$ are multiples of the coordinate functions $\pi_1,\pi_2,\pi_3$.
We decompose $\tilde \psi$ and $\tilde g$:
\begin{align}
\label{decomp1}
\tilde \psi^\perp = \sum_{j=0}^\infty \tilde \psi^\perp_j e_j, \quad
\tilde g^\perp= \sum_{j=0}^\infty \tilde g^\perp_j e_j,
\end{align}
where
\begin{align*}
\tilde \psi^\perp_j = \langle \tilde \psi^\perp,e_j\rangle_{L^2(S^2)},
\quad
\tilde g^\perp_j = \langle \tilde g^\perp,e_j\rangle_{L^2(S^2)}.
\end{align*}

Then
\begin{align*}
\frac{1}{2}\int_{\R^2} \phi g
&= \sum_{j=0}^\infty ( \tilde g^\perp_j)^2
+\frac{1}{2} \sum_{j=0}^\infty (\lambda_j -2)(\tilde \psi^\perp_j)^2
\\
&=  \sum_{j=0}^\infty ( \tilde g^\perp_j)^2
- (\tilde \psi^\perp_0)^2
+\frac{1}{2} \sum_{j=4}^\infty (\lambda_j -2)(\tilde \psi^\perp_j)^2 .
\end{align*}

Equation \eqref{eqTildePsiP} gives us that
\begin{align}
\label{liouvilleQeD}
(\lambda_j-2)\tilde \psi^\perp_j = 2\tilde g^\perp_j,
\end{align}
and then
\begin{align*}
\frac{1}{2}\int_{\R^2} \phi g
&=  \sum_{j=1}^\infty ( \tilde g^\perp_j)^2
+ \sum_{j=4}^\infty \frac{2}{\lambda_j-2} (\tilde g^\perp_j)^2 .
\end{align*}
By Lemma~\ref{lemma-orthG} $\tilde g^\perp_1 = \tilde g^\perp_2 = \tilde g^\perp_3 = 0$. Therefore
\begin{align}
\label{qx2}
\frac{1}{2}\int_{\R^2} \phi g
&=  \sum_{j=4}^\infty \frac{\lambda_j}{\lambda_j-2} (\tilde g^\perp_j)^2
\end{align}
and
\[
\frac{1}{2}\int_{\R^2} (g^\perp)^2 U = \sum_{j=4}^\infty (\tilde g^\perp_j)^2.
\]
This proves \eqref{q1}.

\end{proof}

\begin{lemma}
\label{lemma-q2}
There exist positive constants $c_1$, $c_2$ such that if
$\phi:\R^2 \to \R$ is radially symmetric and satisfies $(1+|y|)^{3+\sigma} \phi \in L^\infty(\R^2)$ with $0<\sigma<1$, and
\begin{align}
\nonumber
\int_{\R^2} \phi(y)dy=0,\quad
\end{align}
then
\begin{align}
\label{q2}
c_1 \int_{\R^2} U (g^\perp)^2
\, \leq\,
\int_{\R^2} U^{-1} (\phi^\perp)^2\, \leq\,  c_2 \int_{\R^2} U (g^\perp)^2 ,
\end{align}
\begin{align}
\label{q3}
\int_{\R^2} U (\psi^\perp)^2\, \leq\,  c_2 \int_{\R^2} U (g^\perp)^2 .
\end{align}
\end{lemma}
\begin{proof}
Using the same notation as in the proof of Lemma~\ref{lemma-q}, we have
\begin{align*}
\frac{1}{2}\int_{\R^2} U^{-1}(\phi^\perp)^2
&=\frac{1}{2}\int_{\R^2} U[(\psi^\perp)^2+2\psi^\perp g^\perp + (g^\perp)^2]
\\
&=\int_{S^2}[ (\tilde \psi^\perp)^2
+ 2 \tilde \psi^\perp \tilde g^\perp
+ (\tilde g^\perp)^2]
\\
&=
\sum_{j=0}^\infty
[(\tilde \psi^\perp_j)^2
+2\tilde \psi^\perp_j \tilde g^\perp_j
+(\tilde g^\perp_j)^2] .
\end{align*}
As in the previous proof, $\tilde g^\perp_j=0$ for $j=0,1,2,3$. Using \eqref{liouvilleQeD} we get
\begin{align*}
\frac{1}{2}\int_{\R^2} U^{-1}(\phi^\perp)^2
&=
\sum_{j=0}^3
(\tilde \psi^\perp_j)^2
+\sum_{j=4}^\infty
\frac{\ch{\lambda_j}^2}{(\lambda_j-2)^2}
(\tilde g^\perp_j)^2 .
\end{align*}
This formula already gives
\begin{align*}
\int_{\R^2} U(g^\perp)^2 \leq C
\int_{\R^2} U^{-1}(\phi^\perp)^2.
\end{align*}

We observe that $\tilde \psi^\perp_1=\tilde \psi^\perp_2=0$ by radial symmetry.
We also have $\tilde \psi^\perp_0=0$, by \eqref{liouvilleQeD}.
Let
\begin{align*}
\hat \psi = \sum_{j=4}^\infty \tilde \psi^\perp_j e_j
\end{align*}
and note that it satisfies
\begin{align*}
-\Delta_{S^2} \hat \psi -2 \hat \psi = 2 \tilde g^\perp
\quad \text{in }  S^2.
\end{align*}
By \eqref{liouvilleQeD},
\begin{align*}
\| \hat \psi\|_{L^2(S^2)} \leq C \| \tilde g^\perp \|_{L^2 (S^2)},
\end{align*}
and from elliptic estimates
\begin{align}
\label{holderPsiHat}
\| \hat \psi\|_{C^\alpha(S^2)} \leq C \| \tilde g^\perp \|_{L^2 (S^2)},
\end{align}
for any $0<\alpha<1$.
Since $(1+|y|)^{3+\sigma}\phi\in L^\infty( \R^2 )$
and $\phi$ has total mass 0, we have $(1+|y|)^{1+\sigma}\psi\in L^\infty( \R^2 )$
(here the functions are radial) and also
$(1+|y|)^{1+\sigma}\psi^\perp\in L^\infty( \R^2 )$.
It follows that $\tilde \psi^\perp(P)$ = 0 where $P= (0,0,1)$. Since $\tilde \psi^\perp$ and $\hat \psi$ differ by a constant times $\ch{\pi_3}$ we have
\[
\tilde \psi^\perp = \hat \psi - \frac{\hat \psi(P)}{\ch{\pi_3}(P)} \ch{\pi_3} ,
\]
\ch{where $\pi_3(x_1,x_2,x_3)=x_3$.}
This implies, by \eqref{holderPsiHat},
\begin{align}
\nonumber
\| \tilde \psi^\perp \|_{L^2(S^2)}
& \leq C \| \hat \psi \|_{L^2(S^2)}
+ C|\hat \psi(P)|
\leq C \|\tilde g^\perp \|_{L^2(S^2)}.
\end{align}
This proves the other inequality in \eqref{q2} and \eqref{q3}.

\end{proof}

\begin{lemma}
\label{lemma-derPhig}
Suppose that $\phi = \phi(y,t)$, $y\in \R^2$, $t>0$ is a function satisfying
\[
|\phi(y,t)|\leq \frac{1}{(1+|y|)^{2+\sigma}} ,
\]
with $0<\sigma<1$,
\begin{align*}
\int_{\R^2} \phi(y,t)\,dy &=0 ,\quad \forall t>0,
\end{align*}
and that $\phi $ is differentiable with respect to $t$ and $\phi_t$ satisfies also
\begin{align*}
|\phi_t(y,t)|\leq \frac{1}{(1+|y|)^{2+\sigma}} .
\end{align*}
Then
\begin{align*}
\int_{\R^2} \phi_t g = \frac{1}{2} \partial_t \int_{\R^2} \phi g
\end{align*}
where for each $t$, $g(y,t)$ is defined as
\begin{align*}
g = \frac{\phi}{U}-(-\Delta^{-1})\phi + c (t)
\end{align*}
and $c(t)\in \R$ is chosen so that
\begin{align*}
\int_{\R^2} g(y,t) U (y)\,dy= 0.
\end{align*}
\end{lemma}
\begin{proof}
Using the notation of the previous lemma, we have
\begin{align*}
\int_{\R^2} \phi_t g
&=
\int_{\R^2} U( g_t + \psi_t ) g
=
\ch{2}
\int_{S^2} ( \tilde g_t \tilde g + \tilde \psi_t \tilde g) .
\end{align*}
We have
\begin{align*}
-\Delta_{S^2} \tilde \psi - 2 \tilde \psi = 2 \tilde g , \quad
\text{in }S^2.
\end{align*}
And differentiating in $t$ we get
\begin{align}
\label{eqTildePsit}
-\Delta_{S^2} \tilde \psi_t - 2 \tilde \psi_t = 2 \tilde g_t , \quad
\text{in }S^2.
\end{align}
Multiplying by $\tilde g$ and integrating we find that
\begin{align*}
\int_{S^2} \tilde \psi_t \tilde g =
- \frac{1}{2}\int_{S^2} \Delta \tilde \psi_t \tilde g
-\int_{S^2} \tilde g_t \tilde g .
\end{align*}
Thus
\begin{align*}
\int_{\R^2} \phi_t g
= \ch{-} \int_{S^2} \Delta \tilde \psi_t \tilde g
\end{align*}
Decompose as in \eqref{decomp1} and find that
\begin{align*}
\int_{\R^2} \phi_t g
=
\sum_{j=0}^\infty \lambda_j (\tilde \psi_j)_t \tilde g_j
\end{align*}
But from \eqref{eqTildePsit}
\begin{align*}
(\lambda_j-2) (\tilde \psi_j)_t = 2(\tilde g_j)_t.
\end{align*}

We note that $\tilde g_j=0$ for $j=0,1,2,3$.
Indeed, this is true for $j=0$ by the assumption
$\int_{\R^2} g U=0$.
By Lemma~\ref{lemma-orthG} this is true also for $j=1,2,3$.
Then
\begin{align*}
\frac{1}{2}
\int_{\R^2} \phi_t g
=
\sum_{j=4}^\infty \frac{\lambda_j}{\lambda_j-2} (\tilde g_j)_t \tilde g_j
\end{align*}
and the desired conclusion follows from \eqref{qx2}.
\end{proof}

\subsection{A Poincar\'e inequality}

\begin{lemma}
\label{lemmaHardy}
Let $B_R(0)\subset \R^2$ be the open ball centered at 0 of radius $R$.
There exists $C>0$ such that, for any $R >0$ large and any $g\in H^1(B_R)$ with $\int_{B_R} g\,U\,dx =0$ we have
\begin{equation}
\nonumber
\frac{C}{ R^2} \int_{B_R} g^2 U\leq
\int_{B_R} |\nabla g|^2 U .
\end{equation}
\end{lemma}

\begin{proof}
Using a Fourier decomposition we only need to consider the radial case, that is, we claim that
if $g(r)$  satisfies
\begin{align}
\label{zeroavg}
\int_0^R g(r) \frac{r}{(1+r^2)^2}dr = 0,
\end{align}
then there is $C$ such that for all $R$ large
\begin{align*}
\int_0^R g(r)^2 \frac{r}{(1+r^2)^2}dr
\leq CR^2 \int_0^R g'(r)^2 \frac{r}{(1+r^2)^2}dr .
\end{align*}

Let $0<\delta<1$ to be fixed later on.
From \eqref{zeroavg} we have
\begin{align*}
\int_\delta^R g(r) \frac{r}{(1+r^2)^2}dr = - \int_0^\delta g(r) \frac{r}{(1+r^2)^2}dr .
\end{align*}
But
\begin{align*}
\int_\delta^R g(r) \frac{r}{(1+r^2)^2}dr
&= -\frac{1}{2} \int_\delta^R g(r) \frac{d}{dr}\Bigl( \frac{1}{1+r^2}\Bigr)dr
\\
&= - \frac{1}{2}\frac{g(R)}{1+R^2}
+ \frac{1}{2}\frac{g(\delta)}{1+\delta^2}
+ \frac{1}{2} \int_\delta^R g'(r) \frac{1}{1+r^2} dr .
\end{align*}
Therefore
\begin{align*}
\frac{1}{2}\frac{|g(\delta)|}{1+\delta^2}
& \leq  \frac{1}{2}\frac{|g(R)|}{1+R^2}
+ \frac{1}{2} \int_\delta^R |g'(r)| \frac{1}{1+r^2} dr
+  \int_0^\delta |g(r)| \frac{r}{(1+r^2)^2}dr .
\end{align*}
By the Cauchy-Schwarz inequality
\begin{align*}
\int_\delta^R |g'(r)| \frac{1}{1+r^2} dr
\leq \Bigl( \int_\delta^R g'(r)^2  \frac{r}{(1+r^2)^2}dr\Bigr)^{1/2}
( \log R - \log \delta)^{1/2}
\end{align*}
\begin{align*}
\int_0^\delta |g(r)| \frac{r}{(1+r^2)^2}dr
\leq \delta \Bigl(  \int_0^\delta g(r)^2 \frac{r}{(1+r^2)^2}dr \Bigr)^{1/2} .
\end{align*}
Hence
\begin{align}
\label{bd}
g(\delta)^2
& \leq
2 \frac{g(R)^2}{R^4}
+ 2 ( \log R - \log \delta)  \int_\delta^R g'(r)^2  \frac{r}{(1+r^2)^2}dr
+ 4 \delta^2   \int_0^\delta g(r)^2 \frac{r}{(1+r^2)^2}dr .
\end{align}

We compute now
\begin{align*}
\int_\delta^R g(r)^2 \frac{r}{(1+r^2)^2}dr
&= -\frac{1}{2} \int_\delta^R g(r)^2 \frac{d}{dr}\Bigl( \frac{1}{1+r^2}\Bigr)dr
\\
&= - \frac{1}{2}\frac{g(R)^2}{1+R^2}
+ \frac{1}{2}\frac{g(\delta)^2}{1+\delta^2}
+ \int_\delta^R g(r) g'(r) \frac{1}{1+r^2} dr .
\end{align*}
Using \eqref{bd} and the Cauchy-Schwartz inequality we get
\begin{align*}
\int_\delta^R g(r)^2 \frac{r}{(1+r^2)^2}dr
& \leq  - \frac{1}{2}\frac{g(R)^2}{1+R^2}
+ \frac{g(R)^2}{R^4}
+  ( \log R - \log \delta)  \int_\delta^R g'(r)^2  \frac{r}{(1+r^2)^2}dr
\\
& \quad
+  2 \delta^2   \int_0^\delta g(r)^2 \frac{r}{(1+r^2)^2}dr
+ A R^2 \int_\delta^R  g'(r)^2 \frac{r}{(1+r^2)^2} dr
\\
&\quad
+ \frac{1}{A R^2} \int_\delta^R  g(r)^2 \frac{1}{r} dr .
\end{align*}
But $\frac{1}{AR^2 r} \leq \frac{1}{2}\frac{r}{(1+r^2)^2}$ for $r\in [\delta,R]$ if $A = 4 (1+\frac{1}{\delta^2})$ and $R\geq 1$. Choosing $A = 4 (1+\frac{1}{\delta^2})$ and $R\geq 2$ we have
\begin{align}
\nonumber
\int_\delta^R g(r)^2 \frac{r}{(1+r^2)^2}dr
& \leq
[ 2AR^2 + 2 ( \log R - \log \delta) ] \int_\delta^R g'(r)^2  \frac{r}{(1+r^2)^2}dr
\\
\label{x1}
& \quad
+  4 \delta^2   \int_0^\delta g(r)^2 \frac{r}{(1+r^2)^2}dr
\end{align}

With $\delta>0$ still to be chosen we get from \eqref{bd} for $0<x<\delta$
\begin{align}
\nonumber
g(x)^2
& \leq
2 \frac{g(R)^2}{R^4}
+ 2 ( \log R - \log x)  \int_0^R g'(r)^2  \frac{r}{(1+r^2)^2}dr
+ 4 x^2   \int_0^\delta g(r)^2 \frac{r}{(1+r^2)^2}dr .
\end{align}
Integrating we get
\begin{align}
\label{x2}
\int_0^\delta
g(r)^2
\frac{r}{(1+r^2)^2}dr
& \leq
\delta^2 \frac{g(R)^2}{R^4}
+ 2 \log R   \int_0^R g'(r)^2  \frac{r}{(1+r^2)^2}dr
+ \delta^4   \int_0^\delta g(r)^2 \frac{r}{(1+r^2)^2}dr .
\end{align}
Using the condition \eqref{zeroavg} we obtain
\begin{align*}
\int_0^Rg(r) \frac{r}{(1+r^2)^2}dr
&= \frac{1}{2} \int_0^R g(r) \frac{d}{dr}\Bigl(
\frac{r^2}{1+r^2}\Bigr)dr
\\
&=
\frac{1}{2} g(R)\frac{R^2}{1+R^2}
- \frac{1}{2} \int_0^R g'(r)
\frac{r^2}{1+r^2}dr .
\end{align*}
Then
\begin{align*}
g(R)^2 \leq 4 R^4  \int_0^R g'(r)^2
\frac{r}{(1+r^2)^2}dr .
\end{align*}
Using this combined with \eqref{x2} we get
\begin{align*}
\int_0^\delta
g(r)^2
\frac{r}{(1+r^2)^2}dr
& \leq
\delta^2  4 \int_0^R g'(r)^2
\frac{r}{(1+r^2)^2}dr
+ 2 \log R   \int_0^R g'(r)^2  \frac{r}{(1+r^2)^2}dr
\\
& \quad
+ \delta^4   \int_0^\delta g(r)^2 \frac{r}{(1+r^2)^2}dr .
\end{align*}
Taking $\delta =\frac{1}{2}$ (this fixes $A$) gives
\begin{align*}
\int_0^\delta
g(r)^2
\frac{r}{(1+r^2)^2}dr
& \leq
4 (\log R +1)  \int_0^R g'(r)^2  \frac{r}{(1+r^2)^2}dr.
\end{align*}
Combining this with \eqref{x1} we get
\begin{align*}
\int_0^R g(r)^2 \frac{r}{(1+r^2)^2}dr
\leq C R^2  \int_0^R g'(r)^2  \frac{r}{(1+r^2)^2}dr.
\end{align*}

\end{proof}

\section{Linear theory: a decomposition}

\label{sect-linear-decomp}

Here we consider
\begin{align}
\label{eq-linear-radial12}
\left\{
\begin{aligned}
\partial_\tau \phi &=
L [\phi] + B[\phi]+ h, \quad \text{in }\R^2 \times (\tau_0,\infty),
\\
\phi(\cdot,\tau_0) &= \phi_0 \quad \text{in }\R^2 .
\end{aligned}
\right.
\end{align}
The results of this section are going to be used later only in the case of radial functions, so we make this assumption here.
We write in the rest of this section
$\phi = \phi ( y,\tau) = \phi(\rho,\tau)$, where $y\in \R^2$,
$\rho = |y|$.


The operator $B$ is assumed to be one of the following two:
\begin{align}
\label{B1}
B[\phi] = \zeta(\tau) ( 2 \phi + y \cdot \nabla \phi) = \zeta(\tau) \nabla \cdot ( y \phi)  ,
\end{align}
or
\begin{align}
\label{B2}
B[\phi] = \zeta(\tau)  y \cdot \nabla \phi ,
\end{align}
where
\begin{align}
\nonumber
\zeta(\tau) = - \frac{\zeta_0 }{\tau \log \tau}+O\Big(\frac{1}{\tau (\log \tau)^{1+\sigma_0}} \Bigr),\quad\text{as }\tau\to\infty,
\end{align}
for some constants $\zeta_0>0$, $0<\sigma_0<1$.

We assume that $\|h\|_{**}<\infty$ where
\begin{align}
\nonumber
& \|h\|_{**} = \inf K , \quad \text{such that }
\\
\nonumber
&   |h(y,\tau)|  \leq
K
\frac{ 1 } {  \tau^\nu (\log\tau)^m }
\frac{  1  }{(1+|y|)^{6+\sigma} }
\min \Bigl( 1, \frac{ \tau^{\epsilon/2}}{|y|^\epsilon} \Bigr) ,
\quad \tau>\tau_0, \ y\in \R^2,
\end{align}
where $\nu>1$, $\epsilon>0$, $\sigma>0$, $m\in\R$.
This is the same norm as in \eqref{normh0b}.

We also assume that $h$ has zero mass
\begin{align}
\label{h-zero-mass-3}
\int_{\R^2} h(y,\tau)dy=0\quad \text{for all }\tau>\tau_0,
\end{align}
and the same for the initial condition
\begin{align}
\label{initial-zero-mass}
\int_{\R^2} \phi_0 dy = 0.
\end{align}
It follows from the equation \eqref{eq-linear-radial12}, \eqref{h-zero-mass-3}, and \eqref{initial-zero-mass} that the solution $\phi$ to \eqref{eq-linear-radial12} defined in \S\ref{sectLT1} satisfies
\begin{align}
\nonumber
\int_{\R^2} \phi(y,\tau)dy=0\quad \text{for all }\tau>\tau_0 .
\end{align}

We recall the decomposition of $\phi$ introduced in \S\ref{subsect-q-form}. Given $\phi:\R^2\to \R$ with sufficient decay and mass zero,
we let $g = \frac{\phi}{U}- (-\Delta^{-1})\phi$,
and define $a$ so that $\int_{\R^2} ( g + a ) U dy= 0$.
Then define $g^\perp = g + a $, $\psi^\perp = \psi - a( 1+\frac{1}{2}z_0)$, and
\begin{align}
\label{decomp-phi-perp}
\phi^\perp = \phi  - \frac{a}{2}Z_0 .
\end{align}
Actually $a$ is directly computed by
\begin{align}
\label{formula-a}
a = - \frac{1}{8 \pi}\int_{\R^2} U g
=  \frac{1}{8 \pi}\int_{\R^2} U (-\Delta)^{-1} \phi
=  \frac{1}{8 \pi}\int_{\R^2}\Gamma_0\phi .
\end{align}
In the time dependent situation $a = a(\tau)$ and all functions depend on $y\in \R^2$  and $\tau$.

A difficulty to obtain estimates is the presence of a kernel in the linear operator if $B=0$, since $Z_0$ satisfies $L[Z_0]=0$.
It can be proved that the solution $\phi$ of \eqref{eq-linear-radial12} with zero initial condition and $\|h\|_{**}<\infty$ has the bound
\begin{align*}
\sup_{y} |\phi(y,\tau)|\leq C \Bigl( \frac{\log \tau_0}{\log \tau} \Bigr)^{2\zeta_0-\sigma_0} \|h\|_{**} ,
\end{align*}
and probably this estimate cannot be improved much. Also $\phi$ has a some decay at spatial infinity and in particular it has finite second moment
\begin{align*}
\int_{\R^2} |\phi(y,\tau)| \, |y|^2\,dy<\infty, \quad \tau>\tau_0.
\end{align*}
Therefore $Z_0$ doesn't describe well the class of solution we want to consider, even for the case $B=0$, in which $\zeta(\tau) \equiv0$.

A better candidate to describe the solutions $\phi$ of \eqref{eq-linear-radial12} with zero initial condition and $\|h\|_{**}<\infty$ is obtained by considering the initial value problem
\begin{align}
\label{eq-ZB}
\left\{
\begin{aligned}
\partial_\tau Z_B &=
L [Z_B] + B[Z_B], \quad \text{in }\R^2 \times (\tau_0,\infty),
\\
Z_B(\cdot,\tau_0) &= \tilde Z_0 \quad \text{in }\R^2 .
\end{aligned}
\right.
\end{align}
where $ \tilde Z_0$ is defined as
\begin{align}
\label{def-tilde-Z0}
\tilde Z_0(\rho ) &= ( Z_0(\rho) - m_{Z_0} U ) \chi_0\Bigl( \frac{\rho}{\sqrt{\tau_0}}\Bigr),
\end{align}
and $m_{Z_0}$ is selected so that
\begin{align*}
\int_{\R^2} \tilde Z_0 dy = 0.
\end{align*}
Note that since $Z_0$ has mass zero and decays like $1/\rho^4$ we have
$ m_{Z_0} = O ( \frac{1}{\tau_0} ) $.

\medskip

We will then consider the problem
\begin{align}
\label{eq-linear-radial20}
\left\{
\begin{aligned}
\partial_\tau \phi &=
L [\phi] + B[\phi]+ h, \quad \text{in }\R^2 \times (\tau_0,\infty),
\\
\phi(\cdot,\tau_0) &= c_1 \tilde Z_0 \quad \text{in }\R^2 ,
\end{aligned}
\right.
\end{align}
for radial functions $\phi$, $h$, $\phi_0$, where $c_1\in \R $ is a parameter.
We assume that $\|h\|_{**}<\infty$.

\begin{prop}
\label{prop-linear-with-energy}
Let us assume that $1<\nu<\frac{7}{4}$.
Then there is $C>0$ such that for any $\tau_0$ sufficiently large the following holds. Suppose that   $\|h\|_{**}<\infty$ is radially symmetric and satisfies the zero mass condition \eqref{h-zero-mass-3}.
Then there exists $c_1$ such that the solution $\phi = \phi^\perp + \frac{a}{2}Z_0$ of  \eqref{eq-linear-radial20} satisfies
\begin{align}
\nonumber
|a(\tau) | &\leq
C \frac{f(\tau) R(\tau)^2}{(\log\tau_0)^{1-q}}\|h\|_{**},
\end{align}
\begin{align}
\label{estPhiPerp}
|\phi^\perp(\rho,\tau)|\leq
C f(\tau) R(\tau)
\frac{1}{1+|y|^2}  \|h\|_{**}.
\end{align}
where $R(\tau)>0$ is defined by
\begin{align}
\label{def-R}
R(\tau)^2 =
\frac{ \tau}{(\log\tau)^{q}},
\end{align}
where $0 < q<1$,
and
\begin{align}
\label{notation-f}
f(\tau) = \frac{1}{\tau^\nu (\log \tau)^m} .
\end{align}
Moreover $c_1$ is a linear function of $h$ and satisfies
\begin{align*}
|c_1| & \leq  C \frac{f(\tau_0)R(\tau_0)^2 }{(\log\tau_0)^{1-q}}  \|h\|_{**}.
\end{align*}
\end{prop}

We always decompose $\phi$ as in \eqref{decomp-phi-perp}:
\begin{align*}
\phi = \phi^\perp + \frac{a(\tau)}{2} Z_0
\end{align*}
and write
\begin{align*}
g = \frac{\phi}{U}- (-\Delta)^{-1} \phi , \quad
g^\perp = \frac{\phi^\perp}{U}- (-\Delta)^{-1} \phi^\perp  .
\end{align*}
Let us denote
\begin{align}
\label{notation-omega}
\omega(\tau) = \Bigl( \int_{\R^2 \setminus B_{R(\tau)}(0)} U g(\tau)^2 \Bigr)^{1/2}.
\end{align}

The strategy for the proof of Proposition~\ref{prop-linear-with-energy} is contained in the following lemmas.
The first one is an a-priori estimate for the solution, assuming that $a(T_2)=0$ for some $T_2$.

\begin{lemma}
\label{lemma.apriori}
There is $C$ such that for $\tau_0$ large the following holds.
Suppose that   $\|h\|_{**}<\infty$ is radially symmetric and satisfies the zero mass condition \eqref{h-zero-mass-3} and
consider \eqref{eq-linear-radial20}.
Let $\phi^\perp$, $a$ be the decomposition \eqref{decomp-phi-perp}.
Suppose that for some $c_1\in \R$ there is $T_2>\tau_0$ is such that
\begin{align*}
a(T_2)=0.
\end{align*}
Then
\begin{align}
\label{est-a1}
|a(\tau)| &\leq C \frac{f(\tau) R(\tau)^2}{(\log\tau_0)^{1-q}}\|h\|_{**},
\quad \tau \in [\tau_0,T_2]
\\
\label{est-omega1}
|\omega(\tau)| &\leq C \frac{f(\tau) R(\tau)}{(\log\tau_0)^{1-q}}\|h\|_{**},
\quad \tau \in [\tau_0,T_2]
\\
\label{esta-c1-1}
|c_1| & \leq  C \frac{f(\tau_0)R(\tau_0)^2 }{(\log\tau_0)^{1-q}} \|h\|_{**} .
\end{align}

The constant $C$ is independent of $T_2$ and $c_1$.
\end{lemma}

There is a variant of the previous lemma, where the hypothesis $a(T_2)=0$ is replaced by an assumption about its time decay.

\begin{lemma}
\label{lemma.apriori2}
There is $C$ such that for $\tau_0$ large the following holds.
Suppose that   $\|h\|_{**}<\infty$ is radially symmetric and satisfies the zero mass condition \eqref{h-zero-mass-3} and
consider \eqref{eq-linear-radial20}.
Let $\phi^\perp$, $a$ be the decomposition \eqref{decomp-phi-perp}.
Suppose that for some  $c_1\in \R$,
\begin{align*}
\frac{a}{f R^2} \in L^\infty(\tau_0,\infty).
\end{align*}
Then
\begin{align}
\label{a11}
|a(\tau)| &\leq C \frac{f(\tau) R(\tau)^2}{(\log\tau_0)^{1-q}}\|h\|_{**},
\quad \tau>\tau_0
\\
\label{o11}
|\omega(\tau)| &\leq C \frac{f(\tau) R(\tau)}{(\log\tau_0)^{1-q}}\|h\|_{**},
\quad \tau>\tau_0,
\\
\label{c11}
|c_1| & \leq  C \frac{f(\tau_0)R(\tau_0)^2 }{(\log\tau_0)^{1-q}} \|h\|_{**} .
\end{align}
\end{lemma}

\begin{lemma}
\label{lemma.Z.nondegen}
Let $Z_B$ be the solution to \eqref{eq-ZB} and write it as $Z_B = Z_B^\perp + \frac{a_Z}{2}Z_0$ according to the decomposition \eqref{decomp-phi-perp}. Then $a_Z(\tau)\not=0$ for all $\tau \geq \tau_0$.
\end{lemma}

\begin{lemma}
\label{lemma.existence}
There is $C$ such that for $\tau_0$ large the following holds.
Suppose that   $\|h\|_{**}<\infty$ is radially symmetric and satisfies the zero mass condition \eqref{h-zero-mass-3}.
Then there is a unique $c_1 \in \R$ such that
the solution $\phi = \phi^\perp + \frac{a}{2}Z_0$ of \eqref{eq-linear-radial20} (as in \eqref{decomp-phi-perp}) satisfies \eqref{a11}, \eqref{o11} and \eqref{c11}.
\end{lemma}

\medskip

In the first results we do some computations and obtain some estimates, which are used as technical steps in the main argument.

\medskip

The next lemma is a calculation to help us deal with the term $B$ when we multiply the equation by a suitable test function. It holds for operators more general than $B$ as in \eqref{B1} and \eqref{B2}. Let
\begin{align*}
\tilde B[\phi] = \zeta_1(\tau) \phi + \zeta_2(\tau) y\cdot \nabla \phi ,
\end{align*}
with $\zeta_1(\tau)$, $\zeta_2(\tau)$ satisfying
\begin{align}
\label{zetai}
|\zeta_i(\tau)| \le   \frac C {\tau \log \tau  }\quad\text{for all } \tau>\tau_0.
\end{align}

\begin{lemma}
\label{lemma-B1}
We have
\begin{align}
\label{estimateB0}
\left|\int_{\R^2} \tilde B[\phi] g^\perp \right|
\leq
\frac{C}{\tau \log \tau } \int_{\R^2} U (g^\perp)^2 dy
+C \frac{|a(\tau)|}{\tau \log \tau}  \| \nabla g^\perp U^{\frac{1}{2}}\|_{L^2} .
\end{align}
\end{lemma}
\begin{proof}
We have
\begin{align*}
\int_{\R^2} \tilde B[\phi] g^\perp  dy
= \int_{\R^2} [ \zeta_1(\tau) \phi + \zeta_2(\tau) y\cdot \nabla \phi  ] g^\perp dy.
\end{align*}
By Lemma~\ref{lemma-q} and the hypothesis \eqref{zetai} we have
\begin{align}
\label{1.33}
\left|
\zeta_1(\tau)
\int_{\R^2}  \phi  g^\perp dy
\right|
& \leq \frac{C}{\tau \log \tau } \int_{\R^2} U (g^\perp)^2 dy.
\end{align}

%

Let us write
\begin{align}
\nonumber
\int_{\R^2} y \cdot \nabla \phi(y)  g^\perp(y)dy
&= \int_{\R^2} y \cdot \nabla \phi^\perp(y)  g^\perp(y)dy
+ \frac{a(\tau)}{2} \int_{\R^2} y \cdot \nabla Z_0(y)  g^\perp(y)dy.
\end{align}
We claim that
\begin{align}
\label{6.33-0}
\left|
\int_{\R^2} y \cdot \nabla \phi^\perp(y)  g^\perp(y)dy\right|\leq
C \int_{\R^2} (g^\perp)^2 Udy.
\end{align}
Indeed, we write
\begin{align}
\label{6.29}
\int_{\R^2} y \cdot \nabla \phi^\perp(y)  g^\perp(y)dy
&= \int_{\R^2} y \cdot \nabla ( U g^\perp ) g^\perp(y)dy
+  \int_{\R^2} y \cdot \nabla (  U \psi^\perp )  g^\perp(y)dy .
\end{align}
But
\begin{align*}
\int_{\R^2} y \cdot \nabla ( U g^\perp ) g^\perp(y)dy
&= \int_{\R^2} y \cdot \nabla  U (g^\perp)^2(y)dy
+ \int_{\R^2} U y \cdot \nabla g^\perp g^\perp(y)dy
\\
&= \int_{\R^2} y \cdot \nabla  U (g^\perp)^2(y)dy
+ \frac{1}{2} \int_{\R^2} U y \cdot \nabla[ (g^\perp)^2 ](y)dy
\\
&= \frac{1}{2}\int_{\R^2} y \cdot \nabla  U (g^\perp)^2(y)dy
- \int_{\R^2} U   (g^\perp)^2 (y)dy ,
\end{align*}
and so
\begin{align}
\label{6.30}
\left|
\int_{\R^2} y \cdot \nabla ( U g^\perp ) g^\perp(y)dy
\right|
\leq C \int_{\R^2} (g^\perp)^2 U dy.
\end{align}

The second term in \eqref{6.29} is:
\begin{align*}
\int_{\R^2} y \cdot \nabla (  U \psi^\perp )  g^\perp(y)dy
& = \int_{\R^2} (y \cdot \nabla  U)  \psi^\perp   g^\perp(y)dy
+ \int_{\R^2} U (y \cdot \nabla  \psi^\perp)   g^\perp(y)dy  .
\end{align*}
We estimate the first term above
\begin{align}
\nonumber
\left|
\int_{\R^2} (y \cdot \nabla  U)  \psi^\perp   g^\perp(y)dy
\right|
&\leq C
\Bigl( \int_{\R^2} (\psi^\perp )^2 Udy\Bigr)^{1/2}
\Bigl( \int_{\R^2} (g^\perp )^2 Udy\Bigr)^{1/2}
\\
\label{6.31}
&\leq C \int_{\R^2} (g^\perp )^2 Udy,
\end{align}
by \eqref{q3}.
To estimate $\int_{\R^2} U (y \cdot \nabla  \psi^\perp) g^\perp(y)dy  $ we write it using radial symmetry:
\begin{align*}
\int_{\R^2} U (y \cdot \nabla  \psi^\perp)   g^\perp(y)dy
= 2 \pi \int_0^\infty U(\rho) (\psi^\perp)'(\rho)
g^\perp (\rho) \rho^2 d\rho.
\end{align*}
We use  that
 $\psi^\perp $ satisfies
\[
-\Delta \psi^\perp - U \psi^\perp = U g^\perp \quad \text{in }\R^2, \quad
\psi^\perp(\rho,\tau)\to 0 \quad \text{as }\rho \to \infty .
\]
Then, by the variations of parameters formula, since that $\int_{\R^2} U g^\perp z_0 dy=0$, we have
\begin{align}
\nonumber
(\psi^\perp)'(\rho) = z_0'(\rho) \int_\rho^\infty U(r) g^\perp (r) \bar z_0(r) r \, d r
+ \bar z_0'(\rho) \int_0^\rho U(r) g^\perp (r)  z_0(r) r \, d r ,
\end{align}
where $\bar z_0$ is a second linear independent function in the kernel of $\Delta+U$ satisfying
\begin{align*}
|\bar z_0(\rho) | \leq C(  |\log \rho| + 1) .
\end{align*}
We then compute
\begin{align*}
\int_0^\infty U(\rho) (\psi^\perp)'(\rho)
g^\perp (\rho) \rho^2 d\rho
=
I_1 + I_2
\end{align*}
where
\begin{align*}
I_1 &=
\int_0^\infty
\int_\rho^\infty
U(\rho)   U(r)
z_0'(\rho) \bar z_0(r) g^\perp (r)
g^\perp (\rho) \rho^2  r  d r  d\rho
\\
I_2 &=
- \int_0^\infty
\int_\rho^\infty
U(\rho)U(r)
\bar z_0'(\rho) z_0(r)
g^\perp (r)
g^\perp (\rho) \rho^2  r  d r  d\rho.
\end{align*}
We directly check that
\begin{align*}
|I_1|+|I_2| \leq C \int_{\R^2} (g^\perp)^2 Udy.
\end{align*}
From this we get that
\begin{align}
\label{6.32}
\left|\int_{\R^2} U (y \cdot \nabla  \psi^\perp)   g^\perp(y)dy \right|
\leq  C \int_{\R^2} (g^\perp)^2 Udy.
\end{align}
Combining \eqref{6.29}, \eqref{6.30}, \eqref{6.31}, \eqref{6.32} we obtain \eqref{6.33-0}.

Next we claim that
\begin{align}
\label{1.39}
\left|  \int_{\R^2} y \cdot \nabla Z_0(y)  g^\perp(y)dy  \right|
&\leq  C  \| \nabla g^\perp U^{\frac{1}{2}}\|_{L^2}.
\end{align}
Indeed, write
\begin{align*}
y \cdot \nabla Z_0 &= \nabla \cdot ( y Z_0) - 2 Z_0
= \nabla \cdot ( y Z_0 - 2 \nabla z_0) - 4 Z_0
\end{align*}
where  $z_0$ is defined in \eqref{defZLiouville} and satisfies the linearized Liouville equation $\Delta z_0 + U z_0=0$. We have used here that  $Z_0 = U z_0$.
So
\begin{align*}
\int_{\R^2} y \cdot \nabla Z_0(y)  g^\perp(y)dy
= - \int_{\R^2} ( y Z_0 - 2 \nabla z_0) \nabla g^\perp dy - 4 \int_{\R^2} g^\perp Z_0 dy .
\end{align*}
But $\int_{\R^2} Z_0 g^\perp dy=\int_{\R^2} U z_0 g^\perp dy=0$ by Lemma~\ref{lemma-orthG}, and $|y Z_0 - 2 \nabla z_0| \leq \frac{C}{|y|^4}$, so
\begin{align*}
\left|
\int_{\R^2} y \cdot \nabla Z_0(y)	 g^\perp(y)dy
\right|
& \leq
C \Bigl( \int_{\R^2} \frac{1}{(1+|y|)^4}|\nabla g^\perp|^2 dy\Bigr)^{\frac{1}{2}}
\leq
C \| \nabla g^\perp U^{\frac{1}{2}}\|_{L^2}.
\end{align*}
This proves \eqref{1.39}.

From \eqref{1.33}, \eqref{6.33-0} and \eqref{1.39} we conclude the validity of \eqref{estimateB0}.

\end{proof}

%

In the next lemma we get an estimate for $\int_{\R^2} \phi g^\perp$, but with right hand side that depends on the solution.

\begin{lemma}
\label{lemma-energy}
We make the same assumptions of Proposition~\ref{prop-linear-with-energy}.
Let $f$ be given by \eqref{notation-f},
$\omega$ be defined in \eqref{notation-omega}
and let $R:[\tau_0,\infty) \to (0,\infty)$ be continuous.
There is $c>0$ , $\varepsilon>0$ and $C>0$ such that for $\tau_0$ sufficiently large, if
\begin{align}
\label{hyR0}
\sup_{\tau\geq \tau_0}\frac{R^2(\tau)}{\tau \log \tau} \leq \varepsilon
\end{align}
then
\begin{align}
\nonumber
\partial_\tau \int_{\R^2} \phi g^\perp +
\frac{c}{R^2} \int_{\R^2} \phi g^\perp
\leq
C f(\tau)^2  \|h\|_{**}^2
+ C \frac{a(\tau)^2}{R^4} + C \frac{\omega(\tau)^2}{R^2},
\end{align}
for some constant $c>0$.
\end{lemma}
\begin{proof}
Equation \eqref{eq-linear-radial20} can be written in the form
\begin{align}
\nonumber
\partial_\tau \phi = \nabla \cdot ( U \nabla g^\perp ) + B[\phi] + h , \quad \text{in }\R^2 \times (\tau_0,\infty) .
\end{align}
We multiply this equation by $g^\perp$ and integrate on $\R^2$, using Lemma~\ref{lemma-derPhig}:
\begin{align}
\label{eq-aa}
\frac{1}{2}
\partial_\tau \int_{\R^2} \phi g^\perp
+ \int_{\R^2 } U |\nabla g^\perp|^2 =
\int_{\R^2} B[\phi] g^\perp
+ \int_{\R^2} h g^\perp .
\end{align}

Let $H = (-\Delta)^{-1}h$, and observe that, since $h$ is radial and $\int_{\R^2} h dy = 0$,
\begin{align*}
|\nabla H ( \rho,\tau)| &= \left|\frac{1}{\rho}\int_{\rho}^\infty h(s,\tau) s ds \right|
\leq C  f(\tau) \|h\|_{**} \frac{1}{ (1+\rho)^{5+\sigma}} .
\end{align*}
It follows that
\begin{align}
\nonumber
\left|\int_{\R^2} h g^\perp \right|
&=
\left| \int_{\R^2} \nabla \cdot \nabla H g^\perp \right|
=
\left| \int_{\R^2} \nabla H \cdot \nabla g^\perp  \right|
\\
\nonumber
& \leq
\frac{1}{2}\int_{\R^2} U |\nabla g^\perp |^2
+ C \int_{\R^2} |\nabla H |^2 U^{-1}
\\
\nonumber
& \leq
\frac{1}{2}\int_{\R^2} U |\nabla g^\perp |^2
+ C  f(\tau)^2  \|h\|_{**}^2 .
\end{align}
This combined with \eqref{eq-aa} gives
\begin{align}
\label{eq-aa2}
\frac{1}{2}
\partial_\tau \int_{\R^2} \phi g^\perp
+ \frac{1}{2}\int_{\R^2 } U |\nabla g^\perp|^2 \leq
\left| \int_{\R^2} B[\phi] g^\perp \right|
+ C f(\tau)^2  \|h\|_{**}^2.
\end{align}

We use the inequality in Lemma~\ref{lemmaHardy} to get
\begin{align}
\label{abc4}
\frac{c}{R^2} \int_{B_{R}} (g^\perp - \bar g^\perp_{R})^2 U \leq \int_{\R^2 } U |\nabla g^\perp|^2 ,
\end{align}
for some $c>0$,
where
\begin{align*}
\bar g^\perp_{R} = \frac{1}{\int_{B_{R}} U }\int_{B_{R}} g^\perp U .
\end{align*}
From
\begin{align*}
\int_{B_R} (g^\perp)^2 U =
\int_{B_R} (g^\perp - \bar g^\perp_R)^2 U
+ 2 \int_{B_R} g^\perp \bar g^\perp_R U
-\int_{B_R} ( \bar g^\perp_R)^2 U
\end{align*}
we get
\begin{align}
\nonumber
\int_{B_R} (g^\perp)^2 U \leq
2 \int_{B_R} (g^\perp - \bar g^\perp_R)^2 U
+ C ( \bar g^\perp_R)^2 .
\end{align}
so, using \eqref{abc4},
\begin{align}
\nonumber
\frac{c}{R^2}\int_{B_R} (g^\perp)^2 U
&\leq
\int_{\R^2 } U |\nabla g^\perp|^2
+C \frac{1}{R^2} ( \bar g^\perp_R)^2 ,
\end{align}
for a new $c>0$.
This implies
\begin{align}
\nonumber
\frac{c}{R^2}\int_{\R^2} (g^\perp)^2 U
&\leq
\int_{\R^2 } U |\nabla g^\perp|^2
+C \frac{1}{R^2} ( \bar g^\perp_R)^2
+ C \frac{1}{R^2} \int_{\R^2\setminus B_R} U(g^\perp)^2.
\end{align}
Using that $g^\perp = g + a$ we get
\begin{align}
\label{abc3a}
\frac{c}{R^2}\int_{\R^2} (g^\perp)^2 U
&\leq
\int_{\R^2 } U |\nabla g^\perp|^2
+C \frac{1}{R^2} ( \bar g^\perp_R)^2
+ C \frac{\omega^2}{R^2}
+ C \frac{a^2}{R^4}.
\end{align}
But
\begin{align}
\nonumber
\int_{\R^2} g^\perp U d y = 0
\end{align}
and this implies
\begin{align*}
\bar g^\perp_{R} &= -
\frac{1}{\int_{B_{R}} U }
\int_{\R^2\setminus B_{R}} g^\perp U
\end{align*}
so
\begin{align*}
(\bar g_{R}^\perp)^2
&\leq \frac{C}{R^2}\int_{\R^2\setminus B_{R}} (g^\perp)^2 U
\leq
\frac{C a^2}{R^4} + \frac{C}{R^2 }\int_{\R^2\setminus B_{R}} g^2 U  .
\end{align*}
This combined with \eqref{abc3a} gives
\begin{align*}
\frac{c}{R^2}
\int_{\R^2} (g^\perp)^2 U
&\leq \int_{\R^2 } U |\nabla g^\perp|^2
+C \frac{a^2}{R^4} + C \frac{\omega^2}{R^2}.
\end{align*}
We use this together with \eqref{eq-aa2} to obtain  (for a new $c>0$)
\begin{align}
\nonumber
\frac{1}{2}
\partial_\tau \int_{\R^2} \phi g^\perp
& + \frac{c}{R^2}\int_{\R^2} (g^\perp)^2 U
+ \frac{1}{4}\int_{\R^2 } U |\nabla g^\perp|^2
\\
\label{eq-aa3}
& \quad
\leq
\left| \int_{\R^2} B[\phi] g^\perp \right|
+ C f(\tau)^2 \|h\|_{**}^2
+ C \frac{a^2}{R^4} + C \frac{\omega^2}{R^2 }.
\end{align}

We obtain from Lemma~\ref{lemma-B1} and the assumption \eqref{hyR0} that
\begin{align}
\nonumber
\left|\int_{\R^2} B[\phi] g^\perp \right|
& \leq
\frac{C}{\tau \log \tau} \int_{\R^2} (g^\perp)^2 Udy
+C \frac{|a(\tau)|}{\tau \log \tau}  \| \nabla g^\perp U^{\frac{1}{2}}\|_{L^2}
\\
\nonumber
& \leq
\frac{C}{\tau \log \tau} \int_{\R^2} (g^\perp)^2 Udy
+ \frac{|a(\tau)|^2}{R^4}
+ C\frac{R^4}{\tau^2 (\log \tau)^2}\| \nabla g^\perp U^{\frac{1}{2}}\|_{L^2}^2.
\\
\label{estimateB}
& =
C \frac{\varepsilon  }{R^2  } \int_{\R^2} (g^\perp)^2 Udy
+ \frac{|a(\tau)|^2}{R^4}
+ C \varepsilon^2 \| \nabla g^\perp U^{\frac{1}{2}}\|_{L^2}^2.
\end{align}
Taking  $\varepsilon>0$ small, and combining \eqref{eq-aa3} and \eqref{estimateB} we get
\begin{align*}
\partial_\tau \int_{\R^2} \phi g^\perp +
\frac{1}{R^2} \int_{\R^2} (g^\perp)^2 U
\leq
C f(\tau)^2 \|h\|_{**}^2
+\frac{C a^2}{R^4} + C \frac{\omega^2}{R^2}.
\end{align*}

By Lemma~\ref{lemma-q} we obtain
\begin{align*}
\partial_\tau \int_{\R^2} \phi g^\perp +
\frac{c}{R^2} \int_{\R^2} \phi g^\perp
\leq
C f(\tau)^2 \|h\|_{**}^2
+ C \frac{a^2}{R^4} + C \frac{\omega^2}{R^2},
\end{align*}
for some constant $c>0$,
which is the desired conclusion.
\end{proof}

The next lemma provides a pointwise estimate for $g = \frac{\phi}{U}- (-\Delta^{-1})\phi$
assuming a certain bound for $\| U^{1/2} g \|_{L^2}$.

\begin{lemma}
\label{lemma-pointwise1}
Assume $\nu>0$. Let $\phi$ be the solution to  \eqref{eq-linear-radial20} as in  \S\ref{sectLT1}. Suppose that $\tau_1\geq \tau_0$ and
\begin{align}
\label{16-0b}
\| g (\tau) U^{\frac{1}{2}} \|_{L^2(\R^2)}
\leq  K_1 f_1(\tau)
,\quad \tau \in [\tau_0,\tau_1] ,
\end{align}
where  $K_1 \geq 0 $ and
\begin{align*}
f_1(\tau) = \frac{(\log \tau)^\mu}{\tau^{\nu-1}},
\end{align*}
where $\mu \in \R$.
Then
\begin{align}
\nonumber
|U(y) g (y,\tau)|
\leq C   \Bigl( K_1 +\frac{\|h\|_{**}}{R(\tau_0)}+ \frac{|c_1|}{f_1(\tau_0)} \Bigr) f_1(\tau)
\frac{1}{ (1+|y|)^{2}}
,\quad \tau \in [\tau_0,\tau_1].
\end{align}
\end{lemma}

\begin{proof}
We define
\[
g_0 = U g ,
\]
and obtain from \eqref{eq-linear-radial12} the equation
\begin{align}
\nonumber
\partial_\tau
g_0 =
U \partial_\tau g
&= \partial_\tau \phi - U (-\Delta^{-1}) \partial_\tau \phi
\\
\nonumber
&=
\nabla \cdot \Bigl[ U \nabla \Bigl( \frac{g_0}{U} \Bigr) \Bigr]  - U  (-\Delta)^{-1}
\left[
\nabla \cdot ( U \nabla g)
\right]
+ h-U(-\Delta)^{-1} h
\\
& \quad
+ B[g_0]
+ B[ U \psi[g_0] ]
- U (-\Delta)^{-1} ( B [ g_0 + U \psi[g_0]] ) ,
\label{638}
\end{align}
where we regard  $\psi[g_0]$ as the operator that maps $g_0$ to the unique radial solution to
\begin{align}
\label{eqLiouville}
-\Delta \psi - U \psi = g_0 \quad \text{in }\R^2,
\quad
\psi(\rho,\tau) \to 0 \quad \text{as }\rho\to \infty.
\end{align}
We note that this problem has indeed a solution since $\int_{\R^2} g_0 z_0 dy = 0$ by Lemma~\ref{lemma-orthG}, which is unique by imposing $\psi(\rho,\tau) \to 0$ as $ \rho \to \infty$ in the radial setting. This solution is given by the variations of parameters formula
\begin{align}
\nonumber
\psi(\rho,\tau ) = z_0(\rho) \int_\rho^\infty g_0(r,\tau) \bar z_0(r) r \, d r
+ \bar z_0(\rho) \int_0^\rho  g_0(r,\tau)  z_0(r) r \, d r ,
\end{align}
where $\bar z_0$ is a second linear independent function in the kernel of $\Delta+U$ satisfying $ |\bar z_0(\rho) | \leq C(  |\log \rho| + 1) $.


We compute
\begin{align*}
\nabla \cdot (U\nabla g)
&=\Delta g U + \nabla U \cdot \nabla g
=\Delta (g U) - \nabla U \cdot \nabla g - g \Delta U,
\end{align*}
and hence
\begin{align*}
(-\Delta)^{-1}[ \nabla\cdot(U\nabla g)]
&= - g U -
(-\Delta)^{-1}
\left[
\nabla U \cdot \nabla g + g \Delta U
\right]
\\
&= - g U - v
\end{align*}
where
\begin{align}
\label{defV}
v:=  (-\Delta)^{-1}
( \nabla \cdot ( g_0 \nabla\Gamma_0  ) ).
\end{align}
We write \eqref{638} as
\begin{align}
\label{eqTildeG2}
\partial_\tau g_0
&=
\Delta g_0 - \nabla g_0 \cdot \nabla \Gamma_0
+ 2 U g_0
+ B[g_0]
+ \tilde h
\end{align}
where
\begin{align}
\label{tilde-h}
\tilde h =
 U v + B[ U \psi[g_0] ] - U (-\Delta)^{-1} ( B [ g_0 + U \psi[g_0]] )+ h-U(-\Delta)^{-1} h .
\end{align}

Note that since we are working with radial functions, we can integrate \eqref{defV} explicitly and obtain
\begin{align}
\label{integral-v}
v(\rho,\tau) = \int_\rho^\infty g_0(s,\tau) \Gamma_0'(s) ds .
\end{align}

We claim that for any $y \in \R^2$:
\begin{align}
\label{estTildeh1}
\| \tilde h \|_{ L^p (B_1(y))}
\leq C \Bigl( K_1 +\frac{\|h\|_{**}}{R(\tau_0)} \Bigr) f_1(\tau)  \frac{1}{(1+|y|)^{4-\frac{4}{p}}}
,\quad \tau \in [\tau_0,\tau_1].
\end{align}

Indeed, let us start with
\begin{align}
\nonumber
\int_0^\infty |v(\rho)|^p U(\rho) \rho d \rho
& \leq  \int_0^\infty \Bigl( \int_\rho^\infty U(s) g(s)^2 s ds\Bigr)^{p/2}
\Bigl( \int_\rho^\infty U(s) \frac{\Gamma_0'(s)^2}{s} ds\Bigr)^{p/2}
U(\rho) \rho d\rho
\\
\label{estV2Lp}
& \leq C \|g U^{\frac{1}{2}}\|_{L^2(\R^2)}^p,
\end{align}
which follows from  \eqref{integral-v} and H\"older's inequality

%
%
%
%
%

Let us write $\psi=\psi[g_0]$ and $\tilde \psi = \psi \circ \Pi$, where $\Pi$ is the stereographic projection.
Writing \eqref{eqLiouville} in $S^2$ and using standard $L^p$ theory we find that for any $p>2$
\[
\|  \tilde \psi \|_{L^\infty(S^2)}+\| \nabla_{S^2} \tilde \psi \|_{L^p(S^2)} \leq C
\|g U^{\frac{1}{2}}\|_{L^2(\R^2)},
\]
which implies
\begin{align}
\label{LpPsi}
\|\psi\|_{L^\infty(\R^2) }
+
\left( \int_{\R^2} |\nabla \psi|^p U^{1-\frac{p}{2}}
\right)^{\frac{1}{p}} \leq C
\|g U^{\frac{1}{2}}\|_{L^2(\R^2)}.
\end{align}

Let $y\in \R^2 $.
From \eqref{16-0b}
we see that
\begin{align}
\nonumber
\| g_0(\cdot,\tau) \|_{L^2(B_1(y))}
\leq C  K_1 f_1(\tau)   \frac{1 }{(1+|y|)^{2}}
,\quad \tau \in [\tau_0,\tau_1],
\end{align}
and from \eqref{16-0b} and \eqref{estV2Lp} we have
\begin{align}
\label{UvLp}
\|U v (\cdot,\tau) \|_{L^p(B_1(y))}\leq C   K_1 f_1(\tau)
\frac{1}{ (1+|y|)^{4-\frac{4}{p}}}
,\quad \tau \in [\tau_0,\tau_1].
\end{align}
Similarly,  inequalities \eqref{LpPsi} and \eqref{16-0b} imply
\begin{align}
\label{BLp}
\| B[U \psi[g_0]] \|_{L^p(B_1(y))}
\leq
C    K_1 f_1(\tau)
\frac{1}{\tau \log \tau}
\frac{1}{ (1+|y|)^{4}}
,\quad \tau \in [\tau_0,\tau_1].
\end{align}

Let's estimate
\begin{align*}
(-\Delta)^{-1} ( B [ g_0 + U \psi[g_0]] )
=
\zeta_1(\tau) (-\Delta)^{-1}( y\cdot \nabla (g_0 + U \psi[g_0]) ) + \zeta_2(\tau) (-\Delta)^{-1}(g_0 + U \psi[g_0]).
\end{align*}
Note that $\psi = (-\Delta)^{-1}\phi = (-\Delta)^{-1}(g_0 + U \psi) $.
But we can estimate $\psi$ from
\begin{align}
\label{psi-var-par3}
\psi(\rho) = z_0(\rho)\int_\rho^\infty \frac{1}{z_0(r)^2 r} \int_r^\infty g_0(s) z_0(s) s ds ,\quad \rho>1 .
\end{align}
Then \eqref{16-0b} yields
\begin{align}
\label{estPsi00a}
|\psi(\rho,\tau) |
& \leq  \frac{C}{1+\rho}
\| g U^{\frac{1}{2}}\|_{L^2( \R^2) }
 \leq
C   K_1 f_1(\tau)    \frac{1 }{1+\rho}
,\quad \tau \in [\tau_0,\tau_1],
\end{align}
and so
\begin{align}
\label{ultimo}
| U (-\Delta)^{-1} ( g_0 + U \psi[g_0] ) |
\leq C   K_1 f_1(\tau)    \frac{1 }{1+|y|^5}
,\quad \tau \in [\tau_0,\tau_1].
\end{align}
Concerning the term $(-\Delta)^{-1} ( y \cdot \nabla ( g_0 + U \psi ))$,
we notice that if we let $w = g_0 + U \psi$, then $\int_{\R^2} y \cdot \nabla w = 0$, and
\begin{align*}
(-\Delta)^{-1} (y \cdot \nabla w)(\rho) =
 \int_\rho^\infty r w(r,\tau)dr - 2 \psi(\rho,\tau).
\end{align*}
Using  \eqref{16-0b} and \eqref{estPsi00a} we get
\begin{align*}
\left|
((-\Delta)^{-1} (y \cdot \nabla (g_0 + U \psi)))(\rho,\tau)
\right|
\leq
C   K_1 f_1(\tau)   \frac{1 }{1+\rho}
,\quad \tau \in [\tau_0,\tau_1].
\end{align*}
From this and \eqref{ultimo} we find that
\begin{align}
\label{ultimo2}
|U (-\Delta)^{-1} ( B [ g_0 + U \psi[g_0]] )(y,\tau)|
\leq C  K_1 f_1(\tau)    \frac{1 }{ \tau \log \tau (1+|y|)^5}
,\quad \tau \in [\tau_0,\tau_1].
\end{align}
Finally the estimates
\begin{align}
\label{ultimo3}
\|h\|_{L^p(B_1(y))}
+ \|U (-\Delta)^{-1} h\|_{L^p(B_1(y))}
\leq C \frac{\|h\|_{**}}{R(\tau_0)} f_1(\tau)  \frac{1}{(1+|y|)^{4-\frac{4}{p}}}
\end{align}
are directly obtained.

Combining \eqref{UvLp}, \eqref{BLp}, \eqref{ultimo2}, and \eqref{ultimo3} we deduce  \eqref{estTildeh1}.

From equation  \eqref{eqTildeG2}, the estimate \eqref{estTildeh1}, standard parabolic $L^p$  estimates restricted to $B_1(y)\times ( \max(\tau-1,\tau_0),\tau)$  and embedding into H\"older spaces, we deduce that
\begin{align}
\label{decayTildeG3}
|g_0 (y,\tau)|
\leq C  \Bigl( K_1 +\frac{\|h\|_{**}}{R(\tau_0)}+ \frac{|c_1|}{f_1(\tau_0)} \Bigr)
f_1(\tau) \frac{1}{ (1+|y|)^{2}}
,\quad \tau \in [\tau_0,\tau_1].
\end{align}
This is the desired conclusion.
We also get from  \eqref{decayTildeG3}:
\begin{align}
\label{estV2}
|v (y,\tau)|
\leq C  \Bigl( K_1 +\frac{\|h\|_{**}}{R(\tau_0)}+ \frac{|c_1|}{f_1(\tau_0)} \Bigr)  f_1(\tau)  \frac{1}{ (1+|y|)^{2}}
,\quad \tau \in [\tau_0,\tau_1].
\end{align}
\end{proof}


In some of the proofs below the following barrier will be useful.
Consider the equation
\begin{align}
\label{heatEq6d}
\begin{aligned}
\partial_\tau \phi &= \Delta_{\R^6} \phi + h \quad \text{in } (\tau_0,\infty)\times \R^6
\\
\phi(\tau_0,\cdot) & = 0
\end{aligned}
\end{align}
where $\Delta_{\R^6}$ is the laplacian in $\R^6$.
Suppose that $h$ has the estimate
\[
|h (y,\tau) | \leq \frac{1}{\tau^{\gamma+1} } \frac{1}{(1+|y|/\sqrt{\tau})^b}
\]
for some $\gamma,b\in\R$.

If $\gamma<3$ and $\gamma<\frac{b}{2}$ then there is a barrier satisfying
\[
C_1 \frac{1}{\tau^{\gamma}  } \frac{1}{ (1+|y|/\sqrt \tau)^b} \leq
\phi (y,\tau)  \leq C_2 \frac{1}{\tau^{\gamma}  } \frac{1}{ (1+|y|/\sqrt \tau)^b}.
\]
Indeed, we can consider all functions to be radial and write $\rho = |y|$, $y\in \R^6$.
Let
\begin{align}
\label{barphi}
\bar\phi(\rho,\tau ) = \frac{1}{\tau^{\gamma}  } g\Bigl( \frac{\rho}{\sqrt \tau}\Bigr) , \quad \zeta = \frac{\rho}{\sqrt\tau}.
\end{align}
Then
\begin{align*}
\partial_\tau\bar\phi - \Bigl(\partial_{\rho\rho} + \frac{5}{\rho} \partial_\rho \Bigr) \bar\phi
= -
\frac{1}{\tau^{\gamma+1}}
\Bigl[
g''(\zeta) + \frac{5}{\zeta} g'(\zeta) +\frac{\zeta}{2}g'(\zeta) + \gamma g(\zeta)
\Bigr].
\end{align*}
Let $g_1(\zeta) = \frac{1}{(1+\zeta^2)^{b/2}}$.
Since $\gamma<\frac{b}{2}$ we have
\[
-\Bigl[
g_1''(\zeta) + \frac{5}{\zeta} g_1'(\zeta) +\frac{\zeta}{2}g_1'(\zeta) + \gamma  g_1(\zeta)
\Bigr]
\geq \frac{c}{\zeta^b}, \quad \zeta \geq M,
\]
for some $c, M>0$.
Let $g_0(\zeta) = e^{-\frac{\zeta^2}{4}}$ be the Gaussian kernel, which satisfies
\[
g_0''(\zeta) + \frac{5}{\zeta} g_0'(\zeta) +\frac{\zeta}{2}g_0'(\zeta) + 3  g_0(\zeta) = 0.
\]
Let $g = C_1 g_0 + g_1$. Since $\gamma<3$, we can find  $C_1$ large so that
\[
-\Bigl[
g''(\zeta) + \frac{5}{\zeta} g'(\zeta) +\frac{\zeta}{2}g'(\zeta) + \gamma  g(\zeta)
\Bigr]
\geq \frac{c}{1+\zeta^b}, \quad \zeta >0.
\]
Then $\bar\phi$ defined by \eqref{barphi} with $g = C_1 g_0 + g_1$ is a supersolution to \eqref{heatEq6d}.

\medskip

In the next lemma we improve the spatial decay of  $g = \frac{\phi}{U}- (-\Delta^{-1})\phi$.

\begin{lemma}
\label{lemma-pointwise2}
Assume $1<\nu<\frac{7}{4}$. Let $\phi$ be the solution to  \eqref{eq-linear-radial20} as in  \S\ref{sectLT1}. Suppose that $\tau_1\geq \tau_0$ and
\begin{align}
\nonumber
\| g (\tau) U^{\frac{1}{2}} \|_{L^2(\R^2)}
\leq K_1 f_1(\tau)
,\quad \tau \in [\tau_0,\tau_1] ,
\end{align}
where $K_1\geq 0$ and
\begin{align*}
f_1(\tau) = \frac{(\log \tau)^\mu}{\tau^{\nu-1}},
\end{align*}
where $\mu\in\R$.
Then
\begin{align}
\label{17-0-b1}
|U (\rho) g(\rho,\tau) |\leq C  \Bigl( K_1 +\frac{\|h\|_{**}}{R(\tau_0)}+ \frac{|c_1|}{f_1(\tau_0)} \Bigr)  f_1(\tau) \frac{1}{(1+\rho)^{4}}
, \quad \tau \in [\tau_0,\tau_1] .
\end{align}
\end{lemma}
\begin{proof}
We us the same notation as in Lemma~\ref{lemma-pointwise1} and consider \eqref{eqTildeG2} for $g_0 = U g$ with $\tilde h$ defined in \eqref{tilde-h}.
We are going to use barriers to estimate $g_0$.

We claim that $\tilde h$ satisfies
\begin{align}
\label{estTildeH1}
|\tilde h(y,\tau)|
\leq
C \Bigl( K_1 +\frac{\|h\|_{**}}{R(\tau_0)}  \Bigr)  f_1(\tau)
\Bigl( \frac{1}{(1+|y|)^6} +\frac{1}{\tau \log \tau(1+|y|)^5} \Bigr)
,\quad \tau \in [\tau_0,\tau_1].
\end{align}
Indeed, from  \eqref{decayTildeG3} and  \eqref{estV2}  we find that
\begin{align}
\label{estPartTildeh}
| -  U v + h-U(-\Delta)^{-1} h |
\leq
C  \Bigl( K_1 +\frac{\|h\|_{**}}{R(\tau_0)}\Bigr)  f_1(\tau)  \frac{1 }{ (1+|y|)^{6}}
,\quad \tau \in [\tau_0,\tau_1].
\end{align}
To estimate $B[U \psi[g_0]]$ we use \eqref{estPsi00a}  a similar estimate for $\partial_\rho \psi$, and the assumptions on $\zeta_1$, $\zeta_2$ in \eqref{zetai}, to obtain
\begin{align*}
|B[U\psi[g_0]|\leq
C K_1  f_1(\tau)   \frac{1  }{\tau \log \tau}  \frac{1}{1+|y|^5}
,\quad \tau \in [\tau_0,\tau_1].
\end{align*}
This, \eqref{estPartTildeh} and \eqref{ultimo2} prove \eqref{estTildeH1}.

To get better spatial decay  we construct a barrier and apply the maximum principle to equation \eqref{eqTildeG2} in $(\R^2 \setminus B_{R_0} (0) ) \times (\tau_0,\tau_1)$, where $R_0$ is a fixed large constant.
Several of constants $C$ below depend on $R_0$ but we will not keep track of the explicit dependence.

The linear operator for $g_0$ in \eqref{eqTildeG2}, acting on radial functions with $\rho = |y|$, is given by:
\begin{align*}
\partial_\tau
g_0
-[  \Delta g_0 - \nabla g_0 \cdot \nabla \Gamma_0
+ B[g_0]  + 2 U g_0]
&= \partial_\tau g_0 - \partial_{\rho\rho} g_0 - \frac{1}{\rho} \partial_\rho g_0
- \frac{4 \rho}{1+\rho^2} \partial_\rho g_0
+ O\Bigl( \frac{1}{1+\rho^4} \Bigr) g_0
\\
& \quad + O\Bigl( \frac{1}{\tau \log \tau} \Bigr) g_0
+ O\Bigl( \frac{1}{\tau \log \tau} \Bigr) \rho \partial_\rho g_0 .
\end{align*}
The main part outside of a ball $B_{R_0}(0)$ with $R_0$ big is given by  $\partial_\tau- \partial_{\rho\rho} - \frac{5}{\rho}\partial_\rho $.

By \eqref{estTildeH1} we need to construct $\bar g_1$  such that
\begin{align*}
\partial_\tau
\bar g_1
-[  \Delta \bar g_1 - \nabla \bar g_1 \cdot \nabla \Gamma_0
+ B[\bar g_1]  +2 U \bar g_1 ] \geq h_1
\end{align*}
where
\begin{align*}
h_1(\rho,\tau) &=  f_1(\tau) \Bigl( \frac{1}{(1+\rho)^6} +\frac{1}{\tau \log \tau(1+\rho)^5}   \Bigr) .
\end{align*}

To construct $\bar g_1$, let $0<\vartheta<1$, and let $\tilde g_1(\rho)$ be radial and solve
\[
-\Delta_6 \tilde g_1 = \frac{1}{1+\rho^{6-\vartheta}} \quad \text{in }\R^6,
\]
such that $ \tilde g_1(\rho) (1+\rho^{4-\vartheta})$ is bounded below and above by positive constants.
Let
\begin{align*}
\bar g_1(\rho,\tau)
&= f_1(\tau) \tilde g_1(\rho) \chi_0\Bigl(\frac{\rho}{\delta\sqrt{\tau}}\Bigr)
+ C_1 \frac{ f_1(\tau) }{\tau^{2-\vartheta/2} (1+\rho/\sqrt{\tau})^{5}}
+ C_2  \frac{ f_1(\tau) }{\tau^{2-\vartheta/2}} e^{-\frac{\rho^2}{4\tau} } ,
\end{align*}
For appropriate $\delta>0$, $C_1$,  and $C_2$, the function $ \bar g_1(\rho,\tau) $ is a supersolution in $(\R^2 \setminus B_{R_0} (0) ) \times (\tau_0,\tau_1)$ for the right hand side $h_1$. More precisely, writing $M = R_0^{2-\vartheta} ( K_1 +\frac{\|h\|_{**}}{R(\tau_0)} + \frac{1}{f_1(\tau_0) } |c_1| )$, we have
\begin{align*}
\left( \partial_\tau
-[  \Delta  - \nabla (\cdot ) \cdot \nabla \Gamma_0
+ B   ] \right)  M  \bar g_1
\geq  |\tilde h|,
\quad \text{in }(\R^2 \setminus B_{R_0} (0) ) \times (\tau_0,\tau_1),
\end{align*}
\begin{align*}
M \bar g_1 \geq |g_0| ,
\quad \text{on } \rho = R_0, \ \tau  \in (\tau_0,\tau_1),
\end{align*}
because of Lemma~\ref{lemma-pointwise1}, and
\begin{align*}
M \bar g_1 (\tau_0) \geq \left| c_1 U g_{\tilde Z_0} \right| , \quad \text{in }\R^2,
\end{align*}
where
\begin{align*}
g_{\tilde Z_0} = \frac{\tilde Z_0}{U} - (-\Delta)^{-1} \tilde Z_0,
\end{align*}
is the function $g$ associated to $\tilde Z_0$ defined in \eqref{def-tilde-Z0}.
We note that $|U g_{\tilde Z_0}(\rho)|\leq C \frac{1}{1+\rho^4}$ and is supported on $\rho \leq 2\sqrt{\tau_0}$.
Here we are using that $\nu<\frac{3}{2}+\frac{\vartheta}{2}$.

Using the maximum principle we get
\begin{align}
\nonumber
|g_0 (y,\tau)|
\leq C    \Bigl( K_1  +\frac{\|h\|_{**}}{R(\tau_0)}+ \frac{|c_1|}{f_1(\tau_0)} \Bigr) f_1(\tau)
\frac{1}{(1+\rho)^{4-\vartheta}}
,\quad \tau \in [\tau_0,\tau_1].
\end{align}
The constant $C$ here depends on $R_0$, but $R_0$ is fixed and we will not keep track of the dependence of $C$ on $R_0$.

By \eqref{integral-v} and \eqref{psi-var-par3} we have
\begin{align*}
|\tilde h (y,\tau)|
\leq C \Bigl( K_1 +\frac{\|h\|_{**}}{R(\tau_0)}+ \frac{|c_1|}{f_1(\tau_0)} \Bigr)
f_1(\tau)\Bigl(
\frac{1}{(1+\rho)^{6+\sigma}}
+\frac{1}{\tau \log \tau ( 1+\rho)^{6-\vartheta}}
\Bigr) .
\end{align*}

We can now repeat the argument with a new barrier.
Consider  $\tilde g_2(\rho)$ the radial solution to
\begin{align}
\label{def-tildeg2}
\begin{aligned}
&-\Delta_6 \tilde g_2 = \frac{1}{1+\rho^{6+\sigma}} \quad \text{in }\R^6,
\quad  c_1 \frac{1}{1+\rho^4} \leq \tilde g_2(\rho) \leq c_2 \frac{1}{1+\rho^4} ,
\end{aligned}
\end{align}
where $c_1$, $c_2>0$.
Let
\begin{align*}
\bar g_2(\rho,\tau)
&= f_1(\tau) \tilde g_2(\rho) \chi_0\Bigl(\frac{\rho}{\delta\sqrt{\tau}}\Bigr)
+ C_1 \frac{f_1(\tau)}{\tau^{2} (1+\rho/\sqrt{\tau})^{6-\vartheta}}
+ C_2  \frac{ f_1(\tau)}{\tau^{2}} e^{-\frac{\rho^2}{4\tau} } .
\end{align*}
For appropriate constants $\delta$, $C_1$, $C_2$, and assuming that $\nu<2-\frac{\vartheta}{2}$ we get a suitable supersolution and we obtain
\begin{align}
\nonumber
|g_0 (y,\tau)|
\leq  C \Bigl( K_1 +\frac{\|h\|_{**}}{R(\tau_0)}+ \frac{|c_1|}{f_1(\tau_0)} \Bigr)
f_1(\tau) \frac{1}{(1+\rho)^4}.
\end{align}
This proves \eqref{17-0-b1}.

The restriction on $\nu$ were $\nu < \frac{3}{2}+\frac{\vartheta}{2}$ and  $\nu<2-\frac{\vartheta}{2}$. Choosing $\vartheta=\frac{1}{4}$ we find that for $\nu<\frac{7}{4}$ both barriers work.
\end{proof}

The next result is a technical step used in several places.
\begin{lemma}
\label{lemma-from-g-to-phi}
Let $\phi:\R^2 \to \R$ be radial such that $\int_{\R^2} \phi = 0$ and $|\phi(y)|\leq \frac{C}{(1+|y|)^{2+\sigma}}$ for some $\sigma>0$. Let $g = \frac{\phi}{U}-(-\Delta)^{-1} \phi$ and assume that $\|g\|_{L^\infty} <\infty$. Then
\begin{align}
\label{bound-phi}
|\phi(y) |\leq C \frac{\|g\|_{L^\infty}}{(1+|y|)^4}.
\end{align}
\end{lemma}
\begin{proof}
Let $\psi = (-\Delta)^{-1}\phi$.
Since $\psi$ satisfies
\[
-\Delta \psi - U \psi = U g \quad \text{in }\R^2, \quad
\psi(\rho)\to 0 \quad \text{as }\rho \to \infty,
\]
we have necessarily
\begin{align}
\nonumber
\int_{\R^2} U g z_0 dy=0 .
\end{align}
We have the variations of parameters formula
\begin{align}
\label{ax-psi-var-par}
\psi (\rho) = z_0(\rho)
\int_\rho^\infty \frac{1}{z_0(r)^2 r}
\int_r^\infty U g (s,\tau) z_0(s) s \, ds \, dr , \quad \rho >1 .
\end{align}
From \eqref{ax-psi-var-par} we find
\begin{align*}
|\psi(\rho,\tau) |
\leq  C \|g\|_{L^\infty}.
\end{align*}
This and the formula $\phi = U g + U \psi$ gives \eqref{bound-phi}.
\end{proof}

Next we give a proof of Proof of Lemma~\ref{lemma.apriori}, but first we point some estimates of $\tilde Z_0$ defined in \eqref{def-tilde-Z0}. Using the general decomposition \eqref{decomp-phi-perp}, we write
\begin{align*}
\tilde Z_0 = \tilde Z_0^\perp + \frac{\tilde a}{2} Z_0 .
\end{align*}
By \eqref{formula-a}
\begin{align*}
\tilde a =\frac{1}{8\pi}\int_{\R^2} \Gamma_0 \tilde Z_0
= 2 + O(\frac{\log \tau_0}{\tau_0}) .
\end{align*}
Hence $\tilde Z_0^\perp $ satisfies
\begin{align}
\nonumber
\tilde Z_0^\perp(\rho)
&= \tilde Z_0(\rho) - \frac{a(\tau_0)}{2}Z_0(\rho)\\
\nonumber
& =  ( Z_0(\rho) - m_{Z_0} U(\rho) ) \chi_0\Bigl( \frac{\rho}{\sqrt{\tau_0}}\Bigr)
- \Bigl( 1 + O (\frac{\log\tau_0}{\tau_0})\Bigr)	Z_0(\rho)\\
\label{best-tildeZ0-perp}
& = O\Bigl( \frac{\log\tau_0}{\tau_0} \frac{1}{1+\rho^4} \Bigr).
\end{align}

\begin{proof}[Proof of Lemma~\ref{lemma.apriori}]

We let $R$ be defined by \eqref{def-R}.
We multiply equation \eqref{eq-linear-radial20} by $g^\perp$ and integrate in $\R^2$. Using Lemmas~\ref{lemma-energy} and \ref{lemma-q} we get
\begin{align}
\label{interm1}
\partial_\tau \int_{\R^2} \phi g^\perp +
\frac{c}{R^2} \int_{\R^2} \phi g^\perp
\leq
C f(\tau)^2 \|h\|_{**}^2 +\frac{C a^2}{R^4} + \frac{C}{R^2 } \omega(\tau)^2 ,
\end{align}
for some $c>0$, where
\begin{align*}
\omega(\tau) = \Bigl( \int_{\R^2\setminus B_{R(\tau)}} g^2 U \Bigr)^{1/2}.
\end{align*}

Let us write
\begin{align*}
\| \varphi \|_{\infty,T_2} = \| \varphi \|_{L^\infty(\tau_0,T_2)} ,
\end{align*}
and note that
\begin{align*}
\Bigl\| \frac{a}{R^2 f} \Bigr\|_{\infty,T_2}^2 <\infty, \quad  \Bigl\| \frac{\omega}{Rf}\Bigr\|_{\infty,T_2}<\infty.
\end{align*}

The following inequalities are valid for $\tau_0<\tau<T_2$.
From \eqref{interm1} we get
\begin{align*}
\partial_\tau \int_{\R^2} \phi g^\perp +
\frac{c}{R^2} \int_{\R^2} \phi g^\perp
\leq
C f(\tau)^2
\Bigl(  \|h\|_{**}^2
+ \Bigl\| \frac{a}{R^2 f} \Bigr\|_{\infty,T_2}^2 + \Bigl\| \frac{\omega}{Rf}\Bigr\|_{\infty,T_2}^2
\Bigr) .
\end{align*}
By Gronwall's inequality and Lemma~\ref{lemma-q} we get
\begin{align}
\label{u-11-0}
\int_{\R^2} (g^\perp)^2 U
& \leq
C  f(\tau)^2 R(\tau)^2
\Bigl(  \|h\|_{**}^2
+  \Bigl\| \frac{a}{R^2 f} \Bigr\|_{\infty,T_2}^2
+ \Bigl\| \frac{\omega}{Rf}\Bigr\|_{\infty,T_2}^2
+ c_1^2 D(\tau_0)^2 \Bigr)
\end{align}
where
\begin{align*}
D(\tau_0) = \frac{1}{f(\tau_0) R(\tau_0)} \frac{\log\tau_0}{\tau_0} ,
\end{align*}
and we have used \eqref{best-tildeZ0-perp}.

From \eqref{u-11-0} we find
\begin{align}
\label{u-16-0}
\int_{\R^2} g^2 U
\leq
C  f(\tau)^2 R(\tau)^4
\Bigl(
\frac{1}{R(\tau_0)^2}  \|h\|_{**}^2
+  \Bigl\| \frac{a}{R^2 f} \Bigr\|_{ \infty,T_2 }^2
+ \frac{1}{R(\tau_0)^2}\Bigl\| \frac{\omega}{R f}\Bigr\|_{ \infty,T_2 }^2
+ c_1^2 \frac{D(\tau_0)^2}{R(\tau_0)^2}
\Bigr)
\end{align}
Using Lemma~\ref{lemma-pointwise2} we get
\begin{align}
\label{u-17-0}
|U g|\leq C
f(\tau) R(\tau)^2
\Bigl(
\frac{1}{R(\tau_0)}  \|h\|_{**}
+  \Bigl\| \frac{a}{R^2 f} \Bigr\|_{\infty,T_2}
+ \frac{1}{R(\tau_0)} \Bigl\| \frac{\omega}{R f}\Bigr\|_{\infty,T_2}
+ |c_1| \frac{1}{f(\tau_0) R(\tau_0)^2}
\Bigr)
\frac{1}{(1+\rho)^{4}} ,
\end{align}
where we have used that for $\tau_0$ large, $ \frac{D(\tau_0)}{R(\tau_0)}< \frac{1}{f(\tau_0) R(\tau_0)^2} $.

We use this  to estimate
\begin{align}
\nonumber
\int_{\R^2\setminus B_{R}} g^2 U
\leq C
f(\tau)^2 R(\tau)^2
\Bigl(
\frac{1}{R(\tau_0)} \|h\|_{**}
+  \Bigl\| \frac{a}{R^2f } \Bigr\|_{ \infty,T_2 }
+ \frac{1}{R(\tau_0)}\Bigl\| \frac{\omega}{R f}\Bigr\|_{\infty,T_2}
+ |c_1| \frac{1}{f(\tau_0) R(\tau_0)^2}
\Bigr)^2,
\end{align}
which implies
\begin{align*}
\frac{\omega(\tau) }{R(\tau)f(\tau)}
\leq C \Bigl(
\frac{1}{R(\tau_0)} \|h\|_{**}
+  \Bigl\| \frac{a}{R^2f } \Bigr\|_{ \infty ,T_2}
+ \frac{1}{R(\tau_0)}\Bigl\| \frac{\omega}{R f}\Bigr\|_{\infty,T_2}
+ |c_1|  \frac{1}{f(\tau_0) R(\tau_0)^2}
\Bigr).
\end{align*}
We deduce that
\begin{align}
\label{ineq.omega}
\Bigl\| \frac{\omega}{Rf}\Bigr\|_{\infty,T_2}
\leq C \Bigl(
\frac{1}{R(\tau_0)} \|h\|_{**}
+  \Bigl\| \frac{a}{R^2f } \Bigr\|_{ \infty ,T_2}
+ |c_1|  \frac{1}{f(\tau_0) R(\tau_0)^2}
\Bigr).
\end{align}
Combining this inequality with \eqref{u-16-0}  we obtain
\begin{align}
\label{u-40}
\int_{\R^2} g^2 U
\leq C
f(\tau)^2 R(\tau)^4
\Bigl(
\frac{1}{R(\tau_0)} \|h\|_{**} +  \Bigl\| \frac{a}{R^2 f} \Bigr\|_{\infty,T_2}
+ |c_1| \frac{1}{f(\tau_0) R(\tau_0)^2}
\Bigr)^2 ,
\end{align}
and with \eqref{u-11-0} we get
\begin{align}
\label{u-11-0b}
\int_{\R^2} (g^\perp)^2 U
\leq
C  f(\tau)^2 R(\tau)^2
\Bigl(
\|h\|_{**}
+  \Bigl\| \frac{a}{R^2 f} \Bigr\|_{\infty,T_2}
+ |c_1|  \frac{1}{f(\tau_0) R(\tau_0)^2}
\Bigr)^2 .
\end{align}
Going back to \eqref{u-17-0} we find
\begin{align}
\label{u-18-0c}
|U g (\rho,\tau)|\leq
Cf(\tau) R(\tau)^2
\Bigl( \frac{1}{R(\tau_0)}  \|h\|_{**}
+  \Bigl\| \frac{a}{R^2 f} \Bigr\|_{\infty}
+ |c_1|   \frac{1}{f(\tau_0) R(\tau_0)^2}
\Bigr)
\frac{1}{1+\rho^4} .
\end{align}
Using Lemma~\ref{lemma-from-g-to-phi} we also obtain
\begin{align}
\label{u-18b-0}
|\phi(\rho,\tau)|\leq
C f(\tau) R(\tau)^2
\Bigl( \frac{1}{R(\tau_0)}  \|h\|_{**}
+ \Bigl\| \frac{a}{R^2 f} \Bigr\|_{\infty,T_2}
+ |c_1| \frac{1}{f(\tau_0) R(\tau_0)^2}
\Bigr)
\frac{1}{1+\rho^4} .
\end{align}

We multiply the equation satisfied by $\phi$ \eqref{eq-linear-radial20}  by $|y|^2 \chi_0(\frac{y}{R})$, and integrate on $\R^2$
\begin{align}
\nonumber
\partial_\tau \int_{\R^2}   \phi |y|^2  \chi_0\Bigl(\frac{y}{R}\Bigr) dy
& =\int_{\R^2} (L[\phi] + h) |y|^2  \chi_0(\frac{y}{R}) dy
+\int_{\R^2} B[\phi]  |y|^2  \chi_0(\frac{y}{R}) dy
\\
\label{xyz1}
& \quad
- \frac{R'(\tau)}{R}\int_{\R^2} \phi |y|^2  \nabla \chi_0(\frac{y}{R}) \cdot \frac{y}{R}	dy,
\end{align}
where $R' = \frac{d R}{d \tau}$.
Now integrate from $\tau$ to $T_2$ and use the decomposition \eqref{decomp-phi-perp}.

We integrate \eqref{xyz1} from $\tau$ to $T_2$, use the decomposition \eqref{decomp-phi-perp} and that $a(T_2)=0$ to get
\begin{align}
\nonumber
|a(\tau) | \log \tau
& \leq
\left| \int_\tau^{T_2} \int_{\R^2} (L[\phi(s)] + h) |y|^2  \chi_0(\frac{y}{R(s)}) dy ds\right|
+ \left| \int_\tau^{T_2} \int_{\R^2} B[\phi(s)]  |y|^2  \chi_0(\frac{y}{R(s)}) dy ds\right|
\\
\nonumber
& \quad
+ \left| \int_\tau^{T_2}  \frac{R'(s)}{R(s)} \int_{\R^2} \phi(s) |y|^2  \nabla \chi_0(\frac{y}{R(s)}) \cdot \frac{y}{R(s)}	dy ds\right|
\\
\label{u-42}
& \quad
+ \left| \int_{\R^2}   \phi^\perp(T_2) |y|^2  \chi_0\Bigl(\frac{y}{R(T_2)}\Bigr) dy\right|
+\left| \int_{\R^2}   \phi^\perp (\tau) |y|^2  \chi_0\Bigl(\frac{y}{R(\tau)}\Bigr) dy\right|.
\end{align}

By Lemma~\ref{lemma-q2} and \eqref{u-11-0b}
\begin{align}
\nonumber
\int_{B_{2R(\tau)}}  |\phi(\tau)^\perp| |y|^2 dy
& \leq C R(\tau) \Bigl( \int_{\R^2} (\phi^\perp(\tau) )^2 U^{-1} \Bigr)^{1/2}
\\
\nonumber
& \leq C R(\tau) \Bigl( \int_{\R^2} (g^\perp(\tau))^2 U \Bigr)^{1/2}
\\
\label{u-32}
&\leq C  f(\tau) R(\tau)^2
\Bigl(  \|h\|_{**}
+  \Bigl\| \frac{a}{R^2 f} \Bigr\|_{\infty,T_2}
+ |c_1| \frac{1}{f(\tau_0) R(\tau_0)^2}
\Bigr) .
\end{align}
Analogously,
\begin{align}
\nonumber
\left|
\int_{\R^2}   \phi^\perp (T_2) |y|^2  \chi_0\Bigl(\frac{y}{R(T_2)}\Bigr) dy
\right|
&
\leq C
f(T_2) R(T_2)^2
\Bigl(  \|h\|_{**}
+  \Bigl\| \frac{a}{R^2 f} \Bigr\|_{\infty,T_2}
+  |c_1|  \frac{1}{f(\tau_0) R(\tau_0)^2}
\Bigr)
\\
\label{u-32b}
&\leq C  f(\tau) R(\tau)^2
\Bigl(  \|h\|_{**}
+  \Bigl\| \frac{a}{R^2 f} \Bigr\|_{\infty,T_2}
+ |c_1| \frac{1}{f(\tau_0) R(\tau_0)^2}
\Bigr) .
\end{align}

Integrating by parts
\begin{align}
\nonumber
\left| \int_\tau^{T_2} \int_{\R^2} B[\phi(s)]  |y|^2  \chi_0(\frac{y}{R(s)}) dy ds\right|
& \leq
C \int_\tau^{T_2} \frac{1}{s \log s}   \int_{\R^2}  |\phi(y,s)|   |y|^2 \chi_0(\frac{y}{R(s)}) dy ds
\\
\label{u-39}
& \quad
+ C \int_\tau^{T_2} \frac{1}{s \log s }   \int_{\R^2}  |\phi(y,s)|   |y|^2 |\nabla \chi_0(\frac{y}{R(s)})| dy ds.
\end{align}
Let's estimate, using  \eqref{u-18b-0}
\begin{align}
\nonumber
& \int_\tau^{T_2} \frac{1}{s \log s}   \int_{\R^2}  |\phi(y,s)|   |y|^2 \chi_0(\frac{y}{R(s)}) dy ds
\\
\nonumber
&\quad  \leq
C\int_\tau^{T_2}
\frac{1}{s}     f(s) R(s)^2 ds
\Bigl( \frac{1}{R(\tau_0)}  \|h\|_{**}
+  \Bigl\| \frac{a}{R^2 f} \Bigr\|_{\infty,T_2}
+ |c_1|  \frac{1}{f(\tau_0) R(\tau_0)^2}
\Bigr)
\\
\nonumber
&\quad  \leq
C  f(\tau) R(\tau)^2
\Bigl( \frac{1}{R(\tau_0)}  \|h\|_{**}
+  \Bigl\| \frac{a}{R^2 f} \Bigr\|_{\infty,T_2}
+  |c_1|  \frac{1}{f(\tau_0) R(\tau_0)^2}
\Bigr).
\end{align}
The second term in \eqref{u-39} is even smaller, and we deduce that
\begin{align}
\nonumber
\left| \int_\tau^{T_2} \int_{\R^2} B[\phi(s)]  |y|^2  \chi_0(\frac{y}{R(s)}) dy ds\right|
&\leq C  f(\tau) R(\tau)^2
\Bigl( \frac{1}{R(\tau_0)}  \|h\|_{**}
+  \Bigl\| \frac{a}{R^2 f} \Bigr\|_{\infty,T_2}
\\
\label{u-eq38}
& \quad
+  |c_1|  \frac{1}{f(\tau_0) R(\tau_0)^2}
\Bigr).
\end{align}

From \eqref{u-18b-0} we also get
\begin{align}
\nonumber
& \left| \int_\tau^{T_2}  \frac{R'(s)}{R(s)} \int_{\R^2} \phi(s) |y|^2  \nabla \chi_0(\frac{y}{R(s)}) \cdot \frac{y}{R(s)}	dy ds\right|
\\
\nonumber
&\quad  \leq
C\int_\tau^{T_2}
\frac{R'(s)}{R(s)}    f(s) R(s)^2 ds
\Bigl( \frac{1}{R(\tau_0)}  \|h\|_{**}
+  \Bigl\| \frac{a}{R^2 f} \Bigr\|_{\infty,T_2}
+ |c_1|  \frac{1}{f(\tau_0) R(\tau_0)^2}
\Bigr)
\\
\label{u-43}
&\quad  \leq
C f(\tau) R(\tau)^2
\Bigl( \frac{1}{R(\tau_0)}  \|h\|_{**}
+ \Bigl\| \frac{a}{R^2 f} \Bigr\|_{\infty,T_2}
+  |c_1|  \frac{1}{f(\tau_0) R(\tau_0)^2}
\Bigr).
\end{align}

Next we look at
\begin{align}
\nonumber
\int_{\R^2} L[\phi ] |y|^2  \chi_0(\frac{y}{R}) dy
&=
- 2 \int_{\R^2} U \nabla g \cdot  y   \chi_0(\frac{y}{R}) dy
- \frac{1}{R}\int_{\R^2} U    |y|^2   \nabla g\cdot \nabla \chi_0(\frac{y}{R}) dy
\\
\nonumber
& =
2 \int_{\R^2} g Z_0  \chi_0(\frac{y}{R}) dy
+\frac{4}{R}\int_{\R^2} g U y \cdot\nabla  \chi_0(\frac{y}{R}) dy
\\
\label{abc20}
& \quad
+\frac{1}{R}\int_{\R^2} g  |y|^2 \nabla U  \cdot\nabla  \chi_0(\frac{y}{R}) dy
+\frac{1}{R^2}\int_{\R^2} g  |y|^2  U \Delta  \chi_0(\frac{y}{R}) dy .
\end{align}
We have $\int_{\R^2} g Z_0 = 0$ by Lemma~\ref{lemma-orthG} and therefore, using \eqref{u-18-0c}, we find that
\begin{align}
\nonumber
\left|  \int_{\R^2} g Z_0  \chi_0(\frac{y}{R(\tau)}) dy \right|
&\leq  \left| \int_{\R^2\setminus B_{R(\tau)}(0) } g Z_0  dy\right|
\\
\nonumber
&
\leq
f(\tau)
\Bigl( \frac{1}{R(\tau_0)}  \|h\|_{**}
+  \Bigl\| \frac{a}{R^2 f} \Bigr\|_{\infty}
+ |c_1|   \frac{1}{f(\tau_0) R(\tau_0)^2}
\Bigr) .
\end{align}
The remaining terms in \eqref{abc20} are estimated using  \eqref{u-40}  or \eqref{u-18-0c}
and we get
\begin{align*}
\left| \int_{\R^2} L[\phi ] |y|^2  \chi_0(\frac{y}{R}) dy \right|
\leq C f(\tau)
\Bigl(
\frac{1}{R(\tau_0)} \|h\|_{**}
+  \Bigl\| \frac{a}{R^2 f} \Bigr\|_{\infty,T_2}
+  |c_1| \frac{1}{f(\tau_0) R(\tau_0)^2}
\Bigr) .
\end{align*}
Therefore
\begin{align}
\nonumber
\left| \int_\tau^{T_2} \int_{\R^2} L[\phi(s)]  |y|^2  \chi_0(\frac{y}{R(s)}) dy ds\right|
& \leq C f(\tau) R(\tau)^2 (\log \tau)^q
\Bigl(
\frac{1}{R(\tau_0)} \|h\|_{**}
+  \Bigl\| \frac{a}{R^2 f} \Bigr\|_{\infty,T_2}
\\
\label{u.41}
&  \quad
+  |c_1| \frac{1}{f(\tau_0) R(\tau_0)^2}
\Bigr) .
\end{align}
Finally
\begin{align}
\label{u.42}
\left| \int_\tau^{T_2} \int_{\R^2} h |y|^2  \chi_0(\frac{y}{R(s)}) dy ds\right|
& \leq  C f(\tau) R(\tau)^2 (\log \tau)^q  \|h\|_{**} .
\end{align}

From \eqref{u-42}, \eqref{u-32}, \eqref{u-32b}, \eqref{u-eq38}, \eqref{u-43}, \eqref{u.41}, and \eqref{u.42} we get
\begin{align}
\label{xyz2}
|a(\tau) | \log \tau \leq C f(\tau) R(\tau)^2 (\log \tau)^q
\Bigl(
\|h\|_{**}
+ \Bigl\| \frac{a}{R^2 f} \Bigr\|_{\infty,T_2}
+  |c_1| \frac{1}{f(\tau_0) R(\tau_0)^2}
\Bigr) .
\end{align}
Assuming $\tau_0$ large, we deduce that
\begin{align}
\label{u-50}
\Bigl\| \frac{a}{R^2 f} \Bigr\|_{\infty,T_2} \leq \frac{C}{(\log \tau_0)^{1-q}}
\Bigl(
\|h\|_{**}
+ |c_1|  \frac{1}{f(\tau_0) R(\tau_0)^2}
\Bigr).
\end{align}
Note that $a(\tau_0)$ and $c_1$ are related. Indeed,
the initial condition is $\phi_0 = c_1 \tilde Z_0 = \phi_0^\perp + \frac{a(\tau_0)}{2} Z_0$ with
\begin{align*}
a(\tau_0) = \frac{c_1}{8\pi} \int_{\R^2} \tilde Z_0 \Gamma_0,
\end{align*}
by \eqref{formula-a}.
We note that $ \int_{\R^2} \tilde Z_0 \Gamma_0 = 16 \pi + O(\frac{\log\tau_0}{\tau_0})$.
So by \eqref{u-50}
\begin{align*}
|c_1|\leq C |a(\tau_0)| \leq C \frac{f(\tau_0)R(\tau_0)^2 }{(\log\tau_0)^{1-q}}\|h\|_{**}
+  C \frac{1 }{(\log\tau_0)^{1-q}}   |c_1|.
\end{align*}
For $\tau_0$ large, we deduce that
\begin{align}
\label{ineq-c1}
|c_1|\leq C |a(\tau_0)| \leq C \frac{f(\tau_0)R(\tau_0)^2 }{(\log\tau_0)^{1-q}}  \|h\|_{**}.
\end{align}
This proves \eqref{esta-c1-1}.
Replacing this in \eqref{u-50} we get
\begin{align}
\label{ineq-a}
\Bigl\| \frac{a}{R^2 f} \Bigr\|_{\infty,T_2} \leq \frac{C}{(\log \tau_0)^{1-q}}
\|h\|_{**} ,
\end{align}
which proves \eqref{est-a1}.
Combining \eqref{ineq.omega}, \eqref{ineq-c1} and \eqref{ineq-a} we obtain \eqref{est-omega1}.

Finally, we also obtain from \eqref{u-18b-0}
\begin{align}
\label{pointwise-phi}
|\phi(\rho,\tau)|\leq
C
\frac{f(\tau) R(\tau)^2}{(\log \tau_0)^{1-q}}
\frac{1}{1+\rho^4} \|h\|_{**}.
\end{align}
\end{proof}

\begin{proof}[Proof of Lemma~\ref{lemma.apriori2}]
The proof is a slight modification of the one of Lemma~\ref{lemma.apriori}.
Using the same notation as in that proof, integrating \eqref{xyz1} from $\tau$ to $T_2>\tau$ yields
\begin{align}
\nonumber
\int_{\R^2} \phi(T_2) |y|^2  \chi_0\Bigl(\frac{y}{R(T_2)}\Bigr) dy
& -
\int_{\R^2} \phi(\tau) |y|^2  \chi_0\Bigl(\frac{y}{R(\tau)}\Bigr) dy
\\
\nonumber
& = \int_\tau^{T_2} \int_{\R^2} (L[\phi(s)] + h) |y|^2  \chi_0(\frac{y}{R(s)}) dy ds
\\
\nonumber
& \quad
+ \int_\tau^{T_2} \int_{\R^2} B[\phi(s)]  |y|^2  \chi_0(\frac{y}{R}) dy ds
\\
\nonumber
& \quad
- \int_\tau^{T_2} \frac{R'(s)}{R(s)}\int_{\R^2} \phi(s) |y|^2  \nabla \chi_0(\frac{y}{R(s)}) \cdot \frac{y}{R(s)}	dy ds,
\end{align}
Similarly to \eqref{xyz2} we obtain
\begin{align}
\nonumber
|a(\tau) | \log \tau
&\leq C f(\tau) R(\tau)^2 (\log \tau)^q
\Bigl(
\|h\|_{**}
+ \Bigl\| \frac{a}{R^2 f} \Bigr\|_{\infty,T_2}
+  |c_1| \frac{1}{f(\tau_0) R(\tau_0)^2}
\Bigr)
\\
\label{u.51}
& \quad
+ C |a(T_2)| \log(T_2).
\end{align}
The assumption  $\frac{a}{f R^2} \in L^\infty(\tau_0,\infty)$ implies that
\begin{align*}
\lim_{\tau\to\infty} a(\tau) \log\tau=0.
\end{align*}
Letting $T_2\to\infty$ in \eqref{u.51} we obtain
\begin{align*}
|a(\tau) | \log \tau \leq C f(\tau) R(\tau)^2 (\log \tau)^q
\Bigl(
\|h\|_{**}
+ \Bigl\| \frac{a}{R^2 f} \Bigr\|_{L^\infty(\tau_0,\infty) }
+  |c_1| \frac{1}{f(\tau_0) R(\tau_0)^2}
\Bigr) .
\end{align*}
Then the same argument as in Lemma~\ref{lemma.apriori} gives the estimates for $a$, $\omega$ and $c_1$.
\end{proof}

\begin{proof}[Proof of Lemma~\ref{lemma.Z.nondegen}]
Assume to the contrary that there is some $T_2>\tau_0$ such that
\[
a_Z(T_2)=0 .
\]
Then by Lemma~\ref{lemma.apriori} $a_Z(\tau)=0$ for $\tau\in [\tau_0,T_2$]. But by \eqref{formula-a}
\begin{align*}
a(\tau_0) = \frac{1}{8\pi}\int_{\R^2}\Gamma_0 \tilde Z_0 = 2 + O(\frac{\log \tau_0}{\tau_0}) \not=0 ,
\end{align*}
which is a contradiction.
\end{proof}

\begin{proof}[Proof of Lemma~\ref{lemma.existence}.]

We let $T_n $ be a sequence such that $T_n\to \infty$ as $n\to \infty$.
Let $\bar\phi$ be the solution to \eqref{eq-linear-radial12} with initial condition equal to $0$. This solution exists but for the moment we don't have any control of its asymptotic behavior as $\tau\to\infty$.
Let $\bar\phi^\perp $, $\bar a(\tau)$ be the decomposition \eqref{decomp-phi-perp} of $\bar\phi$. Let $Z_B^\perp $, $a_Z(\tau)$ be the decomposition \eqref{decomp-phi-perp} of $Z_B$.
Using Lemma~\ref{lemma.Z.nondegen} there is $c_n \in \R$ such that
\begin{align*}
\bar a(T_n) + c_n a_Z(T_n) = 0.
\end{align*}
Let us define
\begin{align*}
\phi_n = \bar\phi + c_n Z_B,
\end{align*}
and let
\begin{align*}
\phi_n  = \phi_n^\perp  + \frac{a_n}{2}Z_0
\end{align*}
be the decomposition \eqref{decomp-phi-perp} of $\phi_n$.
Then by Lemma~\ref{lemma.apriori} we have
\begin{align}
\nonumber
|a_n(\tau)| &\leq C \frac{f(\tau) R(\tau)^2}{(\log\tau_0)^{1-q}}\|h\|_{**},
\quad \tau \in [\tau_0,T_n]
\\
\nonumber
|\omega_n(\tau)| &\leq C \frac{f(\tau) R(\tau)}{(\log\tau_0)^{1-q}}\|h\|_{**},
\quad \tau \in [\tau_0,T_n]
\\
\nonumber
|c_n| & \leq  C \frac{f(\tau_0)R(\tau_0)^2 }{(\log\tau_0)^{1-q}}  \|h\|_{**}.
\end{align}
Moreover, we also have the uniform estimate
\begin{align}
\nonumber
|\phi_n(\rho,\tau)|\leq
C
\frac{f(\tau) R(\tau)^2}{(\log \tau_0)^{1-q}}
\frac{1}{1+\rho^4} \|h\|_{**}
\end{align}
for $\tau \in [\tau_0,T_n]$ from \eqref{pointwise-phi}.

By using standard parabolic estimates,  passing to a subsequence we may assume that $c_n\to c_1$ and $\phi_n\to \phi$ locally uniformly in space-time, and that $\phi$ is a solution of \eqref{eq-linear-radial20} for some $c_1$ such that
\begin{align*}
|c_1| & \leq  C \frac{f(\tau_0)R(\tau_0)^2 }{(\log\tau_0)^{1-q}}  \|h\|_{**}.
\end{align*}
Moreover $\phi$ satisfies
\begin{align}
\nonumber
|\phi(\rho,\tau)|\leq
C
\frac{f(\tau) R(\tau)^2}{(\log \tau_0)^{1-q}}
\frac{1}{1+\rho^4} \|h\|_{**}
\end{align}
and writing the decomposition \eqref{decomp-phi-perp} as $\phi = \phi^\perp + \frac{a}{2}Z_0$ we have
\begin{align*}
|a(\tau)| &\leq C \frac{f(\tau) R(\tau)^2}{(\log\tau_0)^{1-q}}\|h\|_{**}.
\end{align*}
We also get
\begin{align*}
|\omega_n(\tau)| &\leq C \frac{f(\tau) R(\tau)}{(\log\tau_0)^{1-q}}\|h\|_{**},
\end{align*}
where $\omega$ is defined in \eqref{notation-omega}.

The uniqueness of $c_1$ is a consequence of Lemma~\ref{lemma.apriori2}.
\end{proof}

\begin{proof}[Proof of Proposition~\ref{prop-linear-with-energy}]
We have already constructed $\phi$ and $c_1$  in Lemma~\ref{lemma.existence},
we have the uniqueness of $\phi$ and the estimates for $a$ and $c_1$ in Lemma~\ref{lemma.apriori2}.

We only need to prove the estimate for $\phi^\perp$ stated in \eqref{estPhiPerp}.
By the construction of $\phi$ in Lemma~\ref{lemma.existence} and
\eqref{u-11-0b}, \eqref{ineq-c1} and \eqref{ineq-a}, we get
\begin{align}
\label{u-GperpL2}
\int_{\R^2} (g^\perp)^2 U
\leq
C  f(\tau)^2 R(\tau)^2
\|h\|_{**}^2 , \quad \tau >\tau_0.
\end{align}

We claim that from this inequality we have
\begin{align}
\nonumber
U |g^\perp (y,\tau)|
& \leq
C f(\tau) R(\tau)  \frac{1}{(1+|y|)^{2}} \|h\|_{**}.
 , \quad \tau >\tau_0.
\end{align}
The proof of this estimate is similar to that of \eqref{17-0-b1} in Lemma~\ref{lemma-pointwise2}.

Indeed, we define
\begin{align*}
g_0^\perp = U g^\perp
\end{align*}
and obtain the equation
\begin{align}
\nonumber
\partial_\tau g_0^\perp
&= \nabla \cdot \Bigl( U \nabla\Bigl(\frac{g_0^\perp}{U}\Bigr)\Bigr)
- U ( -\Delta)^{-1} \nabla \cdot \Bigl( U \nabla\Bigl(\frac{g_0^\perp}{U}\Bigr)\Bigr)
\\
\nonumber
& \quad + h - U(-\Delta)^{-1} h
\\
\nonumber
& \quad + B[g_0^\perp]-U(-\Delta)^{-1} B[g_0^\perp] + B[U \psi[g_0^\perp]-U(-\Delta)^{-1}  B[U \psi[g_0^\perp]
\\
\label{eqTildeGPerp}
& \quad + a'(\tau) U + \frac{a}{2}B[Z_0]-\frac{a}{2}U(-\Delta)^{-1} B[Z_0] .
\end{align}
Here the notation $\psi[g_0^\perp]$ is the one introduced in the proof of Lemma~\ref{lemma-pointwise1} in \eqref{eqLiouville}.

To get an estimate for the solution we need an estimate for $a'(\tau)$.
Since $g^\perp = g  + a$ and $\int_{\R^2} U g^\perp =0$ we have
\begin{align*}
a(\tau) =-\frac{1}{8\pi}   \int_{\R^2} U g(\tau) dy=-\frac{1}{8\pi}  \int_{\R^2} g_0(\tau) dy.
\end{align*}
But integrating \eqref{638} we find
\begin{align*}
\partial_\tau   \int_{\R^2} g_0(\tau) dy
&=
-  \int_{\R^2} U  (-\Delta)^{-1} \Bigl(
\nabla \cdot ( U \nabla \frac{g_0}{U}) \Bigr) d y
-\int_{\R^2} U(-\Delta)^{-1} h dy
\\
& \quad
- \int_{\R^2} U (-\Delta)^{-1} ( B [ g_0 + U \psi[g_0]] ) dy,
\end{align*}
which gives the expression
\begin{align*}
a'(\tau) &=
\frac{1}{8\pi} \int_{\R^2} U  (-\Delta)^{-1} \Bigl(
\nabla \cdot ( U \nabla \frac{g_0}{U}) \Bigr) dy
+ \frac{1}{8\pi}\int_{\R^2} U(-\Delta)^{-1} h dy
\\
& \quad
+ \frac{1}{8\pi}\int_{\R^2} U (-\Delta)^{-1} ( B [ g_0 + U \psi[g_0]] ) dy.
\end{align*}

We claim that
\begin{align}
\label{est-aprime}
|a'(\tau) |\leq  C    f(\tau) R(\tau)  \|h\|_{**}.
\end{align}
Indeed, we have
\begin{align*}
 \int_{\R^2} U  (-\Delta)^{-1} \Bigl(
\nabla \cdot ( U \nabla \frac{g_0}{U}) \Bigr) dy
&= \int_{\R^2} \Gamma_0 \nabla \cdot ( U \nabla g^\perp )  dy
=-\int_{\R^2} \nabla U  \cdot  \nabla g^\perp   dy
\\
&= \int_{\R^2} \Delta U g^\perp .
\end{align*}
Then, by \eqref{u-GperpL2}
\begin{align}
\nonumber
\left|
\int_{\R^2} U  (-\Delta)^{-1} \Bigl( \nabla \cdot ( U \nabla \frac{g_0}{U}) \Bigr) dy
\right|
& \leq C \Bigl(
\int_{\R^2} (g^\perp)^2 U \Bigr)^{1/2}
\\
\nonumber
& \leq  C  f(\tau) R(\tau)  \|h\|_{**}.
\end{align}
We also have, for the case of the operator \eqref{B1},
\begin{align*}
\int_{\R^2}
U (-\Delta)^{-1}( B[g_0 ]) \, dy
&= \int_{\R^2} \Gamma_0 B[g_0 ]
=  \zeta(\tau) \int_{\R^2} \Gamma_0 \nabla \cdot (y g_0) dy
\\
& = - \zeta(\tau) \int_{\R^2} \nabla \Gamma_0 \cdot y U g  dy
\end{align*}
But by construction and   \eqref{u-40}, \eqref{ineq-c1} and \eqref{ineq-a}, we get
\begin{align}
\label{u.50}
\Bigl( \int_{\R^2} g^2 U  \Bigr)^{1/2}
\leq  \frac{C}{(\log \tau_0)^{1-q}}
f(\tau) R(\tau)^2 \|h\|_{**}.
\end{align}
so, using \eqref{u.50}
\begin{align*}
\left| \int_{\R^2}
U (-\Delta)^{-1}( B[g_0 ]) \, dy \right|
& \leq \frac{C}{\tau \log \tau} \Bigr( \int_{\R^2} U g^2\Bigl)^{1/2}
\\
& \leq  \frac{C}{\tau \log \tau}  \frac{1}{(\log \tau_0)^{1-q}}
f(\tau) R(\tau)^2 \|h\|_{**}
\\
& \leq C f(\tau)  \|h\|_{**}
\\
& \leq C f(\tau) R(\tau) \|h\|_{**}.
\end{align*}
The last term is estimated similarly and we get \eqref{est-aprime}.

Repeating the argument in of Lemma~\ref{lemma-pointwise1} we obtain from \eqref{u-GperpL2}
\begin{align}
\nonumber
|g_0^\perp  (y,\tau)|
\leq C
f(\tau) R(\tau)\|h\|_{**} \frac{1}{(1+|y|)^{2}} .
\end{align}

An argument similar to Lemma~\ref{lemma-from-g-to-phi} gives
\begin{align*}
|\phi^\perp(\rho,\tau)|\leq
C f(\tau) R(\tau)
\frac{1}{(1+|y|)^2}  \|h\|_{**}.
\end{align*}
\end{proof}

We have an estimate for $\phi^\perp$ stronger than \eqref{estPhiPerp} under a stricter assumption on $\nu$.

\begin{lemma}
\label{lemma-estPhiPerpv2}
Let us assume that $1<\nu<\frac{3}{2}$.
Under the same assumption of Proposition~\ref{prop-linear-with-energy}
let  $\phi = \phi^\perp + \frac{a}{2}Z_0$ be the solution of  \eqref{eq-linear-radial20}.
Then
\begin{align}
\nonumber
|\phi^\perp(y,\tau) |
&\leq
C R(\tau)f(\tau)\|h\|_{**}
\begin{cases}
\frac{1}{(1+|y|)^{2}} & |y|\leq \sqrt {\tau}
\medskip
\\
\frac{\tau }{|y|^4} & |y|\geq \sqrt {\tau},
\end{cases}
\end{align}
\end{lemma}
\begin{proof}
We write \eqref{eqTildeGPerp} as
\begin{align}
\label{eqTildeGPerp2}
\partial_\tau
g_0^\perp
&=
\Delta g_0 ^\perp - \nabla g_0^\perp  \cdot \nabla \Gamma_0
+ 2 U g_0^\perp+ B[g_0^\perp ]  + \tilde h_1
\end{align}
where
\begin{align*}
\tilde h_1 &=
- U (-\Delta)^{-1}( \nabla \cdot ( g_0^\perp \nabla\Gamma_0  ) )
\\
& \quad-U(-\Delta)^{-1} B[g_0^\perp] + B[U \psi[g_0^\perp]]-U(-\Delta)^{-1}  B[U \psi[g_0^\perp]]
\\
& \quad + a'(\tau) U + \frac{a}{2}B[Z_0]-\frac{a}{2}U(-\Delta)^{-1} B[Z_0]
\\
& \quad + h - U (-\Delta)^{-1} h.
\end{align*}
Then, similarly to \eqref{estTildeH1}, we have
\begin{align}
\nonumber
|\tilde h_1(y,\tau)|
\leq
C   f(\tau) R(\tau) \|h\|_{**} \frac{1}{(1+|y|)^4}  .
\end{align}
Let
\begin{align*}
\bar g^\perp(\rho,\tau)
&=
f(\tau) R(\tau)
\tilde g_3(\rho)
\chi_0\Bigl(\frac{\rho}{\delta\sqrt \tau}\Bigr)
+A_1 \frac{f(\tau) R(\tau) }{\tau }
\frac{1}{(1+\rho/\sqrt \tau)^{4}}
+A_2 \frac{f(\tau) R(\tau) }{\tau }e^{-\frac{\rho^2}{4\tau} }
\end{align*}
where $-\Delta_6 \tilde g_3 = \frac{1}{1+\rho^4}$ with $\tilde g_2(\rho)\to0$ as $\rho\to \infty$.
If $\nu<\frac{3}{2}$, for appropriate positive constants $\delta$, $A_1$, $A_2$, and $C$, the function $C \|h\|_{**} \bar g^\perp$  is supersolution to \eqref{eqTildeGPerp2} in $\{ (y,\tau) | \tau > \tau_0, \ |y|> R_0 \}$.
We deduce that
\begin{align}
\nonumber
|g_0^\perp (y,\tau)|
\leq
C   f(\tau) R(\tau) \|h\|_{**}
\frac{\min(1,\frac{\tau}{|y|^2})}{(1+|y|)^{2}} .
\end{align}

An argument similar to Lemma~\ref{lemma-from-g-to-phi} gives
\begin{align*}
|\phi^\perp(\rho,\tau)|\leq
C f(\tau) R(\tau) \|h\|_{**}
\begin{cases}
\frac{1}{1+|y|^2} & |y |\leq \sqrt\tau\\
\frac{\tau}{|y|^4}& |y |\geq  \sqrt\tau .
\end{cases}
\end{align*}
\end{proof}

\begin{proof}[Proof of Proposition~\ref{prop-linear-without-second-moment}]
\label{proof-linear-without-second-moment}
By Proposition~\ref{prop-linear-with-energy} there is $c_1$ such that the solution $\phi$ to
\eqref{eq-linear-radial20} has the properties stated in  Proposition~\ref{prop-linear-with-energy}.
We recall that by \eqref{pointwise-phi} $\phi$ satisfies
\begin{align}
\label{est-phix}
|\phi(\rho,\tau)|\leq
C
\frac{f(\tau) R(\tau)^2}{(\log \tau_0)^{1-q}}
\frac{1}{1+\rho^4} \|h\|_{**} .
\end{align}

We will construct a barrier to estimate $\phi$ for $|y|\geq R_0$, where $R_0$ is a large constant.
We consider the equation \eqref{eq-linear-radial20} in $\R^2 \setminus B_{R_0}(0)$ written in the form
\begin{align}
\label{eqLinear-000}
\partial_\tau \phi = \Delta \phi - 4 \nabla \Gamma_0 \nabla \phi  + 2 U \phi + B[\phi] + \bar h ,
\end{align}
where
\[
\bar h = - \nabla U \nabla \psi + h.
\]

Since $\psi = ( -\Delta)^{-1}\phi$, from \eqref{est-phix} we get
\begin{align}
\nonumber
| \nabla \psi(\rho,\tau)|\leq
C
\frac{f(\tau) R(\tau)^2}{(\log \tau_0)^{1-q}}
\frac{1}{1+\rho^3} \|h\|_{**} .
\end{align}
This gives
\begin{align}
\label{estGradPsi2}
|\nabla U \cdot \nabla \psi |
\leq
C  \frac{f(\tau) R(\tau)^2}{(\log \tau_0)^{1-q}}
\frac{1}{1+\rho^8} \|h\|_{**} .
\end{align}
By \eqref{estGradPsi2} and the definition of the norm $ \|h\|_{**}$,
\[
|\bar h(y,\tau) |\leq
C \frac{f(\tau) R(\tau)^2}{(\log\tau_0)^{1-q}} \frac{1}{(1+\rho)^{6+ \sigma}} \min\Bigl( 1 , \frac{\tau^{\epsilon/2}}{\rho^\epsilon}\Bigr)\|h\|_{**} ,
\]
where we have used that $\sigma+\epsilon<2$.
Let $\tilde g_2$ be defined by \eqref{def-tildeg2} and let
\begin{align*}
\bar\phi (\rho,\tau)
&=
f(\tau) R(\tau)^2
\tilde g_2(\rho)
\chi_0\Bigl(\frac{\rho}{\delta\sqrt \tau}\Bigr)
+A_1
\frac{f(\tau) R(\tau)^2 }{\tau^2 }
\frac{1}{(1+\rho/\sqrt \tau)^{6+\sigma+\epsilon}}
\\
& \quad
+A_2
\frac{f(\tau) R(\tau)^2 }{\tau^2 }e^{-\frac{\rho^2}{4\tau} } .
\end{align*}
Then for suitable positive constants $\delta$, $A_1$, $A_2$, and $C$, the function $C (\log\tau_0)^{q-1} \|h\|_{**} \bar \phi $ is a supersolution to \eqref{eqLinear-000} in $\{ (y,\tau) | \tau > \tau_0, \ |y|> R_0 \}$.
For this we need $\nu<2$.
Moreover
$|\phi(\rho,\tau)| \leq C \bar \phi(\rho,\tau) (\log\tau_0)^{q-1} \|h\|_{**}$ at $\rho = R_0$ by \eqref{est-phix}.
By the maximum principle
\begin{align*}
|\phi(y,\tau) |\leq C \bar \phi(y,\tau) (\log\tau_0)^{q-1} \|h\|_{**}, \quad |y|>R_0.
\end{align*}
This gives the explicit bound
\begin{align*}
|\phi(\rho,\tau) | \leq C \frac{f(\tau) R(\tau)^2}{(\log\tau_0)^{1-q}}
\frac{1}{(1+\rho^4) } \min\Bigl( 1 , \frac{\tau^{1/2}}{\rho}\Bigr)^{2+\sigma+\epsilon}
\|h\|_{**}
\end{align*}
\end{proof}

We include here some results  that will be useful later.
Let
\begin{align}
\nonumber
\hat Z_0 = L[ \tilde Z_0].
\end{align}

\begin{lemma}
\label{lemma-hatZ0}
The function $\hat Z_0$ satisfies
\begin{align}
\label{est-hZ0}
|\hat Z_0(\rho)| \leq C \frac{1}{\tau_0 ( 1 + \rho)^4}
\end{align}
and is supported on $\rho \leq 2 \tau_0$.
\end{lemma}
\begin{proof}
Let $ \psi = (-\Delta)^{-1} \tilde Z_0$ and  $g = \frac{\tilde Z_0}{U} - \psi$. By \eqref{def-tilde-Z0} and using that $Z_0 = U z_0$, $z_0$ defined in \eqref{defZLiouville},
\begin{align*}
g &= \frac{(Z_0-m_{Z_0} U ) \chi}{U} - \psi
= z_0 \chi - m_{Z_0} \chi - \psi ,
\end{align*}
where $\chi(\rho) = \chi_0(\frac{\rho}{\sqrt{\tau_0}})$.
Note that $\tilde Z_0$ has mass zero and support in $B_{2\sqrt{\tau_0}}$. It follows that $\psi$ has also support contained in $B_{2\sqrt{\tau_0}}$ and then $g$ has support contained in $B_{2\sqrt{\tau_0}}$.
Therefore $\hat Z_0 = L[\tilde Z_0] = \nabla\cdot( U \nabla g)$ has also support contained in $B_{2\sqrt{\tau_0}}$.

To get an estimate for $\hat Z_0$ let us write
\begin{align*}
\psi &= (-\Delta)^{-1} (Z_0-m_{Z_0} U ) \chi) = (-\Delta)^{-1} Z_0 + \psi_1,
\end{align*}
where
\begin{align*}
\psi_1 =  (-\Delta)^{-1} (Z_0(\chi-1)-m_{Z_0} U  \chi).
\end{align*}
Since $\Delta z_0 + U z_0 = 0$ and $\lim_{\rho\to\infty} z_0(\rho) = -2$ we have
$(-\Delta)^{-1} Z_0 = z_0+2$. So
\begin{align*}
\psi = z_0 + 2 + \psi_1
\end{align*}
Hence
\begin{align*}
g = z_0 ( \chi -1 ) - 2 - m_{Z_0} \chi - \psi_1
\end{align*}
and so
\begin{align}
\nonumber
\hat Z_0 &= L[\tilde Z_0] = \nabla\cdot( U \nabla g)
\\
\label{formula-hZ0}
&= \nabla\cdot( U ( \nabla z_0(\chi-1) + z_0 \nabla \chi - m_{Z_0} \nabla \chi - \nabla \psi_1)).
\end{align}
Using radial symmetry and $m_{Z_0} = O ( \frac{1}{\tau_0})$ we get
\begin{align*}
|\nabla \psi_1(\rho) | \leq C \frac{1}{\tau_0 (1+\rho)}.
\end{align*}
From this and \eqref{formula-hZ0} we get \eqref{est-hZ0}.
\end{proof}

Consider the initial value problem
\begin{align}
\label{linear-004}
\left\{
\begin{aligned}
\partial_\tau \phi_1 &= L[\phi_1] + B[\phi_1]  \quad \text{in }\R^2 \times (\tau_0,\infty),
\\
\phi_1(\cdot,\tau_0) &= \hat Z_0 \quad \text{in }\R^2 .
\end{aligned}
\right.
\end{align}

\begin{lemma}
\label{lemma-est-phi1}
Let $0<\gamma<2$.
Let $1<\nu_0<\frac{7}{4}$
\begin{align}
\nonumber
f_0(\tau) = \frac{1}{\tau^{\nu_0}} .
\end{align}
and let $R(\tau) $ be as in \eqref{def-R}.
Then the solution $\phi_1$ of \eqref{linear-004} satisfies
\begin{align}
\nonumber
| \phi_1(\rho,\tau) | \leq C  \frac{f_0(\tau) R(\tau)^2}{\tau_0 f_0(\tau_0) R(\tau_0)^2}
\frac{1}{(1+\rho^4) } \min\Bigl( 1 , \frac{\tau^{1/2}}{\rho}\Bigr)^{2+\gamma} .
\end{align}
\end{lemma}
\begin{proof}
A suitable modification in the proof of Proposition~\ref{prop-linear-with-energy} gives the following result. Consider
\begin{align}
\label{linear-005}
\left\{
\begin{aligned}
\partial_\tau \phi &= L[\phi] + B[\phi]  \quad \text{in }\R^2 \times (\tau_0,\infty),
\\
\phi(\cdot,t_0) &= \phi_0 + c_1 \tilde Z_0 \quad \text{in }\R^2 ,
\end{aligned}
\right.
\end{align}

Then there is $C>0$ such that for any $\tau_0$ sufficiently large the following holds. Suppose that $\phi_0$ is a radial function with zero mass in $\R^2$, supported in $B_{2\sqrt{\tau_0}}(0)$, and such that
\begin{align*}
|\phi_0(\rho)| \leq M \frac{1}{1+\rho^4}.
\end{align*}
Then there exists $c_1$ such that the solution $\phi$ of \eqref{linear-005} satisfies
\begin{align*}
|\phi(\rho,\tau) | \leq C M \frac{f_0(\tau) R(\tau)^2}{f_0(\tau_0) R(\tau_0)^2}
\frac{1}{(1+\rho^4) } \min\Bigl( 1 , \frac{\tau^{1/2}}{\rho}\Bigr)^{2+\gamma} .
\end{align*}
Moreover $c_1$ is a linear function of $\phi_0$ and satisfies
\begin{align*}
|c_1| & \leq  C M \frac{1 }{(\log\tau_0)^{1-q}}  .
\end{align*}

Let us apply this statement to $\phi_0 = L[\tilde Z_0]$, which is radial, with mass zero, support in $B_{2\sqrt{\tau_0}}(0)$, and satisfies
\begin{align*}
|\phi_0(\rho)|\leq \frac{1}{\tau_0} \frac{1}{1+\rho^4} ,
\end{align*}
by Lemma~\ref{lemma-hatZ0}.
Then there exists $c_1$ such that the solution $\tilde \phi$ to \eqref{linear-005} with $\phi_0 = L[\tilde Z_0]$ satisfies
\begin{align}
\label{bound-phi2}
|\tilde\phi(\rho,\tau) | \leq C  \frac{f_0(\tau) R(\tau)^2}{\tau_0 f_0(\tau_0) R(\tau_0)^2}
\frac{1}{(1+\rho^4) } \min\Bigl( 1 , \frac{\tau^{1/2}}{\rho}\Bigr)^{2+\gamma} .
\end{align}

We claim that $c_1=0$. To prove this, we multiply \eqref{linear-005} by $|y|^2$ and integrate on $\R^2 \times (\tau_0,\infty)$.
Let's work with
\begin{align*}
B[\phi] = \zeta(\tau) \nabla \cdot (y \phi) .
\end{align*}
The case of the operator \eqref{B2} is similar.
Then we get
\begin{align*}
\partial_\tau \int_{\R^2} \tilde\phi(y,\tau)|y|^2d y
= -2 \zeta(\tau)\int_{\R^2} \tilde\phi(y,\tau)|y|^2d y  ,
\end{align*}
because $\int_{\R^2} L[\phi]|y|^2dy=0$, see Remark~\ref{remark-secondm-Lphi}.
Integrating
\begin{align*}
\int_{\R^2} \tilde\phi(y,\tau)|y|^2d y
& = e^{-2\int_{\tau_0}^\tau \zeta }
\int_{\R^2} \tilde\phi(y,\tau_0)|y|^2d y
= c_1 e^{-2\int_{\tau_0}^\tau \zeta }
\int_{\R^2} \tilde Z_0 (y)|y|^2 dy,
\end{align*}
because $\int_{\R^2} L[\tilde Z_0] |y|^2 dy = 0$.
Using the asymptotic expansion of $\zeta$ one gets
\begin{align*}
e^{-2\int_{\tau_0}^\tau \zeta }\to \infty, \quad \text{as }\tau\to \infty .
\end{align*}
But the bound \eqref{bound-phi2} implies that
\begin{align*}
\lim_{\tau\to\infty} \int_{\R^2} \tilde\phi(y,\tau)|y|^2d y = 0.
\end{align*}
This only can happen if $c_1=0$.

We deduce that $\phi_1$ defined in \eqref{linear-004} coincides with $\tilde \phi$, and then \eqref{bound-phi2} holds for $\phi_1$.
\end{proof}


\section{Linear estimate with second moment (radial)}
\label{sect-theorem-linear-with-second-moment}

We will prove in this section Proposition~\ref{prop-linear-with-second-moment} in the radial case $h(\rho,\tau)$.


\begin{prop}
\label{prop-radial-linear-with-second-moment}
Let $0<\sigma<1$, $\epsilon>0$ with $\sigma+\epsilon<2$ and  $1<\nu< \min( 1+\frac{\epsilon}{2},3-\frac{\sigma}{2}, \frac{5}{4})$.
Let $0<q<1$.
Then there is $C$ such that for \ch{$\tau_0$} large the following holds.
Suppose that  $h$ satisfies  $\|h\|_{\nu,m,6+\sigma,\epsilon}<\infty$ and
\begin{align}
\nonumber
\int_{\R^2} h(y,\tau)dy&=0 , \quad
\int_{\R^2} h(y,\tau)|y|^2dy=0.
\end{align}
Then the solution $\phi(y,\tau)$ of problem  \ch{\eqref{linear-002}}
satisfies
\begin{align}
\nonumber
\|  \phi \|_{\nu-\frac 12 ,m+\frac {q}{2},4,2+\sigma+\epsilon} \leq C \|h\|_{\nu,m,6+\sigma,\epsilon}.
\end{align}
\end{prop}

To describe the idea of the proof more easily let us consider for a moment the equation \eqref{linear-002} without $B$:
\begin{align}
\label{linear1-noB}
\left\{
\begin{aligned}
\partial_\tau \phi &= L[\phi] +  h(y,t)  \quad \text{in }\R^2 \times (\tau_0,\infty)
\\
\phi(\cdot,\tau_0) &= 0 \quad \text{in }\R^2 ,
\end{aligned}
\right.
\end{align}
The idea is to formally apply a suitable left inverse $L^{-1}$ of $L$ to \eqref{linear1-noB} (to be defined later on in Lemma~\ref{lemmaLinv}).
If we call $\Phi = L^{-1} \phi$, $H = L^{-1}h$, then we would like to solve
\begin{align}
\label{linear-Phi-noB}
\left\{
\begin{aligned}
\partial_\tau \Phi &= L[\Phi] +  H(y,t)  \quad \text{in }\R^2 \times (\tau_0,\infty)
\\
\Phi(\cdot,\tau_0) &= 0 \quad \text{in }\R^2 .
\end{aligned}
\right.
\end{align}
In order to get good properties of $H$, in this step we have already used that $h$ satisfies the second moment condition. At this point we would like to apply Proposition~\ref{prop-linear-with-energy}, which gives a decomposition
\begin{align*}
\Phi = \Phi^\perp + \frac{a(\tau)}{2}Z_0.
\end{align*}
Note that $\Phi^\perp$ decays in time like $1/\tau^{\nu-1/2}$ and so $\phi = L \Phi$ also decays in time like $1/\tau^{\nu-1/2}$, which is better than the estimate provided by Proposition~\ref{prop-linear-without-second-moment}.
It turns out that $H$ decays in space like $1/\rho^{4+\sigma}$ so we can't apply directly Proposition~\ref{prop-linear-with-energy} to \eqref{linear-Phi-noB}. What we do is \emph{concentrate} $H$ by solving first a nicer problem. We write $\Phi = \Phi_1 + \Phi_2$ where $\Phi_1$ is asked to solve
\begin{align}
\nonumber
\left\{
\begin{aligned}
\partial_\tau \Phi_1 &= L_0[\Phi_1] +  H(y,t)  \quad \text{in }\R^2 \times (\tau_0,\infty)
\\
\Phi_1(\cdot,\tau_0) &= 0 \quad \text{in }\R^2 .
\end{aligned}
\right.
\end{align}
where
\begin{align}
\label{def-L0}
L_0[\phi] = \nabla \cdot \Bigl( U \nabla \Bigl(\frac{\phi}{U} \Bigr) \Bigr)
= \Delta \phi - \nabla \phi \cdot \nabla \Gamma_0 + U \phi  .
\end{align}
Lemma~\ref{lemma-error-concentration} below deals with $\Phi_1$.
Then the problem for $\Phi_2$ becomes
\begin{align}
\nonumber
\left\{
\begin{aligned}
\partial_\tau \Phi_2 &= L_0[\Phi_2] +  L[\Phi_1]-L_0[\Phi_1]  \quad \text{in }\R^2 \times (\tau_0,\infty)
\\
\Phi_1(\cdot,\tau_0) &= 0 \quad \text{in }\R^2 .
\end{aligned}
\right.
\end{align}
It turns out that the right hand side in this equation has better spatial decay and we can apply Proposition~\ref{prop-linear-with-energy}.

\medskip

In the next lemmas we give some preliminary results, and the proof of Proposition~\ref{prop-radial-linear-with-second-moment} is given at the end of this section.

We define the inverse of $L$ that we use.
For $h:\R^2\to \R$ define $\|h\|_{\tau,6+\sigma,\epsilon}$
\ch{as the smallest $K$ such that}
\begin{align}
\nonumber
|h(y)|  \leq
\frac{  K  }{(1+|y|)^{6+\sigma}}
\begin{cases}
1
& |y|\leq \sqrt{\tau}
\\
\frac{\tau^{\epsilon/2}}{|y|^\epsilon}
& |y|\geq \sqrt{\tau}  .
\end{cases}
\end{align}
which depends on $\tau$, treated as parameter here, $\sigma$, and $\varepsilon$.

\begin{lemma}
\label{lemmaLinv}
Let $\sigma,\epsilon>0$.
Let $h=h(\rho)$ be radial and satisfy
$\|h\|_{\tau,6+\sigma,\epsilon}<\infty$ and
\begin{align*}
\int_{\R^2} h dy  = \int_{\R^2} h|y|^2 dy = 0.
\end{align*}
Then there exists $H$ radially symmetric such that $L[H] = h$ in $\R^2$ and satisfies
\begin{align}
\label{bound-H2}
\| H \|_{\tau,4+\sigma,\epsilon}
\leq
C
\|h\|_{\tau,6+\sigma,\epsilon}
\end{align}
Moreover, $H$ defines a linear operator of $h$ and satisfies
\begin{align}
\label{mass0H}
\int_{\R^2} H d y = 0.
\end{align}
\end{lemma}
\begin{proof}
Write the equation $L[H] = h$ as
\[
\nabla \cdot ( U \nabla g ) = h
\]
where $g = \frac{H}{U}- (-\Delta)^{-1} H$.
We choose $g$ as
\[
g(\rho) = -\int_\rho^\infty \frac{1}{r U(r)} \int_0^r h(s) s ds dr.
\]
Using that $\int_{\R^2} h = 0$ we check that
\[
|g(\rho) | \leq
C\|h\|_{\tau,6+\sigma,\epsilon}
\begin{cases}
\frac{1}{(1+\rho)^\sigma}  & \rho \leq \sqrt \tau
\\
\frac{ \tau^{\epsilon/2}}{\rho^{\sigma+\epsilon}}  & \rho \geq \sqrt \tau
\end{cases}
\]
Now we solve Liouville's equation
\[
-\Delta \psi - U \psi = U g \quad \text{in }\R^2, \quad\psi(\rho)\to 0 \quad \text{as }\rho\to\infty,
\]
Since $\int_{\R^2} h |y|^2 dy=0$ we check that
\[
\int_{\R^2} \ch{ g Z_0} dy = 0.
\]
Then we can use the variations of parameter formula, and get
\[
|\psi(\rho)|\leq
C\|h\|_{\tau,6+\sigma,\epsilon}
\begin{cases}
\frac{1}{(1+\rho)^{2+\sigma}} & \rho \leq \sqrt \tau
\\
\frac{\tau^{\epsilon/2}}{\rho^{2+\sigma+\varepsilon}} & \rho \geq \sqrt \tau
\end{cases}
\]
Then define $H = U ( g + \psi)$, which is the desired solution, and note that it satisfies \eqref{bound-H2}.
Property \eqref{mass0H} follows from $H = - \Delta \psi$ and the decay of $\psi$.
\end{proof}


To take into account the operator $B$ we
define
\begin{align*}
\Lambda [ \phi ] = y \cdot \nabla \phi,
\end{align*}
and compute
\begin{align}
\label{commutator1}
\Lambda \circ L [\Phi] - L \circ \Lambda [\Phi]
&=
\nabla \cdot ( \Phi \ch{U} y )
-2 L[\Phi]
- \nabla\cdot( (y \cdot\nabla \ch{U}+  2 \ch{U}) \nabla (-\Delta)^{-1}\Phi).
\end{align}
Indeed, write $\Psi = (-\Delta)^{-1} \Phi$. Then
\begin{align}
\label{opL}
L \Phi = \Delta \Phi - \nabla \Gamma_0 \cdot \nabla \Phi - \nabla U \cdot \nabla \Psi + 2 U \Phi.
\end{align}
By direct computation
\begin{align}
\label{LDelta}
\Lambda \Delta \Phi
&= \Delta \Lambda \Phi - 2 \Delta \Phi
\\
\label{LambdaNGNPhi}
\Lambda(  \nabla \Gamma_0 \cdot \nabla \Phi  )
&=
\nabla (\Lambda \Gamma_0) \cdot \nabla \Phi +  \nabla \Gamma_0 \cdot \nabla ( \Lambda \Phi ) - 2   \nabla \Gamma_0 \cdot \nabla \Phi
\\
\label{LambdaNUNPsi}
\Lambda(  \nabla U \cdot \nabla \Psi )
&=
\nabla (\Lambda U) \cdot \nabla \Psi +  \nabla U \cdot \nabla ( \Lambda \Psi ) - 2   \nabla U \cdot \nabla \Psi .
\end{align}
But $-\Delta \Psi = \Phi$ and therefore
\begin{align*}
- \Delta( \Lambda \Psi) + 2 \Delta \Psi = \Lambda \Phi.
\end{align*}
Applying $(-\Delta)^{-1}$ gives
\begin{align*}
\Lambda \Psi = ( -\Delta)^{-1}( \Lambda \Phi) + 2 \Psi.
\end{align*}
Substituting this into \eqref{LambdaNUNPsi} we obtain
\begin{align}
\nonumber
\Lambda(  \nabla U \cdot \nabla \Psi )
&=
\nabla (\Lambda U) \cdot \nabla \Psi +  \nabla U \cdot \nabla [ ( -\Delta)^{-1}( \Lambda \Phi) + 2 \Psi] - 2   \nabla U \cdot \nabla \Psi
\\
\label{LambdaNUNPsi2}
&=
\nabla (\Lambda U) \cdot \nabla \Psi +  \nabla U \cdot \nabla [ ( -\Delta)^{-1}( \Lambda \Phi) ].
\end{align}
Combining \eqref{opL}, \eqref{LDelta}, \eqref{LambdaNGNPhi}, \eqref{LambdaNUNPsi2} we find that
\begin{align*}
\Lambda L \Phi = L \Lambda \Phi - 2 L \Phi + 4 U \Phi - 2\nabla U  \cdot\nabla \Psi - \nabla (\Lambda \Gamma_0)\cdot \nabla \Phi - \nabla (\Lambda U) \cdot \nabla \Psi + 2 \Lambda(U) \Phi.
\end{align*}
But
\begin{align*}
- 2\nabla U  \cdot\nabla \Psi- \nabla (\Lambda U) \cdot \nabla \Psi
&= - \nabla Z_0 \cdot \nabla \Psi
\\
&= - \nabla \cdot (Z_0 \nabla \Psi) - Z_0 \Phi ,
\end{align*}
so that
\begin{align*}
\Lambda L \Phi = L \Lambda \Phi - 2 L \Phi + 4 U \Phi  - \nabla (\Lambda \Gamma_0)\cdot \nabla \Phi  + 2 \Lambda(U) \Phi  - \nabla \cdot (Z_0 \nabla \Psi) - Z_0 \Phi .
\end{align*}
Using that
\begin{align*}
2 \Lambda(U) \Phi - Z_0 \Phi
&= -2 U \Phi + \Lambda(U) \Phi
\end{align*}
we then obtain
\begin{align*}
\Lambda L \Phi = L \Lambda \Phi - 2 L \Phi + 2 U \Phi  - \nabla (\Lambda \Gamma_0)\cdot \nabla \Phi  +  \Lambda(U) \Phi  - \nabla \cdot (Z_0 \nabla \Psi)
\end{align*}
Let's consider the terms $2 U \Phi  - \nabla (\Lambda \Gamma_0)\cdot \nabla \Phi  +  \Lambda(U) \Phi$.
Noting that $ \nabla (\Lambda \Gamma_0) = \nabla ( y \cdot \nabla \Gamma+2) = \nabla z_0$ and that $Z_0 = 2 U + \Lambda(U)$,
we can write
\begin{align*}
2 U \Phi  - \nabla (\Lambda \Gamma_0)\cdot \nabla \Phi  +  \Lambda(U) \Phi
&= 2 U \Phi  - \nabla z_0 \cdot \nabla \Phi  +  \Lambda(U) \Phi
\\
&= Z_0  \Phi  - \nabla \cdot ( \nabla z_0  \Phi )  + \Delta z_0 \Phi .
\end{align*}
But $\Delta z_0 + Z_0=0$, so
\begin{align*}
\Lambda L \Phi = L \Lambda \Phi - 2 L \Phi  - \nabla \cdot ( \nabla z_0  \Phi )  - \nabla \cdot (Z_0 \nabla \Psi) .
\end{align*}
We can again write $\nabla z_0 = \nabla ( y \cdot \nabla \Gamma_0)$ and using the radial symmetry of the functions $\Gamma_0$, $z_0$ and the notation $\rho = |y|$
\begin{align*}
\nabla z_0 &= \frac{y}{\rho} \partial_\rho z_0
= \frac{y}{\rho} \partial_\rho ( \rho \partial_\rho \Gamma_0)
= y \Delta \Gamma_0 = - y U.
\end{align*}
Then
\begin{align*}
\Lambda L \Phi = L \Lambda \Phi - 2 L \Phi  + \nabla \cdot ( y U  \Phi )  - \nabla \cdot (Z_0 \nabla \Psi) .
\end{align*}
This proves \eqref{commutator1}.

Formula \eqref{commutator1} leads us to consider the following equation for $\Phi = L^{-1}[\phi]$:
\begin{align}
\label{H4plusNO0}
\left\{
\begin{aligned}
\partial_\tau \Phi
&= L [\Phi ]
+ \tilde B[\Phi]
+ \zeta_1(\tau) A[\Phi]+ H
\quad \text{in }\R^2 \times \ch{(}\tau_0,\infty)
\\
\Phi(\cdot,\tau_0)&=0
\end{aligned}
\right.
\end{align}
where
\begin{align}
\nonumber
A[\Phi] =
L^{-1}[ \nabla \cdot ( \Phi \ch{U} y )  - \nabla\cdot( Z_0 \nabla (-\Delta)^{-1}\Phi) ]  ,
\end{align}
$\ch{Z_0(y)=  2 U(y) + y \cdot \nabla_y U(y)}$,
and $\tilde B$ has the same form as $B$:
\[
\tilde B[\Phi] = \tilde \zeta_1(\tau) y \cdot \nabla \Phi + \tilde \zeta_2(\tau) \Phi
\]
with $\tilde \zeta_1(\tau)$, $\tilde \zeta_2(\tau)$ satisfying
\begin{align}
\label{tildeZetai}
|\tilde \zeta_i(\tau)| \le   \frac C {\tau \log \tau  }\quad\text{for all } \tau>\tau_0.
\end{align}
and $\zeta_1$ satisfies the same restriction, that is, \eqref{zetai}.

The next lemma allows us to reduce to an equation like \eqref{H4plusNO0} but with a right hand side with more spatial decay.

\begin{lemma}
\label{lemma-error-concentration}
Let $\sigma>0$, $\epsilon>0$ and  $1<\nu< \min( 1+\frac{\epsilon}{2},3-\frac{\sigma}{2})$.
Let $H(y,\tau)$ be radial in $y$ and satisfy
\begin{align}
\label{H1}
\int_{\R^2} H(\cdot,\tau)=0
\end{align}
and $\|H\|_{\nu,m,4+\sigma,\epsilon}<\infty$.
Then there exists $H_1$ and $\Phi_1$ such that
\begin{align}
\nonumber
\left\{
\begin{aligned}
\partial_\tau \Phi_1 &=
L [\Phi_1]+ \tilde B[\Phi_1] + H - H_1 , \quad \text{in } \R^2 \times \ch{(}\tau_0,\infty)
\\
\Phi_1(\cdot,\tau_0) &= 0\quad \text{in }\R^2 .
\end{aligned}
\right.
\end{align}
Moreover $\Phi_1$ and $H_1$ are linear operators of $H$ and satisfy
\begin{align}
\label{Phi1-sigma-positive}
|\Phi_1(\rho,\tau)|
&\leq  \frac{C}{\tau^\nu (\log \tau)^m }
\|H\|_{\nu,m,4+\sigma,\epsilon}
\begin{cases}
\frac{1}{ (1+\rho)^{2+\sigma}}  & \rho \leq \sqrt \tau
\\
\frac{\tau^{1+\epsilon/2}}{ (1+\rho)^{4+\sigma+\epsilon} }  &  \rho \geq \sqrt \tau .
\end{cases}
\\
\label{H1-sigma-positive}
|H_1(\rho,\tau)|
&\leq  \frac{C}{\tau^\nu (\log \tau)^m }
\|H\|_{\nu,m,4+\sigma,\epsilon}
\begin{cases}
\frac{1}{ (1+\rho)^{6+\sigma}}  & \rho \leq \sqrt \tau
\\
\frac{\tau^{\epsilon/2}}{ (1+\rho)^{6+\sigma+\epsilon} }  &  \rho \geq \sqrt \tau .
\end{cases},
\end{align}
\begin{align}
\label{Phi1zeroMass}
\int_{\R^2}  \Phi_1 dy &=0
\\
\label{H1-zero-mass}
\int_{\R^2} H_1(\cdot,\tau)&=0.
\end{align}

\end{lemma}

\begin{proof}

Write the operator $L$ as
\[
L[\phi] = L_0[\phi] - \nabla \cdot ( U \nabla (-\Delta)^{-1}  \phi )
\]
where $L_0$ is defined in \eqref{def-L0}.
Consider the problem
\begin{align}
\nonumber
\left\{
\begin{aligned}
\partial_\tau \Phi_1 &=
L_0 [\Phi_1] +\tilde B[\Phi_1] +H  , \quad \text{in }\R^2 \times \ch{(}\tau_0,\infty),
\\
\Phi_1(\cdot,\tau_0) &= 0\quad \text{in }\R^2  .
\end{aligned}
\right.
\end{align}

The idea is to formally apply $L_0^{-1}$ to this equation.
Similarly to the proof of \eqref{commutator1} we compute

\begin{align*}
\Lambda \circ L_0 [\Phi] - L_0 \circ \Lambda[\Phi]
&=
\nabla \cdot ( \Phi U y )
- 2 L_0 [ \Phi ].
\end{align*}

This leads us to consider the problem
\begin{align}
\label{eq-linear-radial6}
\left\{
\begin{aligned}
\partial_\tau \tilde \Phi &=
L_0 [\tilde \Phi ]
+ B_1[\tilde \Phi]
+\tilde H , \quad \text{in }\R^2 \times \ch{(}\tau_0,\infty),
\\
\tilde \Phi(\cdot,\tau_0) &= 0\quad \text{in }\R^2 ,
\end{aligned}
\right.
\end{align}
where $\tilde H$ is a radial function satisfying
\[
L_0[\tilde H] = H \quad \text{in }\R^2
\]
and
\[
B_1[\tilde \Phi] = \hat \zeta_1(\tau) y \cdot \nabla \tilde \Phi
+\hat \zeta_2(\tau)  \tilde \Phi
\]
with
\begin{align}
\label{hatZetai}
\hat \zeta_1(\tau)  = \tilde \zeta_1(\tau) = O\Bigl( \frac{1}{\tau \log \tau}\Bigr) , \quad
\hat \zeta_2(\tau) = \tilde \zeta_2(\tau)-2\tilde \zeta_1 (\tau) = O\Bigl( \frac{1}{\tau \log \tau}\Bigr) ,
\end{align}
by \eqref{tildeZetai}.

We claim that there is a choice of $\tilde H$, which defines a linear operator of $H$, and satisfies
\begin{align}
\label{bound-tildeH}
|\tilde H|
+ (1+\rho)|\nabla \tilde H|
\leq
C\frac{1}{\tau^\nu (\log \tau)^m}
\|H\|_{\nu,m,4+\sigma,\epsilon}
\begin{cases}
\frac{1}{(1+\rho)^{2+\sigma}}   & \rho \leq \sqrt \tau
\\
\frac{\tau^{\epsilon/2}}{(1+\rho)^{2+\sigma+\epsilon}}    & \rho \geq \sqrt \tau.
\end{cases}
\end{align}
Indeed, the equation $L_0[\tilde H] = H$ for radial functions has the form
\[
\partial_\rho \Bigl( \rho U \partial_\rho \Bigl( \frac{\tilde H}{U} \Bigr) \Bigr)  = \rho H .
\]
We select the solution
\begin{align*}
\tilde H(\rho,\tau)
=  U(\rho) \int_0^\rho \frac{1}{r U(r)} \int_0^r
H(s,\tau) s ds dr .
\end{align*}
Using \eqref{H1} we get \eqref{bound-tildeH}.

Instead of \eqref{eq-linear-radial6} we consider
\begin{align}
\label{eq-linear-radial7}
\left\{
\begin{aligned}
\partial_\tau \tilde \Phi_1 &=
\Delta_{\R^2} \tilde \Phi_1
-  \nabla \Gamma_0 \cdot \nabla \tilde \Phi_1
+ B_1[\tilde \Phi_1]
+\tilde H , \quad \text{in }\R^2 \times \ch{(}\tau_0,\infty),
\\
\tilde \Phi_1(\cdot,\tau_0) &= 0\quad \text{in }\R^2 ,
\end{aligned}
\right.
\end{align}

We then have the following estimate for $\tilde \Phi_1$:
\begin{align}
\label{decayPhi1}
|\tilde\Phi_1|\leq  \frac{C}{\tau^\nu (\log \tau)^m }
\|H\|_{\nu,m,4+\sigma,\epsilon}
\begin{cases}
\frac{1}{ (1+\rho)^\sigma}  & \rho \leq \sqrt \tau
\\
\frac{\tau^{1+\epsilon/2}}{ (1+\rho)^{2+\sigma+\epsilon} }  &  \rho \geq \sqrt \tau .
\end{cases}
\end{align}
For the proof of this we construct a barrier. First we find a solution to
\begin{align*}
\Delta_{\R^2}  \phi_1  - \nabla \Gamma_0\cdot \nabla \phi_1 + \frac{1}{(1+\rho)^{2+\sigma}} &= 0 \quad \text{in } \R^2,
\\
\phi_1  (\rho) &\to 0 \quad \text{as } \rho \to \infty.
\end{align*}
The equation may be integrated explicitly, noting that
\[
\Delta_{\R^2} \phi - \nabla \Gamma_0 \cdot \nabla \phi
= \phi_{\rho\rho} + \Bigl( \frac{1}{\rho}+ \frac{4\rho}{1+\rho^2}\Bigr) \phi_\rho
\]
and that the constants are in the kernel of this operator.
We then have
\begin{align*}
\phi_1 (\rho) = \int_\rho^\infty \frac{1}{r(1+r^2)^2}\int_0^r \frac{1}{(1+s)^{2+\sigma}} s (1+s^2)^2 ds dr
\end{align*}
and this implies
\begin{align*}
| \phi_1 (\rho) | + (1+\rho) | \phi_1   '(\rho) |  \leq \frac{C}{(1+\rho)^\sigma}
\end{align*}

Let
\begin{align}
\nonumber
\chi(\rho,\tau) = \chi_0(\frac{\rho}{\delta\sqrt{\tau}}) ,
\end{align}
where $\chi_0\in C^\infty(\R)$, $\chi_0(s)=1$ for $s\leq 1$ and $\chi_0(s)=0$ for $s\geq 2$.
Define $\tilde \phi_1 = \frac{1}{\tau^\nu (\log \tau)^m} \phi_1\chi$. We have
\begin{align}
\nonumber
& (\partial_\tau - \Delta_{\R^2} + \nabla \Gamma_0\cdot \nabla ) \tilde \phi_1
\\
\nonumber
&\quad \geq  \frac{1}{\tau^\nu (\log \tau)^m(1+\rho)^{2+\sigma}}  \chi
- \frac{C_1}{\tau^{\nu+\sigma/2+1}(\log \tau)^m} \chi_{\{\delta\sqrt \tau \leq \rho \leq 2 \delta\sqrt \tau\}} ,
\end{align}
for some $C_1>0$, $\delta>0$ (assuming $\tau_0$  large).
Now consider
\begin{align*}
\phi_2(\rho,\tau) = \frac{1}{\tau^{\nu+\sigma/2}(\log\tau)^m}\frac{1}{(1+\rho/\sqrt \tau )^{2+\sigma+\epsilon}} , \qquad
\phi_3(\rho,\tau) = \frac{1}{\tau^{\nu+\sigma/2}(\log\tau)^m} e^{-\frac{\rho^2}{4 \tau}}.
\end{align*}
A computation, using \eqref{hatZetai}, shows that
\begin{align*}
\bar \phi
=A_1 \tilde \phi_1 + A_2 \phi_2 + A_3 \phi_3
\end{align*}
satisfies
\begin{align*}
(\partial_\tau - \Delta_{\R^2} + \nabla \Gamma_0\cdot \nabla+ B_1 ) \bar \phi \geq
\frac{c}{\tau^\nu (\log\tau)^m}
\begin{cases}
\frac{1}{(1+\rho)^{2+\sigma}}   & \rho \leq \sqrt \tau
\\
\frac{\tau^{\epsilon/2}}{(1+\rho)^{2+\sigma+\epsilon}}    & \rho \geq \sqrt \tau.
\end{cases}
\end{align*}
for some $c>0$.
This step needs $\nu-1<\frac{\epsilon}{2}$ and $\nu+\frac{\sigma}{2}<3$.
By comparison, we find that $\tilde \Phi_1$ satisfies \eqref{decayPhi1}.

The solution  $\tilde \Phi_1$ of \eqref{eq-linear-radial7} satisfies
\begin{align}
\nonumber
\partial_\tau \tilde \Phi_1 &=
L_0[ \tilde \Phi_1 ] - U \tilde \Phi_1
\ch{+ B_1[\tilde \Phi_1]}
+\tilde H 
\end{align}
Applying $L_0$ to this equation we find that
\[
\Phi_1 = L_0[ \tilde \Phi_1]
\]
satisfies
\begin{align*}
\partial_\tau \Phi_1 &=
L [\Phi_1]+ \tilde B[\Phi_1]+ H - H_1
\end{align*}
with
\begin{align}
\label{defH1}
H_1 = - \nabla \cdot ( U \nabla \Psi_1)+L_0[ U \tilde \Phi_1]
+ \tilde \zeta_1 \nabla \cdot ( \tilde \Phi_1 U y )
, \quad
\Psi_1 = (-\Delta)^{-1} \Phi_1.
\end{align}

Let us verify that $\Phi_1$ and $H_1$   satisfy the conditions stated in
\eqref{Phi1-sigma-positive}, \eqref{H1-sigma-positive}, \eqref{H1-zero-mass}.
Indeed, from standard parabolic estimates and \eqref{decayPhi1} we have
\begin{align}
\label{decayNablaPhi1}
|\nabla \tilde\Phi_1|\leq  \frac{C}{\tau^\nu (\log \tau)^m }
\|H\|_{\nu,m,4+\sigma,\epsilon}
\begin{cases}
\frac{1}{ (1+\rho)^{1+\sigma}}  & \rho \leq \sqrt \tau
\\
\frac{\tau^{1+\epsilon/2}}{ (1+\rho)^{3+\sigma+\epsilon} }  &  \rho \geq \sqrt \tau .
\end{cases}
\end{align}
Differentiating in $y_j$, $j=1,2$ the equation \eqref{eq-linear-radial7} and using standard parabolic estimates, together with \eqref{bound-tildeH}, \eqref{decayNablaPhi1} we obtain
\begin{align}
\label{decayD2Phi1}
| D^2 \tilde\Phi_1|
\leq  \frac{C}{\tau^\nu (\log \tau)^m }
\|H\|_{\nu,m,4+\sigma,\epsilon}
\begin{cases}
\frac{1}{ (1+\rho)^{2+\sigma}}  & \rho \leq \sqrt \tau
\\
\frac{\tau^{1+\epsilon/2}}{ (1+\rho)^{4+\sigma+\epsilon} }  &  \rho \geq \sqrt \tau .
\end{cases}
\end{align}

The definition  $\Phi_1 =L_0[\tilde \Phi_1]$ and the estimates \eqref{decayPhi1}, \eqref{decayNablaPhi1}, \eqref{decayD2Phi1} give the estimate \eqref{Phi1-sigma-positive}.

We compute
\begin{align*}
H_1 = -\nabla U \cdot \nabla \Psi_1 + U \Phi_1
+ \nabla U \cdot \nabla \tilde \Phi_1 + U \Delta \tilde \Phi_1
+\tilde \zeta_1 \nabla \cdot ( \tilde \Phi_1 U y ) .
\end{align*}
Note that $\int_{\R^2} \Phi_1(\cdot,\tau)=0$. So,
by a direct radial computation of $\Psi_1 = (-\Delta)^{-1} \Phi_1$ and \eqref{Phi1-sigma-positive} we obtain
\begin{align}
\nonumber
|\nabla \Psi_1(\rho,\tau)|
\leq  \frac{C}{\tau^\nu (\log \tau)^m }
\|H\|_{\nu,m,4+\sigma,\epsilon}
\begin{cases}
\frac{1}{ (1+\rho)^{1+\sigma}}  & \rho \leq \sqrt \tau
\\
\frac{\tau^{1+\epsilon/2}}{ (1+\rho)^{3+\sigma+\epsilon} }  &  \rho \geq \sqrt \tau .
\end{cases}
\end{align}
This estimate and the ones already obtained for $\tilde \Phi_1$ \eqref{decayNablaPhi1}, \eqref{decayD2Phi1} and for $\Phi_1$ \eqref{Phi1-sigma-positive} yield
\begin{align}
\nonumber
|H_1(\rho,\tau)|
\leq  \frac{C}{\tau^\nu (\log \tau)^m }
\|H\|_{\nu,m,4+\sigma,\epsilon}
\begin{cases}
\frac{1}{ (1+\rho)^{6+\sigma}}  & \rho \leq \sqrt \tau
\\
\frac{\tau^{\epsilon/2}}{ (1+\rho)^{6+\sigma+\epsilon} }  &  \rho \geq \sqrt \tau .
\end{cases},
\end{align}
which is the desired estimate \eqref{H1-sigma-positive}.

Finally, the zero mass condition \eqref{H1-zero-mass} follows from the form of $H_1$ \eqref{defH1} and its decay. The mass condition for $\Phi_1$ \eqref{Phi1zeroMass} follows from $\Phi_1 =L_0[\tilde \Phi_1]$ and the decay of $\tilde \Phi_1$ \eqref{decayPhi1} and \eqref{decayNablaPhi1}.
\end{proof}

Next we would like to obtain a result similar to Proposition~\ref{prop-linear-with-energy} for the problem \eqref{H4plusNO0}. In order to simplify this step, we will modify this equation by allowing a parameter in the initial condition. This technical obstruction will be removed in the proof of Proposition~\ref{prop-radial-linear-with-second-moment}.
Thus we consider
\begin{align}
\label{H4plusNO0-modified}
\left\{
\begin{aligned}
\partial_\tau \Phi
&= L [\Phi ]
+ \tilde B[\Phi]
+ \zeta_1(\tau) A[\Phi]+ H
\quad \text{in }\R^2 \times \ch{(}\tau_0,\infty)
\\
\Phi(\cdot,\tau_0)&= c_1 \tilde Z_0 ,
\end{aligned}
\right.
\end{align}
where $\tilde Z_0$ is defined in \eqref{def-tilde-Z0}.

The next result allows us to say that if in equation  \eqref{H4plusNO0-modified} the right hand side has fast decay, then we can decompose the solution similarly as in  Proposition~\ref{prop-linear-with-energy}. This result is an extension of that proposition to an equation that has the extra operator $A$ in it, which is treated as a perturbation.

\begin{lemma}
\label{lemmaMainLinearRadial}
Let $0<\sigma<1$, $\epsilon>0$, $\sigma+\epsilon<2$,
$1<\nu<\min(1+\frac{\epsilon}{2},3-\frac{\sigma}{2},\frac{3}{2})$.
Let $0<q<1$.
Then there is $C>0$ such that for $\tau_0$ sufficiently large and for  $H$ radially symmetric with $\|H\|_{\nu,m,4+\sigma,\epsilon}<\infty$ and
\begin{align*}
\int_{\R^2} H(y,\tau)dy=0 \quad \text{for all } \tau>\tau_0
\end{align*}
the solution $\Phi$ to \eqref{H4plusNO0-modified} can be decomposed as $\Phi = \Phi_0 + \frac{a(\tau)}{2} Z_0$ with the estimates
\begin{align*}
|\Phi_0(\rho,\tau) |&\leq C \|H\|_{\nu,m,4+\sigma,\epsilon} \frac{1}{\tau^{\nu-\frac{1}{2}} (\log\tau)^{m+\frac{q}{2}} }
\min\Bigl( \frac{1}{(1+|y|)^{2}}  , \frac{\tau}{|y|^4} \Bigr)
\\
|a(\tau)| & \leq
 C \|H\|_{\nu,m,4+\sigma,\epsilon}
 \frac{1}{\tau^{\nu-1} (\log \tau)^{m+q} } .
\end{align*}
Moreover $\Phi_0$ and $a$ are linear operators of $H$.

\end{lemma}

\begin{proof}[Proof of Lemma~\ref{lemmaMainLinearRadial}]

We will treat the operator $A$ as a perturbation and therefore consider
\begin{align}
\label{H4plusNO}
\left\{
\begin{aligned}
\partial_\tau \Phi
&= L [\Phi ]
+ \tilde B[\Phi]+ H
\quad \text{in }\R^2 \times \ch{(}\tau_0,\infty)
\\
\Phi(\cdot,\tau_0)&=c_1 \tilde Z_0 .
\end{aligned}
\right.
\end{align}

Let $\Phi_1$, $H_1$ be the functions constructed in Lemma~\ref{lemma-error-concentration}.
Setting $\Phi = \Phi_1 + \Phi_2$, \eqref{H4plusNO} is equivalent to the following equation for
$\Phi_2$
\begin{align}
\label{eq-linear-radial4}
\left\{
\begin{aligned}
\partial_\tau \Phi_2 &=
L [\Phi_2] + \tilde B [\Phi_2] + H_1, \quad \text{in }\R^2 \times \ch{(}\tau_0,\infty),
\\
\Phi_2(\cdot,\tau_0) &= c_1 \tilde Z_0 \quad \text{in }\R^2 .
\end{aligned}
\right.
\end{align}

We now apply Proposition~\ref{prop-linear-with-energy} to \eqref{eq-linear-radial4}. We have that
$\|H_1\|_{\nu,m,6+\sigma,\epsilon}<\infty$ by \eqref{H1-sigma-positive}, $H_1$ is radial
and satisfies the zero mass condition \eqref{H1-zero-mass}.
By Proposition~\ref{prop-linear-with-energy}  and Lemma~\ref{lemma-estPhiPerpv2}
there exists $c_1$ such that the solution $\Phi_2 $ of  \eqref{eq-linear-radial4} satisfies
\begin{align*}
\Phi_2(y,\tau) = \Phi_2^\perp (y,\tau)  + \frac{a(\tau)}{2} Z_0(y) ,
\end{align*}
with the estimates
\begin{align}
\label{estPhiPerpb}
|\Phi_2^\perp(y,\tau) |
&\leq C
\frac{ \| H_1 \|_{\nu,m,6+\sigma,\epsilon}}{\tau^{\nu-\frac{1}{2}} (\log\tau)^{m+\frac{q}{2}} }
\begin{cases}
\displaystyle
\frac{1}{(1+|y|)^{2}} & |y|\leq \sqrt \tau
\\
\displaystyle
\frac{\tau}{|y|^4} & |y|\geq \sqrt \tau,
\end{cases}
\\
\label{40b}
|a(\tau) | &\leq
C
\frac{\|H_1\|_{\nu,m,6+\sigma,\epsilon}}{\tau^{\nu-1} (\log \tau)^{m+q}} .
\end{align}
(We are ignoring the factor $\frac{1}{(\log\tau_0)^{1-q}}$ in the estimate of $a(\tau)$.)
We also know that $c_1$ is a linear function of $H_1$ and satisfies
\begin{align*}
|c_1| & \leq
C \frac{\|H_1\|_{\nu,m,6+\sigma,\epsilon}}{\tau_0^{\nu-1} (\log \tau_0)^{m+1}} .
\end{align*}

Combining \eqref{Phi1-sigma-positive} and \eqref{estPhiPerpb} we conclude that $\Phi$, the solution to \eqref{H4plusNO},
can be decomposed as
\[
\Phi = \Phi_0 + \frac{a(\tau)}{2} Z_0
\]
where $\Phi_0(y,\tau) = \Phi_1 + \Phi_2^\perp$ is radial and satisfies
\begin{align*}
|\Phi_0(y,\tau)|
\leq C
\frac{\|H\|_{\nu,m,4+\sigma,\epsilon} }{\tau^{\nu-\frac{1}{2}} (\log\tau)^{m+\frac{q}{2}} }
\min\Bigl( \frac{1}{(1+|y|)^{2}}  , \frac{\tau}{|y|^4} \Bigr)
\end{align*}
and $a(\tau)$ satisfies, combining \eqref{H1-sigma-positive} and \eqref{40b},
\begin{align}
\nonumber
|a(\tau) | &\leq
C
\frac{1}{\tau^{\nu-1} (\log \tau)^{m+q}}
\|H\|_{\nu,m,4+\sigma,\epsilon} .
\end{align}
We summarize the previous finding as follows. Given $H$ radial satisfying $\int_{\R^2} H(\cdot,\tau)=0$ for $\tau>\tau_0$ and $\|H\|_{\nu,m,4+\sigma,\epsilon} <\infty$, let us denote $T_0(H) =\Phi_0 = \Phi_1 + \Phi_2^\perp$ and $T_a(H) = a(\tau) $ so that the solution $\Phi $ of  \eqref{H4plusNO}, is $\Phi =  \Phi_0 + \frac{a(\tau)}{2} Z_0 = T_0[H] + \frac{1}{2} T_a[H] Z_0$. Then $T_0$, $T_a$ are linear and have the estimates
\begin{align}
\label{estT0}
\| T_0[H] \|_0 &\leq  C\|H\|_{\nu,m,4+\sigma,\epsilon}
\\
\label{estTa}
\| T_a[H] \|_a &\leq C \|H\|_{\nu,m,4+\sigma,\epsilon} ,
\end{align}
where
\begin{align*}
\|\Phi_0\|_{0} & = \sup_{ \tau>\tau_0 , \, y \in \R^2}
\tau^{\nu-\frac{1}{2}} (\log\tau)^{m+\frac{q}{2}}
\frac{1}{\min\Bigl( \frac{1}{(1+|y|)^{2}}  , \frac{\tau}{|y|^4} \Bigr) }
|\Phi_0(y,\tau)|
\\
\|a\|_{a} &=  \sup_{ \tau>\tau_0 }
\tau^{\nu-1} (\log \tau)^{m+q} |a(\tau)| .
\end{align*}
Moreover $c_1$ is a linear function of $H$ and satisfies
\begin{align*}
|c_1| & \leq
C \frac{\|H\|_{\nu,m,4+\sigma,\epsilon}}{\tau_0^{\nu-1} (\log \tau_0)^{m+1}} .
\end{align*}

We will apply these estimates to treat problem \eqref{H4plusNO0-modified}, which can be written as the fixed point problem
\begin{align*}
\Phi_0 = T_0 [ H + \zeta_1 A[ \Phi_0 + a Z_0] ]
\\
a= T_a [ H + \zeta_1 A[ \Phi_0 + a Z_0] ]
\end{align*}

By \eqref{estT0} and \eqref{estTa}
\begin{align*}
\| T_0[ \zeta_1A[ \Phi_0 + a Z_0 ] ] \|_0
+\| T_a[ \zeta_1A[ \Phi_0 + a Z_0 ] ] \|_a
\leq
C  \|  \zeta_1 A[ \Phi_0 + a Z_0 ]  \|_{\nu,m,4+\sigma,\epsilon}.
\end{align*}
We claim that
\begin{align}
\label{estAPhi0}
\|  \zeta_1 A[ \Phi_0 ]  \|_{\nu,m,4+\sigma,\epsilon} \leq  C \tau_0^{-\vartheta}
 \|\Phi_0\|_0 ,
\end{align}
for some $\vartheta>0$,
where $C$ is independent of $\tau_0$, and
\begin{align}
\label{estAa}
\|  \zeta_1 A[ a Z_0 ]  \|_{\nu,m,4+\sigma,\epsilon} \leq \frac{C}{(\log\tau_0)^{1+q}}\|a\|_a .
\end{align}

Assume for the moment that \eqref{estAPhi0}, \eqref{estAa} hold. The we see that
\begin{align*}
\|\Phi_0\|_0 + \|a\|_a \leq \frac{C}{(\log \tau_0)^{1+q}} (\|\Phi_0\|_0 + \|a\|_a)
+ C \|H\|_{\nu,m,4+\sigma,\epsilon}.
\end{align*}
For $\tau_0$ large this gives
\[
\|\Phi_0\|_0 + \|a\|_a   \leq C \|H\|_{\nu,m,4+\sigma,\epsilon},
\]
which is the desired result.

For the proof of estimates \eqref{estAPhi0}, \eqref{estAa} we will need the following property.
If $\Phi$ satisfies $|\Phi(y)|\leq \frac{1}{(1+|y|)^{2+\kappa}}$ for some $\kappa>0$ and $\int_{\R^2}\Phi dy=0$, then
\begin{align}
\label{secondMargA}
\int_{\R^2}
\nabla \cdot [   \Phi U y  - Z_0 \nabla \Psi ]|y|^2 dy=0 , \quad \Psi = (-\Delta)^{-1}\Phi.
\end{align}
Indeed,
\begin{align*}
\int_{\R^2}  \nabla \cdot ( \Phi U y ) |y|^2dy
&=
- 2 \int_{\R^2}   \Phi U |y|^2dy
= 2 \int_{\R^2} \Delta \Psi U |y|^2dy
\\
&= -2 \int_{\R^2} \nabla \Psi \cdot \nabla (U |y|^2) dy
\\
&= -2 \int_{\R^2} \nabla \Psi \cdot y Z_0 dy
\end{align*}
and
\begin{align*}
\int_{\R^2} \nabla \cdot ( Z_0 \nabla \Psi ) |y|^2 dy
&=
-2 \int_{\R^2} Z_0 \nabla \Psi \cdot y dy.
\end{align*}

To prove \eqref{estAPhi0}, let us write  $\Psi_0 = (-\Delta)^{-1} \Phi_0$.
Then
\begin{align*}
A[\Phi_0]
= L^{-1}[ \nabla \cdot ( \Phi_0 \ch{U} y - Z_0 \nabla \Psi_0) ]  .
\end{align*}
Using the definition of $L^{-1}$ given in Lemma~\ref{lemmaLinv} we have
that
\[
L^{-1}[ \nabla \cdot ( \Phi \ch{U} y )  - \nabla\cdot( Z_0 \nabla (-\Delta)^{-1}\Phi) ]  = U g + U \psi
\]
where
\begin{align}
\label{g1}
g(\rho,\tau) =  -\int_\rho^\infty
\Bigl[ \Phi_0(s,\tau) s -\frac{Z_0(s)}{U(s)} \partial_\rho \Psi_0(s,\tau) \Bigr]ds ,
\end{align}
and $\psi$ is the decaying solution
 to the Liouville equation
\begin{align*}
-\Delta \psi - U \psi = U g .
\end{align*}

From the definition $\Psi_0 = (-\Delta)^{-1} \Phi_0$ and using that $\int_{\R^2} \Phi_0 dy=0$ we have
\begin{align*}
\partial_\rho \Psi_0(\rho,\tau)=\frac{1}{\rho}\int_\rho^\infty \Phi_0(s,\tau)sds
\end{align*}
which gives the estimate
\begin{align*}
| \partial_\rho \Psi_0(\rho,\tau)|\leq C \|\Phi_0\|_0
\frac{1}{\tau^{\nu-\frac{1}{2}} (\log \tau)^{m+\frac{q}{2}}}
\begin{cases}
\frac{\log( \frac{2\sqrt \tau}{1+\rho})}{1+\rho}  & \rho \leq \sqrt \tau ,
\\
\frac{\tau}{\rho^3} & \rho \geq \sqrt \tau.
\end{cases}
\end{align*}
Then formula \eqref{g1} gives
\begin{align*}
|g(\rho,\tau) |
& \leq C \|\Phi_0\|_0
\frac{1}{\tau^{\nu-\frac{1}{2}} (\log \tau)^{m+\frac{q}{2}}}
\begin{cases}
\log^2( \frac{2\sqrt \tau}{1+\rho})  & \rho \leq \sqrt \tau ,
\\
\frac{\tau}{\rho^2} & \rho \geq \sqrt \tau.
\end{cases}
\\
& \leq C \|\Phi_0\|_0
\frac{1}{\tau^{\nu-\frac{1}{2}} (\log \tau)^{m+\frac{q}{2}-2}}
\min \Bigl(1,\frac{\tau}{\rho^2}\Bigr).
\end{align*}
We note that by \eqref{secondMargA} we have $\int_{\R^2} U g z_0 dy = 0$.
Then,  $\psi$ has the estimate
\begin{align*}
|\psi(\rho,\tau) |\leq C \|\Phi_0\|_0
\frac{1}{\tau^{\nu-\frac{1}{2}} (\log \tau)^{m+\frac{q}{2}-2}}
\frac{1}{(1+\rho)^2}
\min \Bigl(1,\frac{\tau}{\rho^2}\Bigr).
\end{align*}
It follows that $A[ \Phi_0] = U g + U \psi$ satisfies
\begin{align*}
|A[ \Phi_0] (\rho,\tau)|
\leq
C \|\Phi_0\|_0
\frac{1}{\tau^{\nu-\frac{1}{2}} (\log \tau)^{m+\frac{q}{2}-2}}
\frac{1}{(1+\rho)^4}
\min \Bigl(1,\frac{\tau}{\rho^2}\Bigr).
\end{align*}
From this inequality we obtain \eqref{estAPhi0}.

The proof of \eqref{estAa} is similar.
This time $A[a Z_0] = U g_1 + U \psi_1$ where
\begin{align}
\nonumber
g_1(\rho,\tau) =  -a(\tau) \int_\rho^\infty
\Bigl[ Z_0(s) s -\frac{Z_0(s)}{U(s)}  z_0'(s) \Bigr]ds ,
\end{align}
and $\psi_1$ is the radial decaying solution   to
\begin{align*}
-\Delta \psi_1 - U \psi_1 = U g_1 .
\end{align*}
We then obtain that
\begin{align*}
| A[aZ_0](\rho,\tau) |\leq C \|a\|_a
\frac{1}{\tau^{\nu-1} (\log \tau)^{m+q}}
\frac{1}{(1+\rho)^6} .
\end{align*}
From this estimate we deduce \eqref{estAa}.
\end{proof}

Before proving Proposition~\ref{prop-radial-linear-with-second-moment} as stated, we obtain a version of it for the problem
\begin{align}
\label{linear-006}
\left\{
\begin{aligned}
\partial_\tau \phi &= L[\phi] + B[\phi]+  h(y,\tau)  \quad \text{in }\R^2 \times (\tau_0,\infty),
\\
\phi(\cdot,\tau_0) &= c_1 \hat Z_0 \quad \text{in }\R^2 ,
\end{aligned}
\right.
\end{align}
where 	
\begin{align}
\nonumber
\hat Z_0 = L[ \tilde Z_0].
\end{align}


\begin{lemma}
\label{lemma-linear-with-second-moment-radial}
Let $0<\sigma<1$, $\epsilon>0$, $\sigma+\epsilon<2$
and  $1<\nu< \min( 1+\frac{\epsilon}{2},3-\frac{\sigma}{2}, \frac{3}{2})$.
Let $0<q<1$.
Then there is $C$ such that for \ch{$\tau_0$} large the following holds.
Suppose that  $h$ is radially symmetric, satisfies  $\|h\|_{\nu,m,6+\sigma,\epsilon}<\infty$ and
\begin{align}
\nonumber
\int_{\R^2} h(y,\tau)dy&=0, \quad
\int_{\R^2} h(y,\tau)|y|^2dy=0,\quad \tau>\tau_0.
\end{align}
Then there exist $c_1\in \R$ and a solution $\phi(y,\tau)$ of problem
\ch{\eqref{linear-006}}
that define linear operators of $h$ and
satisfy
\begin{align}
\nonumber
\|  \phi \|_{\nu-\frac 12 ,m+\frac {q-1}{2},4,2+\sigma+\epsilon} \leq C \|h\|_{\nu,m,6+\sigma,\epsilon}.
\end{align}
\begin{align*}
|c_1| & \leq  C \frac{1 }{ \tau_0^{\nu-1} (\log\tau_0)^{m+1}}  \|h\|_{\nu,m,6+\sigma,\epsilon}.
\end{align*}
\end{lemma}
\begin{proof}
Consider equation \eqref{H4plusNO0-modified}, where $H$ is the function constructed in Lemma~\ref{lemmaLinv}.
By Lemma~\ref{lemmaMainLinearRadial}, there is $c_1$ such that the solution $\Phi$ of \eqref{H4plusNO0-modified}
can be decomposed as $\Phi = \Phi_0 + \frac{a(\tau)}{2} Z_0$,
where $\Phi_0$ and $a$ satisfy the estimates stated in that proposition. In combination with \eqref{bound-H2} we find
\begin{align}
\label{estPhi0}
|\Phi_0(\rho,\tau) |&\leq C  \|h\|_{\nu,m,6+\sigma,\epsilon} \frac{1}{\tau^{\nu-\frac{1}{2}} (\log\tau)^{m+\frac{q}{2}} }
\min\Bigl( \frac{1}{(1+|y|)^{2}}  , \frac{\tau}{|y|^4} \Bigr)
\\
\nonumber
|a(\tau)| & \leq
 C \|h\|_{\nu,m,6+\sigma,\epsilon}
 \frac{1}{\tau^{\nu-1} (\log \tau)^{m+q} } .
\end{align}
\begin{align}
\label{est-c1}
|c_1| & \leq  C \frac{1 }{ \tau_0^{\nu-1} (\log\tau_0)^{m+1}}  \|h\|_{\nu,m,6+\sigma,\epsilon}.
\end{align}
Moreover $\Phi_0$, $a$, $c_1$ are linear operators of $H$.

From standard parabolic estimates and \eqref{estPhi0} we obtain
\begin{align}
\label{estNablaPhi0}
|\nabla \Phi_0(\rho,\tau) |&\leq C  \|h\|_{\nu,m,6+\sigma,\epsilon}\frac{1}{\tau^{\nu-\frac{1}{2}} (\log\tau)^{m+\frac{q}{2}} }
\min\Bigl( \frac{1}{(1+|y|)^{3}}  , \frac{\tau}{|y|^5} \Bigr).
\end{align}
We consider the equation for $\Phi_0 = \Phi- \frac{a(\tau)}{2}Z_0$, obtained from \eqref{H4plusNO0-modified}, and
differentiate with respect to $y_j$, $j=1,2$. Using standard parabolic estimates, together with \eqref{estPhi0},  \eqref{estNablaPhi0}, and the bound for $a'(\tau)$ in \eqref{est-aprime},  we obtain
\begin{align}
\label{estD2Phi0}
| D^2 \Phi_0(\rho,\tau)|
\leq C  \|h\|_{\nu,m,6+\sigma,\epsilon} \frac{1}{\tau^{\nu-\frac{1}{2}} (\log\tau)^{m+\frac{q}{2}} }
\min\Bigl( \frac{1}{(1+|y|)^{4}}  , \frac{\tau}{|y|^6} \Bigr).
\end{align}

Let us define $\phi = L[\Phi]$. Then $\phi$ satisfies \eqref{linear-006} because $L[Z_0]=0$ and thanks to \eqref{estPhi0}, \eqref{estNablaPhi0}, \eqref{estD2Phi0} we find
\begin{align}
\label{est-phi}
|\phi(\rho,\tau)|
\leq C  \|h\|_{\nu,m,6+\sigma,\epsilon} \frac{1}{\tau^{\nu-\frac{1}{2}} (\log\tau)^{m+\frac{q}{2}} }
\min\Bigl( \frac{1}{(1+|y|)^{4}}  , \frac{\tau}{|y|^6} \Bigr).
\end{align}

In the rest of the proof we show that
\begin{align*}
|\phi(\rho,\tau)|
\leq C  \|h\|_{\nu,m,6+\sigma,\epsilon} \frac{1}{\tau^{\nu-\frac{1}{2}} (\log\tau)^{m+\frac{q}{2}} }
\frac{1}{(1+\rho)^4}
\begin{cases}
1 & \rho \leq \sqrt \tau
\\
\frac{\tau^{1+\sigma/2+\epsilon/2}}{\rho^{2+\sigma  + \epsilon}} & \rho \geq \sqrt \tau .
\end{cases}
\end{align*}
For this we consider the equation \eqref{linear-006} written in the form
\begin{align}
\label{eqLinear2}
\partial_\tau \phi = \Delta \phi -  \nabla \Gamma_0 \nabla \phi  + 2 U \phi + B[\phi] + \bar h ,
\end{align}
where
\[
\bar h = - \nabla U \nabla \psi + h.
\]
Using \eqref{est-phi} and the radial formula for $\psi = (-\Delta)^{-1} \phi$, we get
\begin{align}
\nonumber
|\nabla \psi(y,\tau)| \leq
C  \|h\|_{\nu,m,6+\sigma,\epsilon} \frac{1}{\tau^{\nu-\frac{1}{2}} (\log\tau)^{m+\frac{q}{2}} }
\begin{cases}
\frac{1}{(1+\rho)^3} & \rho \leq \sqrt \tau
\\
\frac{\tau}{\rho^{5}} & \rho \geq \sqrt \tau .
\end{cases}
\end{align}
This estimate and the definition of the norm $ \|h\|_{\nu,m,6+\sigma,\epsilon}$, give
\[
|\bar h(y,\tau) |\leq
C \frac{1}{\tau^{\nu-\frac{1}{2}} ( \log \tau)^{m+\frac{q}{2}} }
\|h\|_{\nu,m,6+\sigma,\epsilon}
\begin{cases}
\frac{1}{(1+|y|)^{6+\sigma}} & |y| \leq \sqrt \tau
\\
\frac{1}{\tau^{3+\sigma/2}(\frac{|y|}{\sqrt \tau})^{6+\sigma+\epsilon}} & |y| \geq \sqrt \tau.
\end{cases}
\]

We now construct a barrier very similar to the proof of Proposition~\ref{prop-linear-without-second-moment}
\begin{align*}
\bar\phi (\rho,\tau)
&=
A_1  \frac{1}{\tau^{\nu-\frac{1}{2}} ( \log \tau)^{m+\frac{q}{2}} }
\tilde g_2(\rho)
\chi_0\Bigl(\frac{\rho}{\sqrt \tau}\Bigr)
+A_2
\frac{1}{\tau^{\nu+\frac{3}{2}} (\log\tau)^{m+\frac{q}{2}} }
\frac{1}{(1+\rho/\sqrt \tau)^{6+\sigma+\epsilon}}
\\
& \quad
+A_3
\frac{1}{\tau^{\nu+\frac{3}{2}} (\log\tau)^{m+\frac{q}{2}} } e^{-\frac{\rho^2}{4\tau} } ,
\end{align*}
where $\tilde g_2$ is the function \eqref{def-tildeg2}.
We consider \eqref{eqLinear2} in $\{ \, (y,\tau) \ | \ \tau > \tau_0, \ |y|> R_0\, \}$ where $R_0>0$ is a large constant.
For suitable constants $A_1$, $A_2$, $A_3$, $C$ the function $C \|h\|_{\nu,m,6+\sigma,\epsilon} \bar \phi $ is a supersolution. This computation requires $\nu<\frac{3}{2}$.

Moreover  $\phi(y,\tau) \leq C  \|h\|_{\nu,m,6+\sigma,\epsilon} \bar \phi(y,\tau)$ at $|y| = R_0$. The initial conditions also compare well. Indeed, by Lemma~\ref{lemma-hatZ0} and  \eqref{est-c1}
\begin{align*}
|\phi(\rho,\tau_0)| &= c_1 |\hat Z_0(\rho)|
\leq C   \frac{1 }{ \tau_0^{\nu-1} (\log\tau_0)^{m+1}}  \|h\|_{\nu,m,6+\sigma,\epsilon} \frac{1}{\tau_0} \frac{1}{1+\rho^6} ,
\end{align*}
and this is supported on $\rho \leq 2 \sqrt \tau_0$, so
\begin{align*}
|\phi(\rho,\tau_0)| \leq C  \|h\|_{\nu,m,6+\sigma,\epsilon} \bar \phi(y,\tau).
\end{align*}

By the maximum principle
\begin{align*}
|\phi(y,\tau) |\leq C \bar \phi(y,\tau)  \|h\|_{\nu,m,6+\sigma,\epsilon} , \quad |y|> R_0.
\end{align*}
This finishes the proof.
\end{proof}

\begin{proof}[Proof of Proposition~\ref{prop-radial-linear-with-second-moment}]
Let $\hat\phi$, $c_1$ be the solution to \eqref{linear-006} constructed in Lemma~\ref{lemma-linear-with-second-moment-radial}.
Let $\phi_1$ be the solution to \eqref{linear-004}.
By Lemma~\ref{lemma-est-phi1} $\phi_1$ satisfies
\begin{align}
\label{est-phi1}
| \phi_1(\rho,\tau) | \leq C  \frac{\tau_0^{\nu_0-1} R(\tau)^2}{ \tau^{\nu_0} R(\tau_0)^2}
\frac{1}{(1+\rho^4) } \min\Bigl( 1 , \frac{\tau^{1/2}}{\rho}\Bigr)^{2+\sigma+\epsilon} ,
\end{align}
where $1<\nu_0<\frac{7}{4}$.
Then the solution $\phi$ to \eqref{linear-002} that we construct is given by
\begin{align*}
\phi  = \hat \phi - c_1 \phi_1 .
\end{align*}
To get the desired estimate on $\phi$ we need to estimate  $|c_1 \phi_1|$.
Let $f$ be given by \eqref{notation-f}.
By \eqref{est-c1} and \eqref{est-phi1}
\begin{align*}
|c_1 \phi_1(\rho,\tau)|
& \leq C \frac{1 }{ \tau_0^{\nu-1} (\log\tau_0)^{m+1}}
\frac{\tau_0^{\nu_0-1} R(\tau)^2}{ \tau^{\nu_0} R(\tau_0)^2}
\frac{1}{(1+\rho^4) } \min\Bigl( 1 , \frac{\tau^{1/2}}{\rho}\Bigr)^{2+\sigma+\epsilon}
\|h\|_{\nu,m,6+\sigma,\epsilon}
\\
\\
&  \leq C
\frac{1}{\log \tau_0 R(\tau_0)}
f(\tau) R(\tau)
\frac{1}{(1+\rho^4) } \min\Bigl( 1 , \frac{\tau^{1/2}}{\rho}\Bigr)^{2+\sigma+\epsilon}
\|h\|_{\nu,m,6+\sigma,\epsilon}
\\
&  \leq C
f(\tau) R(\tau)
\frac{1}{(1+\rho^4) } \min\Bigl( 1 , \frac{\tau^{1/2}}{\rho}\Bigr)^{2+\sigma+\epsilon}
\|h\|_{\nu,m,6+\sigma,\epsilon}
\end{align*}
provided $\frac{1}{2}+\nu-\nu_0<0$. But $\nu_0$ can be taken close to $\frac{7}{4}$, so we obtain the result by assuming $\nu < \frac{5}{4}$ in addition to the other constraints needed in Lemma~\ref{lemma-linear-with-second-moment-radial}, namely  $1<\nu< \min( 1+\frac{\epsilon}{2},3-\frac{\sigma}{2}, \frac{3}{2})$.
\end{proof}

\section{Linear estimate with second moment (general)}
\label{sect-linear-nonradial}

A convenient property of problem \equ{linear1} is that it can be split into Fourier modes. If we decompose
\begin{align}
\label{decomp-h}
h(y,\tau) &= h_0(|y|,\tau)  +  h_1(y,\tau)  , \quad h_0(\rho,\tau) = \frac 1{2\pi} \int_0^{2\pi}  h(\rho e^{i\theta}, \tau) d\theta
\\
\label{decomp-phi}
\phi(y,\tau) &= \phi_0(|y|,\tau)  +  \phi_1(y,\tau)  , \quad \phi_0(\rho,\tau) = \frac 1{2\pi} \int_0^{2\pi}  \phi(\rho e^{i\theta}, \tau) d\theta,
\end{align}
then $\phi$ solves \equ{linear1} if and only if $\phi_i$ solves \equ{linear1} where $h$ is replaced with $h_i$, for $i=0,1$.  If $h=h_1$ we say that $h$ {\em has no radial mode}.

For the proof Proposition~\ref{prop-linear-with-second-moment} in the general case we will consider in a first step the equation \eqref{linear1} but without the operator $B$, namely,
\begin{align}
\label{eq-linear-nonradial120}
\left\{
\begin{aligned}
\partial_\tau \phi &=
L [\phi] + h, \quad \text{in }\R^2 \times (\tau_0,\infty),
\\
\phi(\cdot,\tau_0) &= 0\quad \text{in }\R^2 ,
\end{aligned}
\right.
\end{align}
for functions with no radial mode, as explained at the beginning of Section~\ref{sect-theorem-linear-with-second-moment}.
Later on, we will consider equation \eqref{linear1} for functions with no radial mode, where we will treat the operator $B[\phi]$ as a perturbation term that can be assimilated to the right hand side.

\medskip
The main step in the proof is the following estimate, valid when the functions involved have no radial mode.

\begin{prop}
\label{prop-linear-nonradial}
Let $0<\sigma<1$, $0<\epsilon<2$,  $0<\nu<\min( 1 + \frac{\epsilon}{2}, \frac{3}{2}-\frac{\sigma}{2})$, $m\in \R$.
Then there is a $C>0$ such that for any $\tau_0$ sufficiently large the following holds.
Suppose that  $h(y,\tau)$ has no radial mode and satisfies $\|h\|_{\nu,m,5+\sigma,\epsilon}<\infty$,
\begin{align}
\label{center-mass-h}
\int_{\R^2} h(y,\tau) y_j dy=0\quad \text{for all }\tau>\tau_0, \quad j=1,2.
\end{align}
Then the solution $\phi(y,\tau)$ of \eqref{eq-linear-nonradial120} satisfies
\begin{align}
\label{est-non-radial1}
|\phi(y,\tau) |
\leq C
\frac{\|h\|_{\nu,m,5+\sigma,\epsilon}}{\tau^\nu (\log\tau)^{m} }
\begin{cases}
\frac{1}{(1+|y|)^{3+\sigma}} , & |y|\leq \sqrt{\tau} .
\\
\frac{\tau^{1+\frac{\epsilon}{2}}}{|y|^{5+\sigma+\epsilon}} , & |y|\geq \sqrt{\tau} .
\end{cases}
\end{align}
\end{prop}

\begin{proof}
Since $h(y,\tau )$ has no radial mode, all functions involved in the proof have also this property.
We use the notation from \S\ref{subsect-q-form}, particular $g = \frac{\phi}{U}-(-\Delta)^{-1}\phi$, $g^\perp = g - a$ with $a(\tau)\in \R$ such that
\[
\int_{\R^2} g^\perp (y,\tau) U dy =0 .
\]
But
\[
\int_{\R^2} g (y,\tau) U dy =0
\]
because $g$ has no radial mode, so that $a(\tau) = 0$, $g^\perp = g$, $\phi^\perp = \phi$.
Then the proof proceeds as the proof of Proposition~\ref{prop-linear-with-energy} with some simplifications, since there is no need to estimate $a$.

We write \eqref{eq-linear-nonradial120} as
\begin{align}
\nonumber
\partial_\tau \phi = \nabla \cdot ( U \nabla g^\perp ) +  h , \quad \text{in }\R^2 \times (\tau_0,\infty) .
\end{align}
We multiply this equation by $g$ and integrate in $\R^2$.

Let $R>0$ be a large fixed constant and let
\begin{align*}
f(\tau) = \frac{1}{\tau^\nu (\log\tau)^m}.
\end{align*}
Let $T_2>\tau_0$ and let
\begin{align*}
\|\varphi\|_{\infty,T_2}  = \sup_{\tau \in [\tau_0,T_2]} |\varphi(\tau)|.
\end{align*}
The following estimates are valid for $\tau \in [\tau_0,T_2]$.
As in the proof of Proposition~\ref{prop-linear-with-energy} we get
\begin{align}
\label{UgL2nonradial}
\int_{\R^2} g^2 U
\leq
Cf(\tau)^2 R^2
\Bigl( \|h\|_{\nu,m,5+\sigma,\epsilon}^2  + \Bigl\| \frac{\omega}{f R}\Bigr\|_{\infty,T_2}^2
\Bigr) ,
\end{align}
where
\[
\omega(\tau) = \Bigl( \int_{\R^2\setminus B_{R}} g(\tau)^2 U \Bigr)^{1/2} .
\]
Similarly as in Lemma~\ref{lemma-pointwise2}, from \eqref{UgL2nonradial} we get
\begin{align}
\label{estUdnonrad}
|U g(y,\tau)|
\leq
Cf(\tau) R
\Bigl( \|h\|_{\nu,m,5+\sigma,\epsilon}
+   \Bigl\| \frac{\omega}{f R}\Bigr\|_{\infty,T_2}
\Bigr)
\frac{  1  }{(1+|y|)^{3+\sigma}} .
\end{align}
The proof is presented below.
We use this  to estimate
\begin{align}
\nonumber
\omega(\tau) =
\Bigl(
\int_{\R^2\setminus B_{R}} g^2 U  \Bigr)^{1/2}
\leq C  f(\tau) R^{1-\sigma}
\Bigl(
\|h\|_{\nu,m,5+\sigma,\epsilon}
+
 \Bigl\| \frac{\omega}{f R}\Bigr\|_{\infty,T_2}
\Bigr),
\end{align}
which implies
\begin{align*}
\frac{\omega(\tau) }{f(\tau) R}
\leq C R^{-\sigma}\|h\|_{\nu,m,5+\sigma,\epsilon}
+  C R^{-\sigma} \Bigl\| \frac{\omega}{f R}\Bigr\|_{\infty,T_2}.
\end{align*}

We deduce that
\begin{align}
\nonumber
 \Bigl\| \frac{\omega}{f R}\Bigr\|_{\infty,T_2}
\leq C R^{-\sigma} \|h\|_{\nu,m,5+\sigma,\epsilon} ,
\end{align}
by choosing $R$ as a large constant.

Now we let $T_2\to\infty$ and find
\begin{align}
\label{15x-0}
\omega(\tau) \leq C f(\tau) R  \|h\|_{\nu,m,5+\sigma,\epsilon}^2  , \quad \tau>\tau_0.
\end{align}
The inequalities that follow hold for $\tau>\tau_0$.

Combining \eqref{15x-0} with \eqref{UgL2nonradial}  we obtain
\begin{align}
\nonumber
\int_{\R^2} g^2 U
\leq
Cf(\tau)^2 R^2
\|h\|_{\nu,m,5+\sigma,\epsilon}^2 , \quad \tau>\tau_0.
\end{align}
and using \eqref{estUdnonrad} we also get
\begin{align}
\nonumber
|U g(y,\tau)|
\leq
C f(\tau) R
\|h\|_{\nu,m,5+\sigma,\epsilon}
\frac{  1  }{(1+|y|)^{3+\sigma}}  .
\end{align}
Let $\psi = (-\Delta)^{-1}\phi$ so that $ \phi = U g + U \psi$. Using Lemma~\ref{lemma-liouville} and the previous estimate we obtain
\begin{align}
\label{estPsinonrad}
|\psi(y,\tau)|
+ (1+|y|)|\nabla \psi(y,\tau)|
\leq
C \frac{R }{\tau^{\nu} (\log \tau)^{m}}
\frac{  1  }{(1+|y|)^{1+\sigma}}
\|h\|_{\nu,m,5+\sigma,\epsilon} .
\end{align}

We consider the equation \eqref{eq-linear-nonradial120} in $\R^2 \setminus B_{R}(0)$ written in the form
\begin{align}
\nonumber
\partial_\tau \phi = \Delta \phi - \nabla \Gamma_0 \nabla \phi  + 2 U \phi  + \bar h ,
\end{align}
where
\[
\bar h = - \nabla U \nabla \psi + h.
\]
By \eqref{estPsinonrad} and the definition of the norm $ \|h\|_{\nu,m,5+\sigma,\epsilon}$,
\[
|\bar h(y,\tau) |\leq
C \|h\|_{\nu,m,5+\sigma,\epsilon}
\frac{1}{\tau^{\nu} ( \log \tau)^{m} }
\frac{1}{(1+|y|)^{5+\sigma}}
\begin{cases}
1 & |y| \leq \sqrt \tau
\\
\frac{\tau^{\epsilon/2}}{|y|^\epsilon} & |y| \geq \sqrt \tau.
\end{cases}
\]
Here we are using $\epsilon<2$.
Using barriers as in the proof of Lemma~\ref{lemma-pointwise2} we get
\begin{align*}
|\phi(y,\tau) |
\leq
C \|h\|_{\nu,m,5+\sigma,\epsilon}
\frac{1}{\tau^{\nu} ( \log \tau)^{m} }
\frac{1}{(1+|y|)^{3+\sigma}}
\begin{cases}
1 & |y| \leq \sqrt \tau
\\
\frac{\tau^{1+\epsilon/2}}{|y|^{2+\epsilon}}
& |y| \geq \sqrt \tau .
\end{cases}
\end{align*}
(For this we need $ \nu < 1+\frac{\epsilon}{2}$, $\nu + \frac{\sigma}{2} < \frac{3}{2}$.)
This proves \eqref{est-non-radial1}.

\end{proof}

\begin{proof}[Proof of \eqref{estUdnonrad}]
We define
\[
g_0 = U g ,
\]
which satisfies the equation
\begin{align}
\label{eqTildeGnonrad}
\partial_\tau
g_0
&=
\Delta g_0 - \nabla g_0 \cdot \nabla \Gamma_0 +2 U g_0
+ \tilde h
\end{align}
where
\begin{align*}
\tilde h =
 U v + h-U(-\Delta)^{-1} h
\end{align*}
and
\begin{align}
\nonumber
v:=  (-\Delta)^{-1}
( \nabla \cdot ( g_0 \nabla\Gamma_0  ) ).
\end{align}

%
As in the proof of Lemma~\ref{lemma-pointwise1} we obtain
\begin{align}
\label{decayTildeG3nonrad}
|g_0 (y,\tau)|
\leq C
\frac{R }{\tau^\nu (\log\tau)^m (1+|y|)^{2}} K,
\end{align}
where
\[
K = \|h\|_{\nu,m,5+\sigma,\epsilon}
+  \Bigl\| \frac{\omega}{f R}\Bigr\|_{\infty,T_2}.
\]
Applying parabolic estimates to \eqref{eqTildeGnonrad} and a scaling argument we find
\begin{align}
\label{estGradTildeg0nonrad}
| \nabla g_0 (y,\tau)|
\leq C
\frac{R K }{\tau^\nu (\log\tau)^m (1+|y|)^3} .
\end{align}

%

Using \eqref{decayTildeG3nonrad}, \eqref{estGradTildeg0nonrad} and $g_0 = g U$ we get that
\begin{align*}
| \nabla U \cdot \nabla g + g \Delta U |\leq C
\frac{R K}{\tau^\nu (\log \tau)^m (1+|y|)^4 } .
\end{align*}

We observe that for $i=1,2$
\begin{align}
\label{intGfirstM}
\int_{\R^2} \nabla (U \nabla g) y_i \,dy = 0 .
\end{align}
Indeed,
\begin{align*}
\int_{\R^2} \nabla (U \nabla g) y_i \,dy
&=
-\int_{\R^2} U \nabla g e_i
=\int_{\R^2} g \nabla U e_i .
\end{align*}
But from $g = \frac{\phi}{U} - \psi $, $ \psi = (-\Delta)^{-1}\phi$ we have
\begin{align*}
-\Delta \psi - U \psi = U g = g_0.
\end{align*}
Multiplying this equation by $z_i = \nabla \Gamma_0 e_i$ defined in \eqref{defZLiouville} and integrating we get
\begin{align*}
\int_{\R^2} g U \nabla \Gamma_0 e_i =0 ,
\end{align*}
which is the desired claim \eqref{intGfirstM}.
We note that
\begin{align}
\nonumber
\left\{
\begin{aligned}
-\Delta v &= \nabla U \cdot \nabla g + g \Delta U
= \nabla \cdot ( g \nabla U) \quad \text{in }\R^2.
\\
v(y) & \to 0 \quad \text{as }|y|\to \infty.
\end{aligned}
\right.
\end{align}
Now we can apply Remark~\ref{rem-newtonian} and deduce that for any $\vartheta\in (0,1)$ there is $C$ such that
\begin{align}
\label{estV2nonrad}
|v(y,\tau) | \leq C
\frac{R K}{\tau^\nu (\log \tau)^m (1+|y|)^{2-\vartheta} } .
\end{align}

We next estimate $\tilde h$. From Remark~\ref{rem-newtonian} and the assumptions on $h$, in particular \eqref{center-mass-h}, we have
\begin{align}
\label{estInvLhnonrad}
|((-\Delta)^{-1}h)(y,\tau)|
\leq C
\frac{\|h\| }{\tau^\nu (\log\tau)^m( 1+|y|)^{2-\vartheta}} ,
\end{align}
for any $\vartheta \in (0,1)$.
Also from \eqref{decayTildeG3nonrad} we have
\begin{align}
\nonumber
|U g_0 (y,\tau)|
\leq C
\frac{R }{\tau^\nu (\log\tau)^m (1+|y|)^{6}} K .
\end{align}

Therefore, from \eqref{estInvLhnonrad}, \eqref{decayTildeG3nonrad}, \eqref{estV2nonrad}  we find that for any $\vartheta>0$
\begin{align}
\nonumber
|\tilde h(y,\tau)|
& \leq
C \frac{R K}{\tau^{\nu} (\log \tau)^{m}}
\Bigl[
\frac{  1  }{(1+|y|)^{5+\sigma}}
\min\Bigl( 1 , \frac{\tau^{\epsilon/2}}{\rho^\epsilon} \Bigr)
+\frac{  1  }{(1+|y|)^{6-\vartheta}}
\Bigr] .
\end{align}

We now use a barrier as in the proof of Lemma~\ref{lemma-pointwise2}, in a domain of the form $(\R^2 \setminus B_{R_0} ) \times (\tau_0,\infty)$ where $R_0$ is a large constant.
We let $\tilde g(y)$ be the radial decaying solution to $-\Delta_6 \tilde g = \frac{1}{(1+|y|)^{5+\sigma}}$ and
\begin{align*}
\bar g(y,\tau) = \frac{1}{\tau^\nu (\log \tau)^m} \tilde g(y) \chi_0\Bigl( \frac{y}{\delta \sqrt{\tau}}\Bigr) +
C_1  \frac{1}{\tau^{\nu+\frac{3}{2}+\frac{\sigma}{2}} (\log \tau)^m}
\Bigl[ \frac{1}{(1+|y|/\sqrt{\tau})^{\mu}}
+ C_2 e^{-\frac{|y|^2}{4\tau}} \Bigr]
\end{align*}
where
\begin{align*}
\mu = \min( 5 + \sigma + \epsilon  , 6 -\vartheta ).
\end{align*}
We assume that  $\nu<\frac{3}{2}-\frac{\sigma}{2}-\frac{\vartheta}{2}$, $ \nu < 1+\frac{\epsilon}{2}$,  $\nu + \frac{\sigma}{2} < \frac{3}{2}$, and $\sigma+\vartheta<1$.
Since $\vartheta>0$ is arbitrary we only need $\nu<\frac{3}{2}-\frac{\sigma}{2}$,  $ \nu < 1+\frac{\epsilon}{2}$ and $\sigma<1$.
Then, for an appropriate choice of $C_1$, $C_2$, the function $R K \bar g(y,\tau)$ is a supersolution. By the maximum principle
\begin{align*}
|g_0(y,\tau)|
& \leq
C R K \bar g(y,\tau).
\end{align*}

This proves the desired estimate \eqref{estUdnonrad}.

\end{proof}

Next we consider equation \eqref{linear1}, which we recall,
\begin{align}
\label{eq-linear-nonrad2}
\left\{
\begin{aligned}
\partial_\tau \phi &= L[\phi] + B[\phi] + h \quad\text{in }\R^2 \times (\tau_0,\infty)
\\
\phi(\cdot,\tau_0)&= 0 \quad \text{in }\R^2 .
\end{aligned}
\right.
\end{align}
For $  \phi$ with no radial mode we can write
\[
B[\phi] =
  ( \zeta_1(t) \phi + \zeta_2(t) y \cdot \nabla \phi)  \chi_0\Bigl( \frac{\lambda y}{5 \sqrt t} \Bigr) .
\]
\begin{corollary}
\label{coro-linear-nonradial}
Let $0<\sigma<1$, $0<\epsilon<2$,  $1<\nu<\min( 1 + \frac{\epsilon}{2}, \frac{3}{2}-\frac{\sigma}{2})$, $m\in \R$.
Then there is a $C>0$ such that for any $\tau_0$ sufficiently large the following holds.
Suppose that  $h(y,\tau)$ has no radial mode and satisfies $\|h\|_{\nu,m,5+\sigma,\epsilon}<\infty$,
\begin{align}
\label{center-mass-h2}
\int_{\R^2} h(y,\tau) y_j dy=0\quad \text{for all }\tau>\tau_0, \quad j=1,2.
\end{align}
Then the solution $\phi(y,\tau)$ of \eqref{eq-linear-nonrad2} satisfies
\begin{align}
\label{non-radial2}
|\phi(y,\tau) |
\leq C
\frac{\|h\|_{\nu,m,5+\sigma,\epsilon}}{\tau^\nu (\log\tau)^{m} }
\begin{cases}
\frac{1}{(1+|y|)^{3+\sigma}} , &  |y|\leq \sqrt{\tau} .
\\
\frac{\tau^{1+\frac{\epsilon}{2}}}{|y|^{5+\sigma+\epsilon}} , & |y|\geq \sqrt{\tau} .
\end{cases}
\end{align}
\end{corollary}
\begin{proof}
Using Proposition~\ref{prop-linear-nonradial}, there is a linear operator $T$ so that given $h$ with  $ \|h\|_{\nu,m,5+\sigma,\epsilon}<\infty$, with no radial mode, and satisfying the condition \eqref{center-mass-h2} associates the solution $\phi$ of \eqref{eq-linear-nonradial120}. Then the solution $\phi$ of \eqref{eq-linear-nonrad2} can be written as
\[
\phi = T [B[\phi] + h].
\]
The estimate \eqref{est-non-radial1} implies
\[
\|\phi\|_{\nu,m,3+\sigma,2+\epsilon} \leq
\| B[\phi] + h \|_{\nu,m,5+\sigma,\epsilon}.
\]
Using standard parabolic estimates we also get
\begin{align*}
\|
|y| \nabla \phi
\|_{\nu,m,3+\sigma,2+\epsilon} \leq
\| B[\phi] + h \|_{\nu,m,5+\sigma,\epsilon}.
\end{align*}
Next we observe that
\begin{align*}
\| B[\phi]  \|_{\nu,m,5+\sigma,\epsilon}
\leq
\frac{C}{\log \tau_0}
\| |\phi| + |y||\nabla \phi|  \|_{\nu,m,3+\sigma,2+\epsilon}  .
\end{align*}
Then for $\tau_0$ large we deduce the estimate \eqref{non-radial2}.
\end{proof}


%

We are now in a position to prove Proposition~\ref{prop-linear-with-second-moment} in the general case.

\begin{proof}[Proof of Proposition~\ref{prop-linear-with-second-moment}]
\label{proof-linear-with-second-moment}
We decompose $h = h_0 + h_1$ and $\phi = \phi_0 + \phi_1$ as in \eqref{decomp-h}, \eqref{decomp-phi}.
We apply Proposition~\ref{prop-radial-linear-with-second-moment} to get
\begin{align}
\nonumber
\|  \phi_0 \|_{\nu-\frac 12 ,m+\frac q2,4,2+\sigma+\epsilon} \leq C \|h\|_{\nu,m,6+\sigma,\epsilon}.
\end{align}
To estimate $\phi_1$ we use Corollary~\ref{coro-linear-nonradial}.
First we select $0<\vartheta<1$. Then note that
\begin{align*}
 \|h_1\|_{\nu,m,6-\vartheta,\sigma+\epsilon+\vartheta}
\leq C \|h\|_{\nu,m,6+\sigma,\epsilon}.
\end{align*}
Then by Corollary~\ref{coro-linear-nonradial} we obtain a solution $\phi_1$ of
\eqref{eq-linear-nonrad2} such that
\begin{align*}
\|\phi_1\|_{\nu,m,4-\vartheta,2+\sigma+\epsilon+\vartheta}
\leq C \|h_1\|_{\nu,m,6-\vartheta,\sigma+\epsilon+\vartheta} .
\end{align*}
This implies
\begin{align*}
\|  \phi_1 \|_{\nu-\frac {1}{2} ,m+\frac{q}{2},4,2+\sigma+\epsilon}
& \leq
\|  \phi_1 \|_{\nu-\frac {\vartheta}{2} ,m+\frac{q}{2},4,2+\sigma+\epsilon}
\\
& \leq C \|h\|_{\nu,m,6+\sigma,\epsilon}.
\end{align*}
To apply Corollary~\ref{coro-linear-nonradial} we need $1<\nu<1+\frac{\epsilon}{2}$ and $\nu<1+\frac{\vartheta}{2}$. Given
$1<\nu< \min( 1+\frac{\epsilon}{2},3-\frac{\sigma}{2}, \frac{5}{4})$
we can select $\vartheta \in (0,\frac{1}{2})$ such that $\nu<1+\frac{\vartheta}{2}$ and then proceed.
This concludes the proof.

\end{proof}

\section{The outer problem}
\label{sect-thmOuter}

We consider the linear outer  problem:
\begin{align}
\label{outer1b2}
\left\{
\begin{aligned}
\partial_t \phi^o
&= L^o  [\phi^o] +  g(x,t), \quad \text{in }\R^2 \times(t_0,\infty)
\\
\phi^o(\cdot,t_0)&=0, \quad \text{in }\R^2 .
\end{aligned}
\right.
\end{align}
where
\begin{align}
\nonumber
L^o [\varphi ] :=\Delta_x \varphi- \nabla_x\Big[ \Gamma_0 \Big (\frac{x-\xi(t)}{\lambda(t)}\Big) \Big]\cdot \nabla_x \vp
= \Delta_x \varphi + 4 \frac{(x-\xi)}{|x-\xi|^2+\lambda^2}\cdot \nabla_x\varphi .
\end{align}

For $g:\R^2 \times (t_0,\infty)\to \R $ we consider the norm $\| g \|_{**,o}$ defined as the least $K$ such that for all
$(x,t)\in \R^2\times (t_0,\infty)$
\begin{align}
\nonumber
|g(x,t)|  \le   K\frac 1{(t-t_0+A)^a (\log t)^\beta}  \frac 1 { 1+ |\zeta |^b} , \quad \zeta = \frac{x-\xi(t)}{\sqrt{t-t_0+A}},
\end{align}
where $A>0$ is a constant.

We also define the norm  $\| \phi \|_{*,o}$  as the least $K$ such that
\begin{align}
\nonumber
|\phi^o(x,t)| +  (\la +|x-\xi|) |\nn_x \phi^o (x,t)| \leq    K\frac 1{(t-t_0+A)^{a-1} (\log t)^\beta}  \frac 1 { 1+ |\zeta |^{b}} , \quad \zeta = \frac{x-\xi}{\sqrt{t-t_0+A}}
\end{align}
for all $(x,t)\in \R^2\times (t_0,\infty)$.

We assume that  the parameters $a,b$  satisfy the constraints
\begin{align}
\label{cond2b1}
1<a < 4,\quad  2<b< 6, \quad  a< 1+ \frac b2.
\end{align}
There is no restriction on $\beta$.

We recall from \eqref{conditions} that we are assuming that
\begin{align}
\label{h-lambda-o}
| \dot\lambda(t)| \leq \frac{C}{t ( \log t)^{3/2}} , \quad t>t_0,
\end{align}
and
\begin{align}
\label{xi2}
|\dot \xi(t)|\leq \frac{C}{t^{\frac{3}{2}+\sigma}},\quad t>t_0,
\end{align}
where $0<\sigma<\frac{1}{2}$.

\begin{prop}
\label{thmOuter2}
Assume that  $a,b$ satisfy \eqref{cond2b1}, $\frac{A}{\lambda(t_0)^2}$ is sufficiently large, and  $\lambda, \xi$ satisfy \eqref{h-lambda-o}, \eqref{xi2}.
Then there is a constant $C$ so that for $t_0$ sufficiently large and for $\|g\|_{**,o}<\infty$ there exists a solution  $\phi^o= \mathcal T ^o_{\textbf{p}}[g]$  of \eqref{outer1b2}, which defines a linear operator of $g$ and satisfies
\begin{align*}
\|\phi^o\|_{*,o}
\leq  C
\|g\|_{ **,o} .
\end{align*}
\end{prop}

Proposition~\ref{thmOuter} in Section~\ref{sect-proof-existence} follows from Proposition~\ref{thmOuter2} with $A = t_0$.

\begin{lemma}
\label{lemma-elliptic-1}
Let $2<\beta<6$ and $h(r)$ satisfy
\begin{align}
\label{ah}
|h(r)|\leq \frac{\lambda^{-2}}{(r/\lambda+1)^\beta}
=\frac{\lambda^{\beta-2}}{(r+\lambda)^\beta} ,
\end{align}
where $\lambda>0$.
Then there is a unique bounded radial function $\varphi(r)$ satisfying
\begin{align}
\nonumber
L^o[\varphi] + h = 0 \quad \text{in }\R^2.
\end{align}
Moreover $\varphi$ satisfies
\begin{align}
\label{xc2b}
|\varphi (r) |
+
(\lambda+r) |\partial_r \varphi (r) |
\leq\frac{C}{(1+r/\lambda)^{\beta-2}}
= C \frac{\lambda^{\beta-2}}{(r+\lambda)^{\beta-2}}
\end{align}
\end{lemma}
\begin{proof}
The equation for $\varphi$ is given by
\begin{align}
\nonumber
\partial_{rr} \varphi(r)
+ \Bigl(\frac{1}{r}+ \frac{4r}{\lambda^2 + r^2}
\Bigr)  \partial_{r} \varphi(r)
+ h(r) = 0 , \quad r>0.
\end{align}
We change variables $\rho = \frac{r}{\lambda}$ and let $\varphi(r) = \bar \varphi(\frac{r}{\lambda})$. Then we need to solve
\[
\partial_{\rho\rho} \bar \varphi
+ \Bigl(\frac{1}{\rho}+ \frac{4\rho}{1 + \rho^2}
\Bigr)  \partial_{\rho} \bar \varphi
+ \bar h(\rho) = 0 , \quad \rho>0,
\]
where
\[
\bar h(\rho) = \lambda^2 h(\lambda \rho).
\]
By \eqref{ah}
\[
|\bar h(\rho)|\leq \frac{1}{(1+\rho)^\beta}.
\]
The bounded solution is given by
\[
 \bar \varphi(\rho)
=\int_\rho^\infty \frac{1}{v(1+v^2)^2}
\int_0^v \bar h(s) s (1+s^2)^2 \, d s \, d v.
\]
By direct computation we get
\[
|\bar \varphi (\rho)|
+ (1+\rho) |\partial_\rho \bar \varphi(\rho)|
\leq \frac{C}{(1+\rho)^{\beta-2}} ,
\]
and this implies \eqref{xc2b}.

\end{proof}

\begin{proof}[Proof of Proposition~\ref{thmOuter2}]

To find a pointwise estimate for the solution $\phi^o$ we construct a barrier.

Using polar coordinates $x-\xi(t) = r e^{i\theta}$,
$L^o$  can be written as:
\begin{align*}
L^o [\varphi ] = \partial_{rr} \varphi
+ \Bigl(\frac{1}{r}+ \frac{4r}{\lambda^2 + r^2}
\Bigr)  \partial_{r} \varphi + \frac{1}{r^2}\partial_{\theta\theta}\varphi.
\end{align*}
First we construct a function $\tilde\psi(r,t)$ such that
\[
\Bigl[ \partial_t - \partial_{rr}
- \Bigl(\frac{1}{r}+ \frac{4r}{\lambda^2 + r^2}
\Bigr)  \partial_{r} \bigr]\tilde \psi \geq  \frac{1}{(t-t_0+A)^a (\log t)^\beta } \frac{1}{(1+r/\sqrt{t-t_0+A})^b}.
\]
Let
\[
\psi_1(r,t) =
\frac{1}{(t-t_0+A)^{a-1} ( \log t)^\beta }
\Bigl[ \frac{1}{(1+\frac{r^2}{t-t_0+A})^{b/2}}
+ C_1 e^{-\frac{r^2}{4(t-t_0+A)}}\Bigr] .
\]
Choosing a large constant $C_1$, $\psi_1$ satisfies
\begin{align*}
\partial_t \psi_1
- \partial_{rr} \psi_1 - \frac{5}{r} \partial_r \psi_1
\geq c \frac{1}{(t-t_0+A)^a ( \log t)^\beta } \frac{1}{(1+ \frac{r}{\sqrt{t-t_0+A}})^b},\quad\text{for } r>0, \ t>t_0,
\end{align*}
where $c>0$.
Here we require $a<4$ and $ a  < 1+ \frac{b}{2}$, which are part of the conditions \eqref{cond2b1}.
Then
\begin{align}
\nonumber
\Bigl[ \partial_t - \partial_{rr}
- \Bigl(\frac{1}{r}+ \frac{4r}{\lambda^2 + r^2}
\Bigr)  \partial_{r} \bigr]\psi_1
&= \Bigl[ \partial_t
-\partial_{rr}  - \frac{5}{r} \partial_r \Bigr]\psi_1
+ 4\frac{\lambda^2}{r(r^2+\lambda^2)} \partial_r \psi_1
\\
\label{superPsi1}
& \geq
c \frac{1}{(t-t_0+A)^a ( \log t)^\beta } \frac{1}{(1+ \frac{r}{\sqrt{t-t_0+A}})^b}
 - 4 \frac{ \lambda^2}{r(r^2+\lambda^2)} |\partial_r \psi_1|.
\end{align}
But
\begin{align*}
\partial_r \psi_1 =  \frac{r}{(t-t_0+A)^a (\log t)^\beta} \Bigl[ -\frac{b}{(1+\frac{r^2}{t-t_0+A})^{b/2+1}} - \frac{C_1}{2}  e^{-\frac{r^2}{4(t-t_0+A)}}\Bigr]
\end{align*}
and so
\begin{align}
\label{estDPsi1}
\frac{\lambda^2}{r(r^2+\lambda^2)} |\partial_r \psi_1|
&\leq C \frac{ \lambda^2}{r^2+\lambda^2}  \frac{1}{(t-t_0+A)^a (\log t)^\beta} \frac{1}{(1+\frac{r^2}{t-t_0+A})^{b/2+1}} .
\end{align}

We note that for $r\leq\sqrt{t-t_0+A}$ we have
\begin{align}
\label{diff}
\frac{\lambda^2}{r(r^2+\lambda^2)} |\partial_r \psi_1|
\leq
\frac{ \lambda^2}{r^2+\lambda^2}  \frac{1}{(t-t_0+A)^a (\log t)^\beta}
\leq
C \frac{ \lambda^2}{(r^2+\lambda^2)^2}  \frac{1}{(t-t_0+A)^{a-1} (\log t)^\beta} ,
\end{align}
where we have used that $A\geq \lambda(t)^2$.

Let $\tilde\psi_2(r;\lambda)$ be the bounded solution of
\begin{align*}
- \Bigl[\partial_{rr}
+ \Bigl(\frac{1}{r}+ \frac{4r}{\lambda^2 + r^2}
\Bigr)  \partial_{r} \Bigr]\tilde\psi_2
= \frac{\lambda^2}{(r^2+\lambda^2)^2} , \quad r>0,
\end{align*}
given by Lemma~\ref{lemma-elliptic-1}.
Then $\tilde\psi_2$ can be written as
\begin{align*}
\tilde\psi_2(r;\lambda) = \bar\psi_2\Bigl(\frac{r}{\lambda}\Bigr) ,
\end{align*}
for a function $\bar\psi_2$ satisfying
\begin{align}
\label{bounds-bar-psi2}
|\bar\psi_2(\rho) | + (1+\rho) |\bar\psi_2'(\rho)|\leq \frac{C}{1+\rho^2}.
\end{align}

Let
\[
\psi_2(r,t) = \frac{1}{(t-t_0+A)t^{a-1} (\log t)^\beta} \tilde \psi_2(r;\lambda(t)) .
\]
Then, using \eqref{bounds-bar-psi2} and \eqref{h-lambda-o}, we get
\begin{align*}
& \Bigl[ \partial_t - \partial_{rr}
- \Bigl(\frac{1}{r}+ \frac{4r}{\lambda^2 + r^2}
\Bigr)  \partial_{r} \bigr]\psi_2
\\
& \quad = \frac{1}{(t-t_0+A)^{a-1}(\log t)^\beta}\frac{\lambda^2}{(r^2+\lambda^2)^2}
\Bigl[ 1 - \Bigl( \frac{a-1}{t-t_0+A} + \frac{\beta}{t\log t}\Bigr) \tilde \psi_2(r)\frac{(r^2+\lambda^2)^2}{\lambda^2}
\\
& \qquad\qquad\qquad - \frac{\dot\lambda}{\lambda} \bar \psi_2'\Bigl(\frac{r}{\lambda}\Bigr)\frac{r}{\lambda}\frac{(r^2+\lambda^2)^2}{\lambda^2}\Bigr]
\\
& \quad \geq \frac{1}{(t-t_0+A)^{a-1}(\log t)^\beta}\frac{\lambda^2}{(r^2+\lambda^2)^2}
\Bigl[ 1 - C \frac{r^2+\lambda^2}{t-t_0+A}\Bigr] .
\end{align*}
Therefore there is $\delta>0$ (fixed independent of $t_0$) such that for all $t_0$ large,
\begin{align}
\label{superPsi2}
\Bigl[ \partial_t - \partial_{rr}
- \Bigl(\frac{1}{r}+ \frac{4r}{\lambda^2 + r^2}
\Bigr)  \partial_{r} \bigr]\psi_2
& \geq \frac{1}{2} \frac{1}{(t-t_0+A)^{a-1}(\log t)^\beta}\frac{\lambda^2}{(r^2+\lambda^2)^2}, \quad \text{for } r\leq 2\delta \sqrt t.
\end{align}

Let $\chi_0 \in C^\infty(\R)$ be such that $\chi_0(s) = 1 $ if $s\leq 1$ and $\chi_0(s)=0$ if $s\geq 2 $ and define
\begin{align}
\nonumber
\chi_{\delta} (r,t) = \chi_0\Bigl(\frac{r}{\delta\sqrt{t-t_0+A}}\Bigr) .
\end{align}

We consider
\begin{align*}
\tilde \psi= \psi_1 + M \psi_2 \chi_{\delta} ,
\end{align*}
where $M>0$ is a constant to be fixed later.
We compute, using \eqref{superPsi1}
\begin{align}
\nonumber
\Bigl[ \partial_t - \partial_{rr}
- \Bigl(\frac{1}{r}+ \frac{4r}{\lambda^2 + r^2}
\Bigr)  \partial_{r} \bigr]\tilde \psi
& \geq
 c \frac{1}{(t-t_0+A)^a ( \log t)^\beta (1+\frac{r}{\sqrt{t-t_0+A}})^b}
 - 4 \frac{ \lambda^2}{r(r^2+\lambda^2)} |\partial_r \psi_1|
 \\
\label{superPsi-a}
 & \quad
 + M \chi_{\delta} \Bigl[ \partial_t - \partial_{rr}
- \Bigl(\frac{1}{r}+ \frac{4r}{\lambda^2 + r^2}
\Bigr)  \partial_{r} \Bigr]\tilde \psi_2
+ R(r,t),
\end{align}
where
\begin{align*}
R &= M \Bigl[ \psi_2 \partial_t \chi_{\delta} -2 \partial_r \tilde\psi_2 \partial_r \chi_{\delta} - \tilde\psi_2\Bigl(
\partial_{rr}\chi_{\delta} + \frac{1}{r}\partial_r \chi_{\delta}+4\frac{r}{r^2+\lambda^2}\partial_r \chi_{\delta}
\Bigr)\Bigr] .
\end{align*}
We have, by \eqref{bounds-bar-psi2},
\begin{align}
\label{boundR}
|R(r,t)|\leq C_2 M   \lambda^2 \frac{1}{(t-t_0+A)^{a+1} (\log t)^\beta} ,
\end{align}
where $C_2$ is independent of $M$ (although it depends on $\delta$),
and is supported on $\delta \sqrt{t-t_0+A} \leq r \leq 2 \delta \sqrt{t-t_0+A}$.

We claim that there is $M>0$ and  $\tilde c>0$ so that for all $t_0$ sufficiently large
\begin{align}
\label{super-x}
\Bigl[ \partial_t - \partial_{rr}
- \Bigl(\frac{1}{r}+ \frac{4r}{\lambda^2 + r^2}
\Bigr)  \partial_{r} \Bigr]\tilde \psi
& \geq \tilde c \frac{1}{t^a ( \log t)^\beta (1+r/\sqrt t)^b} ,
\end{align}
for all $r>0$, $t>t_0$.

Indeed, if  $r\leq \delta \sqrt{t-t_0+A}$, then from \eqref{superPsi-a}, \eqref{superPsi1}, \eqref{superPsi2} and \eqref{diff} we get
\begin{align}
\nonumber
\Bigl[ \partial_t - \partial_{rr}
- \Bigl(\frac{1}{r}+ \frac{4r}{\lambda^2 + r^2}
\Bigr)  \partial_{r} \bigr]\tilde \psi
& \geq
 c \frac{1}{(t-t_0+A)^a ( \log t)^\beta (1+\frac{r}{\sqrt{t-t_0+A}})^b}
\\
\nonumber
& \quad
 - C \frac{ \lambda^2}{(r^2+\lambda^2)}
\frac{1}{(t-t_0+A)^{a-1} (\log t)^\beta}
\\
\nonumber
 & \quad
 + M \chi_{\delta} \Bigl[ \partial_t - \partial_{rr}
- \Bigl(\frac{1}{r}+ \frac{4r}{\lambda^2 + r^2}
\Bigr)  \partial_{r} \Bigr]\tilde \psi_2
\\
\label{reg1}
&\geq  c \frac{1}{(t-t_0+A)^a ( \log t)^\beta (1+\frac{r}{\sqrt{t-t_0+A}})^b} ,
\end{align}
if $M\geq C$. Here we fix $M=C$.

If $ \delta \sqrt{t-t_0+A} \leq r\leq 2\delta \sqrt{t-t_0+A}$,  then by \eqref{superPsi-a}, \eqref{superPsi1}, \eqref{diff} and \eqref{boundR} we get
\begin{align*}
\Bigl[ \partial_t - \partial_{rr}
- \Bigl(\frac{1}{r}+ \frac{4r}{\lambda^2 + r^2}
\Bigr)  \partial_{r} \bigr]\tilde \psi
& \geq
c \frac{1}{(t-t_0+A)^a ( \log t)^\beta (1+\frac{r}{\sqrt{t-t_0+A}})^b}
- C_2 M \lambda^2 \frac{1}{(t-t_0+A)^{a+1} (\log t)^\beta}
\\
&=
\frac{1}{(t-t_0+A)^a ( \log t)^\beta }
\Bigl( \frac	{c}{3^ b} - \frac{C_2 M \lambda^2}{t-t_0+A} \Bigr)
\end{align*}
By taking $\frac{A}{\lambda(t_0)^2}$ large, we get
\begin{align}
\label{reg2}
\Bigl[ \partial_t - \partial_{rr}
- \Bigl(\frac{1}{r}+ \frac{4r}{\lambda^2 + r^2}
\Bigr)  \partial_{r} \bigr]\tilde \psi
& \geq
\frac{c}{2} \frac{1}{(t-t_0+A)^a ( \log t)^\beta (1+ \frac{r}{\sqrt{t-t_0+A}})^b} ,
\end{align}
for $ \delta \sqrt{t-t_0+A} \leq r\leq 2\delta \sqrt{t-t_0+A}$.

If $r\geq 2\delta \sqrt{t-t_0+A}$, by \eqref{superPsi-a} and \eqref{estDPsi1}
\begin{align}
\nonumber
\Bigl[ \partial_t - \partial_{rr}
- \Bigl(\frac{1}{r}+ \frac{4r}{\lambda^2 + r^2}
\Bigr)  \partial_{r} \bigr]\psi_1
& \geq
 c \frac{1}{(t-t_0+A)^a ( \log t)^\beta (1+\frac{r}{\sqrt{t-t_0+A}})^b}
 \\
 \nonumber
&\quad -C \frac{ \lambda^2}{r^2+\lambda^2}  \frac{1}{(t-t_0+A)^a (\log t)^\beta} \frac{1}{(1+\frac{r^2}{t-t_0+A})^{b/2+1}}
\\
\nonumber
& \geq
\frac{1}{(t-t_0+A)^a ( \log t)^\beta (1+\frac{r}{\sqrt{t-t_0+A}})^b}
\Bigl[c - C \frac{\lambda^2}{t-t_0+A}
\Bigr]
\\
\label{reg3}
& \geq
\frac{c}{2}
\frac{1}{(t-t_0+A)^a ( \log t)^\beta (1+\frac{r}{\sqrt{t-t_0+A}})^b}
\end{align}
if $\frac{A}{\lambda(t_0)^2}$ is sufficiently large.

Combining \eqref{reg1}, \eqref{reg2} and \eqref{reg3} we deduce the estimate \eqref{super-x}.

Let
\[
\psi(x,t) = \tilde \psi(|x-\xi|,t).
\]
Then by \eqref{super-x}
\begin{align*}
(\partial_t - L^ o)[\psi]
&=
\Bigl[ \partial_t - \partial_{rr}
- \Bigl(\frac{1}{r}+ \frac{4r}{\lambda^2 + r^2}
\Bigr)  \partial_{r} \Bigr]\tilde \psi
-\partial_r \tilde\psi \frac{(x-\xi)\cdot\dot\xi}{|x-\xi|}
\\
& \geq
\tilde c \frac{1}{(t-t_0+A)^a ( \log t)^\beta (1+\frac{r}{\sqrt{t-t_0+A}})^b}
- |\dot\xi| \, |\partial_r \tilde\psi|.
\end{align*}
But
\begin{align*}
|\partial_r \psi|
& \leq C\frac{1}{(t-t_0+A)^{a-1/2}(\log t)^\beta} \frac{1}{(1+ \frac{r}{\sqrt{t-t_0+A}})^{b+1}}
+ C\frac{1}{(t-t_0+A)^{a-1}(\log t)^\beta} \frac{1}{\lambda} \frac{1}{(1+r/\lambda)^3}
\chi_{\delta}(r,t)
\\
& \quad + C\frac{1}{\delta (t-t_0+A)^{a-1/2}(\log t)^\beta} \frac{1}{(1+r/\lambda)^2}
\chi_0'\Bigl( \frac{r}{\delta \sqrt{t-t_0+A}} \Bigr) .
\end{align*}
Using \eqref{xi2} we see that if $t_0$ is sufficiently large,
\begin{align*}
(\partial_t - L^ o)[\psi]
& \geq
\frac{\tilde c}{2} \frac{1}{(t-t_0+A)^a ( \log t)^\beta (1+\frac{r}{\sqrt{t-t_0+A}})^b} .
\end{align*}
\end{proof}

A direct consequence of the proof of Proposition~\ref{thmOuter2} (using the same barriers) is the following, for the initial value problem
\begin{align}
\label{outer1c}
\left\{
\begin{aligned}
\partial_t \phi^o
&= L^o  [\phi^o] , \quad \text{in }\R^2 \times(t_0,\infty)
\\
\phi^o(\cdot,t_0)&= \phi^o_0, \quad \text{in }\R^2 .
\end{aligned}
\right.
\end{align}
Consider the norm
\begin{align*}
\| \phi^o_0 \|_{*,b} =& \inf K \quad \text{such that }
\\
& |\phi^o_0(x)| \leq \frac{K}{( 1+ \frac{|x-\xi(0)|}{\sqrt{t-t_0+A}} )^b}
\end{align*}
where $b \in (2,6)$, $A>0$.

\begin{prop}
Assume that  $a,b$ satisfy \eqref{cond2b1}, $\frac{A}{\lambda(t_0)^2}$ is sufficiently large, and  $\lambda, \xi$ satisfy \eqref{h-lambda-o}, \eqref{xi2}.
Then there is a constant $C$ so that for $t_0$ sufficiently large and for $\| \phi^o_0 \|_{*,b}<\infty$ there exists a solution  $\phi^o$  of \eqref{outer1c}, which defines a linear operator of $ \phi^o_0$ and satisfies
\begin{align*}
\|\phi^o\|_{*,o}
\leq  C A^{a-1} (\log t_0)^\beta
\| \phi^o_0 \|_{*,b} .
\end{align*}

\end{prop}

\subsection*{Acknowledgments}
	J.~D\'avila has been supported  by  a Royal Society  Wolfson Fellowship, UK and  Fondecyt grant 1170224, Chile.
	M.~del Pino has been supported by a Royal Society Research Professorship, UK.
	M. Musso has been supported by EPSRC research Grant EP/T008458/1. The  research  of J.~Wei is partially supported by NSERC of Canada.
We are grateful to Federico Buseghin for many valuable comments and corrections.

\bibliographystyle{siam}
\bibliography{KS}

\blfootnote{Data sharing not applicable.}

\end{document}